\documentclass{amsbook}
\usepackage{epsf}
\usepackage[all]{xy}
\usepackage{enumerate}
\usepackage{amssymb}
\usepackage{color}
\usepackage{graphicx}

\newcommand{\ie}{{\em i.e.\ }}
\newcommand{\eg}{{\em e.g.\ }}
\newcommand{\cf}{{\em cf.\ }}

\newtheorem{theorem}{Theorem}[section]
\newtheorem{lemma}[theorem]{Lemma}
\newtheorem{proposition}[theorem]{Proposition}
\newtheorem{corollary}[theorem]{Corollary}
\newtheorem*{thm}{Theorem}
\newtheorem*{cor}{Corollary}
\newtheorem*{prop}{Proposition}
\newtheorem*{defin}{Definition}
\newtheorem*{notat}{Notation}
\newtheorem*{ex}{Example}
\newtheorem*{teorema}{Teorema}
\newtheorem*{corolario}{Corolario}
\newtheorem*{proposicion}{Proposici\'{o}n}
\newtheorem*{lema}{Lema}
\newtheorem*{definicion}{Definici\'on}
\newtheorem*{notacion}{Notaci\'on}
\newtheorem*{ej}{Ejemplo}

\theoremstyle{definition}
\newtheorem{definition}[theorem]{Definition}

\newtheorem{notation}[theorem]{Notation}
\newtheorem{example}[theorem]{Example}
\newtheorem*{question}{\textbf{Question}}

\newtheorem{remark}[theorem]{Remark}

\numberwithin{equation}{section}

\newcommand{\internalcomment}[1]{}

\newcommand{\N}{\mathbf{N}}
\newcommand{\Z}{\mathbf{Z}}

\newcommand{\ko}{\: , \;}

\newcommand{\ol}{\overline}
\newcommand{\ul}{\underline}
\newcommand{\we}{\wedge}
\newcommand{\che}{\vee}
\renewcommand{\tilde}[1]{\widetilde{#1}}

\newcommand{\ra}{\rightarrow}

\newcommand{\arr}[1]{\stackrel{#1}{\rightarrow}}

\newcommand{\opname}[1]{\operatorname{\mathsf{#1}}}

\renewcommand{\mod}{\opname{mod}\nolimits}
\newcommand{\Mod}{\opname{Mod}\nolimits}
\newcommand{\Sub}{{\mbox{Sub}}}
\newcommand{\Gen}{{\mbox{Gen}}}
\newcommand{\Sum}{{\mbox{Sum}}}
\newcommand{\Tria}{{\mbox{Tria}}}
\newcommand{\Susp}{{\mbox{Susp}}}
\newcommand{\aisle}{{\mbox{aisle}}}
\newcommand{\add}{{\mbox{add}}}
\newcommand{\ann}{{\mbox{ann}}}
\newcommand{\lann}{{\mbox{lann}}}

\newcommand{\proj}{\opname{proj}\nolimits}
\newcommand{\Proj}{\opname{Proj}\nolimits}

\newcommand{\id}{\mathbf{1}}
\newcommand{\R}{\mathbf{R}}
\renewcommand{\L}{\mathbf{L}}

\newcommand{\ten}{\otimes}

\newcommand{\cok}{\opname{cok}\nolimits}
\newcommand{\im}{\opname{im}\nolimits}
\renewcommand{\ker}{\opname{ker}\nolimits}

\newcommand{\colim}{\varinjlim}
\newcommand{\lid}{\varinjlim}

\newcommand{\can}{\opname{can}}
\newcommand{\Mcolim}{\opname{Mcolim}}

\newcommand{\op}[1]{\opname{#1}\nolimits}
\newcommand{\Zy}[1]{\op{Z}^{#1}}
\newcommand{\Bo}[1]{\op{B}^{#1} \,}
\renewcommand{\H}[1]{{H}^{#1}}

\newcommand{\ca}{{\mathcal A}}
\newcommand{\cb}{{\mathcal B}}
\newcommand{\cc}{{\mathcal C}}
\newcommand{\cd}{{\mathcal D}}
\newcommand{\ce}{{\mathcal E}}
\newcommand{\cF}{{\mathcal F}}
\newcommand{\cg}{{\mathcal G}}
\newcommand{\ch}{{\mathcal H}}
\newcommand{\ci}{{\mathcal I}}

\newcommand{\ck}{{\mathcal K}}

\newcommand{\cm}{{\mathcal M}}
\newcommand{\cn}{{\mathcal N}}

\newcommand{\cp}{{\mathcal P}}
\newcommand{\cR}{{\mathcal R}}
\newcommand{\cq}{{\mathcal Q}}
\newcommand{\cs}{{\mathcal S}}
\newcommand{\ct}{{\mathcal T}}
\newcommand{\cu}{{\mathcal U}}
\newcommand{\cv}{{\mathcal V}}

\newcommand{\cx}{{\mathcal X}}
\newcommand{\cy}{{\mathcal Y}}
\newcommand{\cz}{{\mathcal Z}}
\newcommand{\eps}{\varepsilon}

\renewcommand{\phi}{\varphi}

\newcommand{\Hom}{\opname{Hom}}
\newcommand{\End}{\opname{End}}

\newcommand{\Ext}{\opname{Ext}}

\newcommand{\cone}{\opname{Cone}\nolimits}

\newcommand{\Fun}{\opname{Fun}}

\newcommand{\per}{\opname{per}}

\makeindex
\begin{document}
\frontmatter


\begin{titlepage}

\begin{center}
\Huge{UNIVERSIDAD DE MURCIA}
\end{center}
\bigskip

\begin{center}
\huge{DEPARTAMENTO DE MATEM\'ATICAS}
\end{center}
\vspace*{1cm}

\begin{center}
\huge{TESIS DOCTORAL}
\end{center}

\begin{center}
\Huge{\textbf{On torsion torsionfree triples}}
\end{center}
\bigskip

\begin{center}
\huge{D. Pedro Nicol\'as Zaragoza}
\end{center}
\vspace*{1cm}

\begin{center}
\huge{2007}
\end{center}
\end{titlepage}



\begin{titlepage}
\textbf{Agradecimientos}
\bigskip

Quiero expresar mi agradecimiento a Manuel Saor\'{i}n Casta\~{n}o, cuya presencia y ayuda han sido determinantes en la realizaci\'{o}n de este trabajo.

Le estoy agradecido a K. Fuller por apuntarnos algunas referencias utilizadas en el Ejemplo \ref{balancedMorita}, a K. Goodearl por comunicarnos algunos resultados claves de \cite{Goodearl, Small1967} utilizados en el Cap\'itulo \ref{Split TTF triples on module categories}, y a J. Trlifaj por sugerirnos el Lema \ref{TrlifajLemma} y el Ejemplo \ref{TrlifajExample}. Tambi\'en le estoy agradecido a B.~Keller por su atenta lectura del manuscrito de \cite{NicolasSaorin2007c}, por diversas conversaciones e intercambio de correos electr\'onicos sobre categor\'ias dg, por su supervisi\'on de \cite{Nicolas2007b} y por su amabil\'isima hospitalidad durante mis estancias en Par\'is.
\vspace{3cm}

\textbf{Acknowledgements}
\bigskip

I would like to express my gratitude to my advisor, Manuel Saor\'{i}n Casta\~{n}o, whose presence and help have been decisive to carry out this work.

I am grateful to K. Fuller for pointing out some references used in Example \ref{balancedMorita}, to K. Goodearl for telling us about key results in \cite{Goodearl, Small1967} used in Chapter \ref{Split TTF triples on module categories}, and to J. Trlifaj for suggesting Lemma \ref{TrlifajLemma} and Example \ref{TrlifajExample}. Also, I am grateful to B.~Keller for carefully reading the manuscript of \cite{NicolasSaorin2007c}, for several conversations and e-mail exchange concerning dg categories, for his supervision of \cite{Nicolas2007b} and for his extremely kind hospitality during my stays in Paris.
\end{titlepage}


\begin{titlepage}
\end{titlepage}


\begin{titlepage}
\rightline{\ }
\bigskip
\bigskip
\bigskip
\bigskip
\bigskip
\bigskip
\bigskip
\bigskip
\bigskip
\bigskip
\bigskip
\bigskip
\bigskip
\bigskip
\bigskip
\bigskip

\rightline{\Large{A mi familia y amigos.}}
\bigskip

\rightline{\Large{A Patri.}}
\end{titlepage}


\begin{titlepage}
\end{titlepage}


\tableofcontents


\begin{titlepage}
\end{titlepage}


\chapter*{Introducci\'{o}n (Spanish)}
\addcontentsline{lot}{chapter}{Introducci\'{o}n}
\bigskip

\textbf{Motivaci\'on}
\bigskip

En su trabajo \cite{Dickson1966}, S. E. Dickson defini\'o la noci\'on de \emph{teor\'ia de torsi\'on} (ahora llamada \emph{par de torsi\'on}) en el marco general de las categor\'ias abelianas, que generaliza el concepto de `torsi\'on' que aparece en la teor\'ia de los grupos abelianos. Un \emph{par de torsi\'on} en una categor\'ia abeliana $\ca$ consiste en un par $(\cx,\cy)$ de subcategor\'ias plenas de $\ca$ tal que:
\begin{enumerate}[1)]
\item $\cx\cap\cy=\{0\}$,
\item $\cx$ es cerrada para cocientes e $\cy$ es cerrada para subobjetos,
\item cada objeto $M$ de $\ca$ aparece en una sucesi\'on exacta corta
\[0\ra M_{\cx}\ra M\ra M^{\cy}\ra 0
\]
en la cual $M_{\cx}$ pertenece a $\cx$ y $M^{\cy}$ pertenece a $\cy$.
\end{enumerate}
Cuando esta sucesi\'on exacta corta escinde para cada $M$, decimos que el par de torsi\'on \emph{escinde}. Desde el trabajo de S. E. Dickson, los pares de torsi\'on han desempe\~{n}ado una funci\'on importante en \'algebra: han sido una herramienta fundamental para una teor\'ia general de localizaci\'on no conmutativa \cite{Stenstrom}, han tenido una gran influencia en la teor\'ia de representaci\'on de \'algebras de Artin \cite{HappelRingel1982, HappelReitenSmalo1996, AssemSaorin2005},\dots Uno de los conceptos importantes relacionados con la teor\'ia de torsi\'on es el de \emph{teor\'ia de torsi\'on y libre de torsi\'on} (ahora llamada \emph{terna TTF}), introducida por J. P. Jans en \cite{Jans}. Consiste en una terna $(\cx,\cy,\cz)$ de subcategor\'ias plenas de una categor\'ia abeliana tal que ambos pares $(\cx,\cy)$ e $(\cy,\cz)$ son de torsi\'on. Decimos que una terna TTF $(\cx,\cy,\cz)$ es \emph{escindida por la izquierda (respectivamente, derecha)} cuando el par $(\cx,\cy)$ (respectivamente, $(\cy,\cz)$) es escindido. Una terna TTF es \emph{centralmente escindida} si es escindida por la izquierda y por la derecha. Cuando la categor\'ia abeliana ambiente es una categor\'ia de m\'odulos (unitarios) $\Mod A$ sobre un anillo $A$ (asociativo, con unidad) entonces tenemos el siguiente resultado de J. P. Jans \cite[Corollary 2.2]{Jans}:

\begin{teorema}
Las ternas TTF en $\Mod A$ est\'an en biyecci\'on con los ideales bi\-l\'a\-te\-ros idempotentes de $A$. M\'as a\'un, esta biyecci\'on induce una biyecci\'on entre las ternas TTF en $\Mod A$ centralmente escindidas e (ideales bil\'ateros generados por) idempotentes centrales de $A$.
\end{teorema}

Surge entonces una pregunta natural:
\bigskip

\noindent \emph{Pregunta 1:} ?`Cu\'ales son los ideales bil\'ateros de $A$ correspondientes a las ternas TTF en $\Mod A$ que escinden por la izquierda o por la derecha?
\bigskip

En su trabajo \cite{BeilinsonBernsteinDeligne} sobre haces perversos, A.~A.~Beilinson, J.~Bernstein y P.~Deligne definieron la noci\'on de \emph{t-estructura}
en una categor\'ia triangulada, que formaliza la noci\'on de `truncaci\'on (inteligente)' de un complejo de cocadena sobre una categor\'ia abeliana. Una \emph{t-estructura} en una categor\'ia triangulada $(\cd,?[1])$ consiste en un par $(\cx,\cy)$ de subcategor\'ias estrictamente(=cerradas para isomorfismos) plenas de $\cd$ que satisfacen:
\begin{enumerate}[1)]
\item $\cd(X,Y)=\{0\}$ para todo $X\in\cx$ e $Y\in\cy$,
\item $\cx$ es cerrada para traslaciones positivas e $\cy$ es cerrada para traslaciones negativas,
\item cada objeto $M$ de $\ca$ aparece en un tri\'angulo
\[ M_{\cx}\ra M\ra M^{\cy}\ra M_{\cx}[1].
\]
\end{enumerate}
Ellos probaron que la teor\'ia general de las t-estructuras es lo suficientemente sofisticada como para aprehender uno de los fen\'omenos m\'as importantes relacionados con la truncaci\'on: la t-estructura $(\cx,\cy)$ define una categor\'ia abeliana $\cx\cap(\cy[1])$ dentro de la categor\'ia triangulada $\cd$ y vinculada a ella mediante un funtor cohomol\'ogico $\cd\ra \cx\cap(\cy[1])$. Para enfatizar este hecho, uno normalmente escribe $(\cx,\cy[1])$ para referirse a la t-estructura $(\cx,\cy)$. Las t-estructuras han dado lugar a una cantidad considerable de resultados y son importantes en muchas ramas de las matem\'aticas: teor\'ia de haces perversos \cite{BeilinsonBernsteinDeligne, KashiwaraShapira}, relaci\'on entre la teor\'ia de la dualidad de Grothendieck-Roos y la teor\'ia inclinante \cite{KellerVossieck88a}, la teor\'ia inclinante misma \cite{KellerVossieck88a, HappelReitenSmalo1996}, cohomolog\'ia mot\'ivica y teor\'ia de la $\mathbb{A}^{1}$-homotop\'ia en el sentido de V. Voe\-vodsky \cite{Voevodsky2000, Morel2003}, \'algebra conmutativa \cite{Neeman1992a, AlonsoJeremiasSaorin2007}, \dots Como se puede ver, la noci\'on de t-estructura es el an\'alogo formal de la noci\'on de par de torsi\'on para ca\-te\-go\-r\'i\-as abelianas. Consecuentemente, el an\'alogo `triangulado' de una terna TTF, llamado \emph{terna TTF triangulada}, consiste en una terna $(\cx,\cy,\cz)$ tal que los pares $(\cx,\cy)$ e $(\cy,\cz)$ son t-estructuras. Aunque una terna TTF triangulada est\'a formada por t-estructuras, no nos provee de categor\'ias abelianas dentro de la categor\'ia triangulada ambiente. A cambio, la filosof\'ia general de las ternas TTF trianguladas es que nos permiten ver una categor\'ia triangulada como aglutinaci\'on de otras dos categor\'ias trianguladas. M\'as precisamente, las ternas TTF trianguladas en una categor\'ia triangulada $\cd$ est\'an en biyecci\'on con las maneras de expresar $\cd$ como \emph{recollement} o \emph{aglutinaci\'on} (una noci\'on definida por A.~A.~Beilinson, 
J.~Bernstein y P.~Deligne en \cite{BeilinsonBernsteinDeligne}) de dos categor\'ias trianguladas. En este sentido, las ternas TTF trianguladas se han independizado de las t-estructuras y establecido por su cuenta. Aparecen, con el nombre de ``recollement'' en: el trabajo de E. Cline, B. Parshall y L. L. Scott \cite{ClineParshallScott1988, ParshallScott1988}, la teor\'ia de las \'algebras de diagramas \cite{MartinGreenParker2006},...

Inspirado por el trabajo de E.~Cline, B.~Parshall y L.~L.~Scott y utilizando el trabajo de J.~Rickard \cite{Rickard1989}, S.~K\"{o}nig dio en \cite[Theorem 1]{Konig1991} condiciones necesarias y suficientes para la existencia de aglutinaciones. Precisamente, demostr\'o lo siguiente:

\begin{teorema}
La categor\'ia derivada acotada por la derecha $\cd^-A$ de un anillo $A$ es una aglutinaci\'on de las categor\'ias derivadas acotadas por la derecha de dos anillos $B$ y $C$ si y s\'olo si existen dos complejos $P$ y $Q$ en $\cd^-A$ cuasi-isomorfos a complejos acotados de $A$-m\'odulos proyectivos tales que:
\begin{enumerate}[1)]
\item $P$ es compacto en $\cd A$ (\ie cuasi-isomorfo a un complejo acotado de $A$-m\'odulos proyectivos finitamente generados) y $(\cd A)(Q,?)$ conmuta con coproductos peque\-\~{n}os de copias de $Q$.
\item $(\cd A)(P,P[n])=0$ y $(\cd A)(Q,Q[n])=0$ para todo $n\in\Z$.
\item Existen isomorfismos de anillos $(\cd A)(P,P)\cong C$ y $(\cd A)(Q,Q)\cong B$.
\item $(\cd A)(P[n],Q)=0$ para todo $n\in\Z$.
\item Si $M$ es un complejo de $\cd^-A$ tal que $(\cd A)(P[n],M)=0$ y $(\cd A)(Q[n],M)=0$ para todo $n\in\Z$, entonces $M=0$.
\end{enumerate}
\end{teorema}

S.~K\"{o}nig demostr\'o su teorema alrededor de 1990, y en aquel momento todav\'ia no se ten\'ia a disposici\'on una teor\'ia de Morita para categor\'ias derivadas no acotadas. Algunos a\~{n}os m\'as tarde, B.~Keller desarroll\'o en \cite{Keller1994a} una teor\'ia de Morita (y de Koszul) para categor\'ias derivadas no acotadas de categor\'ias dg. As\'i, surge una pregunta natural:
\bigskip

\noindent \emph{Pregunta 2:} ?`Podemos utilizar la teor\'ia de B. Keller para parametrizar todas las maneras de expresar la categor\'ia derivada no acotada $\cd A$ de un anillo $A$ como una aglutinaci\'on de categor\'ias derivadas no acotadas de otros dos anillos? 
\bigskip

Si la respuesta es afirmativa, 
\bigskip

\noindent \emph{Pregunta 3:} ?`podemos desarrollar una aproximaci\'on `no acotada' al trabajo de S.~K\"{o}nig? Es decir, ?`podemos valernos del estudio de las aglutinaciones a nivel no acotado para comprender las aglutinaciones a nivel acotado por la derecha?
\bigskip

Seg\'un \cite{BokstedtNeeman1993}, la noci\'on de \emph{subcategor\'ia aplastante} de una categor\'ia triangulada $\cd$ se origin\'o en el trabajo de D. Ravenel. Una subcategor\'ia plena $\cx$ de una categor\'ia triangulada $\cd$ con coproductos peque\~{n}os es una \emph{subcategor\'ia aplastante} si existe una t-estructura $(\cx,\cy)$ tal que $\cy$ es cerrada para coproductos peque\~{n}os. Si la ca\-te\-go\-r\'ia triangulada ambiente es, por ejemplo, compactamente generada, entonces las subcategor\'ias aplastantes est\'an en biyecci\'on con las ternas TTF trianguladas. La \emph{conjetura aplastante generalizada} es una generalizaci\'on a categor\'ias compactamente generadas arbitrarias de una conjetura debida a D.~Ravenel \cite[1.33]{Ravenel1984} y, originalmente, a A.~K.~Bousfield \cite[3.4]{Bousfield1979}. Predice lo siguiente:
\bigskip

\noindent Si $\cx$ es una subcategor\'ia aplastante de una categor\'ia triangulada compactamente generada $\cd$, entonces existe un conjunto $\cp$ de objetos compactos de $\cd$ tal que $\cx$ es la menor subcategor\'ia triangulada plena de $\cd$ cerrada para coproductos peque\~{n}os que contiene a $\cp$.
\bigskip

B. Keller dio en \cite{Keller1994} un contraejemplo a esta conjetura. No obstante, a veces resultados importantes demostrados para objetos compactos siguen siendo verdaderos (o admiten bonitas generalizaciones) despu\'es de sustituir  ``compacto'' por ``perfecto'', ``superperfecto'' o algo parecido. As\'i es que uno podr\'ia preguntarse lo siguiente:
\bigskip

\noindent \emph{Pregunta 4:} ?`Resulta la conjetura cierta despu\'es de sustituir ``compacto'' por algo parecido?
\bigskip

En su aproximaci\'on algebraica a la conjetura aplastante generalizada \cite{Krause2000, Krause2005}, H.~Krause demostr\'o lo siguiente:

\begin{teorema}
Sea $\cd$ una categor\'ia triangulada compactamente generada. Las subcategor\'ias aplastantes de $\cd$ est\'an en biyecci\'on con los ideales bil\'ateros idempotentes saturados de $\cd^c$(=la subcategor\'ia de $\cd$ formada por los objetos compactos) cerrados para traslaciones en ambos sentidos.
\end{teorema}

Aqu\'i \emph{saturado} significa que siempre que exista un tri\'angulo
\[P'\arr{u} P\arr{v} P''\ra P'[1]
\]
en $\cd^c$ y un morfismo $f\in\cd^c(P,Q)$ con $fu\ko v\in\ci$, entonces $f\in\ci$.

N\'otese la analog\'ia con el teorema de J. P. Jans de m\'as arriba. Esta analog\'ia no es extra\~{n}a ya que las categor\'ias de m\'odulos son categor\'ias abelianas `compactamente generadas'. Sin embargo, varios hechos sugieren que el verdadero an\'alogo `triangulado' de las categor\'ias de m\'odulos no son simplemente las categor\'ias trianguladas compactamente generadas sino m\'as bien las categor\'ias trianguladas compactamente generadas \emph{algebraicas}. \'Estas, gracias al teorema de B. Keller \cite[Theorem 4.3]{Keller1994a}, son precisamente las categor\'ias derivadas de categor\'ias dg.
\bigskip

\noindent \emph{Pregunta 5:} ?`Podemos parametrizar todas las ternas TTF trianguladas en (equivalentemente, subcategor\'ias aplastantes de) la categor\'ia derivada $\cd\ca$ de una categor\'ia dg $\ca$ en t\'erminos de la propia $\ca$ o, al menos, en `t\'erminos dg'? 
\bigskip

Si la respuesta es positiva, 
\bigskip

\noindent \emph{Pregunta 6:} ?`cu\'al es la relaci\'on entre esta parametrizaci\'on y la de H. Krause?
\bigskip

Las categor\'ias pretrianguladas (en el sentido de A. Beligiannis \cite{Beligiannis2001}) son una generalizaci\'on de las categor\'ias abelianas y las trianguladas. En las categor\'ias pretrianguladas uno puede definir la noci\'on de \emph{par de torsi\'on pretriangulado}, que es una generalizaci\'on de las nociones de par de torsi\'on y de t-estructura. Los pares de torsi\'on pretriangulados (y as\'i los pares de torsi\'on y las t-estructuras) tienen muchas propiedades b\'asicas que se deber\'ian deducir de una teor\'ia de torsi\'on en categor\'ias aditivas. M\'as a\'un, varios hechos sugieren que la categor\'ia aditiva subyacente es esencialmente suficiente para comprender las ternas TTF centralmente escindidas en una categor\'ia abeliana o ternas TTF trianguladas centralmente escindidas (la escisi\'on en este caso se define como lo hicimos para categor\'ias abelianas). Para formalizar esta observaci\'on uno deber\'ia intentar responder a la siguiente pregunta:
\bigskip

\noindent \emph{Pregunta 7:} ?`Podemos definir una teor\'ia de torsi\'on para categor\'ias aditivas de modo que las ternas TTF escindias en categor\'ias m\'as ricas puedan ser esencialmente comprendidas simplemente mirando a las ternas TTF aditivas subyacentes?
\bigskip

Al responder a esta pregunta uno se da cuenta de que la propiedad de ser una categor\'ia \emph{equilibrada} desempe\~{n}a una funci\'on importante en la caracterizaci\'on de las ternas TTF centralmente escindidas. N\'otese que tanto las categor\'ias abelianas como las trianguladas son equilibradas. Sin embargo, no es \'este el caso de las ca\-te\-go\-r\'i\-as pretrianguladas. Un ejemplo t\'ipico de categor\'ia pretriangulada (no necesariamente abeliana o triangulada) es la categor\'ia estable de una categor\'ia abeliana asociada a una subcategor\'ia funtorialmente finita (\cf los art\'iculos de A. Beligiannis \cite{Beligiannis2000, Beligiannis2001}). De este modo, la siguiente es una pregunta natural:
\bigskip

\noindent \emph{Pregunta 8:} ?`Podemos caracterizar cu\'ando categor\'ias estables `homol\'ogicamente buenas' de categor\'ias abelianas son equilibradas?
\bigskip 

Este trabajo est\'a estructurado en seis cap\'itulos a lo largo de los cuales es\-tu\-dia\-mos y respondemos las preguntas formuladas m\'as arriba.
\bigskip

\textbf{Cap\'itulo 1}
\bigskip

Este cap\'itulo versa sobre la pregunta 7, y desarrolla los rudimentos de una teor\'ia de torsi\'on para categor\'ias aditivas. La definici\'on principal es la de par de torsi\'on aditivo.

\begin{definicion}
Sea $\cd$ una categor\'ia aditiva. Un \emph{par de torsi\'on} en $\cd$ es un par $(\cx,\cy)$ de subcategor\'ias estrictamente plenas de $\cd$ tales que:
\begin{enumerate}[i)]
\item $\cd(X,Y)=0$ para cada $X\in\cx$ e $Y\in\cy$.
\item El funtor inclusi\'on $x:\cx\ra\cd$ tiene un adjunto por la derecha $\tau_{\cx}$\index{$\tau_{\cx}$}. Escribimos
\[\theta_{\cx}:\cd(xN,M)\arr{\sim}\cx(N,\tau_{\cx}M)
\] 
para referirnos al isomorfismo de adjunci\'on, $\eta_{\cx}$\index{$\eta_{\cx}$} para referirnos a la unidad y $\delta_{\cx}$\index{$\delta_{\cx}$} para referirnos a la counidad.
\item El funtor de inclusi\'on $y:\cy\ra\cd$ tiene un adjunto por la izquierda $\tau^{\cy}$\index{$\tau^{\cy}$}. Escribimos
\[\theta^{\cy}:\cy(\tau^{\cy}N,M)\arr{\sim}\cd(N,yM)
\] 
para referirnos al isomorfismo de adjunci\'on, $\eta^{\cy}$\index{$\eta^{\cy}$} para referirnos a la unidad y $\delta^{\cy}$\index{$\delta^{\cy}$} para referirnos a la counidad.
\item Para cada $M\in\cd$ la sucesi\'on
\[x\tau_{\cx}M\arr{\delta_{\cx,M}}M\arr{\eta^{\cy}_{M}}y\tau^{\cy}M
\]
es d\'ebilmente exacta, es decir $\delta_{\cx,M}$ es un n\'ucleo d\'ebil de $\eta^{\cy}_{M}$ y $\eta^{\cy}_{M}$ es un con\'ucleo d\'ebil de $\delta_{\cx,M}$.
\end{enumerate}
\end{definicion}

En la secci\'on \ref{(Co)suspended, triangulated and pretriangulated categories} revisamos la noci\'on de \emph{categor\'ia (co)suspendida} \cite{KellerVossieck1987, Keller1996} (que tambi\'en aparece en la literatura bajo el nombre de \emph{categor\'ia triangulada por la izquierda (o derecha)} \cite{BeligiannisMarmaridis1994}). En particular, recordamos las dos maneras t\'ipicas de obtener categor\'ias suspendidas: como categor\'ias estables de categor\'ias exactas (en el sentido de D. Quillen \cite{Quillen73}, v\'ease tambi\'en el art\'iculo de B. Keller \cite{Keller1990}) con suficientes inyectivos y como categor\'ias estables de categor\'ias abelianas asociadas a subcategor\'ias covariantemente finitas. Revisamos tambi\'en la noci\'on de categor\'ia triangulada (y su relaci\'on con las categor\'ias de Frobenius mediante el teorema de D. Happel \cite[Theorem 2.6]{Happel1988}), de categor\'ia pretriangulada (en el sentido de A.~Beligiannis \cite{Beligiannis2001}), de t-estructura y de par de torsi\'on pretriangulado.

En la secci\'on \ref{Compatible torsion theories}, demostramos que los pares de torsi\'on abelianos, triangulados, pretriangulados, \dots\  son precisamente los pares de torsi\'on aditivos que satisfacen ciertas condiciones extra. 

Decimos que un par de torsi\'on aditivo $(\cx,\cy)$ en una categor\'ia aditiva $\cd$ es \emph{escindido} si para cada objeto $M$ de $\cd$ tenemos que en la sucesi\'on d\'ebilmente exacta asociada
\[x\tau_{\cx}M\arr{\delta_{\cx,M}}M\arr{\eta^{\cy}_{M}}y\tau^{\cy}M
\]
el morfismo $\delta_{\cx,M}$ es una secci\'on y el morfismo $\eta^{\cy}_{M}$ es una retracci\'on. N\'otese que esto implica que $M\cong x\tau_{\cx}M\oplus y\tau^{\cy}M$. Una \emph{terna TTF aditiva} sobre una categor\'ia aditiva $\cd$ es una terna $(\cx,\cy,\cz)$ tal que tanto $(\cx,\cy)$ como $(\cy,\cz)$ son pares de torsi\'on aditivos. Decimos que la terna TTF aditiva $(\cx,\cy,\cz)$ es \emph{escindida por la izquierda (respectivamente, derecha)} si $(\cx,\cy)$ (respectivamente, $(\cy,\cz)$) es escindido, y decimos que la terna es \emph{centralmente escindida} si ambos pares de torsi\'on son escindidos.
En la secci\'on \ref{Characterization of centrally split TTF triples}, damos una caracterizaci\'on de las ternas TTF centralmente escindidas que ahora presentamos. Pero antes, observemos que en las categor\'ias pretrianguladas la siguiente versi\'on (estrictamente, v\'ease el Ejemplo \ref{weakly balanced not balanced}) d\'ebil de la noci\'on de ``equilibrio'' es posible:

\begin{definicion}
Una categor\'ia pretriangulada $(\cd,\Omega,\Sigma)$ es \emph{d\'ebilmente equilibrada} si un morfismo $f:M\ra N$ es un isomorfismo siempre que exista un tri\'angulo por la izquierda de la forma
\[\Omega N\ra 0\ra M\arr{f}N
\]
y un tri\'angulo por la derecha de la forma
\[M\arr{f}N\ra 0\ra\Sigma M.
\]
\end{definicion}

He aqu\'i la caracterizaci\'on mencionada (Proposici\'on \ref{centrally split TTF triples on additive}): 

\begin{proposicion}
Sea $(\cx,\cy,\cz)$ una terna TTF aditiva en una categor\'ia aditiva $\cd$. Las siguientes propiedades son equivalentes:
\begin{enumerate}[1)]
\item Es escindida por la izquierda y $\cx=\cz$.
\item Es escindida por la derecha y $\cx=\cz$.
\item Es centralmente escindida.
\end{enumerate}
Cuando  $\cd$ tiene \emph{factorizaciones can\'onicas} (es decir, un morfismo se factoriza `de manera \'unica' como un epimorfismo seguido de un monomorfismo), o es una categor\'ia pretriangulada d\'ebilmente equilibrada y ambos pares de torsi\'on $(\cx,\cy)$ e $(\cy,\cz)$ son pretriangulados, entonces las propiedades anteriores son equivalentes a:
\begin{enumerate}[4)]
\item $\cx=\cz$.
\end{enumerate}
\end{proposicion}

Es destacable que, en el caso de que la categor\'ia ambiente $\cd$ sea triangulada, entonces una terna TTF triangulada 
$(\cx,\cy,\cz)$ es centralmente escindida si y s\'olo si es escindida por la izquierda, si y s\'olo si es escindida por la derecha, si y s\'olo si $\cx=\cz$ (\cf Proposici\'on \ref{split TTF triple}). Por lo tanto, en el `mundo triangulado', o bien una terna TTF triangulada es centralmente escindida o bien no es escindida en modo alguno.

En la secci\'on \ref{Parametrization of centrally split TTF triples}, usamos idempotentes para parametrizar:
\begin{enumerate}[a)]
\item Proposici\'on \ref{centraly split TTF, 2-decomposition and idempotents}: Ternas TTF aditivas centralmente escindidas en categor\'ias aditivas en las que los idempotentes escinden.
\item Corolario \ref{centrally split TTF triples on abelian categories}: Ternas TTF centralmente escindidas en categor\'ias abelianas de una cierta clase que incluye a las categor\'ias de Grothendieck con un generador proyectivo. En particular, generalizamos y damos una demostraci\'on alternativa del hecho de que las ternas TTF en una categor\'ia de m\'odulos $\Mod A$ est\'a en biyecci\'on con los idempotentes centrales de $A$.
\item Corolario \ref{centrally split TTF in compactly generated triangulated categories}: Ternas TTF trianguladas centralmente escindidas en categor\'ias trianguladas compactamente generadas. En particular, demostramos que las ternas TTF trianguladas centralmente escindidas en la categor\'ia derivada 
$\cd\ca$ de una categor\'ia dg peque\~{n}a est\'an en biyecci\'on con los idempotentes (centrales) $(e_{A})_{A\in\ca}$ de $\prod_{A\in\ca}\H 0\ca(A,A)$ tales que para cada entero $n$ y cada $f\in\H n\ca(A,B)$ tenemos $e_{B}\cdot f=f\cdot e_{A}$ en $\H n\ca(A,B)$. As\'i, si $A$ es un \'algebra ordinaria, las ternas TTF trianguladas centralmente escindidas en $\cd A$ est\'an en biyecci\'on con los idempotentes centrales de $A$.
\item Corolario \ref{centrally split TTF in derived categories of abelian categories}: Ternas TTF trianguladas centralmente escindidas en la categor\'ia derivada de una categor\'ia abeliana completa y cocompleta.
\end{enumerate}
\bigskip

\textbf{Cap\'itulo 2}
\bigskip

Los resultados de este cap\'itulo, que tratan sobre la pregunta 8, aparecen en \cite{NicolasSaorin2007b}. Una de las conclusiones de este cap\'itulo es que la propiedad de ser equilibrada es bastante restrictiva para una categor\'ia estable de una categor\'ia abeliana, y que est\'a relacionada con fen\'omenos de escisi\'on. Esto queda patente en el siguiente resultado (Corolario \ref{balanced in Serre case}):

\begin{notacion}
Sea $\ca$ una categor\'ia aditiva y sea $\ct$ una clase de objetos de $\ca$. Denotamos por $\ct^{\bot}$ a la clase de objetos $M$ de $\ca$ tales que $\ca(T,M)=0$ para todo objeto $T$ de $\ct$. Dualmente, definimos $\ ^{\bot}\ct$.
\end{notacion}

\begin{corolario}
Sea $\ca$ una categor\'ia abeliana y sea $\ct$ una clase de Serre de $\ca$. Las siguientes afirmaciones son equivalentes:
\begin{enumerate}[1)]
\item $\ct$ es contravariantemente (respectivamente, covariantemente) finita en $\ca$ y la categor\'ia estable $\ul{\ca}$ de $\ca$ asociada a $\ct$ es equilibrada.
\item $\ct$ es funtorialmente finita en $\ca$ y $\ul{\ca}$ es d\'ebilmente equilibrada.
\item $\ct$ satisface las siguientes propiedades:
\begin{enumerate}[3.1)]
\item $\cd(M,T)=0$ para todo $T\in\ct$ y $M\in\ct^{\bot}$.
\item Todo objeto de $\ca$ es la suma directa de un objeto de $\ct$ y un objeto de $\ct^{\bot}$.
\end{enumerate}
\end{enumerate}
\end{corolario}

El siguiente teorema (Teorema \ref{characterization balanced}), que es el resultado principal de la secci\'on \ref{Balance when T consists of projective objects}, caracteriza completamente cu\'ando una categor\'ia estable de una categor\'ia abeliana asociada a una subcategor\'ia de proyectivos es equilibrada:

\begin{teorema}
Sea $\ca$ una categor\'ia abeliana y $\ct$ una subcategor\'ia de objetos proyectivos que asumimos cerrada para sumas directas finitas y sumandos directos. Las siguientes afirmaciones son equivalentes:
\begin{enumerate}[1)]
\item La categor\'ia estable $\ul{\ca}$ de $\ca$ asociada a $\ct$ es equilibrada.
\item Si $\mu:T\ra M$ es un monomorfismo no escindido con $T\in\ct$, entonces existe un morfismo $h:M\ra T'$, con $T'\in\ct$, tal que nig\'un morfismo $\tilde{h}:M\ra (h\mu)(T)$ coincide con $h$ en $T$. Es decir, no tenemos un diagrama conmutativo como el que sigue:
\[\xymatrix{T\ar@{->>}[d]\ar[r]^{\mu} & M\ar[d]^{h}\ar@{.>}[dl]_{\tilde{h}} \\
(h\mu)(T)\ar@{^(->}[r] & T'
}
\]
\item Si $f:M\ra N$ es un epimorfismo satisfaciendo las condiciones $i)$ y $ii)$ de abajo, entonces es una retracci\'on:
\begin{enumerate}
\item[i)] Su n\'ucleo $f^k:\ker(f)\ra M$ factoriza a trav\'es de un objeto de $\ct$.
\item[ii)] Para cada $h:M\ra T$, con $T\in\ct$, el epimorfismo can\'onico $\ker(f)\twoheadrightarrow h(\ker(f))$ factoriza a trav\'es de $f^k:\ker(f)\ra M$.
\end{enumerate}
\end{enumerate}
\end{teorema}

Como consecuencia obtenemos que si $\ct$ consta de objetos inyectivos, entonces la correspondiente categor\'ia estable es siempre equilibrada ya que la condici\'on 2) siempre se satisface. Esto nos permite dar un ejemplo (Ejemplo \ref{hereditary non-abelian non-triangulated}) de una categor\'ia pretriangulada equilibrada que no es ni abeliana ni triangulada.

El resultado principal (Teorema \ref{characterization weakly-balanced}) de la secci\'on \ref{Weak balance when T consists of projective objects} trata el caso de las categor\'ias estables pretrianguladas d\'ebilmente equilibradas:

\begin{teorema} 
Sea $\ca$ una categor\'ia abeliana y $\ct$ una subcategor\'ia plena funtorialmente finita que consta de objetos proyectivos y que asumimos cerrada para sumandos directos. Consideremos las siguientes afirmaciones:
\begin{enumerate}[1)]
\item Para todo $T\in\ct\setminus\{0\}$, existe un morfismo no nulo $\varphi:T\ra T'$, con $T'\in\ct$, que factoriza a trav\'es de un objeto inyectivo de $\ca$. 
\item La categor\'ia estable $\ul{\ca}$ de $\ca$ asociada $\ct$ (que es pretriangulada) es d\'ebilmente equilibrada. 
\item Si $j: T\ra M$ es un monomorfismo no nulo en $\ca$, con $T\in\ct$, entonces existe un morfismo $h:M\ra T'$ tal que $h j\neq 0$, para alg\'un $T'\in\ct$. 
\end{enumerate}
Entonces $1)\Rightarrow 2)\Leftrightarrow 3)$ y, en el caso de que $\ca$ tenga suficientes inyectivos, todas las afirmaciones son equivalentes.
\end{teorema}

La relaci\'on entre equilibrio y escisi\'on en categor\'ias estables vuelve a quedar clara gracias a la siguiente consecuencia (Proposici\'on \ref{hereditary weakly balanced}) del teorema anterior, que caracteriza el equilibrio (d\'ebil) de categor\'ias estables de categor\'ias abelianas cuya clase de proyectivos es cerrada para subobjetos.

\begin{proposicion}
Sea $\ch$ una categor\'ia abeliana cuya clase de objetos proyectivos es cerrada para subobjetos. Sea $\ct$ una subcategor\'ia plena covariantemente finita en $\ch$ cuyos objetos son proyectivos y que asumimos cerrada para sumandos directos. Las siguientes afirmaciones son equivalentes: 
\begin{enumerate}[1)]
\item La categor\'ia estable $\ul{\ch}$ de $\ch$ asociada a $\ct$ es equilibrada. 
\item  $^\perp\ct$ es cerrada para subobjetos. 
\item El par $(^\perp\ct,\Sub(\ct))$ es un par de torsi\'on hereditario (escindido) en $\ch$, donde $\Sub(\ct)$ es la subcategor\'ia plena de $\ch$ formada por los subobjetos de los objetos de $\ct$.
\end{enumerate}
Cuando $\ct$ es contravariantemente finita en $\ch$, las afirmaciones anteriores son equivalentes a:
\begin{enumerate}[4)]
\item  $\ul{\ch}$ es d\'ebilmente equilibrada.
\end{enumerate}
Cuando $\ch$ tiene suficientes inyectivos, las afirmaciones $1)-3)$ son tambi\'en e\-qui\-va\-len\-tes a:
\begin{enumerate}[5)]
\item Para todo objeto $T\in\ct$, existe un monomorfismo $T\ra E$, cuando $E$ es un objeto inyectivo(-proyectivo) de $\ch$ que pertenece a $\ct$.
\end{enumerate}
\end{proposicion}
\bigskip

\textbf{Cap\'itulo 3}
\bigskip

Los resultados de este cap\'itulo, que dan una respuesta completa a la pregunta 1, han aparecido publicados en \cite{NicolasSaorin2007a} (v\'ease tambi\'en \cite{NicolasSaorin2007e}).

Sea $A$ un anillo arbitrario. Recordemos que un $A$-m\'odulo $M$ es \emph{$\Sigma$-inyectivo hereditario} si todo cociente de un coproducto (posiblemente infinito) de copias de $M$ es $A$-m\'odulo inyectivo. 

Resulta que las ternas TTF escindidas por la izquierda en $\Mod A$ admiten una caracterizaci\'on relativamente f\'acil en t\'erminos de m\'odulos $\Sigma$-inyectivos hereditarios (Corolario \ref{clasificacion left split}):

\begin{corolario}
La biyecci\'on de Jans entre ternas TTF en $\Mod A$ e ideales bi\-l\'a\-te\-ros idempotentes de $A$ se restringe a una biyecci\'on entre:
\begin{enumerate}[1)]
\item Ternas TTF escindidas a izquierda en $\Mod A$.
\item Ideales de $A$ de la forma $I=eA$ donde $e$ es un idempotente de $A$ tal que $eA(1-e)$ es $\Sigma$-inyectivo hereditario como $(1-e)A(1-e)$-m\'odulo por la derecha.
\end{enumerate}
\end{corolario}

Un $A$-m\'odulo $M$ es \emph{$\Pi$-proyectivo hereditario} si todo subm\'odulo de un producto (posiblemente infinito) de copias de $M$ es proyectivo.
N\'otese que uno puede considerar varios `duales' de un $A$-m\'odulo por la izquierda $M$. Por ejemplo, uno puede considerar el $A$-m\'odulo por la derecha $\Hom_{\Z}(M,\mathbf{Q}/\Z)$. Tambi\'en, uno podr\'ia poner $S:=\End_{A}(M)^{op}$ y considerar el $A$-m\'odulo por la derecha $\Hom_{S}(M,Q)$ donde $Q$ es un cogenerador inyectivo minimal de $\Mod S$. La Proposici\'on \ref{proposicion de hereditary pi-projective dual} dice que uno de esos duales es $\Pi$-proyectivo hereditario si y s\'olo si tambi\'en lo es el otro. En ese caso, decimos que $M$ \emph{tiene dual $\Pi$-proyectivo hereditario}.
 
Una caracterizaci\'on completa de las ternas TTF que escinden a derecha es m\'as dif\'icil que su an\'aloga a izquierda y requiere algo m\'as que m\'odulos $\Pi$-proyectivos hereditarios. Sin embargo, a veces estos m\'odulos bastan, como demostramos en la siguiente caracterizaci\'on parcial (\cf Corolario \ref{right split para buenos2}):

\begin{corolario}
La biyecci\'on de Jans entre ternas TTF en $\Mod A$ e ideales bi\-l\'a\-te\-ros idempotentes de $A$ se restringe a una biyecci\'on entre:
\begin{enumerate}[1)]
\item Ternas TTF que escinde a derecha en $\Mod A$ cuyo ideal idempotente asociado $I$ es finitamente generado por la izquierda.
\item Ideales de la forma $I=Ae$, donde $e$ es un idempotente de $A$ tal que $(1-e)Ae$ tiene dual $\Pi$-proyectivo hereditario visto como ${(1-e)A(1-e)}$-m\'odulo por la izquierda .
\end{enumerate}
En particular, cuando $A$ satisface cualesquiera de las siguientes condiciones, la clase $1)$ de m\'as arriba cubre todas las ternas TTF que escinden por la derecha en $\Mod A$:
\begin{enumerate}[i)]
\item $A$ es semiperfecto.
\item Todo ideal idempotente de $A$ que es puro por la izquierda es tambi\'en finitamente generado por la izquierda (\eg si $A$ es N\oe theriano por la izquierda).
\end{enumerate}
\end{corolario}

Ahora presentamos la parametrizaci\'on general de las ternas TTF que escinden a derecha. Pero antes necesitamos algunas definiciones.

\begin{definicion}
Sea $I$ un ideal bil\'atero idempotente de $A$ y sea $M$ un $A$-m\'odulo por la derecha. Un subm\'odulo $N$ de $M$ es \emph{$I$-saturado} si un elemento $m$ de $M$ pertenece a $N$ siempre y cuando $mI$ est\'e contenido en $N$.
\end{definicion}

\begin{definicion}
Un ideal bil\'atero idempotente $I$ de $A$ \emph{escinde a derecha} si:
\begin{enumerate}[1)]
\item para todo entero $n\geq 1$ y todo subm\'odulo $I$-saturado $K$ de $A^n$, el cociente $A^n/(K+I^n)$ es un $A/I$-m\'odulo proyectivo,
\item $I$ es un $A$-m\'odulo puro por la izquierda,
\item el anulador por la izquierda de $I$ en $A$ se anula, \ie $\lann_{A}(I)=0$.
\end{enumerate}
\end{definicion}

\begin{corolario}
La biyecci\'on de Jans entre ternas TTF en $\Mod A$ e ideales bi\-l\'a\-te\-ros idempotentes de $A$ se restringe a una biyecci\'on entre:
\begin{enumerate}[1)]
\item Ternas TTF en $\Mod A$ que escinden a derecha.
\item Ideales bil\'ateros idempotentes $I$ tales que $\lann_A(I)=(1-e)A$ para alg\'un idempotente $e\in A$, y adem\'as $I$ es un ideal de $eAe$ que escinde a derecha tal que $eAe/I$ es un anillo perfecto hereditario.
\end{enumerate}
\end{corolario}

Este resultado (Corolario \ref{clasification right split2}) nos permite dar un ejemplo (Ejemplo \ref{l=c<r}) de un anillo conmutativo tal que: toda terna TTF que escinde por la izquierda es de hecho centralmente escindida, pero admite ternas TTF que escinden por la derecha que no son centralmente escindidas.
\bigskip

\textbf{Cap\'itulo 4}
\bigskip

Este cap\'itulo est\'a esencialmente dedicado a recordar resultados bien conocidos sobre categor\'ias trianguladas y a establecer la terminolog\'ia para los cap\'itulos si\-guien\-tes. No obstante, contiene algunos resultados aparentemente nuevos que quiz\'as merece la pena resaltar aqu\'i. Por ejemplo, en la siguiente proposici\'on `me\-di\-mos' la distancia entre la propiedad de ser compacto y la de ser: `autocompacto', perfecto, superperfecto, \dots 

\begin{notacion}
Dada una clase $\cq$ de objetos de una categor\'ia triangulada $\cd$ denotamos por $\Tria(\cq)$ a la menor subcategor\'ia triangulada de $\cd$ que contiene a $\cq$ y es cerrada para coproductos peque\~{n}os.
\end{notacion}

\begin{proposicion}
Sea $\cd$ una categor\'ia triangulada con coproductos peque\~{n}os y sea $P$ un objeto de $\cd$. Las siguientes condiciones son equivalentes:
\begin{enumerate}[1)]
\item $P$ es compacto en $\cd$.
\item $P$ satisface:
\begin{enumerate}[2.1)]
\item $P$ es perfecto en $\cd$.
\item $P$ es compacto en la subcategor\'ia plena $\mbox{Sum}(\{P[n]\}_{n\in\Z})$ de $\cd$ formada por los coproductos peque\~{n}os de las traslaciones en ambos sentidos de $P$.
\item $\Tria(P)^{\bot}$ es cerrada para coproductos peque\~{n}os.
\end{enumerate}
\item $P$ satisfies:
\begin{enumerate}[3.1)]
\item $P$ es compacto en $\Tria(P)$.
\item $\Tria(P)^{\bot}$ es cerrada para coproductos peque\~{n}os.
\end{enumerate}
\item $P$ satisface:
\begin{enumerate}[4.1)]
\item $P$ es superperfecto en $\cd$.
\item $P$ es compacto en $\Sum(\{P[n]\}_{n\in\Z})$.
\end{enumerate}
\end{enumerate}
\end{proposicion}

Tambi\'en definimos la \emph{categor\'ia derivada acotada por la derecha de una categor\'ia dg} y estudiamos algunas de sus propiedades b\'asicas. Esbocemos brevemente aqu\'i su definici\'on, que necesitaremos para el resumen del cap\'itulo 6. Sea $\ca$ una categor\'ia dg peque\~{n}a. Sea $\Susp(\ca)$ la menor subcategor\'ia suspendida plena de $\cd\ca$ que contiene a los $\ca$-m\'odulos dg por la derecha $A^{\we}$ representados por los objetos $A$ de $\ca$ y cerrada para coproductos peque\~{n}os. Decimos que la \emph{categor\'ia derivada acotada por la derecha de $\ca$} es la subcategor\'ia triangulada plena de $\cd\ca$ obtenida como la uni\'on de todos los trasladados en ambos sentidos de $\Susp(\ca)$:
\[\cd^-\ca:=\bigcup_{n\in\Z}\Susp(\ca)[n].
\]
Esta definici\'on coincide con la cl\'asica cuando $\ca$ es la categor\'ia dg asociada a un \'algebra ordinaria $A$, y con la natural cuando $\ca$ tiene cohomolog\'ia concentrada en grados no positivos. En efecto, el Lema \ref{looking for coherence} nos dice lo siguiente:

\begin{lema}
Sea $\ca$ una categor\'ia dg peque\~{n}a con cohomolog\'ia concentrada en grados $(-\infty,m]$ para alg\'un entero $m\in\Z$. Para un $\ca$-module dg $M$ consideramos las siguientes afirmaciones:
\begin{enumerate}[1)]
\item $M\in\Susp(\ca)[s]$.
\item $\H iM(A)=0$ para cada entero $i>m+s$ y cada objeto $A$ de $\ca$.
\end{enumerate}
Entonces $1)\Rightarrow 2)$ y, en caso de que $m=0$, tambi\'en tenemos $2)\Rightarrow 1)$.
\end{lema}
\bigskip

\textbf{Cap\'itulo 5}
\bigskip

Los resultados de este cap\'itulo aparecen en el art\'iculo \cite{NicolasSaorin2007c}. Fijemos algo de terminolog\'ia antes de responder a la pregunta 2 de m\'as arriba.

\begin{definicion}
Sea $\cd$ una categor\'ia triangulada y sea $\cq$ un conjunto de objetos de $\cd$. Decimos que $\cq$ \emph{genera} $\cd$ si un objeto $M$ de $\cd$ se anula siempre y cuando
\[\cd(Q[n],M)=0
\]
para todo objeto $Q$ de $\cq$ y todo entero $n\in\Z$.
\end{definicion}

\begin{definicion}
Una categor\'ia triangulada $\cd$ es \emph{alada} si tiene un conjunto de generadores, tiene coproductos peque\~{n}os y para todo conjunto $\cq$ de objetos de $\cd$ tenemos que $\Tria(\cq)$ es la primera clase de objetos de una t-estructura de $\cd$.
\end{definicion}

\begin{ej}
La categor\'ia derivada de una categor\'ia dg peque\~{n}a es alada (\cf Corolario \ref{derived categories are aisled}).
\end{ej}

\begin{definicion}
Un objeto $M$ de una categor\'ia triangulada $\cd$ es \emph{excepcional} si no tiene `autoextensiones', \ie $\cd(M,M[n])=0$ para todo $n\in\Z\setminus\{0\}$.
\end{definicion}

Ahora podemos presentar nuestro Corolario \ref{recollements unbounded derived}, que es la versi\'on no acotada del teorema de S.~K\"{o}nig:

\begin{corolario}
Sea $\cd$ una categor\'ia triangulada que es o bien perfectamente ge\-ne\-ra\-da o bien alada. Las siguientes afirmaciones son equivalentes:
\begin{enumerate}[1)]
\item $\cd$ es una aglutinaci\'on de categor\'ias trianguladas generadas por un solo objeto compacto (y excepcional).
\item Existen objetos (excepcionales) $P$ y $Q$ de $\cd$ tales que:
\begin{enumerate}[2.1)]
\item $P$ es compacto.
\item $Q$ es compacto en $\Tria(Q)$.
\item $\cd(P[n],Q)=0$ para cada $n\in\Z$.
\item $\{P, Q\}$ genera $\cd$.
\end{enumerate}
\item Existe un objeto compacto (y excepcional) $P$ tal que $\Tria(P)^{\bot}$ es generado por un objeto compacto (y excepcional) en $\Tria(P)^{\bot}$.
\end{enumerate}
En el caso de que $\cd$ sea compactamente generada por un objeto compacto podemos a\~{n}adir:
\begin{enumerate}[4)]
\item Existe un objeto compacto (y excepcional) $P$ (tal que $\Tria(P)^{\bot}$ est\'a generada por un objeto compacto excepcional).
\end{enumerate}
En caso de que $\cd$ sea algebraica podemos a\~{n}adir:
\begin{enumerate}[5)]
\item $\cd$ es una aglutinaci\'on de categor\'ias derivadas de \'algebras dg (concentradas en grado $0$, \ie \'algebras ordinarias).
\end{enumerate}
\end{corolario}

Este resultado es consecuencia de una parametrizaci\'on general de todas las formas en que una categor\'ia triangulada $\cd$, que suponemos con un conjunto de generadores, aparece como aglutinaci\'on de dos categor\'ias trianguladas (v\'ease la Proposici\'on \ref{naive general parametrization}).

\begin{definicion}
Sea $k$ un anillo conmutativo. Una categor\'ia dg $k$-lineal $\ca$ es \emph{$k$-plana} si para cualquier par de objetos $A$ y $A'$ de $\ca$ resulta que tensorizar con el correspondiente espacio de morfismos
\[?\otimes_{k}\ca(A,A'): \cc k\ra\cc k
\]
conserva complejos ac\'iclicos.
\end{definicion}

El siguiente teorema (Teorema \ref{TTF are he}) es una respuesta a la pregunta 5 de arriba. Utiliza lo que llamamos \emph{epimorfismos homol\'ogicos de categor\'ias dg}, que son una generalizaci\'on natural de los ``epimorfismos homol\'ogicos'' de W. Geigle y H. Lenzing \cite{GeigleLenzing}.

\begin{teorema}
Sea $k$ un anillo conmutativo y sea $\ca$ una categor\'ia dg $k$-plana. Para toda terna TTF triangulada $(\cx,\cy,\cz)$ en $\cd\ca$ existe un epimorfismo homol\'ogico $F:\ca\ra\cb$ (biyectivo en objetos) tal que la imagen esencial del funtor de restricci\'on de escalares $F^{*}:\cd\cb\ra\cd\ca$ es $\cy$.
\end{teorema}

Gracias al trabajo de G. Tabuada \cite{Tabuada2005a}, la asunci\'on de $k$-planitud en el teorema anterior es inofensiva. En efecto, uno puede probar (\cf Lemma \ref{k flat up to quasi-equivalence}) que toda categor\'ia triangulada algebraica compactamente generada $k$-lineal es trianguladamente equivalente a la categor\'ia derivada de una categor\'ia dg $k$-plana.

El siguiente teorema (Teorema \ref{several descriptions}) a\'una varios resultados obtenidos en el transcurso de los cap\'itulos 4 y 5:

\begin{teorema}
Sea $k$ un anillo conmutativo, sea $\cd$ una categor\'ia triangulada algebraica compactamente generada $k$-lineal, y sea $\ca$ una categor\'ia dg $k$-plana cuya categor\'ia derivada es trianguladamente equivalente a $\cd$. Existe una biyecci\'on entre:
\begin{enumerate}[1)]
\item Subcategor\'ias aplastantes $\cx$ de $\cd$.
\item Ternas TTF trianguladas $(\cx,\cy,\cz)$ en $\cd$.
\item (Clases de equivalencia de) aglutinaciones para $\cd$.
\item (Clases de equivalencia de) epimorfismos homol\'ogicos de categor\'ias dg de la forma $F:\ca\ra\cb$ (que podemos escoger de modo que sean biyectivos en objetos).
\end{enumerate}
M\'as a\'un, si denotamos por $\cs$ a cualquiera de los conjuntos (equipotentes) de arriba, entonces existe una aplicaci\'on suprayectiva $\cR\ra\cs$, donde $\cR$ es la clase de objetos $P$ de $\cd$ tales que $\{P[n]\}^{\bot}_{n\in\Z}$ es cerrada para coproductos peque\~{n}os.
\end{teorema}

Los siguientes resultados pretenden complementar el trabajo de H.~Krause sobre subcategor\'ias aplastantes de categor\'ias trianguladas compactamente generadas \cite{Krause2005}. Empecemos con el Teorema \ref{from ideals to devissage wrt hc}:
\begin{teorema}
Sea $\cd$ una categor\'ia compactamente generada y sea $\ci$ un ideal bil\'atero idempotente de $\cd^c$ cerrado para traslaciones en ambos sentidos. Entonces existe una terna TTF triangulada $(\cx,\cy,\cz)$ en $\cd$ tal que:
\begin{enumerate}[1)]
\item $\cx=\Tria(\cp)$, para cierto conjunto $\cp$ de col\'imites de Milnor (\cf Definici\'on \ref{Milnor colimit}) de sucesiones de morfismos de $\ci$.
\item $\cy=\ci^{\bot}$, donde $\ci^{\bot}$ es por definici\'on la clase de objetos $M$ de $\cd$ tales que $\cd(f,M)=0$ para todo morfismo $f\in\ci$.
\item Un morfismo de $\cd^c$ pertenece a $\ci$ si y s\'olo si factoriza a trav\'es de un objeto de $\cp$.
\end{enumerate}
\end{teorema}

Esto, junto con la afirmaci\'on 2') de \cite[Theorem 4.2]{Krause2000} da una prueba breve y directa del siguiente teorema (Teorema \ref{our result}), que est\'a en la l\'inea de la biyecci\'on que H.~Krause establece en \cite[Corollary 12.5, Corollary 12.6]{Krause2005}:

\begin{teorema}
Sea $\cd$ una categor\'ia triangulada compactamente generada. Si $\cy$ es una clase de objetos de $\cd$ denotamos por $\cm or(\cd^c)^{\cy}$ a la clase de morfismos $f$ de $\cd^c$ tales que $\cd(f,Y)=0$ para todo objeto $Y\in\cy$. Entonces, las aplicaciones
\[(\cx,\cy,\cz)\mapsto\cm or(\cd^c)^{\cy}\ \ \ \ \text{ e }\ \ \ \ \ci\mapsto(^{\bot}(\ci^{\bot}),\ci^{\bot}, \ci^{\bot\bot})
\] 
definen una biyecci\'on entre el conjunto de todas las ternas TTF trianguladas en $\cd$ y el conjunto de todos los ideales bil\'ateros idempotentes cerrados (\cf Definici\'on \ref{closed ideals}) $\ci$ de $\cd^c$ tales que $\ci[1]=\ci$. 
\end{teorema}

Como consecuencia obtenemos el Corolario \ref{gsc}, que responde positivamente a la pregunta 4.

\begin{corolario}
Sea $\cd$ una categor\'ia triangulada compactamente generada. Entonces toda subcategor\'ia aplastante $\cx$ de $\cd$ es de la forma $\cx=\Tria(\cp)$, donde $\cp$ es un conjunto de col\'imites de Milnor de objetos compactos.
\end{corolario}

El Teorema \ref{our result} tambi\'en implica el Corolario \ref{Krause algebraic}, que para el caso algebraico da una nueva demostraci\'on de la biyecci\'on de H. Krause:

\begin{corolario}
Sea $\cd$ una categor\'ia triangulada algebraica compactamente ge\-ne\-ra\-da. Las aplicaciones
\[(\cx,\cy,\cz)\mapsto\cm or(\cd^c)^{\cy}\ \ \ \ \text{ e }\ \ \ \ \ci\mapsto(^{\bot}(\ci^{\bot}),\ci^{\bot}, \ci^{\bot\bot})
\] 
definen una biyecci\'on entre el conjunto de todas las ternas TTF trianguladas en $\cd$ y el conjunto de todos los ideales bil\'ateros idempotentes saturados $\ci$ de $\cd^c$ cerrados para traslaciones en ambos sentidos.
\end{corolario}

Una posible respuesta a la pregunta 6 ser\'ia la siguiente: en el caso algebraico, la `omnipresencia' de los epimorfismos homol\'ogicos de categor\'ias dg nos permite dar una demostraci\'on sencilla de la afirmaci\'on 2') del \cite[Theorem 4.2]{Krause2000} (v\'ease la demostraci\'on de la Proposici\'on \ref{for the idempotency}), la cual es el resultado necesario, junto con el Teorema \ref{from ideals to devissage wrt hc}, para dar la nueva demostraci\'on de la biyecci\'on de H.~Krause.
\bigskip

\textbf{Cap\'itulo 6}
\bigskip

Los resultados de este cap\'itulo, que tratan la pregunta 3, aparecer\'an en el art\'iculo \cite{NicolasSaorin2007d} en preparaci\'on. El primer resultado importante, y muy general, es el siguiente (Proposici\'on \ref{parametrization right bounded recollements}):

\begin{proposicion}
Sea $\ca$ una categor\'ia dg. Las siguientes afirmaciones son e\-qui\-va\-len\-tes:
\begin{enumerate}[1)]
\item $\cd^-\ca$ es una aglutinaci\'on de $\cd^-\cb$ y $\cd^-\cc$, para ciertas categor\'ias dg $\cb$ and $\cc$.
\item Existen conjuntos $\cp\ko \cq$ de objetos de $\cd^-\ca$ tales que:
\begin{enumerate}[2.1)]
\item $\cp$ est\'a contenido en $\Susp(\ca)[n_{\cp}]$ y $\cq$ est\'a contenido en $\Susp(\ca)[n_{\cq}]$ para ciertos enteros $n_{\cp}$ y $n_{\cq}$. 
\item $\cp$ y $\cq$ son dualmente acotados por la derecha.
\item $\Tria(\cp)\cap\cd^-\ca$ est\'a exhaustivamente generada a izquierda por $\cp$ y los objetos de $\cp$ son compactos en $\cd\ca$.
\item $\Tria(\cq)\cap\cd^-\ca$ est\'a exhaustivamente generada a izquierda por $\cq$ y los objetos de $\cq$ son compactos en $\Tria(\cq)\cap\cd^-\ca$.
\item $(\cd\ca)(P[n],Q)=0$ para cada $P\in\cp\ko Q\in\cq$ y $n\in\Z$.
\item $\cp\cup\cq$ genera $\cd\ca$.
\end{enumerate}
\end{enumerate}
\end{proposicion}

Hay varias cosas de este resultado que precisan aclaraci\'on.

En primer lugar, nos encontramos con la noci\'on de ser ``dualmente acotado por la derecha''. La definici\'on precisa se puede encontrar en la Definici\'on \ref{dually right bounded}. Dicha definici\'on es algo abstracta ya que usa resoluciones fibrantes en una cierta estructura modelo en la categor\'ia de $\ca$-m\'odulos dg por la derecha. Sin embargo, a veces admite una caracterizaci\'on mucho m\'as expl\'icita. Por ejemplo, si $A$ es un \'algebra ordinaria y $P$ es un complejo acotado por la derecha de $A$-m\'odulos tal que $(\cd A)(P,P[n])=0$ para $n\geq 1$, entonces $P$ es dualmente acotado por la derecha si y s\'olo si es cuasi-isomorfo a un complejo acotado de $A$-m\'odulos proyectivos. Para demostrarlo se puede utilizar el Lema \ref{characterization of dually right bounded} junto con el criterio de S.~K\"{o}nig que caracteriza a $\ch^b(\Proj A)$ dentro de $\cd^-A$ (\cf el principio de la demostraci\'on de \cite[Theorem 1]{Konig1991}).

En segundo lugar, nos encontramos con la noci\'on de estar ``exhaustivamente generada a izquierda''. Decimos que una categor\'ia triangulada $\cd$ est\'a \emph{exhaustivamente generada (a izquierda)} por una clase $\cp$ de objetos de $\cd$ si se satisfacen las siguientes condiciones:
\begin{enumerate}[1)]
\item La existencia de coproductos peque\~{n}os de extensiones finitas de coproductos peque\~{n}os (de traslaciones no negativas) de objetos de $\cp$ est\'a garantizada en $\cd$.
\item Para todo objeto $M$ de $\cd$ existe un entero $i\in\Z$ junto con un tri\'angulo
\[\coprod_{n\geq 0}Q_{n}\ra\coprod_{n\geq 0}Q_{n}\ra M[i]\ra \coprod_{n\geq 0}Q_{n}[1]
\]
de $\cd$ tal que cada $Q_{n}$ es una extensi\'on finita de coproductos peque\~{n}os (de traslaciones no negativas) de objetos de $\cp$.
\end{enumerate}

Las categor\'ias trianguladas exhaustivamente generadas aparecen frecuentemente en la pr\'actica como categor\'ias compactamente o incluso perfectamente generadas.

Cuando s\'olo tratamos con categor\'ias dg con cohomolog\'ia concentrada en grados no positivos (por ejemplo, \'algebras ordinarias), la proposici\'on anterior admite la siguiente reformulaci\'on m\'as sencilla (Corolario \ref{parametrization right bounded hn recollements}):

\begin{corolario}
Sea $\ca$ una categor\'ia dg. Las siguientes afirmaciones son e\-qui\-va\-len\-tes:
\begin{enumerate}[1)]
\item $\cd^-\ca$ es una aglutinaci\'on de $\cd^-\cb$ y $\cd^-\cc$, parta ciertas categor\'ias dg $\cb$ y $\cc$ con cohomolog\'ia concentrada en grados no positivos.
\item Existen conjuntos $\cp\ko \cq$ de objetos de $\cd^-\ca$ tales que:
\begin{enumerate}[2.1)]
\item $\cp$ est\'a contenido en $\Susp(\ca)[n_{\cp}]$ y $\cq$ est\'a contenido en $\Susp(\ca)[n_{\cq}]$ para ciertos enteros $n_{\cp}$ y $n_{\cq}$.
\item $\cp$ y $\cq$ son dualmente acotados por la derecha.
\item Los objetos de $\cp$ son compactos en $\cd\ca$ y satisfacen
\[(\cd\ca)(P,P'[n])=0
\] 
para todos $P\ko P'\in\cp$ y $n\geq 1$.
\item Los objetos de $\cq$ son compactos en $\Tria(\cq)\cap\cd^-\ca$ y satisfacen 
\[(\cd\ca)(Q,Q'[n])=0
\] 
para todos $Q\ko Q'\in\cp$ y $n\geq 1$.
\item $(\cd\ca)(P[n],Q)=0$ para cada $P\in\cp\ko Q\in\cq$ y $n\in\Z$.
\item $\cp\cup\cq$ genera $\cd\ca$.
\end{enumerate}
\end{enumerate}
\end{corolario}

Gracias al Lema \ref{route to Konig2}, la versi\'on para \'algebras ordinarias (\cf Teorema \ref{Konig}) es esencialmente el teorema de S.~K\"{o}nig \cite[Theorem 1]{Konig1991}:

\begin{teorema}
Sean $A$, $B$ y $C$ \'algebras ordinarias. Las siguientes afirmaciones son equivalentes: 
\begin{enumerate}[1)]
\item $\cd^-A$ es una aglutinaci\'on de $\cd^-C$ y $\cd^-B$. 
\item Existen dos objetos $P\ko Q\in\cd^-A$ que satisfacen las siguientes propiedades:
\begin{enumerate}[2.1)]
\item Existen isomorfismos de \'algebras $C\cong(\cd A)(P,P)$ y $B\cong(\cd A)(Q,Q)$. 
\item $P$ es excepcional e isomorfo en $\cd A$ a un complejo acotado de $A$-m\'odulos proyectivos finitamente generados. 
\item $(\cd A)(Q,Q[n]^{(\Lambda )})=0$ para cualquier conjunto $\Lambda$ y todo $n\in\Z\setminus\{0\}$, la aplicaci\'on can\'onica $(\cd A)(Q,Q)^{(\Lambda)}\ra(\cd A)(Q,Q^{(\Lambda )})$ es un isomorfismo, y $Q$ es isomorfo en $\cd A$ a un complejo acotado de $A$-m\'odulos proyectivos. 
\item $(\cd A)(P[n],Q)=0$ para todo $n\in\Z$. 
\item $P\oplus Q$ genera $\cd A$.
\end{enumerate}
\end{enumerate}
\end{teorema}
\bigskip



\chapter*{Introduction}
\addcontentsline{lot}{chapter}{Introduction}
\bigskip

\textbf{Motivation}
\bigskip

In his work \cite{Dickson1966}, S. E. Dickson defined the notion of \emph{torsion theory} (now called \emph{torsion pair}) in the general framework of abelian categories, which generalizes  the concept of `torsion' appearing in the theory of abelian groups. A \emph{torsion pair} on an abelian category $\ca$ consists of a pair $(\cx,\cy)$ of full subcategories of objects of $\ca$ such that:
\begin{enumerate}[1)]
\item $\cd(X,Y)=0$ for all $X\in\cx$ and $Y\in\cy$,
\item $\cx$ is closed under quotients and $\cy$ is closed under subobjects,
\item each object $M$ of $\ca$ occurs in a short exact sequence
\[0\ra M_{\cx}\ra M\ra M^{\cy}\ra 0
\]
in which $M_{\cx}$ belongs to $\cx$ and $M^{\cy}$ belongs to  $\cy$.
\end{enumerate}
When this short exact sequence splits for each $M$, we say that the torsion pair is \emph{split}. Since S. E. Dickson's work, torsion pairs have played an important r\^{o}le in algebra: they have been a fundamental tool for developing a general theory of noncommutative localization \cite{Stenstrom}, they have had a great impact in the representation theory of Artin algebras \cite{HappelRingel1982, HappelReitenSmalo1996, AssemSaorin2005},\dots One of the important concepts related to torsion theory is that of a \emph{torsion-torsionfree(=TTF) theory} (now called \emph{TTF triple}), introduced by J. P. Jans in \cite{Jans}. It consists of a triple $(\cx,\cy,\cz)$ of full subcategories of an abelian category such that both $(\cx,\cy)$ and $(\cy,\cz)$ are torsion pairs. We say that the TTF triple $(\cx,\cy,\cz)$ is \emph{left (respectively, right) split} when $(\cx,\cy)$ (respectively, $(\cy,\cz)$) is split. A TTF triple is \emph{centrally split} if it is both left and right split. When the ambient abelian category is the category of modules $\Mod A$ over a ring $A$ (associative, with unit), then we have the following result due to J. P. Jans \cite[Corollary 2.2]{Jans}:

\begin{thm}
TTF triples on $\Mod A$ are in bijection with idempotent two-sided ideals of $A$. Moreover, this bijection induces a bijection between centrally split TTF triples on $\Mod A$ and (two-sided ideals generated by) central idempotents of $A$.
\end{thm}

A natural question arises:
\bigskip

\noindent \emph{Question 1:} Which are the two-sided ideals of $A$ corresponding to left or right split TTF triples on $\Mod A$?
\bigskip

In their work \cite{BeilinsonBernsteinDeligne} on perverse sheaves,  A.~A.~Beilinson, J.~Bernstein and P.~Deligne defined the notion of \emph{t-structure} on a triangulated category, which formalizes the notion of `(intelligent) truncation' of a cochain complex on an abelian category. A  \emph{t-structure} on a triangulated category $(\cd,?[1])$ consists of a pair $(\cx,\cy)$ of strictly(=closed under isomorphisms) full subcategories of $\cd$ satisfying:
\begin{enumerate}[1)]
\item $\cd(X,Y)=0$ for all $X\in\cx$ and $Y\in\cy$,
\item $\cx$ is closed under positive shifts and $\cy$ is closed under negative shifts,
\item each object $M$ of $\cd$ occurs in a triangle
\[ M_{\cx}\ra M\ra M^{\cy}\ra M_{\cx}[1].
\]
\end{enumerate}
They proved that the general theory of t-structures is sophisticated enough to capture one of the most important phenomena related to truncation: the t-structure $(\cx,\cy)$ defines an abelian category $\cx\cap(\cy[1])$ inside the triangulated category $\cd$ and linked to it by a cohomological functor
$\cd\ra \cx\cap(\cy[1])$. To emphasize this fact, one usually writes $(\cx,\cy[1])$  to refer to the t-structure $(\cx,\cy)$. t-structures have given rise to a fair amount of results and are important in many branches of mathematics: theory of perverse sheaves \cite{BeilinsonBernsteinDeligne, KashiwaraShapira}, link between Grothendieck-Roos duality theory and tilting theory \cite{KellerVossieck88a}, tilting theory itself \cite{KellerVossieck88a, HappelReitenSmalo1996}, motivic cohomology and $\mathbb{A}^{1}$-homotopy theory in the sense of V. Voevodsky \cite{Voevodsky2000, Morel2003}, commutative algebra \cite{Neeman1992a, AlonsoJeremiasSaorin2007}, \dots As one can see, the notion of t-structure is the formal analogue of the notion of torsion pair for abelian categories. Accordingly, the `triangulated' analogue of a TTF triple, called \emph{triangulated TTF triple}, consists of a triple $(\cx,\cy,\cz)$ such that both $(\cx,\cy)$ and $(\cy,\cz)$ are t-structures. Although a triangulated TTF triple is made of t-structures, it does not provide us abelian categories inside the ambient triangulated category. Instead, the general philosophy of triangulated TTF triples is that they allow us to regard a triangulated category as glued together from two others triangulated categories. More precisely, triangulated TTF triples on a triangulated category $\cd$ are in `bijection' with ways of expressing $\cd$ as a \emph{glueing} or \emph{recollement} (a notion defined by A.~A.~Beilinson, J.~Bernstein and P.~Deligne in \cite{BeilinsonBernsteinDeligne}) of two others triangulated categories. In this sense, triangulated TTF triples have made a life by their own, apart from t-structures. They appear under the name of ``recollement'' in: the work of E. Cline, B. Parshall and L. L. Scott \cite{ClineParshallScott1988, ParshallScott1988}, the theory of diagram algebras \cite{MartinGreenParker2006},...

Inspired by the work of E.~Cline, B.~Parshall and L.~L.~Scott and using the work of J.~Rickard \cite{Rickard1989}, S.~K\"{o}nig gave in \cite[Theorem 1]{Konig1991} necessary and sufficient conditions for the existence of recollement situations. Precisely, he proved the following:

\begin{thm}
The right bounded derived category $\cd^-A$ of a ring $A$ is a recollement of the right derived categories of two rings $B$ and $C$ if and only if there exist two complexes $P$ and $Q$ in $\cd^-A$ quasi-isomorphic to a bounded complex of projective $A$-modules such that:
\begin{enumerate}[1]
\item $P$ is compact in $\cd A$ (\ie quasi-isomorphic to a bounded complex of finitely generated projective $A$-modules) and $(\cd A)(Q,?)$ commutes with small coproducts of copies of $Q$.
\item $(\cd A)(P,P[n])=0$ and $(\cd A)(Q,Q[n])=0$ for all $n\in\Z$.
\item There exist ring isomorphisms $(\cd A)(P,P)\cong C$ and $(\cd A)(Q,Q)\cong B$.
\item $(\cd A)(P[n],Q)=0$ for all $n\in\Z$.
\item If $M$ is an object of $\cd^-A$ such that $(\cd A)(P[n],M)=0$ and $(\cd A)(Q[n],M)=0$ for all $n\in\Z$, then $M=0$.
\end{enumerate}
\end{thm}

S.~K\"{o}nig proved his theorem at around 1990, and at that time a Morita theory for unbounded derived categories was not available. Some years latter, 
B.~Keller developed in \cite{Keller1994a} a Morita (and Koszul) theory for unbounded derived category of dg categories. Then, natural questions arise:
\bigskip

\noindent \emph{Question 2:} Can we use B. Keller's theory to parametrize all the ways of expressing the unbounded derived category $\cd A$ of a ring $A$ as a recollement of unbounded derived categories of two other rings? 
\bigskip

If the answer is yes, 
\bigskip

\noindent \emph{Question 3:} can we develop an `unbounded' approach to S.~K\"{o}nig's work? Namely, can we use the study of recollements at the unbounded level to understand recollements at the right bounded level?
\bigskip

According to \cite{BokstedtNeeman1993}, the notion of \emph{smashing subcategory} of a triangulated category $\cd$ originated in the work of D. Ravenel. A full triangulated subcategory $\cx$ of a triangulated category $\cd$ with small coproducts is a \emph{smashing subcategory} if there exists a t-structure $(\cx,\cy)$ such that $\cy$ is closed under small coproducts. If the ambient triangulated category is, for instance, compactly generated, then smashing subcategories  are in bijection with triangulated TTF triples. The \emph{generalized smashing conjecture} is a generalization to arbitrary compactly generated triangulated categories of a conjecture due to D.~Ravenel \cite[1.33]{Ravenel1984} and, originally, A.~K.~Bousfield \cite[3.4]{Bousfield1979}. It predicts the following:
\bigskip

\noindent If $\cx$ is a smashing subcategory of a compactly generated triangulated category $\cd$, then there exists a set $\cp$ of compact objects of $\cd$ such that $\cx$ is the smallest full triangulated subcategory of $\cd$ closed under small coproducts and containing $\cp$.
\bigskip

This conjecture was disproved by B.~Keller in \cite{Keller1994}. However, sometimes important results proved for compact objects remain true (or admit nice generalizations) after replacing ``compact'' by ``perfect'', ``superperfect'' or something related. Therefore one could ask:
\bigskip

\noindent \emph{Question 4:} Does the conjecture become true after replacing ``compact'' by something related?
\bigskip

In his algebraic approach to the generalized smashing conjecture \cite{Krause2000, Krause2005}, H.~Krause proved the following:

\begin{thm}
Let $\cd$ be a compactly generated triangulated category. Smashing subcategories of $\cd$ are in bijection with saturated idempotent two-sided ideals of $\cd^c$(=the full subcategory of $\cd$ formed by the compact objects) closed under shifts in both directions.
\end{thm}

Here \emph{saturated} means that whenever there exists a triangle
\[P'\arr{u} P\arr{v} P''\ra P'[1]
\]
in $\cd^c$ and a morphism $f\in\cd^c(P,Q)$ with $fu\ko v\in\ci$, then $f\in\ci$.

Notice the analogy with J. P. Jans' theorem above. This analogy is not strange, since module categories are `compactly generated' abelian categories. However, several facts suggest that the true `triangulated' analogue of module categories are not just compactly generated triangulated categories but compactly generated \emph{algebraic} triangulated categories. These, thanks to B. Keller's theorem \cite[Theorem 4.3]{Keller1994a}, are precisely the derived categories of dg categories.
\bigskip

\noindent \emph{Question 5:} Can we parametrize all the triangulated TTF triples on (equivalently, smashing subcategories of) the derived category $\cd\ca$ of a dg category $\ca$ in terms of $\ca$ itself or, at least, in `dg terms'? 
\bigskip

If the answer is yes, 
\bigskip

\noindent \emph{Question 6:} which is the link between this parametrization and Krause's?
\bigskip

Pretriangulated categories (in the sense of A. Beligiannis \cite{Beligiannis2001}) are a generalization of abelian and triangulated categories. In pretriangulated categories one can define the notion of \emph{pretriangulated torsion pair}, which is a generalization of the notion of torsion pair and t-structure. Pretriangulated torsion pairs (and so t-structures and torsion pairs) have many basic properties which should be deduced from a theory of torsion in additive categories. Moreover,  several facts suggest that the underlying additive category is essentially enough to understand centrally split TTF triples on an abelian category or centrally split triangulated TTF triples (splitness here is defined as we did for abelian categories). To formalize this observation one should try to answer the following question:
\bigskip

\noindent \emph{Question 7:} Can we define a torsion theory for additive categories so that split TTF triples on richer categories can be essentially understood by just looking at the underlying additive TTF triples?
\bigskip

When answering the question above one realizes that the property of being a \emph{balanced} category plays an important r\^{o}le in the characterization of centrally split TTF triples. Notice that both abelian and triangulated categories are necessarily balanced. However, this is not the case for pretriangulated categories. A typical example of a (not necessarily abelian or triangulated) pretriangulated category is the stable category of an abelian category associated to a functorially finite subcategory (\cf A. Beligiannis' papers \cite{Beligiannis2000, Beligiannis2001}). Hence, the following is a natural question:
\bigskip

\noindent \emph{Question 8:} Can we characterize when `homologically nice' stable categories of abelian categories are balanced?
\bigskip 

This work is structured in six chapters in which we study and answer the questions above.
\bigskip

\textbf{Chapter 1}
\bigskip

This chapter deals with question 7, and it develops the basics of a torsion theory for additive categories. In section \ref{Compatible torsion theories}, we prove that abelian, pretriangulated, triangulated,\dots torsion pairs are precisely the additive torsion pairs which satisfy certain extra conditions. In section \ref{Split torsion pairs}, we define what we mean by \emph{split additive torsion pair}. In section \ref{Characterization of centrally split TTF triples}, we define and characterize the so called \emph{centrally split additive TTF triples}. 

Notice that in pretriangulated categories the following (strictly, see Example \ref{weakly balanced not balanced}) weak version of the notion of ``balance'' is possible:

\begin{defin}
A pretriangulated category $(\cd,\Omega,\Sigma)$ is \emph{weakly balanced} if a morphism $f:M\ra N$ is an isomorphism whenever there exist a left triangle of the form
\[\Omega N\ra 0\ra M\arr{f}N
\]
and a right triangle of the form
\[M\arr{f}N\ra 0\ra\Sigma M.
\]
\end{defin}

In case the ambient additive category has canonical factorizations, or it is a weakly balanced pretriangulated (and the additive TTF triples under consideration are also pretriangulated), the characterization of centrally split additive TTF triples is improved in Proposition \ref{centrally split TTF triples on additive}: 

\begin{prop}
Let $(\cx,\cy,\cz)$ be an additive TTF triple on an additive category $\cd$. The following assertions are equivalent:
\begin{enumerate}[1)]
\item It is left split and $\cx=\cz$.
\item It is right split and $\cx=\cz$.
\item It is centrally split.
\end{enumerate}
When  $\cd$ has canonical factorizations or it is a weakly balanced pretriangulated category with both $(\cx,\cy)$ and $(\cy,\cz)$ pretriangulated torsion pairs, then the above conditions are equivalent to:
\begin{enumerate}[4)]
\item $\cx=\cz$.
\end{enumerate}
\end{prop}

It is remarkable (\cf Proposition \ref{split TTF triple}) that in case the ambient category $\cd$ is triangulated, then a triangulated TTF triple $(\cx,\cy,\cz)$ is centrally split if and only if it is left split, if and only if it is right split, if and only if $\cx=\cz$. Therefore, in the triangulated world, either a triangulated TTF triple is centrally split or it does not split at all.

In section \ref{Parametrization of centrally split TTF triples}, we use idempotents to parametrize:
\begin{enumerate}[a)]
\item Proposition \ref{centraly split TTF, 2-decomposition and idempotents}: Centrally split additive TTF triples on additive categories with splitting idempotents.
\item Corollary \ref{centrally split TTF triples on abelian categories}: Centrally split TTF triples on abelian categories of a certain type which includes the Grothendieck categories having a projective generator. In particular, we generalize and give an alternative proof of a well-known fact concerning torsion pairs on module categories: centrally split TTF triples on the category $\Mod A$ of modules over an algebra $A$ are in bijection with central idempotents of $A$.
\item Corollary \ref{centrally split TTF in compactly generated triangulated categories}: Centrally split triangulated TTF triples on compactly generated triangulated categories. In particular, we prove that centrally split triangulated TTF triples on the derived category $\cd\ca$ of a small dg category $\ca$ are in bijection with (central) idempotents $(e_{A})_{A\in\ca}$ of $\prod_{A\in\ca}\H 0\ca(A,A)$ such that for each integer $n$ and each $f\in\H n\ca(A,B)$ we have $e_{B}\cdot f=f\cdot e_{A}$ in $\H n\ca(A,B)$. Hence, if $A$ is an ordinary algebra, centrally split triangulated TTF triples on $\cd A$ are in bijection with central idempotents of $A$.
\item Corollary \ref{centrally split TTF in derived categories of abelian categories}: Centrally split triangulated TTF triples on derived categories of complete and cocomplete abelian categories.
\end{enumerate}
\bigskip

\textbf{Chapter 2}
\bigskip

The results of this chapter, which deals with question 8, will appear in \cite{NicolasSaorin2007b}.  

One of the conclusions of this chapter is that the condition of being balanced is quite restrictive for the stable category of an abelian category and has to do with splitness. This becomes clear with the following result (Corollary \ref{balanced in Serre case}):

\begin{notat}
Let $\ca$ be an additive category and let $\ct$ be a class of objects of $\ca$. We denote by $\ct^{\bot}$ the class of objects $M$ of $\ca$ such that $\ca(T,M)=0$ for every object $T$ in $\ct$. Dually, one defines $\ ^{\bot}\ct$.
\end{notat}

\begin{cor}
Let $\ca$ be an abelian category and $\ct$ be a Serre class in $\ca$. The following assertions are equivalent:
\begin{enumerate}[1)]
\item $\ct$ is contravariantly (respectively, covariantly) finite in $\ca$ and the stable category $\ul{\ca}$ of $\ca$ associated to $\ct$ is balanced.
\item $\ct$ is functorially finite in $\ca$ and $\ul{\ca}$ is weakly balanced.
\item $\ct$ satisfies the following properties:
\begin{enumerate}[3.1)]
\item $\cd(M,T)=0$ for all $T\in\ct$ and $M\in\ct^{\bot}$.
\item Every object of $\ca$ is the direct sum of an object of $\ct$ and an object of $\ct^{\bot}$.
\end{enumerate}
\end{enumerate}
\end{cor}

The next theorem (Theorem \ref{characterization balanced}), which is the main result of section \ref{Balance when T consists of projective objects}, completely characterizes when a stable category of an abelian category associated to a subcategory of projectives is balanced:

\begin{thm}
Let $\ca$ be an abelian category and  $\ct$ be a full subcategory consisting of projective objects which is closed under finite direct sums and direct summands. The following assertions are equivalent:
\begin{enumerate}[1)]
\item The stable category $\ul{\ca}$ of $\ca$ associated to $\ct$ is balanced.
\item If $\mu:T\ra M$ is a non-split monomorphism with $T\in\ct$, then there exists a morphism $h:M\ra T'$, with $T'\in\ct$, such that no morphism $\tilde{h}:M\ra (h\mu)(T)$ coincides with $h$ on $T$.
That is to say, we do not have any commutative diagram as follows:
\[\xymatrix{T\ar@{->>}[d]\ar[r]^{\mu} & M\ar[d]^{h}\ar@{.>}[dl]_{\tilde{h}} \\
(h\mu)(T)\ar@{^(->}[r] & T'
}
\]
\item If $f:M\ra N$ is an epimorphism satisfying conditions $i)$ and $ii)$ below, then it is a retraction:
\begin{enumerate}
\item[i)] Its kernel $f^k:\ker(f)\ra M$ factors through an object of $\ct$.
\item[ii)] For every $h:M\ra T$, with $T\in\ct$, the canonical epimorphism $\ker(f)\twoheadrightarrow h(\ker(f))$ factors through $f^k:\ker(f)\ra M$.
\end{enumerate}
\end{enumerate}
\end{thm}

As a consequence we get that if $\ct$ consists of injective objects, then the corresponding stable category is always balanced since condition 2) always holds. This allows to give an example (Example \ref{hereditary non-abelian non-triangulated}) of a balanced pretriangulated category which is neither abelian nor triangulated.

The main result (Theorem \ref{characterization weakly-balanced}) of section \ref{Weak balance when T consists of projective objects} deals with weakly balanced pretriangulated stable categories:

\begin{thm} 
Let $\ca$ be an abelian category and $\ct\subseteq\ca$ be a functorially finite full subcategory consisting of projective objects and closed under direct summands. Consider the following assertions:
\begin{enumerate}[1)]
\item For every $T\in\ct\setminus\{0\}$, there is a non-zero morphism $\varphi:T\ra T'$, with $T'\in\ct$, which factors through an injective object of $\ca$. 
\item The stable category $\ul{\ca}$ of $\ca$ associated to $\ct$ (which is pretriangulated) is weakly balanced. 
\item If $j: T\ra M$ is a non-zero monomorphism in $\ca$, with $T\in\ct$, then there is a morphism $h:M\ra T'$ such that $h j\neq 0$, for some $T'\in\ct$. 
\end{enumerate}
Then $1)\Rightarrow 2)\Leftrightarrow 3)$ and, in case $\ca$ has enough injectives, all assertions are equivalent.
\end{thm}

The link between balance and splitness in stable categories becomes again clear thanks to the following consequence (Proposition \ref{hereditary weakly balanced}) of the theorem above. It characterizes (weak) balance in abelian categories in which the class of projective objects is closed under subobjects.

\begin{prop}
Let $\ch$ be an abelian category such that its class of projective objects is closed under subobjects. Let $\ct$ be a covariantly finite full subcategory of $\ch$ consisting of projective objects and closed under direct summands. The following assertions are equivalent: 
\begin{enumerate}[1)]
\item The stable category $\ul{\ch}$ of $\ch$ associated to $\ct$ is  balanced. 
\item  $^\perp\ct$ is closed under subobjects. 
\item The pair $(^\perp\ct,\Sub(\ct))$ is a hereditary (split) torsion pair in $\ch$, where $\Sub(\ct)$ is the full subcategory of $\ch$ formed by the subobjects of objects in $\ct$.
\end{enumerate}
When $\ct$ is contravariantly finite in $\ch$, the above assertions are equivalent to:
\begin{enumerate}[4)]
\item  $\ul{\ch}$ is weakly balanced
\end{enumerate}
When $\ch$ has enough injectives,  assertions $1)-3)$ are also  equivalent to:
\begin{enumerate}[5)]
\item For every object $T\in\ct$, there is a monomorphism $T\ra E$, where $E$ is an injective(-projective) object of $\ch$ which belongs to $\ct$.
\end{enumerate}
\end{prop}
\bigskip

\textbf{Chapter 3}
\bigskip

The results of this chapter, which give a complete answer to question 1 above, have appeared published in \cite{NicolasSaorin2007a} (see also \cite{NicolasSaorin2007e}).

Let $A$ be an arbitrary ring. Recall that an $A$-module $M$ is \emph{hereditary $\Sigma$-injective} if every quotient of a (possibly infinite) coproduct of copies of $M$ is an injective $A$-module. 

It turns out that left split TTF triples on $\Mod A$ are relatively easily characterized in terms of hereditary $\Sigma$-injective modules (Corollary \ref{clasificacion left split}):

\begin{cor}
Jans' bijection between TTF triples on $\Mod A$ and idempotent two-sided ideals of $A$ restricts to a bijection between:
\begin{enumerate}[1)]
\item Left split TTF triples on $\Mod A$.
\item Ideals of $A$ of the form $I=eA$ where $e$ is an idempotent of $A$ such that $eA(1-e)$ is hereditary $\Sigma$-injective as a right $(1-e)A(1-e)$-module.
\end{enumerate}
\end{cor}

An $A$-module $M$ is \emph{hereditary $\Pi$-projective} if every submodule of a (possible infinite) product of copies of $M$ is projective. 
Notice that one can consider several `duals' of a left $A$-module $M$. For instance, one can consider the right $A$-module $\Hom_{\Z}(M,\mathbf{Q}/\Z)$. Also, one can put $S:=\End_{A}(M)^{op}$ and consider the  right $A$-module $\Hom_{S}(M,Q)$ where $Q$ is a minimal injective cogenerator of $\Mod S$.
Proposition \ref{proposicion de hereditary pi-projective dual} says that one of these duals is hereditary $\Pi$-projective if and only if so is the other. In this case, we say that $M$ \emph{has hereditary $\Pi$-projective dual}. 

A complete characterization of right split TTF triples is more difficult and needs something more than hereditary $\Pi$-projective modules. However, sometimes this kind of modules suffices, as shown in the following partial characterization (\cf Corollary \ref{right split para buenos2}):

\begin{cor}
Jans' bijection between TTF triples on $\Mod A$ and idempotent two-sided ideals of $A$ restricts to a bijection between:
\begin{enumerate}[1)]
\item Right split TTF triples on $\Mod A$ whose associated idempotent ideal $I$ is finitely generated on the left.
\item Ideals of the form $I=Ae$, where $e$ is an idempotent of $A$ such that the left ${(1-e)A(1-e)}$-module $(1-e)Ae$ has hereditary $\Pi$-projective dual.
\end{enumerate}
In particular, when $A$ satisfies either one of the two following conditions, the class $1)$ above covers all the right split TTF triples on $\Mod A$:
\begin{enumerate}[i)]
\item $A$ is semiperfect.
\item Every idempotent Êideal of $A$ which is pure on the left is also finitely generated on the left (\eg if $A$ is left N\oe therian).
\end{enumerate}
\end{cor}

Now we will present the general parametrization of right split TTF triples. For this we need some previous definitions.

\begin{defin}
Let $I$ be an idempotent two-sided ideal of $A$ and let $M$ be a right $A$-module. A submodule $N$ of $M$ is \emph{$I$-saturated} if an element $m$ of $M$ belongs to $N$ provided $mI$ is contained in $N$.
\end{defin}

\begin{defin}
An idempotent two-sided ideal $I$ of $A$ is \emph{right splitting} if:
\begin{enumerate}[1)]
\item for every integer $n\geq 1$ and every $I$-saturated submodule $K$ of $A^n$, the quotient $A^n/(K+I^n)$ is a projective $A/I$-module,
\item $I$ is a pure left $A$-module,
\item the left annihilator of $I$ in $A$ vanishes, \ie $\lann_{A}(I)=0$.
\end{enumerate}
\end{defin}

\begin{cor}
Jans' bijection between TTF triples on $\Mod A$ and idempotent two-sided ideals of $A$ restricts to a bijection between:\begin{enumerate}[1)]
\item Right split TTF triples on $\Mod A$.
\item Idempotent ideals $I$ such that, for some idempotent $e\in A$, $\lann_A(I)=(1-e)A$ and $I$ is a right splitting ideal of
$eAe$ with $eAe/I$ a hereditary perfect ring.
\end{enumerate}
\end{cor}

This result (Corollary \ref{clasification right split2}) allows us to give an example (Example \ref{l=c<r}) of a commutative ring such that: every left split TTF triple is centrally split, but there are right split TTF triples which are not centrally split.
\bigskip

\textbf{Chapter 4}
\bigskip

In this chapter we recall well-known results on triangulated categories and set the terminology for the subsequent chapters. Also, we define the \emph{right bounded derived category of a dg category} and study some of its basic properties. Let us briefly state here this definition, since it will be needed for the summary of chapter 6.

Let $\ca$ be a small dg category. Let $\Susp(\ca)$ be the smallest full suspended subcategory of $\cd\ca$ containing the right dg $\ca$-modules $A^{\we}$ represented by objects $A$ of $\ca$ and closed under small coproducts. We say that the \emph{right bounded derived category of $\ca$} is the full triangulated subcategory of $\cd\ca$ obtained as the union of the shifts of $\Susp(\ca)$:
\[\cd^-\ca:=\bigcup_{n\in\Z}\Susp(\ca)[n].
\]
This definition agrees with the classical one when $\ca$ is the dg category associated to an ordinary algebra $A$, and with the natural one when $\ca$ has cohomology concentrated in non-positive degrees (see Lemma \ref{looking for coherence}).
\bigskip

\textbf{Chapter 5}
\bigskip

The results of this chapter appear in the preprint \cite{NicolasSaorin2007c}. Let us fix some terminology before giving the answer to question 2 above.

\begin{notat}
Given a class $\cq$ of objects of a triangulated category $\cd$ we denote by $\Tria(\cq)$ the smallest full triangulated subcategory of $\cd$ containing $\cq$ and closed under small coproducts.
\end{notat}

\begin{defin}
Let $\cd$ be a triangulated category and let $\cq$ be a set of objects of $\cd$. We say that $\cq$ \emph{generates} $\cd$ if an object $M$ of $\cd$ vanishes provided
\[\cd(Q[n],M)=0
\]
for every object $Q$ of $\cq$ and every integer $n\in\Z$.
\end{defin}

\begin{defin}
A triangulated category $\cd$ is \emph{aisled} if it has a set of generators, it has small coproducts and for every set $\cq$ of objects of $\cd$ we have that $\Tria(\cq)$  is the first class of objects of a t-structure on $\cd$.
\end{defin}

\begin{ex}
The derived category of a small dg category is aisled (\cf Corollary \ref{derived categories are aisled}).
\end{ex}

\begin{defin}
An object $M$ of a triangulated category $\cd$ is \emph{exceptional} provided it has no `self-extensions', \ie $\cd(M,M[n])=0$ for all $n\in\Z\setminus\{0\}$.
\end{defin}

Now we can state our Corollary \ref{recollements unbounded derived}, which is the unbounded version of S.~K\"{o}nig's theorem:

\begin{cor}
Let $\cd$ be a triangulated category which is either perfectly generated or aisled. The following assertions are equivalent:
\begin{enumerate}[1)]
\item $\cd$ is a recollement of triangulated categories generated by a single compact (and exceptional) object.
\item There are (exceptional) objects $P$ and $Q$ of $\cd$ such that:
\begin{enumerate}[2.1)]
\item $P$ is compact.
\item $Q$ is compact in $\Tria(Q)$.
\item $\cd(P[n],Q)=0$ for each $n\in\Z$.
\item $\{P, Q\}$ generates $\cd$.
\end{enumerate}
\item There is a compact (and exceptional) object $P$ such that $\Tria(P)^{\bot}$ is generated by a compact (and exceptional) object in $\Tria(P)^{\bot}$.
\end{enumerate}
In case $\cd$ is compactly generated by a single object we can add:
\begin{enumerate}[4)]
\item There is a compact (and exceptional) object $P$ (such that $\Tria(P)^{\bot}$ is generated by an exceptional compact object).
\end{enumerate}
In case $\cd$ is algebraic we can add:
\begin{enumerate}[5)]
\item $\cd$ is a recollement of derived categories of dg algebras (concentrated in degree $0$, \ie ordinary algebras).
\end{enumerate}
\end{cor}

\begin{defin}
Let $k$ be a commutative ring. A $k$-linear dg category $\ca$ is \emph{$k$-flat} if for every two objects $A$ and $A'$ of $\ca$ we have that tensoring with the corresponding morphisms space
\[?\otimes_{k}\ca(A,A'): \cc k\ra\cc k
\]
preserves acyclic complexes.
\end{defin}

The following theorem (Theorem \ref{TTF are he}) is an answer to question 5 above. It uses what we call \emph{homological epimorphisms of dg categories}, which are a natural generalization of the ``homological epimorphisms'' of W. Geigle and H. Lenzing \cite{GeigleLenzing}.

\begin{thm}
Let $k$ be a commutative ring and let $\ca$ be a $k$-flat dg category. For every triangulated TTF triple $(\cx,\cy,\cz)$ on $\cd\ca$ there exists a homological epimorphism $F:\ca\ra\cb$ (bijective on objects) such that the essential image of the restriction of scalars functor $F^{*}:\cd\cb\ra\cd\ca$ is $\cy$.
\end{thm}

Thanks to the work of G. Tabuada \cite{Tabuada2005a}, the $k$-flatness assumption in the theorem above is harmless. Indeed, one can prove (\cf Lemma \ref{k flat up to quasi-equivalence}) that every compactly generated algebraic $k$-linear triangulated category is triangle equivalent to the derived category of a $k$-flat dg category.

The following theorem (Theorem \ref{several descriptions}) summarizes many results obtained in the course of chapters 4 and 5:

\begin{thm}
Let $k$ be a commutative ring, let $\cd$ be a compactly generated algebraic $k$-linear triangulated category, and let $\ca$ be a $k$-flat dg category whose derived category is triangle equivalent to $\cd$. There exists a bijection between:
\begin{enumerate}[1)]
\item Smashing subcategories $\cx$ of $\cd$.
\item Triangulated TTF triples $(\cx,\cy,\cz)$ on $\cd$.
\item (Equivalence classes of) recollements for $\cd$.
\item (Equivalence classes of) homological epimorphisms of dg categories of the form $F:\ca\ra\cb$ (which can be taken to be bijective on objects).
\end{enumerate}
Moreover, if we denote by $\cs$ any of the given (equipotent) sets, then there exists a surjective map $\cR\ra\cs$, where $\cR$ is the class of objects 
$P$ of $\cd$ such that $\{P[n]\}^{\bot}_{n\in\Z}$ is closed under small coproducts.
\end{thm}

The next results are meant to complement H.~Krause's work on smashing subcategories of compactly generated triangulated categories \cite{Krause2005}. Let us start with Theorem \ref{from ideals to devissage wrt hc}:

\begin{thm}
Let $\cd$ be a compactly generated triangulated category and let $\ci$ be an idempotent two-sided ideal of $\cd^c$ closed under shifts in both directions. There exists a triangulated TTF triple $(\cx,\cy,\cz)$ on $\cd$ such that:
\begin{enumerate}[1)]
\item $\cx=\Tria(\cp)$, for a certain set $\cp$ of Milnor colimits (\cf Definition \ref{Milnor colimit}) of sequences of morphisms of $\ci$.
\item $\cy=\ci^{\bot}$, where $\ci^{\bot}$ is the class of those objects $M$ of $\cd$ such that $\cd(f,M)=0$ for every $f\in\ci$.
\item A morphism of $\cd^c$ belongs to $\ci$ if and only if it factors through an object of $\cp$.
\end{enumerate}
\end{thm}

This, together with the assertion 2') of \cite[Theorem 4.2]{Krause2000}, gives a short proof of the following result (Theorem \ref{our result}) in the spirit of H.~Krause's bijection of \cite[Corollary 12.5, Corollary 12.6]{Krause2005}:

\begin{thm}
Let $\cd$ be a compactly generated triangulated category. If $\cy$ is a class of objects of $\cd$ we denote by $\cm or(\cd^c)^{\cy}$ the class of morphisms $f$ of $\cd^c$ such that $\cd(f,Y)=0$ for every object $Y\in\cy$. Then, the maps
\[(\cx,\cy,\cz)\mapsto\cm or(\cd^c)^{\cy}\ \ \ \ \text{ and }\ \ \ \ \ci\mapsto(^{\bot}(\ci^{\bot}),\ci^{\bot}, \ci^{\bot\bot})
\] 
define a bijection between the set of triangulated TTF triples on $\cd$ and the set of closed (\cf Definition \ref{closed ideals}) idempotent two-sided ideals $\ci$ of $\cd^c$ such that $\ci[1]=\ci$.
\end{thm}

As a consequence, we get Corollary \ref{gsc}, which gives a positive answer to question 4.

\begin{cor}
Let $\cd$ be a compactly generated triangulated category. Then every smashing subcategory $\cx$ of $\cd$ is of the form $\cx=\Tria(\cp)$, where $\cp$ is a set of Milnor colimits of compact objects.
\end{cor}

Another consequence of Theorem \ref{our result} is Corollary \ref{Krause algebraic}, which gives a new proof of H.~Krau\-se's bijection for the algebraic case:

\begin{cor}
Let $\cd$ be a compactly generated algebraic triangulated category. The maps
\[(\cx,\cy,\cz)\mapsto\cm or(\cd^c)^{\cy}\ \ \ \ \text{ and }\ \ \ \ \ci\mapsto(^{\bot}(\ci^{\bot}),\ci^{\bot}, \ci^{\bot\bot})
\] 
define a bijection between the set of triangulated TTF triples on $\cd$ and the set of saturated idempotent two-sided ideals $\ci$ of $\cd^c$ closed under shifts in both directions.
\end{cor}

A possible answer to question 6 is the following: in the algebraic setting, the `omnipresence' of homological epimorphisms of dg categories enables us to give a simple proof of assertion 2') of \cite[Theorem 4.2]{Krause2000} (see the proof of Proposition \ref{for the idempotency}), which is the needed statement, together with Theorem \ref{from ideals to devissage wrt hc}, in order to give a new proof of H. Krause's bijection.
\bigskip

\textbf{Chapter 6}
\bigskip

The results of this chapter, which deals with question 3, will appear in a preprint \cite{NicolasSaorin2007d}. The first main, very general, result is the following (Proposition \ref{parametrization right bounded recollements}):

\begin{prop}
Let $\ca$ be a dg category. The following assertions are equivalent:
\begin{enumerate}[1)]
\item $\cd^-\ca$ is a recollement of $\cd^-\cb$ and $\cd^-\cc$, for certain dg categories $\cb$ and $\cc$.
\item There exist sets $\cp\ko \cq$ in $\cd^-\ca$ such that:
\begin{enumerate}[2.1)]
\item $\cp$ is contained in $\Susp(\ca)[n_{\cp}]$ and $\cq$ is contained in $\Susp(\ca)[n_{\cq}]$ for some integers $n_{\cp}$ and $n_{\cq}$. 
\item $\cp$ and $\cq$ are dually right bounded.
\item $\Tria(\cp)\cap\cd^-\ca$ is exhaustively generated to the left by $\cp$ and the objects of $\cp$ are compact in $\cd\ca$.
\item $\Tria(\cq)\cap\cd^-\ca$ is exhaustively generated to the left by $\cq$ and the objects of $\cq$ are compact in $\Tria(\cq)\cap\cd^-\ca$.
\item $(\cd\ca)(P[n],Q)=0$ for each $P\in\cp\ko Q\in\cq$ and $n\in\Z$.
\item $\cp\cup\cq$ generates $\cd\ca$.
\end{enumerate}
\end{enumerate}
\end{prop}

There are several things to clarify in the proposition above. 

First, one encounters the notion of being ``dually right bounded''. One can find the precise definition in Definition \ref{dually right bounded}. This definition is a little bit abstract since it uses fibrant resolutions for a certain model category structure on the category of right dg $\ca$-modules. However, sometimes it admits a much more explicit characterization. For instance, if $A$ is an ordinary algebra and $P$ is a right bounded complex of $A$-modules such that $(\cd A)(P,P[n])=0$ for $n\geq 1$, then $P$ is dually right bounded if and only if it is quasi-isomorphic to a bounded complex of projective $A$-modules. To prove this one can use Lemma \ref{characterization of dually right bounded} together with S.~K\"{o}nig's criterion which characterizes $\ch^b(\Proj A)$ inside $\cd^-A$ (\cf the beginning of the proof of \cite[Theorem 1]{Konig1991}).

Secondly, one encounters the notion of ``being exhaustively generated to the left''. We say that a triangulated category $\cd$ is \emph{exhaustively generated (to the left)} by a class $\cp$ of objects of $\cd$ if the following two conditions hold:
\begin{enumerate}[1)]
\item The existence of small coproducts of finite extensions of small coproducts (of non-negative shifts) of objects of $\cp$ is guaranteed in $\cd$.
\item For every object $M$ of $\cd$ there exists an integer $i\in\Z$ together with a triangle
\[\coprod_{n\geq 0}Q_{n}\ra\coprod_{n\geq 0}Q_{n}\ra M[i]\ra \coprod_{n\geq 0}Q_{n}[1]
\]
in $\cd$ such that each $Q_{n}$ is a finite extension of small coproducts (of non-negative shifts) of objects of $\cp$.
\end{enumerate}

Exhaustively generated triangulated categories appear frequently in practice as compactly or even perfectly generated triangulated categories.

When we are only concerned with dg categories with cohomology concentrated in non-positive degrees (for instance, ordinary algebras), the proposition above can be simplified as follows (Corollary \ref{parametrization right bounded hn recollements}):

\begin{cor}
Let $\ca$ be a dg category. The following assertions are equivalent:
\begin{enumerate}[1)]
\item $\cd^-\ca$ is a recollement of $\cd^-\cb$ and $\cd^-\cc$, for certain dg categories $\cb$ and $\cc$ with cohomology concentrated in non-positive degrees.
\item There exist sets $\cp\ko \cq$ in $\cd^-\ca$ such that:
\begin{enumerate}[2.1)]
\item $\cp$ is contained in $\Susp(\ca)[n_{\cp}]$ and $\cq$ is contained in $\Susp(\ca)[n_{\cq}]$ for some integers $n_{\cp}$ and $n_{\cq}$.
\item $\cp$ and $\cq$ are dually right bounded.
\item The objects of $\cp$ are compact in $\cd\ca$ and satisfy 
\[(\cd\ca)(P,P'[n])=0
\] 
for all $P\ko P'\in\cp$ and $n\geq 1$.
\item The objects of $\cq$ are compact in $\Tria(\cq)\cap\cd^-\ca$ and satisfy 
\[(\cd\ca)(Q,Q'[n])=0
\] 
for all $Q\ko Q'\in\cp$ and $n\geq 1$.
\item $(\cd\ca)(P[n],Q)=0$ for each $P\in\cp\ko Q\in\cq$ and $n\in\Z$.
\item $\cp\cup\cq$ generates $\cd\ca$.
\end{enumerate}
\end{enumerate}
\end{cor}

After Lemma \ref{route to Konig2}, the version for ordinary algebras (\cf Theorem \ref{Konig}) is essentially S.~K\"{o}nig's theorem \cite[Theorem 1]{Konig1991}:

\begin{thm}
Let $A$, $B$ and $C$ be ordinary algebras. The following assertions are equivalent: 
\begin{enumerate}[1)]
\item $\cd^-A$ is a recollement of $\cd^-C$ and $\cd^-B$. 
\item There are two objects $P\ko Q\in\cd^-A$ satisfying the following properties:
\begin{enumerate}[2.1)]
\item There are isomorphisms of algebras $C\cong(\cd A)(P,P)$ and $B\cong(\cd A)(Q,Q)$. 
\item $P$ is exceptional and isomorphic in $\cd A$ to a bounded complex of finitely generated projective $A$-modules. 
\item $(\cd A)(Q,Q[n]^{(\Lambda )})=0$ for every set $\Lambda$ and every $n\in\Z\setminus\{0\}$, the canonical map $(\cd A)(Q,Q)^{(\Lambda)}\ra(\cd A)(Q,Q^{(\Lambda )})$ is an isomorphism, and $Q$ is isomorphic in $\cd A$ to a bounded complex of projective $A$-modules. 
\item $(\cd A)(P[n],Q)=0$ for all $n\in\Z$. 
\item $P\oplus Q$ generates $\cd A$.
\end{enumerate}
\end{enumerate}
\end{thm}
\bigskip


\mainmatter
\chapter{Torsion theory in additive categories}\label{Torsion theory in additive categories}
\addcontentsline{lot}{chapter}{Cap\'itulo 1. Teor\'ias de torsi\'on en categor\'ias aditivas}

\section{Introduction}
\addcontentsline{lot}{section}{1.1. Introducci\'on}

\subsection{Motivation}
\addcontentsline{lot}{subsection}{1.1.1. Introducci\'on}

The first observation is that the essence of splitting phenomena concerning torsion pairs (\eg split t-structures on triangulated categories) can be basically understood already at the additive level. This motivates the study of torsion theory in arbitrary additive categories and its interplay with more sophisticated torsion theories, like pretriangulated torsion pairs or t-structures, specially in the situation of splitting.

\subsection{Outline of the chapter}
\addcontentsline{lot}{subsection}{1.1.2 Esbozo del cap\'itulo}

In section \ref{Additive torsion pairs}, we introduce the notion of torsion pair of an arbitrary additive category. In section \ref{(Co)suspended, triangulated and pretriangulated categories}, we recall the definition of (co)suspended (or left/right triangulated) categories, and remind two ways in which they appear in nature, namely, as the stable categories of exact categories with enough injectives (or projectives), or as stable categories of abelian categories associated to a covariantly (or contravariantly) finite subcategory. We also recall the notion of pretriangulated category in the sense of A. Beligiannis, and the definition of pretriangulated torsion pair. In section \ref{Compatible torsion theories}, we prove that abelian, pretriangulated, triangulated,\dots torsion pairs are precisely the additive torsion pairs which satisfy certain extra conditions. In section \ref{Split torsion pairs}, we define what we mean by split torsion pair. In section \ref{Characterization of centrally split TTF triples}, we define and characterize the so called centrally split torsion torsionfree(=TTF) triples. In case the ambient additive category has canonical factorizations, or it is a weakly balanced pretriangulated (and the TTF triples under consideration are also pretriangulated), this characterization is improved. In section \ref{Parametrization of centrally split TTF triples}, we use idempotents to parametrize:
\begin{enumerate}[a)]
\item Centrally split TTF triples on additive categories with splitting idempotents.
\item Centrally split TTF triples on abelian categories of a certain type which includes the Grothendieck categories having a projective generator. In particular, we generalize and give an alternative proof of the second statement of J. P. Jans' theorem: centrally split TTF triples on the category $\Mod A$ of modules over an algebra $A$ are in bijection with central idempotents of $A$.
\item Centrally split triangulated TTF triples on compactly generated triangulated categories. In particular, we prove that centrally split triangulated TTF triples on the derived category $\cd\ca$ of a small dg category $\ca$ are in bijection with (central) idempotents $(e_{A})_{A\in\ca}$ of $\prod_{A\in\ca}\H 0\ca(A,A)$ such that for each integer $n$ and each $f\in\H n\ca(A,B)$ we have $e_{B}\cdot f=f\cdot e_{A}$ in $\H n\ca(A,B)$. Hence, if $A$ is an ordinary algebra, centrally split triangulated TTF triples on $\cd A$ are in bijection with central idempotents of $A$.
\item Centrally split triangulated TTF triples on derived categories of complete and cocomplete abelian categories.
\end{enumerate}

\subsection{Notation}
\addcontentsline{lot}{subsection}{1.1.3. Notaci\'on}

By \emph{strictly full}\index{subcategory!strictly full} subcategory we mean `full and closed under isomorphisms'. Given a class $\cq$ of objects of an additive category $\cd$, we denote by $\cq^{\bot_{\cd}}$\index{$\cq^{\bot_{\cd}}$} (or $\cq^{\bot}$\index{$\cq^{\bot}$} if the category $\cd$ is clear) the full subcategory of $\cd$ formed by the objects $M$ which are \emph{right orthogonal}\index{orthogonal!right} to each object of $\cq$, \ie such that $\cd(Q,M)=0$ for all $Q$ in $\cq$. Dually for $\ ^{\bot_{\cd}}\cq$\index{$\ ^{\bot_{\cd}}\cq$} denotes the full subcategory of $\cd$ formed by the objects which are \emph{left orthogonal}\index{orthogonal!left} to each object of $\cq$. Also, we denote by $\add(\cq)$\index{$\add(?)$} the class of objects of $\cd$ which are direct summands of finite coproducts of objects of $\cq$. If $\cq$ is a class of objects of an abelian category, we denote by $\Gen(\cq)$\index{$\Gen(?)$} the class of objects which are quotients of small coproducts of objects of $\cq$.

\section{Additive torsion pairs}\label{Additive torsion pairs}
\addcontentsline{lot}{section}{1.2. Pares de torsi\'on aditivos}

We start with some elementary properties of adjoint pairs of functors.

\begin{lemma}\label{properties of adjunctions}
Let
\[\xymatrix{\cc\ar@<1ex>[d]^R \\
\cd\ar@<1ex>[u]^{L}
}
\]
be an adjoint pair of functors between arbitrary categories. Put 
\[\theta: \cc(LN,M)\arr{\sim}\cd(N,RM)
\] 
for the adjunction isomorphism. The \emph{unit}\index{adjunction!unit of an} will be denoted by $\eta_{?}:=\theta(\id_{L?})$ and the \emph{counit}\index{adjunction!counit of an} by $\delta_{?}:=\theta^{-1}(\id_{R?})$. The following assertions hold:
\begin{enumerate}[1)]
\item $R(\delta_{M})\eta_{RM}=\id_{RM}$ for each $M\in\cc$.
\item $\delta_{LN}L(\eta_{N})=\id_{LN}$ for each $N\in\cd$.
\item $\theta(f)=R(f)\eta_{N}$ for each $f\in\cc(LN,M)\ko N\in\cd\ko M\in\cc$.
\item $\theta^{-1}(g)=\delta_{M}L(g)$ for each $g\in\cd(N,RM)\ko N\in\cd\ko M\in\cc$.
\item $L$ is fully faithful if and only if $\eta$ is an isomorphism.
\item $R$ is fully faithful if and only if $\delta$ is an isomorphism.
\item If $L$ is full, then for each $M\in\cc$ we have that $\delta_{M}$ is a retraction if and only if $\delta_{M}$ is an isomorphism.
\item If $R$ is full, then for each $N\in\cd$ we have that $\eta_{N}$ is a section if and only if $\eta_{N}$ is an isomorphism.
\item Assume $\cc$ and $\cd$ are additive. Then $\eta_{N}=0$ if and only if $LN=0$.
\item Assume $\cc$ and $\cd$ are additive. Then $\delta_{M}=0$ if and only if $RM=0$.
\item If $L$ is full, a morphism $f$ satisfies $\delta_{M}=\delta_{M}f$ if and only if $f=\id_{LRN}$.
\item If $R$ is full, a morphism $g$ satisfies $\eta_{N}=g\eta_{N}$ if and only if $g=\id_{RLN}$.
\end{enumerate}
\end{lemma}
\begin{proof}
Assertions $1)-4)$ are well-known.

5) Thanks to \cite[Proposition II.7.5]{HiltonStammbach}, we have that if $L$ is fully faithful then $\eta$ is an isomorphism. Conversely, assume that $\eta$ is an isomorphism. If $f\in\cc(LN,LM)$, then $f=L(\eta_{N})^{-1}\circ LR(f)\circ L(\eta_{M})=L(\eta^{-1}_{N}\circ R(f)\circ\eta_{M})$ and so $L$ is full. Finally, if $L(f)=L(g)$, then $RL(f)=RL(g)$ and, since the unit is an isomorphism, we deduce that $f=g$. This proves that $L$ is faithful.

6) Follows from 5) by duality.

7) Assume there exists a morphism $s$ such that $\delta_{M}\circ s=\id_{M}$. If $L$ is full, then $s\circ\delta_{M}=L(f)$ for some $f\in\cd(RM,RM)$. Now
\[\theta^{-1}(f)=\delta_{M}\circ L(f)=\delta_{M}\circ s\circ \delta_{M}=\delta_{M}=\theta^{-1}(\id_{RM}),
\] 
and so $f=\id_{RM}$, which implies $s\circ\delta_{M}=\id_{LRM}$.

8) Follows from 7) by duality.

9) $\theta(\id_{LN})=\eta_{N}=0=\theta(0_{LN})$ is equivalent to $\id_{LN}=0_{LN}$.

10) Follows from 9) by duality.

11) Assume $\delta_{M}=\delta_{M}f$. Since $L$ is full, there exists a morphism $g$ such that $L(g)=f$. But then we have
$\theta^{-1}(g)=\delta_{M}L(g)=\delta_{M}f=\delta_{M}=\theta^{-1}(\id_{RM})$. This is equivalent to $g=\id_{RM}$, which implies $f=\id_{LRM}$.

12) Follows from 11) by duality.
\end{proof}

\begin{definition}
Let $L\arr{f}M\arr{g}N$ be a sequence of morphisms in an additive category $\cd$. We say that $f$ is a \emph{weak kernel}\index{weak!kernel} of $g$ if the induced sequence
\[\cd(?,L)\ra\cd(?,M)\ra\cd(?,N)
\]
is exact, \ie if the following conditions hold:
\begin{enumerate}[1)]
\item $gf=0$.
\item If $h$ is a morphism such that $gh=0$, then there exists a morphism $h'$ such that $fh'=h$.
\[\xymatrix{L\ar[r]^f & M\ar[r]^g & N \\
 K\ar[ur]_{h}\ar@{.>}[u]^{h'}\ar@/_1pc/[urr]_{0} &&
}
\]
\end{enumerate}
\bigskip

Dually, $g$ is a \emph{weak cokernel}\index{weak!cokernel} of $f$ is the induced sequence
\[\cd(N,?)\ra\cd(M,?)\ra\cd(L,?)
\]
is exact, \ie, the following conditions hold:
\begin{enumerate}[1)]
\item $gf=0$.
\item If $h$ is a morphism such that $hf=0$, then there exists a morphism $h'$ such that $h'g=h$.
\[\xymatrix{L\ar[r]^f\ar@/_1pc/[drr]_{0} & M\ar[r]^g\ar[dr]_{h} & N\ar@{.>}[d]^{h'} \\
 &&C
}
\]
\end{enumerate}
The sequence $L\arr{f}M\arr{g}N$ is \emph{weakly exact}\index{weakly!exact} if $f$ is a weak kernel of $g$ and $g$ is a weak cokernel of $f$.
\end{definition}

\begin{proposition}\label{properties of additive torsion pairs}
Let $\cd$ be an additive category and let $(\cx,\cy)$ be a pair of strictly full subcategories of $\cd$ such that:
\begin{enumerate}[i)]
\item $\cd(X,Y)=0$ for each $X\in\cx$ and $Y\in\cy$.
\item The inclusion functor $x:\cx\ra\cd$ has a right adjoint $\tau_{\cx}$\index{$\tau_{\cx}$}. We put 
\[\theta_{\cx}:\cd(xN,M)\arr{\sim}\cx(N,\tau_{\cx}M)
\] 
for the adjunction isomorphism, $\eta_{\cx}$\index{$\eta_{\cx}$} for the unit and $\delta_{\cx}$\index{$\delta_{\cx}$} for the counit.
\item The inclusion functor $y:\cy\ra\cd$ has a left adjoint $\tau^{\cy}$\index{$\tau^{\cy}$}. We put 
\[\theta^{\cy}:\cy(\tau^{\cy}N,M)\arr{\sim}\cd(N,yM)
\] 
for the adjunction isomorphism, $\eta^{\cy}$\index{$\eta^{\cy}$} for the unit and $\delta^{\cy}$\index{$\delta^{\cy}$} for the counit.
\item For each $M\in\cd$ the sequence
\[x\tau_{\cx}M\arr{\delta_{\cx,M}}M\arr{\eta^{\cy}_{M}}y\tau^{\cy}M
\]
is weakly exact.
\end{enumerate}
Then, the following assertions hold:
\begin{enumerate}[1)]
\item $\cx=\ ^{\bot}\cy$ and $\cy=\cx^{\bot}$.
\item For an object $M\in\cd$ the following properties are equivalent:
\begin{enumerate}[a)]
\item $M$ belongs to $\cx$. 
\item $\delta_{\cx,M}$ is an isomorphism. 
\item $\tau^{\cy}M=0$.
\end{enumerate}
\item For an object $M\in\cd$ the following properties are equivalent:
\begin{enumerate}[a)]
\item $M$ belongs to $\cy$. 
\item $\eta^{\cy}_{M}$ is an isomorphism. 
\item $\tau_{\cx}M=0$.
\end{enumerate}
\item The endofunctors $x\tau_{\cx}$ and $y\tau^{\cy}$ are idempotent.
\item For an object $M\in\cd$ the following properties are equivalent:
\begin{enumerate}[a)]
\item $\delta_{\cx,M}$ is a section. 
\item $\eta^{\cy}_{M}$ is a retraction.
\end{enumerate}
In this case, $M\cong x\tau_{\cx}M\oplus y\tau^{\cy}M$.
\end{enumerate}
\end{proposition}
\begin{proof}
1) The inclusion $\cx\subseteq\ ^{\bot}\cy$ is clear. Conversely, let $M\in\ ^{\bot}\cy$. Then $\eta^{\cy}_{M}=0$ and, since $\delta_{\cx,M}$ is a weak kernel of $\eta^{\cy}_{M}$, we have that the identity morphism $\id_{M}$ factors through $\delta_{\cx,M}$. Hence, $\delta_{\cx,M}$ is a retraction and, by the lemma above, it is in fact an isomorphism. This implies that $M\in\cx$. Dually, one gets $\cy=\cx^{\bot}$.

2) Thanks to the proof of 1) we already know that an object $M\in\cd$ belongs to $\cx$ if and only if $\delta_{\cx,M}$ is an isomorphism, and that if $\tau^{\cy}M=0$ then $M$ belongs to $\cx$. Now, if $M\in\cx$ then $\eta^{\cy}_{M}=0$ and, by Lemma \ref{properties of adjunctions}, we have that $\tau^{\cy}M=0$.

3) It follows from 2) by duality.

4) For every $M\in\cd$, since $x\tau_{\cx}M\in\cx$ we have that
\[\delta_{\cx,x\tau_{\cx}M}:x\tau_{\cx}x\tau_{\cx}M\ra x\tau_{\cx}M
\]
is an isomorphism. This proves that $x\tau_{\cx}$ is idempotent. Dually, one proves that $y\tau^{\cy}$ is idempotent.

5) Assume $\delta_{\cx,M}$ is a section, \ie there exists a morphism $r$ such that $r\delta_{\cx,M}=\id$. Since $(\id-\delta_{\cx,M}r)\delta_{\cx,M}=0$, there exists a morphism $s$ such that $s\eta^{\cy}_{M}=\id-\delta_{\cx,M}r$. We depict the situation:
\[\xymatrix{x\tau_{\cx}M\ar[r]^{\delta_{\cx,M}}\ar@/_1pc/[dr]_{0} & M\ar[d]_{\id-\delta_{\cx,M}r}\ar[r]^{\eta^{\cy}_{M}} & y\tau^{\cy}M\ar@{.>}[ld]^{s} \\
& M &
}
\]
Hence, $\eta^{\cy}_{M}=\eta^{\cy}_{M}s\eta^{\cy}_{M}$ and, by Lemma \ref{properties of adjunctions}, we have $\eta^{\cy}_{M}s=\id$, \ie $\eta^{\cy}_{M}$ is a retraction.
Notice that, since $\delta_{\cx,M}$ is a section then it is the kernel of $\eta^{\cy}_{M}$. Conversely, since $\eta^{\cy}_{M}$ is a retraction, then it is the cokernel of $\delta_{\cx,M}$. In this case we have that $M\cong x\tau_{\cx}M\oplus y\tau^{\cy}M$, thanks to Lemma \ref{split in additive categories} below.
\end{proof}

\begin{lemma}\label{split in additive categories}
Let $\cd$ be an additive category, and let
\[L\arr{f}M\arr{g}N
\]
be a pair of morphisms in $\cd$ such that:
\begin{enumerate}[1)]
\item it is an exact pair, \ie $f$ is the kernel of $g$ and $g$ is the cokernel of $f$,
\item $f$ is a section, \ie there exists a morphism $\rho$ such that $\rho f=\id_{L}$,
\item $g$ is a retraction, \ie there exists a morphism $\sigma$ such that $g\sigma=\id_{N}$.
\end{enumerate}
Then, the morphism
\[\left[\begin{array}{c}\rho \\ g\end{array}\right]:M\ra L\oplus N
\]
is an isomorphism making commutative the following diagram
\[\xymatrix{L\ar[rr]^{f}\ar[d]^{\id_{L}} && M\ar[rr]^{g}\ar[d] && N\ar[d]^{\id_{N}} \\
L\ar[rr]_{\scriptsize{\left[\begin{array}{c}\id_{L} \\ 0\end{array}\right]}} && L\oplus N\ar[rr]_{\scriptsize{\left[\begin{array}{cc}0&\id_{N}\end{array}\right]}} && N
}
\]
\end{lemma}
\begin{proof}
Put $\sigma':=(\id_{M}-f\rho)\sigma$. Of course, the morphism
\[\left[\begin{array}{cc}f&\sigma'\end{array}\right]:L\oplus N\ra M
\]
is a right inverse to the morphism
\[\left[\begin{array}{c}\rho \\ g\end{array}\right]:M\ra L\oplus N.
\]
Conversely, we have
\[\left[\begin{array}{cc}f&\sigma'\end{array}\right] \left[\begin{array}{c}\rho \\ g\end{array}\right]=f\rho+\sigma' g,
\]
and it turns out that this sum is $\id_{M}$. Indeed, since $g(\id_{M}-f\rho-\sigma' g)=0$, there exists a morphism $h:M\ra L$ such that $\id_{M}-f\rho-\sigma' g=fh$. Therefore, $h=\rho fh=0$.
\end{proof}

In the sight of Proposition \ref{properties of additive torsion pairs}, the following seems to be a natural definition.

\begin{definition}
Let $\cd$ be an additive category and let $(\cx,\cy)$ be a pair of strictly full subcategories. We say that it is an \emph{(additive) torsion pair}\index{pair!torsion!(additive)} on $\cd$ if it satisfies the conditions i)--iv) of the proposition above. The sequences 
\[x\tau_{\cx}M\arr{\delta_{\cx,M}}M\arr{\eta^{\cy}_{M}}y\tau^{\cy}M
\] 
will be called \emph{torsion sequences}\index{torsion!sequence}. The functor $\tau_{\cx}$ (respectively, $\tau^{\cy}$) is said to be the \emph{torsion functor}\index{torsion!functor} (respectively, \emph{torsionfree functor}\index{torsionfree!functor}) associated to the torsion pair $(\cx,\cy)$.
\end{definition}

\section{(Co)suspended, triangulated and pretriangulated categories}\label{(Co)suspended, triangulated and pretriangulated categories}
\addcontentsline{lot}{section}{1.3. Categor\'ias (co)suspendidas, trianguladas y pretrianguladas}

We recall here the definition of \emph{(co)suspended category}, due to B.~Keller and D.~Vossieck \cite{KellerVossieck1987}, but also used by A. Beligiannis and N. Marmaridis \cite{BeligiannisMarmaridis1994} under the name of \emph{left/right triangulated category}:

\begin{definition}\label{definition suspended category}
Let $\cd$ be an additive category endowed with an additive endofunctor $\Sigma:\cd\ra\cd$. We put $\cR\ct(\cd,\Sigma)$ for the category whose objects are diagrams in $\cd$ of the form $L\arr{f} M\arr{g} N\arr{h}\Sigma L$ and a morphism between two such diagrams is a triple of morphisms $(l,m,n)$ making the following diagram commutative:
\[\xymatrix{L\ar[r]^{f}\ar[d]^{l} & M\ar[r]^{g}\ar[d]^m & N\ar[r]^{h}\ar[d]^n & \Sigma L\ar[d]^{\Sigma l} \\
L'\ar[r]^{f'} & M'\ar[r]^{g'} & N'\ar[r]^{h'} & \Sigma L 
}
\]
A \emph{suspended or right triangulated category}\index{category!suspended}\index{category!right triangulated} is an additive category $\cd$ endowed with an additive endofunctor $\Sigma:\cd\ra\cd$, called \emph{suspension functor}\index{functor!suspension}, and a strictly full subcategory of $\cR\ct(\cd,\Sigma)$, whose objects are called \emph{right triangles}\index{triangle!right}, satisfying the following axioms:
\begin{enumerate}[RT1)]
\item For each object $M\in\cd$ the sequence $0\ra M\arr{\id_{M}}M\ra 0$ is a right triangle.
\item If $L\arr{f}M\arr{g}N\arr{h}\Sigma L$ is a right triangle, then so is $M\arr{g}N\arr{h}\Sigma L\arr{-\Sigma f}\Sigma M$.
\item If $L\arr{f}M\arr{g}N\arr{h}\Sigma L$ and $L'\arr{f'}M'\arr{g'}N'\arr{h'}\Sigma L'$ are right triangles and $l$ and $m$ are morphisms such that $mf=f'l$, then there is a morphism $n$ such that $(l,m,n)$ is a morphism of triangles
\[\xymatrix{L\ar[r]^{f}\ar[d]^{l} & M\ar[r]^{g}\ar[d]^{m} & N\ar[r]^{h}\ar@{.>}[d]^{n} & \Sigma L\ar[d]^{\Sigma l} \\
L'\ar[r]^{f'} & M'\ar[r]^{g'} & N'\ar[r]^{h'} & \Sigma L'
}
\]
\item For each pair of morphisms $L\arr{f}M\arr{g}N$ there is a commutative diagram
\[\xymatrix{L\ar[r]^{f}\ar@{=}[d] & M\ar[r]^{m}\ar[d]^{g} & N'\ar[r]\ar[d] & \Sigma L\ar@{=}[d] \\
L\ar[r]^{gf} & N\ar[r]\ar[d] & M'\ar[r]\ar[d] & \Sigma L\ar[d]^{\Sigma f} \\
& L'\ar@{=}[r]\ar[d]^{l} & L'\ar[r]^{l}\ar[d] & \Sigma M \\
& \Sigma M\ar[r]^{\Sigma m} & \Sigma N' &
}
\]
such that the two first rows and the two central columns are right triangles.
\end{enumerate}
A \emph{full suspended subcategory}\index{subcategory!suspended} $\cx$ of a suspended category $(\cd,\Sigma)$ consists of a full additive subcategory $\cx$ closed under $\Sigma$ and under extensions, \ie if
\[L\arr{f}M\arr{g}N\arr{h}\Sigma L
\]
is a right triangle of $\cd$ with $L\ko N\in\cx$, then $\Sigma L\ko M\in\cx$.

The dual notions are that of \emph{cosuspended or left triangulated category}\index{category!cosuspended}\index{category!left triangulated} and \emph{full cosuspended subcategory}\index{subcategory!cosuspended}. In this case the endofunctor is the \emph{loop functor}\index{functor!loop}, denoted by $\Omega:\cd\ra\cd$, and the distinguished sequences are \emph{left triangles}\index{triangle!left}.
\end{definition}

\begin{remark}
Let $\cd$ be a suspended category. For each right triangle $L\arr{f}M\arr{g}N\arr{h}\Sigma L$ of $\cd$ and each object $V\in\cd$, the induced sequence
\[\dots\ra\cd(\Sigma L,V)\ra\cd(N,V)\ra\cd(M,V)\ra\cd(L,V)
\]
is exact. In particular, $gf=hg=(\Sigma f)h=0$. See \cite[Proposition 1.5.3]{KashiwaraShapira} for a proof.
\end{remark}

(Co)suspended categories arise in nature as quotients of exact categories in the sense of D. Quillen \cite{Quillen73}. 

\begin{definition}\label{ideal of an additive category}
A \emph{two-sided ideal}\index{ideal!two-sided... of an additive category} of an additive category $\cc$ is a class $\ci$ of morphisms of $\cc$ such that:
\begin{enumerate}
\item[I.1)] For any two objects $M\ko N\in\cc$, the set $\ci\cap\cc(M,N)$ is a subgroup of $\cc(M,N)$.
\item[I.2)] In any sequence of three composable morphisms
\[L\arr{f}M\arr{g}N\arr{h}U
\]
we have $hgf\in\ci$ whenever $g\in\ci$.
\end{enumerate}
\end{definition}

\begin{definition}\label{quotient category}
If $\ci$ is a two-sided ideal of the additive category $\cc$, we define the \emph{quotient category}\index{category!quotient (by a two-sided ideal)} $\cc/\ci$ as follows: it has the same objects as $\cc$ and its morphisms space $(\cc/\ci)(M,N)$ is the quotient of the abelian group $\cc(M,N)$ by the subgroup $\ci\cap\cc(M,N)$. In this category the composition is given by the composition of the representatives. Associated to the quotient category we have the \emph{quotient functor}
\[\cc\ra\cc/\ci,
\]
which takes the object $X$ to itself, and the morphism $f$ to its class $\ol{f}$.
\end{definition}

It is well-known that the quotient category $\cc/\ci$ is also an additive category making the quotient functor into an additive functor.

\begin{definition}\label{stable category}
The \emph{stable category}\index{category!stable} of an additive category $\cc$ associated to a full additive subcategory $\ct$ is the quotient $\cc/\langle\ct\rangle$ of $\cc$ by the two-sided ideal $\langle\ct\rangle$\index{$\langle\ct\rangle$} of morphisms factoring through an object of $\ct$. When $\ct$ is clear, the stable category is denoted by $\ul{\cc}$. 
\end{definition}

\begin{definition}\label{exact category}
An \emph{exact category}\index{category!exact} is an additive category $\cc$ endowed with a distinguished class of sequences $\ce$ closed under isomorphisms
\[L\arr{i}M\arr{p}N,
\]
such that $(i,p)$ is an \emph{exact pair}\index{pair!exact}, \ie $i$ is the kernel of $p$ and $p$ is the cokernel of $i$. Following \cite{GabrielRoiter}, the morphisms $p$ are called \emph{deflations}\index{deflation}, the morphisms $i$ \emph{inflations}\index{inflation} and the pairs $(i, p)$ \emph{conflations}\index{conflation}. The class of conflations have to satisfy the following axioms:
\begin{enumerate}
\item[Ex0\ )] The identity morphism of the zero object is a deflation.
\item[Ex1\ )] The composition of two deflations is a deflation.
\item[Ex1')] The composition of two inflations is an inflation.
\item[Ex2\ )] Deflations admit and are stable under base change (\ie pullbacks).
\item[Ex2')] Inflations admit and are stable under cobase change (\ie pushouts).
\end{enumerate}
\end{definition}

\begin{remark}\label{Keller Quillen}
As shown by B. Keller \cite{Keller1990}, these axioms are equivalent to D. Quillen's \cite{Quillen73} and they imply that if $\cc$ is
small, then there is a fully faithful functor from $\cc$ into an ambient abelian category $\cc'$ whose image
is an additive subcategory closed under extensions and such that a sequence of $\cc$ is a
conflation if and only if its image is a short exact sequence of $\cc'$. Conversely, one easily checks that
an extension closed full additive subcategory $\cc$ of an abelian category $\cc'$ endowed with all exact pairs which induce short exact sequences in $\cc'$ is always exact.
\end{remark}

\begin{definition}
An object $I$ of an exact category $(\cc,\ce)$ is \emph{$\ce$-injective}\index{object!injective} if the sequence
\[\cc(M,I)\arr{i^{\che}}\cc(L,I)\ra 0
\]
is exact for each conflation $(i,p)\in\ce$. The exact category has \emph{enough $\ce$-injectives}\index{enough!injectives} if for each object $M$ there exists a conflation
\[M\arr{i_{M}}IM\arr{p_{M}}\Sigma M
\]
such that $IM$ is $\ce$-injective. Dually, one has the notion of \emph{enough $\ce$-projectives}\index{enough!projectives}. An exact category with enough $\ce$-injectives, enough $\ce$-projectives and in which an object is $\ce$-injective if and only if it is $\ce$-projective is said to be a \emph{Frobenius category}\index{category!Frobenius}.
\end{definition}

Recall that if $(\cc,\ce)$ is an exact category with enough $\ce$-injectives, then its stable category $\ul{\cc}$ associated to the subcategory of the $\ce$-injective objects is an additive category making the quotient functor $\cc\ra\ul{\cc}$ into an additive functor. However, it hardly carries an exact structure making the quotient functor into an \emph{exact functor}\index{functor!exact} (\ie an additive functor taking conflations to conflations). Nevertheless, it admits a natural structure of suspended category \cite{KellerVossieck1987, Keller1996}. The suspension functor $\Sigma$ of $\ul{\cc}$ is obtained by choosing a conflation 
\[M\arr{i_{M}}IM\arr{p_{M}}\Sigma M
\]
for each object $M$, with $IM$ being $\ce$-injective. Each triangle is isomorphic to a standard triangle $(\ol{i},\ol{p},\ol{e})$ obtained by embedding a conflation $(i,p)$ into a commutative diagram
\[\xymatrix{L\ar[r]^{i}\ar[d]^{\id} & M\ar[r]^{p}\ar[d] & N\ar[d]^{e} \\
L\ar[r]^{i_{L}} & IL\ar[r]^{p_{L}} & \Sigma L
}
\]
Dually, the stable category of an exact category with enough $\ce$-projectives associated to the subcategory of the $\ce$-projective objects becomes a cosuspended category.

Therefore, if we start with an exact structure and then we quotient by those objects which behave like injectives with respect to this exact structure, we get a suspended category. As pointed out by A. Beligiannis and N. Marmaridis \cite{BeligiannisMarmaridis1994}, sometimes one can also get suspended categories by reversing this procedure. Namely, we can start with a certain class of objects and then to pick up those short sequences which regard these objects as `injectives'. More precisely, we can consider quotients of abelian categories by covariantly finite subcategories in the sense of M.~Auslander and S.~O.~Smal{\o} \cite{AuslanderSmalo1980}.

\begin{definition}\label{preenvelope}
Let $\cc$ be an additive category and let $\ct$ be a full subcategory of $\cc$. A morphism $i_{M}:M\ra T_{M}$ in $\cc$ is a \emph{left $\ct$-approximation}\index{approximation!left} \cite{AuslanderSmalo1980} or a \emph{$\ct$-preenvelope}\index{preenvelope} \cite{Enoch1981} of $M$ if $T_{M}$ is an object of $\ct$ and the induced map
\[\cc(T_{M},T)\ra\cc(M,T)
\]
is surjective for every object $T$ of $\ct$. The subcategory $\ct$ is \emph{covariantly finite}\index{subcategory!covariantly finite} \cite{AuslanderSmalo1980} in $\cc$ if any object of $\cc$ admits a $\ct$-preenvelope. In this case, the stable category $\ul{\cc}$ is the quotient of $\cc$ by the ideal of morphisms which factors through an object of $\ct$. The dual notions are \emph{right $\ct$-approximation}\index{approximation!right} or \emph{$\ct$-precover}\index{precover}, and \emph{contravariantly finite}\index{subcategory!contravariantly finite}. The subcategory $\ct$ is \emph{functorially finite}\index{subcategory!functorially finite} in $\cc$ if it is both covariantly and contravariantly finite.
\end{definition}

If $\ct$ is a covariantly finite subcategory in an abelian category $\cc$, then the corresponding stable category $\ul{\cc}$ admits a structure of right triangulated category \cite{BeligiannisMarmaridis1994}. Indeed, the suspension functor $\Sigma$ is constructed by fixing, for each object $M$, a right exact sequence
\[M\arr{i_{M}}T_{M}\arr{i^c_{M}}\Sigma M\ra 0,
\]
where $i_{M}$ is a $\ct$-preenvelope of $M$. The right triangles are those objects of $\cR\ct(\ul{\cc},\Sigma)$ isomorphic to `induced triangles', constructed as follows: let
\[L\arr{i}M\arr{p}N\ra 0
\]
be a right exact sequence such that
\[\cc(M,T)\arr{i^{\che}}\cc(L,T)\ra 0
\]
is exact for each $T\in\ct$. Then we have a commutative diagram
\[\xymatrix{L\ar[r]^{i}\ar[d]^{\id} & M\ar[r]^{p}\ar[d] & N\ar[d]^{e} \\
L\ar[r]^{i_{L}} & T_{L}\ar[r]^{i^c_{L}} & \Sigma L
}
\]
which produces the induce triangle
\[L\arr{\ol{i}}M\arr{\ol{p}}N\arr{\ol{e}}\Sigma L.
\]
Alternatively \cite[Proposition 2.9]{BeligiannisMarmaridis1994}, one can take the right triangles to be those objects of $\cR\ct(\ul{\cc},\Sigma)$ isomorphic to `distinguished triangles', constructed as follows: given an arbitrary morphism $f\in\cc(L,M)$, we have a commutative diagram
\[\xymatrix{L\ar[r]^{i_{L}}\ar[d]^{f} & T_{L}\ar[r]^{i^c_{L}}\ar[d] & \Sigma L\ar[d]^{\id} \\
M\ar[r]^{g} & N\ar[r]^{h} & \Sigma L
}
\]
in which the left square is cocartesian and which produces the distinguished triangle
\[L\arr{\ol{f}}M\arr{\ol{g}}N\arr{\ol{h}}\Sigma L.
\]
Dually, the stable category of an abelian category associated to a contravariantly finite subcategory becomes a left triangulated category.

When an additive category admits both a structure of right and left triangulated category in a compatible way, then it is a \emph{pretriangulated category}. This notion is due to A. Beligiannis \cite{Beligiannis2001} and it is inspired by the \emph{pretriangulated categories} in the sense M. Hovey \cite{Hovey1999} (do not confuse with the \emph{pretriangulated categories} in the sense of A. I. Bondal and M. M. Kapranov \cite{BondalKapranov1990} or in the sense of A. Neeman \cite{Neeman2001}).

\begin{definition}
A \emph{pretriangulated category}\index{category!pretriangulated} is an additive category $\cd$ endowed with an adjoint pair $(\Sigma,\Omega)$ of additive endofunctors such that:
\begin{enumerate}[PT1)]
\item The pair $(\cd,\Omega)$ can be completed to a left triangulated category.
\item The pair $(\cd,\Sigma)$ can be completed to a right triangulated category.
\item Given a right triangle $L\arr{f}M\arr{g}N\arr{h}\Sigma L$, a left triangle $\Omega N'\arr{f'}L'\arr{g'}M'\arr{h'}N'$ and morphisms $l, m, n', l'$ such that $mf=f'l$ and $l'h=h'n'$, there exist morphisms $n$ and $m'$ making the following diagrams commutative
\[\xymatrix{L\ar[r]^{f}\ar[d]^{l} & M\ar[r]^{g}\ar[d]^{m}&N\ar[r]^{h}\ar@{.>}[d]^{n}&\Sigma L\ar[d]^{\delta_{N'}(\Sigma l)} \\
\Omega N'\ar[r]^{f'}&L'\ar[r]^{g'}&M'\ar[r]^{h'}&N'
}\]
\[\xymatrix{L\ar[r]^{f}\ar[d]_{(\Omega l')\eta_{L}} & M\ar[r]^{g}\ar@{.>}[d]_{m'}&N\ar[r]^{h}\ar[d]_{n'}&\Sigma L\ar[d]_{l'} \\
\Omega N'\ar[r]^{f'}&L'\ar[r]^{g'}&M'\ar[r]^{h'}&N'
}
\]
where $\delta$ is the counit and $\eta$ is the unit of the adjoint pair $(\Sigma,\Omega)$.
\end{enumerate}
\end{definition}

\begin{example}\label{abelian categories are pretriangulated}
An abelian category $\cd$ becomes pretriangulated when we take $\Sigma=\Omega=0$ and we take the exact sequences of the form $L\ra M\ra N\ra 0$ (respectively, $0\ra L\ra M\ra N$) to be the right (respectively, left) triangles. 
\end{example}

\begin{example}
The homotopy category of an additive closed model category in the sense of D. Quillen \cite{Quillen1967,Hovey1999} becomes a pretriangulated category. In particular, the stable category of an abelian category associated to a functorially finite category is pretriangulated \cite{Beligiannis2000, Beligiannis2001}.
\end{example}

It is easy to prove that the following definition agrees with the standard one \cite{Verdier1996}:

\begin{definition}\label{triangulated category}
A \emph{triangulated category}\label{category!triangulated} is a suspended category $(\cd,\Sigma)$ in which the suspension functor $\Sigma$ is an equivalence. In this case, we put $\Sigma=?[1]$\index{$?[1]$} and call it \emph{shift functor}\index{functor!shift}. A quasi-inverse will be denoted by $?[-1]$\index{$?[-1]$}, and right triangles are called \emph{distinguished triangles}\index{triangle!distinguished} or simply \emph{triangles}\index{triangle}. Alternatively, one can also use cosuspended categories to define triangulated categories. A \emph{full triangulated subcategory}\index{subcategory!triangulated} $\cx$ of a triangulated category $(\cd,?[1])$ is a full additive subcategory $\cx$ of $\cd$ which is closed under shifts and closed under extensions, \ie if
\[L\arr{f}M\arr{g}N\arr{h}L[1]
\]
is a triangle of $\cd$ with $L\ko N\in\cx$, then $L[1]\ko L[-1]\ko M\in\cx$.
\end{definition}

Notice that a triangulated category $(\cd,[1])$ always gives rise to a pretriangulated category $(\cd,\Sigma,\Omega)$ in which $\Sigma=[1]$, $\Omega=[-1]$ and both the right triangles and the left triangles are given by the distinguished triangles. 

\begin{definition}\label{triangle functor}
Let $\cd$ and $\ct$ be triangulated categories. A \emph{triangle functor}\index{functor!triangle} from $\cd$ to $\ct$ is a pair $(F,\alpha)$ such that $F:\cd\ra\ct$ and $\alpha$ defines a natural transformation 
\[\alpha_{M}: F(M[1])\ra (FM)[1]
\]
such that for each triangle
\[L\arr{f}M\arr{g}N\arr{h}L[1]
\]
of $\cd$ we have a triangle
\[FL\arr{Ff}FM\arr{Fg}FN\arr{\alpha_{L}(Fh)}(FL)[1]
\]
in $\ct$.
\end{definition}

\begin{remark}
If $(F,\alpha)$ is a triangle functor, then $F$ is additive (\cf \cite[Remarque II.1.2.7]{Verdier1996}) and $\alpha$ is an isomorphism of functors (\cf \cite[section 8]{Keller1996}).
\end{remark}

\begin{definition}\label{triangle equivalence}
A triangle functor $(F,\alpha)$ is a \emph{triangle equivalence}\index{equivalence!triangle} if $F$ is an equivalence of categories. 
\end{definition}

\begin{remark}
Despite the former definition is not `correct' \emph{prima facie}, since it seems to forget the presence of $\alpha$, it turns out to be equivalent to the `correct one' (\cf \cite{KellerVossieck1987}). 
\end{remark}

The following is a very important example of triangulated category:

\begin{example}
Thanks to D. Happel's theorem (\cf \cite[Theorem 2.6]{Happel1988}), the stable category of a Frobenius category is always triangulated. According to 
B.~Keller (see for instance \cite{Keller2007}), we say that a triangulated category is \emph{algebraic}\index{category!triangulated!algebraic} if it is triangle equivalent to the stable category of a Frobenius category.
\end{example}

In pretriangulated categories we also have a reasonable notion of `torsion pair' \cite{BeligiannisReiten2001}:

\begin{definition}\label{definition pretriangulated category}
Let $(\cd,\Sigma,\Omega)$ be a pretriangulated category. A pair $(\cx,\cy)$ of strictly full triangulated subcategories of $\cd$ is a \emph{pretriangulated torsion pair}\index{pair!torsion!pretriangulated} in $\cd$ if:
\begin{enumerate}[1)]
\item $\cd(X,Y)=0$ for each $X\in\cx$ and $Y\in\cy$.
\item $\Sigma\cx\subseteq\cx$ and $\Omega\cy\subseteq\cy$.
\item For each object $M\in\cd$ there exists a left triangle
\[\Omega(M^{\cy})\arr{g_{M}}M_{\cx}\arr{f_{M}}M\arr{g^{M}}M^{\cy}
\]
and a right triangle
\[M_{\cx}\arr{f_{M}}M\arr{g^M}M^{\cy}\arr{f^{M}}\Sigma(M_{\cx})
\]
with $M_{\cx}\in\cx$ and $M^{\cy}\in\cy$.
\end{enumerate}
If $(\cx,\cy)$ is a pretriangulated torsion pair in a triangulated category $(\cd,?[1])$ regarded as a pretriangulated category, then we say that $(\cx,\cy)$ is a \emph{triangulated torsion pair}\index{pair!torsion!triangulated} and that $(\cx,\cy[1])$ is a \emph{t-structure}\index{t-structure}, according with A.~A.~Beilinson, J.~Bernstein and P.~Deligne \cite{BeilinsonBernsteinDeligne}.
\end{definition}

\begin{remark}
In the definition above we have that:
\begin{enumerate}[i)]
\item The map $M\mapsto M_{\cx}$ underlies a functor $\tau_{\cx}:\cd\ra\cx$ which is right adjoint to the inclusion.
\item The map $M\mapsto M^{\cy}$ underlies a functor $\tau^{\cy}:\cd\ra\cy$ which is left adjoint to the inclusion.
\item For each $M\in\cd$ we have that $M_{\cx}\arr{f_{M}}M\arr{g^M}M^{\cy}$ is a weakly exact.
\end{enumerate}
Therefore, every pretriangulated torsion pair is an additive torsion pair.
\end{remark}

When dealing with t-structures, it is extremely useful the point of view developed by B.~Keller and D.~Vossieck \cite{KellerVossieck88b}. They introduced the following definition:

\begin{definition}\label{aisle}
An \emph{aisle}\index{aisle} in a triangulated category $\cd$ is a strictly full suspended subcategory $\cx$ of $\cd$ whose corresponding inclusion functor has a right adjoint. Dually, a \emph{coaisle}\index{coaisle} in $\cd$ consists of a strictly full cosuspended subcategory whose corresponding inclusion functor admits a left adjoint.
\end{definition}

We know that if $(\cx,\cy[1])$ is a t-structure on a triangulated category $\cd$, then $\cx$ (respectively, $\cy$) is an aisle (respectively, a coaisle) in $\cd$, called the \emph{aisle (respectively, coaisle) of the t-structure}. The point is that the converse also holds, as B.~Keller and D.~Vossieck pointed out in \cite[Proposition 1.1]{KellerVossieck88b}

\begin{proposition}\label{aisles are t-structures}
Let $\cd$ be a triangulated category. The map 
\[\cx\mapsto (\cx,\cx^{\bot}[1])
\]
induces a one-to-one correspondence between the class of aisles in $\cd$ and the class of t-structures on $\cd$.
\end{proposition}

From now on, we will use without mention the fact that aisles `are' t-structures.

\section{Compatible torsion theories}\label{Compatible torsion theories}
\addcontentsline{lot}{section}{1.4. Teor\'ias de torsi\'on compatibles}

Abelian categories are particular cases of additive categories with `canonical factorizations':

\begin{definition}
An additive category $\cd$ has \emph{canonical factorizations}\index{factorization!canonical} if: 
\begin{enumerate}[1)]
\item Every morphism $f$ factors as the composition $f=ip$ of an epimorphism $p$ followed by a monomorphism $i$.
\item If $f: M\arr{p} L\arr{i} N$ and $f:M\arr{p'}L'\arr{i'}N$ are two such factorizations, then there exists an isomorphism $g:L\arr{\sim}L'$ such that the following diagram is commutative
\[\xymatrix{M\ar[rr]^{f\ \ \ \ }\ar@{=}[dd]\ar[dr]_{p} & & N\ar@{=}[dd] \\
&L\ar[ur]_{i}\ar[dd]^{g}& \\
M\ar[rr]^{f\ \ \ \ }\ar[dr]_{p'}&&N \\
&L'\ar[ur]_{i'}&
}
\]
\end{enumerate}
\end{definition}

\begin{remark}\label{canonical factorizations implies balanced}
Notice that additive categories with canonical factorizations are \emph{balanced}, \ie a morphism is an isomorphism if and only if it is both an epimorphism and a monomorphism.
\end{remark}

If the additive category $\cd$ has additional structure, then the most interesting additive torsion pairs are those compatible with the extra structure, which allows us to recover the notions of: torsion pairs for abelian or exact categories, pretriangulated torsion pairs, t-structures,\dots The following is a sample of these compatible torsion pairs.

\begin{proposition}\label{compatibility proposition}
Let $(\cx,\cy)$ be an additive torsion pair on an additive category $\cd$. Then:
\begin{enumerate}[1)]
\item If $\cd$ has canonical factorizations, then every torsion sequence yields a short exact sequence
\[0\ra x\tau_{\cx}M\arr{\delta_{\cx,M}}M\arr{\eta^{\cy}_{M}}y\tau^{\cy}M\ra 0.
\]
\item If $\cd$ is an exact category, then $\delta_{\cx,M}$ is an inflation if and only if $(\delta_{\cx,M},\eta^{\cy}_{M})$ is a conflation if and only if $\eta^{\cy}_{M}$ is a deflation.
\item If $(\cd,\Sigma,\Omega)$ is a pretriangulated category, then $\Sigma(\cx)\subseteq\cx$ if and only if $\Omega(\cy)\subseteq\cy$. In this case, the pair $(\cx,\cy)$ is a pretriangulated torsion pair if and only if every torsion sequence is both the beginning of a right triangle and the end of a left triangle.
\item If $(\cd,?[1])$ is a triangulated category, then $(\cx,\cy[1])$ is a t-structure on $\cd$ if and only if $\cx[1]\subseteq\cx$.
\end{enumerate}
\end{proposition}
\begin{proof}
1) Let $\delta_{\cx,M}=pi$ be a canonical factorization of $\delta_{\cx,M}$,
\[\xymatrix{x\tau_{\cx}M\ar[dr]_{p}\ar[rr]^{\delta_{\cx,M}} & & M\ar[r]^{\eta^{\cy}_{M}}& y\tau^{\cy}M \\
& T\ar[ur]_{i} & &
}
\]
Since $p$ is an epimorphism and $x\tau_{\cx}M\in\cx=\ ^{\bot}\cy$, then $T\in\ ^{\bot}\cy=\cx$ and so $T=xT'$ for some $T'\in\cx$. By using the isomorphisms
\[\cd(T,M)\underset{\sim}{\arr{\theta_{\cx}}}\cx(T',\tau_{\cx}M)\underset{\sim}{\arr{x}}\cd(T,x\tau_{\cx}M)
\]
we map $i$ to a morphism $j\in\cd(T,x\tau_{\cx}M)$ such that $\delta_{\cx,M}j=i$. Hence $\delta_{\cx,M}jp=\delta_{\cx,M}$. Since $x$ is full, by Lemma \ref{properties of adjunctions} we have $jp=\id$, which implies that $p$ is an isomorphism and thus $\delta_{\cx,M}$ is a monomorphism. Therefore, $\delta_{\cx,M}$ is the kernel of $\eta^{\cy}_{M}$. Dually, $\eta^{\cy}_{M}$ is the cokernel of $\delta_{\cx,M}$.

2) Assume that $\delta_{\cx,M}$ is an inflation, and we have to prove that $\eta^{\cy}_{M}$ is its cokernel. Let $c:M\ra C$ be the cokernel of $\delta_{\cx,M}$. Hence, there exists a morphism $\beta$ such that $\beta\eta^{\cy}_{M}=c$ and a unique morphism $\alpha$ such that $\alpha c=\eta^{\cy}_{M}$.
\[\xymatrix{x\tau_{\cx,M}\ar[r]^{\delta_{\cx,M}} & M\ar[r]^{\eta^{\cy}_{M}}\ar[dr]_{c} & y\tau^{\cy}M\ar@<1ex>[d]^{\beta} \\
&& C\ar@<1ex>[u]^{\alpha}
}
\]
Then, $\beta\alpha c=c$, which implies $\beta\alpha=\id$. Also, $\alpha\beta\eta^{\cy}_{M}=\eta^{\cy}_{M}$, and by Lemma \ref{properties of adjunctions} we have $\alpha\beta=\id$. Hence $\alpha$ and $\beta$ are isomorphisms and $\eta^{\cy}_{M}$ is the cokernel of $\delta_{\cx,M}$. Dually, if $\eta^{\cy}_{M}$ is a deflation, then $\delta_{\cx,M}$ is the corresponding inflation.

3) Assume $\Sigma\cx\subseteq\cx$, and let $Y$ be an object of $\cy$. We want to prove $\Omega Y\in\cy$. This happens if and only if $\cd(X,\Omega Y)=0$ for all $X\in\cx$, which follows by adjunction. Dually, one proves that $\Omega\cy\subseteq\cy$ implies $\Sigma\cx\subseteq\cx$. The remainder of the assertion is just a direct application of the definitions.

4) Since $\cx=\ ^{\bot}\cy$, then $\cx$ is closed under extensions. Therefore, if we assume $\cx$ to be closed under positive shifts, then $\cx$ is a full suspended subcategory of $\cd$ whose inclusion functor admits a right adjoint, \ie $\cx$ is an aisle in $\cd$.
\end{proof}

\begin{remark} Let $\cd$ be an additive category.
\begin{enumerate}[1)]
\item If $\cd$ is abelian, an (abelian) torsion pair on $\cd$ is precisely an additive torsion pair on $\cd$. If we regard the abelian category $\cd$ as a pretriangulated category (\cf Example \ref{abelian categories are pretriangulated}), then the pretriangulated torsion pairs are precisely the additive torsion pairs. Therefore, when dealing with (abelian) torsion pairs on abelian categories, we will refer to them just as \emph{torsion pairs}\index{pair!torsion!(abelian)}, without using the word ``
abelian'' before.
\item If $\cd$ is exact, an exact torsion pair on $\cd$ is precisely an additive torsion pair $(\cx,\cy)$ on $\cd$ such that each $\delta_{\cx,M}$ is an inflation.
\item If $(\cd,\Omega,\Sigma)$ is pretriangulated, a pretriangulated torsion pair on $\cd$ is an additive torsion pair $(\cx,\cy)$ on $\cd$ such that $\Sigma\cx\subseteq\cx$ and every torsion sequence is both at the beginning of a right triangle and at the end of a left triangle.
\item If $(\cd,?[1])$ is a triangulated category, a t-structure is precisely a pair $(\cx,\cy[1])$ such that $(\cx,\cy)$ is an additive torsion pair on $\cd$ satisfying $\cx[1]\subseteq\cx$.
\end{enumerate}
\end{remark}

\section{Split torsion pairs}\label{Split torsion pairs}
\addcontentsline{lot}{section}{1.5. Pares de torsi\'on escindidos}

\begin{lemma}\label{split torsion pair lemma}
Let $(\cx,\cy)$ be a pair of strictly full subcategories of an additive category $\cd$ such that:
\begin{enumerate}[1)]
\item $\cd(X,Y)=0$ for each $X\in\cx\ko Y\in\cy$.
\item Every object $M\in\cd$ is the coproduct of an object of $\cx$ and an object of $\cy$.
\end{enumerate}
Then $(\cx,\cy)$ is a torsion pair.
\end{lemma}
\begin{proof}
For each object $M\in\cd$ we fix a coproduct $M=M_{\cx}\oplus M^{\cy}$ with $M_{\cx}\in\cx$ and $M^{\cy}\in\cy$. The map $M\mapsto M_{\cx}$ defines $\tau_{\cx}$ on objects. Notice that an arbitrary morphism 
\[f:M=M_{\cx}\oplus M^{\cy}\ra N_{\cx}\oplus N^{\cy}=N
\] 
is of the form
\[f=\left[\begin{array}{cc}f_{11} & f_{12} \\ 0& f_{22}\end{array}\right].
\]
Then, the map $f\mapsto f_{11}$ defines $\tau_{\cx}$ on morphisms. It is easy to check that $\tau_{\cx}$ so defined is a functor right adjoint to the inclusion of $\cx$ in $\cd$. Indeed, any morphism $X\ra M=M_{\cx}\oplus M^{\cy}$ from an object $X$ of $\cx$ to an arbitrary object $M$ of $\cd$ is uniquely determined by its component $X\ra M_{\cx}$ since its component $X\ra M^{\cy}$ vanishes. Similarly, we define $\tau^{\cy}$.
\end{proof}

\begin{definition}
A torsion pair $(\cx,\cy)$ on an additive category $\cd$ \emph{splits}\index{pair!torsion!split} if in each torsion sequence
\[x\tau_{\cx}M\arr{\delta_{\cx,M}}M\arr{\eta^{\cy}_{M}}y\tau^{\cy}M
\]
the morphism $\delta_{\cx,M}$ is a section (if and only if $\eta^{\cy}_{M}$ is a retraction). Notice that this implies that $M=x\tau_{\cx}M\oplus y\tau^{\cy}M$.
A pretriangulated torsion pair or a t-structure \emph{splits}, if so does the underlying (additive) torsion pair.
\end{definition}

\begin{remark}
This is the same as requiring that $(\cx,\cy)$ is a pair satisfying conditions 1) and 2) of the lemma above.
\end{remark}

\begin{proposition}
Let $\cd$ be an additive category and let $\cd'$ be a full subcategory closed under finite direct sums and direct summands. If $(\cx,\cy)$ is a split torsion pair on $\cd$, then $(\cd'\cap\cx,\cd'\cap\cy)$ is a split torsion pair on $\cd'$.
\end{proposition}
\begin{proof}
Clearly, the pair $(\cd'\cap\cx,\cd'\cap\cy)$ satisfies the conditions of the Lemma \ref{split torsion pair lemma}.
\end{proof}

\section{Characterization of centrally split TTF triples}\label{Characterization of centrally split TTF triples}
\addcontentsline{lot}{section}{1.6. Caracterizaci\'on de las ternas TTF centralmente escindidas}

\begin{definition}\label{definition weakly balanced}
A pretriangulated category $(\cd,\Omega,\Sigma)$ is \emph{weakly balanced}\index{category!pretriangulated!weakly balanced} if a morphism $f:M\ra N$ is an isomorphism whenever there exist a left triangle of the form
\[\Omega N\ra 0\ra M\arr{f}N
\]
and a right triangle of the form
\[M\arr{f}N\ra 0\ra\Sigma M.
\]
\end{definition}

\begin{remark}\label{triangulated categories are weakly balanced}
If $(\cd,?[1])$ is a triangulated category, then it is weakly balanced. Indeed, if there exists a triangle
\[M\arr{f}N\ra 0\ra M[1],
\]
then we have a morphism of triangles
\[\xymatrix{M\ar[r]^{f}\ar[d]_{f} & N\ar[r]\ar[d]^{\id_{N}} & 0\ar[r]\ar[d]^{\id_{0}} & M[1]\ar[d]^{f[1]} \\
N\ar[r]^{\id_{N}} & N \ar[r] & 0\ar[r] & N[1]
}
\]
But from the fact that the vertical morphisms $\id_{N}$ and $\id_{0}$ are isomorphism, we deduce that the third vertical morphism $f$ is also an isomorphism. In this argument we use the following property of triangulated categories: If in the morphism of triangles
\[\xymatrix{X\ar[r]^{u}\ar[d]^f & Y\ar[r]^{v}\ar[d]^g & Z\ar[r]^{w}\ar[d]^h & X[1]\ar[d]^{f[1]} \\
X'\ar[r]^{u'} & Y'\ar[r]^{v'} & Z'\ar[r]^{w'} & X'[1]
}
\]
both $f$ and $g$ are isomorphisms, then so is $h$. The proof is as follows. For any object $U$ we get a morphism of long exact sequences
\[\xymatrix{\cd(U,X)\ar[r]^{u^{\we}}\ar[d]_{\wr} & \cd(U,Y)\ar[r]^{v^{\we}}\ar[d]_{\wr} & \cd(U,Z)\ar[r]^{w^{\we}}\ar[d]^{h^{\we}} & \cd(U,X[1])\ar[r]\ar[d]_{\wr} &\cd(U,Y[1])\ar[d]_{\wr} \\
\cd(U,X')\ar[r]^{u^{\we}} & \cd(U,Y')\ar[r]^{v^{\we}} & \cd(U,Z')\ar[r]^{w^{\we}} & \cd(U,X'[1])\ar[r] &\cd(U,Y'[1]) 
}
\]
Hence, 5-lemma implies that $h^{\we}$ is an isomorphism for every $U\in\cd$, and Yoneda's Lemma implies that $h$ is an isomorphism. The problem with pretriangulated categories is that we can not extend this kind of exact sequences arbitrarily to both sides. Indeed, if 
\[\Omega N\ra L\ra M\ra N
\]
is a left triangle, then any object $U\in\cd$ induces a long exact sequence
\[\dots\ra\cd(U,\Omega M)\ra\cd(U,\Omega M)\ra\cd(U,\Omega N)\ra\cd(U,L)\ra\cd(U,M)\ra\cd(U,N),
\]
but nothing more is guaranteed for the right side.
\end{remark}

Chapter \ref{Balance in stable categories} is devoted to the study of (weak) balance in stable categories of abelian categories.

\begin{definition}\label{TTF triple}
Let $\cd$ be an additive category and let $(\cx,\cy,\cz)$ be a triple of full subcategories. We say that it is an \emph{torsion torsionfree(=TTF) triple}\index{TTF triple} on $\cd$ in case both $(\cx,\cy)$ and $(\cy,\cz)$ are torsion pairs on $\cd$. We will say that the TTF triple is \emph{left (respectively, right) split}\index{TTF triple!left split}\index{TTF triple!right split} if $(\cx,\cy)$ (respectively, $(\cy,\cz)$) is split. A left and right split TTF triple will be called \emph{centrally split}\index{TTF triple!centrally split}. This terminology also applies in case $\cd$ is abelian.

Let $\cd$ be a pretriangulated category and let $(\cx,\cy,\cz)$ be a triple of full subcategories. We say that it is a \emph{pretriangulated TTF triple}\index{TTF triple!pretriangulated} on $\cd$ in case both $(\cx,\cy)$ and $(\cy,\cz)$ are pretriangulated torsion pairs on $\cd$. In case $\cd$ is not only pretriangulated but even triangulated, then we will speak of \emph{triangulated TTF triples}\index{TTF triple!triangulated}; notice that in this case these are precisely the triples $(\cx,\cy,\cz)$ such that $(\cx,\cy)$ and $(\cy,\cz)$ are t-structures. We will say that the pretriangulated TTF triple $(\cx,\cy,\cz)$ is \emph{left (respectively, right, centrally) split} if so is its underlying (additive) TTF triple.
\end{definition}

\begin{proposition}\label{centrally split TTF triples on additive}
Let $(\cx,\cy,\cz)$ be a TTF triple on an additive category $\cd$. The following assertions are equivalent:
\begin{enumerate}[1)]
\item It is left split and $\cx=\cz$.
\item It is right split and $\cx=\cz$.
\item It is centrally split.
\end{enumerate}
When  $\cd$ has canonical factorizations or it is a weakly balanced pretriangulated category with both $(\cx,\cy)$ and $(\cy,\cz)$ pretriangulated torsion pairs, then the above conditions are equivalent to:
\begin{enumerate}[4)]
\item $\cx=\cz$.
\end{enumerate}
\end{proposition}
\begin{proof}
$1)\Rightarrow 3)$ is clear thanks to lemma \ref{split torsion pair lemma}.

$3)\Rightarrow 1)$ For an object $X\in\cx$, we consider its decomposition $X=y\tau_{\cy}X\oplus z\tau^{\cz}X$. Since the projection on $y\tau_{\cy}X$ vanishes, we get that $y\tau_{\cy}X=0$, \ie $X\in\cz$. Conversely, one proves the inclusion $\cz\subseteq\cx$.

The equivalence $2)\Leftrightarrow 3)$ follows dually.

$4)\Rightarrow 1)$ \emph{Suppose that $\cd$ has canonical factorizations:}

We have to prove that each $M\in\cd$ admits a decomposition $M=X\oplus Y$ with $X\in\cx$ and $Y\in\cy$. For this, consider the short exact sequences:
\[0\ra x\tau_{\cx}M\arr{\delta_{\cx,M}}M\arr{\eta^{\cy}_{M}}y\tau^{\cy}M\ra 0
\]
and
\[0\ra y\tau_{\cy}M\arr{\delta_{\cy,M}}M\arr{\eta^{\cz}_{M}}z\tau^{\cz}M\ra 0.
\]
Let us prove that
\[\left[\begin{array}{cc}\delta_{\cx,M}&\delta_{\cy,M}\end{array}\right]:x\tau_{\cx}M\oplus y\tau_{\cy}M\ra M
\]
is an isomorphism. For this, thanks to Remark \ref{canonical factorizations implies balanced} it suffices to prove that it is an epimorphism and a monomorphism. Consider a morphism
\[\left[\begin{array}{c} u \\ -v\end{array}\right]:N\ra x\tau_{\cx}M\oplus y\tau_{\cy}M
\]
such that 
\[0=\left[\begin{array}{cc}\delta_{\cx,M}&\delta_{\cy,M}\end{array}\right]\left[\begin{array}{c} u \\ -v\end{array}\right]=\delta_{\cx,M}u-\delta_{\cy,M}v.
\]
Put $f:=\delta_{\cx,M}u=\delta_{\cy,M}v$ and consider canonical factorizations:
\[\xymatrix{N\ar[rr]^f\ar[dr]_{p}&&M,\\
&L\ar[ur]_{i}&
}
\xymatrix{N\ar[rr]^u\ar[dr]_{u'}&&x\tau_{\cx}M\\
&U\ar[ur]_{u''}&
}
\text{ and }
\xymatrix{N\ar[rr]^v\ar[dr]_{v'}&&y\tau_{\cy}M\\
&V\ar[ur]_{v''}&
}
\]
Since $\delta_{\cx,M}$ and $\delta_{\cy,M}$ are monomorphisms, then $f$ also admits the following factorizations:
\[\xymatrix{N\ar[rr]^f\ar[dr]_{u'}&&M\\
&U\ar[ur]_{\delta_{\cx,M}u''}&
}
\text{ and }
\xymatrix{N\ar[rr]^f\ar[dr]_{v'}&&M\\
&V\ar[ur]_{\delta_{\cy,M}v''}&
}
\]
Uniqueness of canonical factorizations implies that we have a monomorphism $L\ra x\tau_{\cx}M$. Since $x\tau_{\cx}M\in\cz=\cy^{\bot}$, then $L\in\cy^{\bot}$. Similarly, we have a monomorphism $L\ra y\tau_{\cy}M$. Since $y\tau_{\cy}M\in\cy=\cx^{\bot}$, then $L\in\cx^{\bot}=\cy$. Therefore, $L=0$ and thus $f=0$. But $\delta_{\cx,M}$ and $\delta_{\cy,M}$ are monomorphisms, which implies that $u=0$ and $v=0$.

This proves that $\left[\begin{array}{cc}\delta_{\cx,M}&\delta_{\cy,M}\end{array}\right]$ is a monomorphism. To prove that it is also an epimorphism, consider a morphism $w:M\ra N$ such that $w\delta_{\cx,M}=0$ and $w\delta_{\cy,M}=0$. Since $\eta^{\cy}_{M}$ is the cokernel of $\delta_{\cx,M}$, there exists a morphism $a:y\tau^{\cy}M\ra N$ such that $a\eta^{\cy}_{M}=w$. Similarly, there exists a morphism $b:z\tau^{\cz}\ra N$ such that $b\eta^{\cz}_{M}=w$. Consider the canonical factorizations:
\[\xymatrix{M\ar[rr]^{w}\ar[dr]_{p}&&N,\\
&L\ar[ur]_{i}&
}
\xymatrix{y\tau^{\cy}M\ar[rr]^{a}\ar[dr]_{a'}&& N\\
&A\ar[ur]_{a''}&
}
\text{ and }
\xymatrix{z\tau^{\cz}M\ar[rr]^{b}\ar[dr]_{b'}&& N\\
&B\ar[ur]_{b''}&
}
\]
Since $\eta^{\cy}_{M}$ and $\eta^{\cz}_{M}$ are epimorphisms, then $w$ also admits the following factorizations:
\[\xymatrix{ M\ar[rr]^{w}\ar[dr]_{a'\eta^{\cy}_{M}} && N \\
& A\ar[ur]_{a''} &
}
\text{ and }
\xymatrix{M\ar[rr]^{w}\ar[dr]_{b'\eta^{\cz}_{M}}&&N\\
&B\ar[ur]_{b''}&
}
\]
Uniqueness of canonical factorizations implies that there exists two epimorphisms $A\ra L$ and $B\ra L$. Since $a'$ is an epimorphism and $y\tau^{\cy}M\in\cy=\ ^{\bot}\cz$, then $A\in^{\bot}\cz$, which implies $L\in^{\bot}\cz=\cy$. Similarly, since $b'$ is an epimorphism and $z\tau^{\cz}M\in\cz=\cx=\ ^{\bot}\cy$, then $B\in\ ^{\bot}\cy$, which implies $L\in\ ^{\bot}\cy=\cx$. Therefore, $L=0$, and $w=0$.

\emph{Suppose that $\cd$ is a weakly balanced pretriangulated category:} Assume that $(\cx,\cy,\cx)$ is a pretriangulated TTF triple. We want to prove that $(\cx,\cy)$ splits,\ie that for each $M\in\cd$ the morphism $\delta_{\cx,M}$ is a section. Indeed, thanks to axiom RT4) of Definition \ref{definition suspended category}, there exists a commutative diagram of the form
\[\xymatrix{x\tau_{\cx}M\ar[r]^{\delta_{\cx,M}}\ar@{=}[d] & M\ar[r]\ar[d]^{\eta^{\cx}_{M}} & y\tau^{\cy}M\ar[r]^{0}\ar[d] & \Sigma(x\tau_{\cx}M)\ar@{=}[d] \\
x\tau_{\cx}M\ar[r]^{\eta^{\cx}_{M}\delta_{\cx,M}} & x\tau^{\cx}M\ar[d]^0\ar[r] & U\ar[r]\ar[d] & \Sigma(x\tau_{\cx}M) \\
& \Sigma(y\tau_{\cy}M)\ar@{=}[r]\ar[d]_{-\Sigma \delta_{\cy,M}} & \Sigma(y\tau_{\cy}M)\ar[d] & \\
& \Sigma M\ar[r] & \Sigma(y\tau^{\cy}M) &
}
\]
in which the two first rows and the two central columns are triangles. Since $\Sigma\cx\subseteq\cx$ and $\cx=\ ^{\bot}\cy$, then $U\in\cx$. Similarly, since $\Sigma\cy\subseteq\cy$ and $\cy=\ ^{\bot}\cx$, then $U\in\cy$. Therefore, $U=0$. Dually, there exists a commutative diagram of the form
\[\xymatrix{& \Omega(y\tau_{\cy}M)\ar[r]\ar[d] & \Omega M\ar[d]^{-\Omega \eta^{\cy}_{M}} & \\
& \Omega(y\tau^{\cy}M)\ar@{=}[r]\ar[d] &  \Omega(y\tau^{\cy}M)\ar[d]^0 & \\
 \Omega(x\tau^{\cx}M)\ar[r]\ar@{=}[d] & V\ar[r]\ar[d] & x\tau_{\cx}M\ar[r]^{\eta^{\cx}_{M}\delta_{\cx,M}}\ar[d]^{\delta_{\cx,M}} & x\tau^{\cx}M\ar@{=}[d] \\
 \Omega(x\tau^{\cx}M)\ar[r]^0 & y\tau_{\cy}M\ar[r] & M\ar[r]^{\eta^{\cx}_{M}} & x\tau^{\cx}M
}
\]
in which the two last rows and the two central columns are triangles. Since $\Omega\cx\subseteq\cx$ and $\cx=\cy^{\bot}$, then $V\in\cx$. Similarly, since $\Omega\cy\subseteq\cy$ and $\cy=\cx^{\bot}$, then $V\in\cy$. Therefore, $V=0$. But we are assuming that $\cd$ is weakly balanced, which implies that $\eta^{\cx}_{M}\delta_{\cx,M}$ is an isomorphism, and thus $\delta_{\cx,M}$ is a section.
\end{proof}

\begin{remark}
The Proposition above implies the equivalence $a)\Leftrightarrow d)$ of \cite[Proposition VI.8.5]{Stenstrom}.
\end{remark}

In the setting of triangulated categories, the former result can be improved. Before the improvement, we recall two well-known and very useful lemmas:

\begin{lemma}\label{quotient by coaisles}
Let $\cn$ be a strictly full triangulated subcategory of a triangulated category $(\cd,?[1])$, and let $\ck:=^{\bot}\cn$ be the left orthogonal to $\cn$ in $\cd$. 
\begin{enumerate}[1)]
\item For two objects $K\in\ck$ and $M\in\cd$, the canonical morphism
\[q:\cd(K,M)\ra(\cd/\cn)(K,M)
\]
is an isomorphism. In particular, the restriction of the quotient functor $q:\cd\ra\cd/\cn$ to $\ck$ is a full embedding.
\item If $(\ck,\cn)$ is a t-structure on $\cd$, then the restriction of the quotient functor $q:\cd\ra\cd/\cn$ to $\ck$ is a triangle equivalence.
\end{enumerate}
\end{lemma}
\begin{proof}
1) \emph{Surjectivity:} Consider a morphism $s^{-1}f$ of $(\cd/\cn)(K,M)$, represented by the left fraction
\[\xymatrix{&L& \\
K\ar[ur]^{f}&&M\ar[ul]_{s}
}
\]
By construction of $\cd/\cn$ (see for instance the books \cite{Verdier1996, Neeman2001}), there exists a triangle
\[M\arr{s}L\ra N\ra M[1]
\]
with $N\in\cn$, and so
\[\cd(K,M)\arr{s^{\we}}\cd(K,L)
\]
is surjective. In particular, $f=sg$ for some $g\in\cd(K,M)$, and so $q(g)=\id^{-1}g=s^{-1}f$.

\emph{Injectivity:} Suppose that $q(f)=\id^{-1}f=0$ in $(\cd/\cn)(K,M)$. This happens if and only if there exists a triangle
\[M\arr{s}L\ra N\ra M[1]
\]
with $N\in\cn$ and a commutative diagram of the form
\[\xymatrix{&M\ar[d]_{s}& \\
K\ar[ur]^{f}\ar[r]\ar[dr]_{0}&L&M\ar[ul]_{\id}\ar[l]_{s} \\
&M\ar[u]^{s}\ar[ur]_{\id}&
}
\]
In particular
\[\cd(K,M)\arr{s^{\we}}\cd(K,L)
\]
is injective, and then we deduce that $f=0$.

2) We already know that the composition
\[\ck\hookrightarrow\cd\arr{q}\cd/\cn
\]
is fully faithful. It remains to prove that it is also essentially surjective. Given an arbitrary object $M$ of $\cd/\cn$, consider a triangle
\[K\ra M\ra N\ra K[1]
\]
in $\cd$ with $K\in\ck$ and $N\in\cn$. The quotient functor $q$ takes this triangle to a triangle in $\cd/\cn$ of the form
\[K\ra M\ra 0\ra K[1],
\]
which proves that $M$ is isomorphic to $K$ in $\cd/\cn$.
\end{proof}

\begin{lemma}\label{x igual a z}
Let $(\cx,\cy,\cz)$ be a triangulated TTF triple on a triangulated category $(\cd,?[1])$. Then the compositions
\[\cx\arr{x}\cd\arr{\tau^{\cz}}\cz
\]
and 
\[\cz\arr{z}\cd\arr{\tau_{\cx}}\cx
\]
are mutually quasi-inverse triangle equivalences.
\end{lemma}
\begin{proof}
The first composition can be regarded as the composition of the following triangle equivalences
\[\cx\arr{\sim}\cd/\cy\arr{\tau^{\cz}}\cz,
\]
and the second composition can be regarded as the composition of the following triangle equivalences
\[\cz\arr{\sim}\cd/\cy\arr{\tau_{\cx}}\cx.
\]
\end{proof}

\begin{proposition}\label{split TTF triple}
Let $(\cx,\cy,\cz)$ be a triangulated TTF triple on a triangulated category $(\cd,?[1])$. The following assertions are equivalent:
\begin{enumerate}[1)]
\item The t-structure $(\cx,\cy)$ splits.
\item The t-structure $(\cy,\cz)$ splits.
\item $\cx=\cz$
\item The canonical composition $y\tau_{\cy}M\arr{\delta_{\cy,M}} M\arr{\eta^{\cy}_{M}} y\tau^{\cy}M$ is an isomorphism for each $M\in\cd$.
\item The canonical composition $y\tau_{\cy}M\arr{\delta_{\cy,M}} M\arr{\eta^{\cy}_{M}} y\tau^{\cy}M$ is a section for each $M\in\cd$.
\item The canonical composition $y\tau_{\cy}M\arr{\delta_{\cy,M}} M\arr{\eta^{\cy}_{M}} y\tau^{\cy}M$ is a retraction for each $M\in\cd$.
\end{enumerate}
\end{proposition}
\begin{proof}
Implications $4)\Rightarrow 5), 6)$ are clear. 

$3)\Rightarrow 1)+2)$ Take for $M\in\cd$ the triangle 
\[x\tau_{\cx}M\ra M\ra y\tau^{\cy}M\arr{w} x\tau_{\cx}M[1]
\] 
associated to the t-structure $(\cx,\cy)$. Since $\cx=\cz$, then $w=0$ and so the triangle splits \cite[Corollary 1.2.7]{Neeman2001}. This proves that $(\cx,\cy)$ splits. Dually, we prove that $(\cy,\cz)$ splits.

$1)\Rightarrow 3)$ Assume $(\cx,\cy)$ splits. Let $f\in\cd(Y,X')$ be an arbitrary morphism with $Y\in\cy\ko X'\in\cx$. From it we form the triangle
\[Y\arr{f}X'\ra M'\ra Y[1].
\]
By rotating we get a triangle of the form
\[X\ra M\ra Y\arr{f}X[1],
\]
with $X:=X'[-1]\in\cx\ko M:=M'[-1]$. This triangle is isomorphic to
\[x\tau_{\cx}M\ra M\ra y\tau^{\cy}M\ra(x\tau_{\cx}M)[1]
\]
which splits. Hence, this implies $f=0$ and so $\cx\subseteq\cy^{\bot}=\cz$.

Therefore, we can take the functor $\tau^{\cz}:\cd\ra\cz$ to be the identity on $\cx$. But then, since the restriction $\cx\arr{x}\cd\arr{\tau^{\cz}}\cz$ is 
an equivalence (\cf Lemma \ref{x igual a z}), in particular it is dense, and so $\cz\subseteq\cx$. That is to say, $\cx=\cz$.

Dually, we prove $2)\Rightarrow 3)$.

$4)\Rightarrow 3)$ We have $M\in\cx$ if an only if $\tau^{\cy}M=0$ if and only if $\tau_{\cy}M=0$ if and only if $M\in\cz$.

$3)\Rightarrow 4)$ By looking to the diagram guaranteed by axiom RT4) of Definition \ref{definition suspended category}
\[\xymatrix{y\tau_{\cy}M\ar[r]^{\delta_{\cy,M}}\ar@{=}[d] & M\ar[r]\ar[d]^{\eta^{\cy}_{M}} & z\tau^{\cz}M\ar[r]\ar[d] & y\tau_{\cy}M[1]\ar@{=}[d] \\
y\tau_{\cy}M\ar[r] & y\tau^{\cy}M\ar[r]\ar[d] & U\ar[r]\ar[d] & y\tau_{\cy}M[1]\ar[d] \\
&x\tau_{\cx}M[1]\ar@{=}[r]\ar[d] & x\tau_{\cx}M[1]\ar[r]\ar[d] & M[1] \\
&M[1]\ar[r]&z\tau^{\cz}M &
}
\]
we realise that the mapping cone $U$ of $\eta^{\cy}_{M}\delta_{\cy,M}$ is in $\cx=\cz$, but also in $\cy$. Hence $U=0$ and so $\eta^{\cy}_{M}\delta_{\cy,M}$ is an isomorphism.

$5)\Rightarrow 3)$ $M\in\cx$ if and only if $\tau^{\cy}M=0$, which implies $\eta^{\cy}_{M}\delta_{\cy,M}=0$, and so $\tau_{\cy}M=0$, \ie $M\in\cz$. Hence $\cx\subseteq\cz$ and so we can assume that the equivalence $\cx\arr{x}\cd\arr{\tau^{\cz}}\cz$ is the identity. This proves $\cx=\cz$.

Dually, we prove $6)\Rightarrow 3)$.
\end{proof}

The proposition above proves that, in the triangulated setting, left split, right split and centrally split TTF triples coincide. The situation is much more subtle in the case of module categories, as we will see in Chapter \ref{Split TTF triples on module categories}.

\section{Parametrization of centrally split TTF triples}\label{Parametrization of centrally split TTF triples}
\addcontentsline{lot}{section}{1.7. Parametrizaci\'on de las ternas TTF centralmente escindidas}

\begin{definition}\label{2-decomposition of an additive category}
Let $\cg$ be a full subcategory of an additive category $\cd$. A \emph{2-decomposition}\index{decomposition!2-} of $\cg$ consists of a pair $(\cx,\cy)$ of strictly full subcategories of $\cg$ such that the functor
\[\cx\times\cy\ra\cg\ko (X,Y)\mapsto X\oplus Y
\]
is an equivalence or, alternatively, such that:
\begin{enumerate}[1)]
\item $\cd(X,Y)=\cd(Y,X)=0$ for each $X\in\cx\ko Y\in\cy$.
\item Every object $M\in\cg$ is the coproduct of an object of $\cx$ and an object of $\cy$.
\end{enumerate}
The class of all 2-decompositions of $\cg$ is denoted by $\cc_{\cg}$.
\end{definition}

\begin{definition}
The \emph{center}\index{center}\index{category!center of an additive} of an additive category $\cd$ is the (big) commutative ring formed by the endomorphisms $\Fun(\id_{\cd},\id_{\cd})$ of the identity functor $\id_{\cd}$.
\end{definition}

\begin{example}\label{centro en categorias aditivas con generadores}
Let $\cd$ be an additive category and let $\cg:=\{G_{i}\}_{i\in I}$ be a family of objects of $\cd$. In this case, the center of $\add(\cg)$ is isomorphic to the subring of $\prod_{i\in I}\cd(G_{i},G_{i})$ formed by the elements $\{f_{i}\}_{i\in I}$ such that $gf_{i}=f_{j}g$ for all morphisms $g\in\cd(G_{i},G_{j})\ko i,j\in I$. Let us sketch how this is proved. Indeed, it is clear that each element of the center of $\add(\cg)$ gives rise to such a family $\{f_{i}\}_{i\in I}$. Conversely, let $\{f_{i}\}_{i\in I}$ be such an element of $\prod_{i\in I}\cd(G_{i},G_{i})$. It defines an element $F$ of the center of $\add(\cg)$ as follows. If $M\in\add(\cg)$, there is a section
\[\sigma:M\ra\bigoplus_{t=1}^{m}G_{i_{t}}
\]
and a retraction
\[\rho:\bigoplus_{t=1}^{m}G_{i_{t}}\ra M,
\]
with $\rho\sigma=\id_{M}$, and where $\bigoplus_{t=1}^{m}G_{i_{t}}$ is a finite direct sum of objects of $\cg$. We convene that $\sigma$ and $\rho$ are the identity morphisms for the objects of $\cg$. Define $F$ on $M$ as the composition
\[M\arr{\sigma}\bigoplus_{t=1}^{m}G_{i_{t}}\arr{\tilde{f}}\bigoplus_{t=1}^{m}G_{i_{t}}\arr{\rho} M,
\]
where $\tilde{f}$ is given by
\[\left[\begin{array}{cccc}f_{i_{1}}& 0 & \dots & 0 \\ 0 & f_{i_{2}} & \dots & 0 \\ \dots &\dots &\dots &\dots \\ 0 & 0 & \dots & f_{i_{m}}\end{array}\right]
\]
\end{example}

\begin{proposition}\label{centraly split TTF, 2-decomposition and idempotents}
Let $\cd$ be an additive category. There is a one-to-one correspondence between:
\begin{enumerate}[1)]
\item Centrally split TTF triples on $\cd$.
\item 2-decompositions of $\cd$.
\end{enumerate}
If idempotents split in $\cd$, then the classes above are in one-to-one correspondence with:
\begin{enumerate}[3)]
\item Idempotents of the center of $\cd$.
\end{enumerate}
When idempotents split in $\cd$ and $\cd=\add(\cg)$ for some family $\cg=\{G_{i}\}_{i\in I}$, the classes above are in one-to-one correspondence with:
\begin{enumerate}[4)]
\item The (central) idempotents $e=(e_{i})_{i\in I}\in\prod_{i\in I}\cd(G_{i},G_{i})$ such that $e_{j}f=fe_{i}$ for each morphism $f\in\cd(G_{i},G_{j})\ko i, j\in I$.
\end{enumerate} 
\end{proposition}
\begin{proof}
From $1)$ to $2)$: If $(\cx,\cy,\cz)$ is a centrally split TTF triple, then by Proposition \ref{centrally split TTF triples on additive} we know that $\cx=\cz$ and so $(\cx,\cy)$ is a 2-decomposition.

From $2)$ to $1)$: Use Lemma \ref{split torsion pair lemma} to prove that, if $(\cx,\cy)$ is a 2-decomposition, then $(\cx,\cy,\cx)$ is a centrally split TTF triple.

From $2)$ to $3)$ For each object $M\in\cd$ fix a decomposition $M=M_{\cx}\oplus M^{\cy}$ with $M_{\cx}\in\cx$ and $M^{\cy}\in\cy$. Take the idempotent $(e_{M})_{M\in\cd}$ of the center of $\cd$ defined by
\[e_{M}:=\left[\begin{array}{cc}\id_{M_{\cx}}&0\\ 0&0\end{array}\right].
\]

From $3)$ to $2)$: Let $(e_{M})_{M\in\cd}$ be an idempotent of the center of $\cd$. In particular, each $e_{M}$ is an idempotent endomorphism of $M$. Since $\cd$ has splitting idempotents, then each $M$ is, up to isomorphism, a coproduct $M_{\cx}\oplus M^{\cy}$ such that $e_{M}$ gets identified with
\[e_{M}=\left[\begin{array}{cc}\id_{M_{\cx}}&0\\ 0&0\end{array}\right].
\]
Put $\cx$ for the full subcategory of $\cd$ formed by the objects $M$ with $M^{\cy}=0$ (or, equivalently, such that $e_{M}$ is an isomorphism), and $\cy$ for the full subcategory of $\cd$ formed by the objects $M$ with $M_{\cx}=0$ (or, equivalently, such that $e_{M}=0$).

For the bijection between $3)$ and $4)$ we use Example \ref{centro en categorias aditivas con generadores}.
\end{proof}

\begin{proposition}\label{2-decompositions arriba abajo}
Let $\cd$ be an additive category, $\cg$ a full subcategory of $\cd$ closed under direct summands and such that $\cg^{\bot}=0$. The map
\[\Psi:\cc_{\cd}\ra\cc_{\cg}\ko (\cx,\cy)\mapsto (\cx\cap\cg,\cy\cap\cg)
\]
is injective. The image of $\Psi$ consists of those pairs $(\cx',\cy')\in\cc_{\cg}$ such that $(\cy'^{\bot_{\cd}}, \cx'^{\bot_{\cd}})$ is a 2-decomposition of $\cd$.
\end{proposition}
\begin{proof}
It is clear that, if $(\cx,\cy)$ is a 2-decomposition of $\cd$, then the map
\[(\cx\cap\cg)\times(\cy\cap\cg)\ra\cg\ko (X,Y)\mapsto X\oplus Y 
\]
is an equivalence of categories, and thus $\Psi$ is well-defined. To prove that it is injective, it suffices to check that $\cy=(\cg\cap\cx)^{\bot}$. The inclusion $\subseteq$ is clear. Conversely, let $M\in(\cx\cap\cg)^{\bot}$ and consider its decomposition $M=x\tau_{\cx}M\oplus y\tau^{\cy}M$. Since both $M$ and $y\tau^{\cy}M$ are in $(\cx\cap\cg)^{\bot}$, then so is $x\tau_{\cx}M$. If $x\tau_{\cx}M\neq 0$, then there exists a non-zero morphism $g:G\ra x\tau_{\cx}M$ with $G\in\cg$. Consider the decomposition of $G$ and put $g$ in the corresponding matrix form
\[g=\left[\begin{array}{cc}g_{1} & g_{2}\end{array}\right]: x\tau_{\cx}G\oplus y\tau^{\cy}G\ra x\tau_{\cx}M
\]
Since $ x\tau_{\cx}G\in\cx\cap\cg$, then $g_{1}=0$, which implies $g_{2}\neq 0$. But this contradicts the fact that $\cx=\cy^{\bot}$. Now, let $(\cx',\cy')\in\cc_{\cg}$ be such that $(\cy'^{\bot_{\cd}},\cx'^{\bot_{\cd}})$ is a 2-decomposition of $\cd$. We have to prove that $\cx'=\cy'^{\bot_{\cd}}\cap\cg$ and $\cy'=\cx'^{\bot_{\cd}}\cap\cg$. But this is easy. Indeed, if $G\in\cy'^{\bot_{\cd}}\cap\cg$, then in the decomposition $G=G_{\cx'}\oplus G^{\cy'}$ with $G_{\cx'}\in\cx'\ko G^{\cy'}\in\cy'$ we have that $G^{\cy'}$ is forced to vanish.
\end{proof}

\begin{definition}\label{detects centrally split TTF triples}
Let $\cd$ be an additive category and $\cg$ a full subcategory closed under direct summands and such that $\cg^{\bot}=0$. We will say that $\cg$ \emph{detects centrally split TTF triples}\index{detects centrally split TTF triples} on $\cd$ in case the map $\Psi:\cc_{\cd}\ra\cc_{\cg}$ of Proposition \ref{2-decompositions arriba abajo} is bijective, \ie every 2-decomposition $(\cx',\cy')$ of $\cg$ induces a 2-decomposition $(\cy'^{\bot_{\cd}},\cx'^{\bot_{\cd}})$ of $\cd$.
\end{definition}

\begin{lemma}\label{2-decompositions and add}
Let $\cd$ be an additive category with splitting idempotents and let $\cg$ be a full subcategory of $\cd$ closed under direct summands. Then
\begin{enumerate}[1)]
\item The map
\[\cc_{\add(\cg)}\ra\cc_{\cg}\ko (\cx,\cy)\mapsto(\cx\cap\cg,\cy\cap\cg)
\]
is bijective.
\item If $\cg$ detects centrally split TTF triples on $\cd$ and it is contained in a full subcategory $\cg'$ of $\cd$ closed under direct summands, then $\cg'$ detects centrally split TTF triples on $\cd$.
\end{enumerate}
\end{lemma}
\begin{proof}
1) By Proposition \ref{2-decompositions arriba abajo}, we already know that the map is injective. Let us prove that it is also surjective. If $(\cx',\cy')$ is a 2-decomposition of $\cg$, consider the pair $(\cy'^{\bot},\cx'^{\bot})$ of full subcategories of $\add(\cg)$, with orthogonals taken in $\add(\cg)$. Notice that $\cy'^{\bot}=\add(\cx')$ and $\cx'^{\bot}=\add(\cy')$. The aim is to prove that $(\add(\cx'),\add(\cy'))$ is a 2-decomposition of $\add(\cg)$. If $M\in\add(\cg)$, then there exists $N\in\add(\cg)$ such that
\[M\oplus N=G_{1}\oplus\dots\oplus G_{n},
\]
with $G_{i}\in\cg$. Now, each $G_{i}$ admits a decomposition $G_{i}=G_{i_{\cx'}}\oplus G_{i}^{\cy'}$ with $G_{i_{\cx'}}\in\cx'$ and $G_{i}^{\cy'}\in\cy'$. Put $G_{\cx'}:=\coprod_{i=1}^nG_{i_{\cx'}}\ko G^{\cy'}:=\coprod_{i=1}^nG_{i}^{\cy'}$ and consider $M\oplus N=G_{\cx'}\oplus G^{\cy'}$. Notice that the endomorphisms ring of $M\oplus N$ decomposes as follows
\[\op{End}_{\ca}(M\oplus N)\cong\left[\begin{array}{cc}\op{End}_{\ca}(G_{\cx'})&0\\0&\op{End}_{\ca}(G^{\cy'})\end{array}\right]
\]
and so the idempotent
\[e_{M}=\left[\begin{array}{cc}\id_{M}&0\\0&0\end{array}\right]:M\oplus N\ra M\oplus N
\]
of $\op{End}_{\ca}(M\oplus N)$ corresponds to 
\[\left[\begin{array}{cc}e_{\cx'} & 0 \\ 0 & e_{\cy'}\end{array}\right],
\]
where $e_{\cx'}$ is an idempotent of $\op{End}_{\ca}(G_{\cx'})$ and $e_{\cy'}$ is an idempotent of $\op{End}_{\ca}(G^{\cy'})$. Since idempotents split, then $e_{\cx'}$ corresponds to a direct summand $M_{\cx'}$ of $G_{\cx'}$ and $e_{\cy'}$ corresponds to a direct summand $M^{\cy'}$ of $G^{\cy'}$. It is clear that $M=M_{\cx'}\oplus M^{\cy'}$, with $M_{\cx'}\in\add(\cx')$ and $M^{\cy'}\in\add(\cy')$.

2) The first part of this lemma tells us that we have the following commutative diagram made of (injective) restriction maps
\[\xymatrix{\cc_{\cd}\ar[r] & \cc_{\add(\cg')}\ar[r]\ar[d]^{\wr} & \cc_{\add(\cg)}\ar[d]^{\wr} \\
&\cc_{\cg'}\ar[r] & \cc_{\cg}
}
\]
Finally, if the composition of the two maps in the top row is bijective, then so is the map $\cc_{\cd}\ra\cc_{\add(\cg')}\arr{\sim}\cc_{\cg'}$.
\end{proof}

\begin{definition}
An object $P$ of an abelian category $\ca$ is a \emph{generator}\index{category!abelian!generator of}\index{generator!of an abelian category} of $\ca$ if the functor $\ca(P,?)$ is faithful.
\end{definition}

We recall here some basic properties of generators of abelian categories:

\begin{lemma}
Let $P$ be an object of an abelian category $\ca$.
\begin{enumerate}[1)]
\item If $\ca$ is cocomplete and $P$ is a generator, then for every object $M$ of $\ca$ there is an epimorphism $P^{(I)}\ra M$ for some index set.
\item If $P$ is projective, then it is a generator if and only if there exists a non-zero morphism $P\ra M$ for each object $M\neq 0$.
\end{enumerate}
\end{lemma}
\begin{proof}
1) is \cite[Proposition IV.6.2]{Stenstrom}, and 2) is \cite[Proposition IV.6.3]{Stenstrom}.
\end{proof}

\begin{corollary}\label{centrally split TTF triples on abelian categories}
Let $\ca$ be a complete and cocomplete abelian category such that for each family $M_{i}\ko i\in I$, of objects of $\ca$, the canonical morphism
\[\coprod_{i\in I}M_{i}\ra\prod_{i\in I}M_{i}
\]
is a monomorphism. If $P$ is a projective generator of $\ca$, then $\add(P)$ detects centrally split TTF triples on $\ca$. In particular, there is a one-to-one correspondence between:
\begin{enumerate}[1)]
\item Centrally split TTF triples on $\ca$.
\item Decompositions $P=P_{1}\oplus P_{2}$ such that $\ca(P_{1},P_{2})=\ca(P_{2},P_{1})=0$.
\item Central idempotents of the endomorphisms ring $\ca(P,P)$.
\end{enumerate}
\end{corollary}
\begin{proof}
Clearly, there is a one-to-one correspondence between 3) and 2) and, on the other side, by using the part 1) of the Lemma \ref{2-decompositions and add} and the techniques of its proof one shows that the decompositions of 2) are in bijection with the 2-decompositions of the additive category \add(P). Our task is hence reduced to prove that the restriction map $\cc_{\ca}\ra\cc_{\add(P)}$ is surjective. The aim is to prove that if $(\cp_{1},\cp_{2})$ is a 2-decomposition of \add(P), then $(\cp_{2}^{\bot},\cp_{1}^{\bot})$ is a 2-decomposition of $\ca$. Since every object of $\ca$ is a quotient of a coproduct of objects of \add(P), one readily sees that $\cp_{1}^{\bot}\subseteq\Gen(\cp_{2})$ and $\cp_{2}^{\bot}\subseteq\Gen(\cp_{1})$. On the other hand, since the objects of the classes $\cp_{i}$ are projective, by using the fact that morphisms from coproducts to products are monomorphisms we conclude that $\Gen(\cp_{2})=\cp_{1}^{\bot}$ and $\Gen(\cp_{1})=\cp_{2}^{\bot}$. Then $\ca(M_{1},M_{2})=\ca(M_{2},M_{1})=0$ for $M_{1}\in\cp_{2}^{\bot}\ko M_{2}\in\cp_{1}^{\bot}$. Now, given an object $M\in\ca$, we denote by $t_{1}(M)$ the trace of $\cp_{1}$ in $M$, \ie the sum of the images of all the morphisms from objects of $\cp_{1}$ to $M$. Then $t_{1}(M)\in\Gen(\cp_{1})=\cp_{2}^{\bot}$ and, by using that objects of $\cp_{1}$ are projective, we also have $M/t_{1}(M)\in\cp_{1}^{\bot}$. It remains to show that the short exact sequence
\[0\ra t_{1}(M)\ra M\ra M/t_{1}(M)\ra 0
\]
splits. But in fact we have that $\Ext_{\ca}^{1}(M_{2},M_{1})=0$ for each $M_{1}\in\cp_{2}^{\bot}\ko M_{2}\in\cp_{1}^{\bot}$. Indeed, consider a projective presentation of $M_{2}$
\[0\ra K\ra P\ra M_{2}\ra 0
\]
with $P\in\Sum(\cp_{2})$. Since both $P$ and $M_{2}$ are in $\cp_{1}^{\bot}$, then so is $K$. Now, we have an exact sequence
\[\dots\ra0=\ca(K,M_{1})\ra\Ext_{\ca}^{1}(M_{2},M_{1})\ra\Ext_{\ca}^{1}(P,M_{1})=0\ra\dots
\]
which implies that $\Ext_{\ca}^{1}(M_{2},M_{1})=0$.
\end{proof}

\begin{example}
If $A$ is a $k$-algebra, we can take $\ca=\Mod A$ and $P=A$. Notice that we have a ring isomorphism $\ca(P,P)\cong A$ and so Corollary \ref{centrally split TTF triples on abelian categories} generalizes the equivalence $c)\Leftrightarrow g)$ of \cite[Proposition VI.8.5]{Stenstrom}.
\end{example}


We refer to Chapter \ref{TTF triples on triangulated categories} for the notion of ``compactly generated triangulated category''.

\begin{corollary}\label{centrally split TTF in compactly generated triangulated categories}
Let $\cd$ be a triangulated category with small coproducts which is compactly generated by the family $\cg=\{G_{i}\}_{i\in I}$. There is a one-to-one correspondence between:
\begin{enumerate}[1)]
\item Centrally split triangulated TTF triples on $\cd$.
\item 2-decompositions $(\add(\cg)_{1},\add(\cg)_{2})$ of $\add(\cg)$ such that 
\[\cd(M_{1},M_{2}[n])=\cd(M_{2},M_{1}[n])=0
\] 
for each $M_{1}\in\add(\cg)_{1}\ko M_{2}\in\add(\cg)_{2}$ and $n\in\Z$.
\item Central idempotents $e=\{e_{i}\}_{i\in I}$ of the ring $\prod_{i\in I}\cd(G_{i},G_{i})$ such that $(e_{j}[n])f=fe_{i}$ for each $i, j\in I\ko n\in\Z$ and each morphism $f\in\cd(G_{i},G_{j}[n])$.
\end{enumerate}
\end{corollary}
\begin{proof}
\emph{First step:} Put $\cc'_{\cd}$ for the class of 2-decompositions of $\cd$ stable under $?[1]$ and $?[-1]$, and similarly for $\cc'_{\hat{\cg}}$, where $\hat{\cg}:=\bigcup_{n\in\Z}\add(\cg)[n]$. Thanks to Proposition \ref{compatibility proposition} and Proposition \ref{centraly split TTF, 2-decomposition and idempotents}, we know that the map $(\cx,\cy,\cx)\mapsto(\cx,\cy)$ induces a bijection between the class of centrally split triangulated TTF triples on $\cd$ and the class $\cc'_{\cd}$. Now, thanks to Proposition \ref{2-decompositions arriba abajo}, we have an injective map
\[\Psi: \cc_{\cd}\ra\cc_{\hat{\cg}},
\]
which induces an injective map
\[\Psi:\cc'_{\cd}\ra\cc'_{\hat{\cg}},
\]
and the claim is that it is also surjective. Indeed, take $(\hat{\cg}_{1},\hat{\cg}_{2})\in\cc'_{\hat{\cg}}$ and let $\cx_{i}:=\Tria(\hat{\cg}_{i})$ be the smallest full triangulated subcategory of $\cd$ containing $\hat{\cg}_{i}$ and closed under small coproducts. The adjoint functor argument (\cf Lemma \ref{adjoint functor argument}) implies that each $\cx_{i}$ is a triangulated aisle in $\cd$. It is easy to prove (\cf Lemma \ref{right orthogonal and devissage}) that $\cx_{1}^{\bot}$ is the class of objects which are right orthogonal to all the shifts in both directions of objects of $\hat{\cg}_{1}$. Let $M\in\hat{\cg}_{1}$ and consider the objects $N\in\cx_{2}$ such that $\cd(M[n],N)=0$ for every $n\in\Z$. They form a full triangulated subcategory of $\cx_{2}$ containing $\hat{\cg}_{2}$ and closed under coproducts. This proves that $\cx_{2}\subseteq\cx_{1}^{\bot}$. Notice that $\cx_{1}^{\bot}$ is a full triangulated subcategory of $\cd$ generated by $\hat{\cg}_{2}$. Then, by Lemma \ref{generators and t-structures} we have that $\cx_{2}=\cx_{1}^{\bot}$, and so $(\cx_{1},\cx_{2})$ belongs to $\cc'_{\cd}$ and is easily seen to be mapped to $(\hat{\cg}_{1},\hat{\cg}_{2})$ by $\Psi$.

\emph{Second step: Bijection between 1) and 2):} To prove this bijection we only need to show that the image of the restriction map
\[\Psi:\cc'_{\hat{\cg}}\ra\cc_{\add(\cg)}
\]
is precisely the class of 2-decompositions described in part 2). It is clear that the 2-decompositions in the image satisfy the conditions of part 2). Conversely, if $(\add(\cg)_{1},\add(\cg)_{2})$ is a 2-decomposition of $\add(\cg)$ satisfying the extra assumption, then $(\hat{\cg}_{1},\hat{\cg}_{2})$, with $\hat{\cg}_{i}:=\cup_{n\in\Z}\add(\cg)_{i}[n]$, belongs to $\cc'_{\hat{\cg}}$. Finally, since any object of $\add(\cg)_{1}[n]\cap\add(\cg)$ is left orthogonal to $\add(\cg)_{2}$, and any object of $\add(\cg)_{2}[n]\cap\add(\cg)$ is right orthogonal to $\add(\cg)_{1}$, we have that $\add(\cg)_{i}[n]\cap\add(\cg)\subseteq\add(\cg)_{i}$ for each $n\in\Z$, and so $(\hat{\cg}_{1},\hat{\cg}_{2})$ is sent to $(\add(\cg)_{1},\add(\cg)_{2})$ via $\Psi$.

\emph{Third step: bijection between 2) and 3):} We use Proposition \ref{centraly split TTF, 2-decomposition and idempotents} to check that the 2-decompositions of $\add(\cg)$ satisfying the condition of assertion 2) are in bijection with the idempotents $\{e_{i}\}_{i\in I}\in\prod_{i\in I}\cd(G_{i},G_{i})$ such that the compositions
\[C_{i}\arr{e_{i}}C_{i}\arr{f}C_{j}[n]\arr{(\id-e_{j})[n]}C_{j}[n]
\]
and
\[C_{i}\arr{\id-e_{i}}C_{i}\arr{f}C_{j}[n]\arr{e_{j}[n]}C_{j}[n]
\]
vanish for every morphism $f$. But this is equivalent to require that $(e_{j}[n])f=fe_{i}$.
\end{proof}

\begin{example}\label{centrally split TTF in compactly generated triangulated categories in examples}
Let $\cd=\cd\ca$ be the derived category of a small dg category $\ca$ (\cf \cite{Keller1994a, Keller2006b}, see also Chapter \ref{TTF triples on triangulated categories}), and take $\cg=\{A^{\we}\}_{A\in\ca}$. Then, Corollary \ref{centrally split TTF in compactly generated triangulated categories} says that there is a bijection between the class of centrally split triangulated TTF triples on $\cd\ca$ and the class of (central) idempotents $\{e_{A}\}_{A\in\ca}$ of $\prod_{A\in\ca}\H 0\ca(A,A)$ such that for each integer $n$ and each $f\in\H n\ca(A,B)$ we have $e_{B}\cdot f=f\cdot e_{A}$ in $\H n\ca(A,B)$. In particular, if $A$ is an ordinary $k$-algebra, centrally split triangulated TTF triples on $\cd A$ are in bijection with central idempotents of $A$.
\end{example}

We refer to \cite{Verdier1996} for the notion of ``derived category of an abelian category''.

\begin{corollary}\label{centrally split TTF in derived categories of abelian categories}
Let $\ca$ be a complete and cocomplete abelian category, let $\cg$ be a full subcategory which detects centrally split TTF triples of $\ca$ and let $\cd$ be a full triangulated subcategory of the derived category $\cd\ca$ containing the objects of $\cg$ regarded as stalk complexes. There is a one-to-one correspondence between:
\begin{enumerate}[1)]
\item Centrally split triangulated TTF triples on $\cd\ca$.
\item Centrally split triangulated TTF triples on $\cd$.
\item Centrally split TTF triples on $\ca$.
\end{enumerate}
\end{corollary}
\begin{proof}
We have a commutative diagram of inclusions of additive categories:
\[\xymatrix{\add(\cg)\ar[r]\ar[d] & \ca\ar[d] \\
\cd\ar[r] &\cd\ca
}
\]
Then, the compositions
\[\cc'_{\cd\ca}\ra\cc'_{\cd}\ra\cc_{\add(\cg)}
\]
and
\[\cc'_{\cd\ca}\ra\cc_{\ca}\ra\cc_{\add(\cg)}
\]
coincide, where $\cc'_{?}$ is the subclass of $\cc_{?}$ formed by the 2-decompositions closed under $?[1]$ and $?[-1]$. Notice that the elements of $\cc'_{\cd}$ are precisely the 2-decompositions corresponding to the centrally split triangulated TTF triples of $\cd$. Hence, we just have to check that one of the two compositions above is bijective, for then all the arrows in these compositions will be bijective maps. Since $\cg$ detects centrally split TTF triples on $\ca$, then by Lemma \ref{2-decompositions and add} the map $\cc_{\ca}\ra\cc_{\add(\cg)}$ is bijective. It remains to prove that $\cc'_{\cd\ca}\ra\cc_{\ca}$ is surjective. For this, take $(\ca_{1},\ca_{2})\in\cc_{\ca}$ and define $\cd_{i}$ to be the full subcategory of $\cd\ca$ formed by those objects which are isomorphic in $\cd\ca$ to cochain complexes of objects of $\ca_{i}$. Clearly, each object of $\cd\ca$ decomposes as the direct sum of an object of $\cd_{1}$ and an object of $\cd_{2}$. It remains to prove that $(\cd\ca)(M_{1},M_{2})=(\cd\ca)(M_{2},M_{1})=0$ for $M_{i}\in\cd_{i}$. Take, for instance, a map in $(\cd\ca)(M_{1},M_{2})$. This is represented by a fraction $\ol{f}\ol{s}^{-1}$ of morphisms of the category up to homotopy $\ch\ca$:
\[\xymatrix{&N\ar[dl]_{\ol{s}}\ar[dr]^{\ol{f}}& \\
M_{1} & &M_{2}
}
\]
where $s$ is a quasi-isomorphism. Consider the decomposition $N=N_{1}\oplus N_{2}$ with $N_{i}\in\cd_{i}$. Since $s$ and $f$ are cochain maps, then we can ensure they are of the form 
\[s=\left[\begin{array}{cc}s_{1}&0\end{array}\right]:N_{1}\oplus N_{2}\ra M_{1}
\] 
and 
\[f=\left[\begin{array}{cc}0&f_{2}\end{array}\right]:N_{1}\oplus N_{2}\ra M_{2}.
\]
But since $s$ is a quasi-isomorphism, then $N_{2}$ is acyclic and so $\ol{f}\ol{s}^{-1}=0$.
\end{proof}
\bigskip

\begin{example}
Let $A$ be a $k$-algebra and let $\ca=\Mod A$ be its category of modules. Then there is a bijection between:
\begin{enumerate}[1)]
\item Centrally split triangulated TTF triples on the unbounded derived category $\cd A$.
\item Centrally split triangulated TTF triples on the right bounded derived category $\cd^{-}A$.
\item Centrally split triangulated TTF triples on the left bounded derived category $\cd^{+}A$.
\item Centrally split triangulated TTF triples on bounded derived category $\cd^bA$.
\item Centrally split triangulated TTF triples on the category of compact objects $\per A$ of $\cd A$.
\item Central idempotents of $A$.
\end{enumerate}
\end{example}

\chapter{Balance in stable categories}\label{Balance in stable categories}
\addcontentsline{lot}{chapter}{Cap\'itulo 2. Equilibrio en categor\'ias estables}

\section{Introduction}
\addcontentsline{lot}{section}{2.1. Introducci\'on}

\subsection{Motivation}
\addcontentsline{lot}{subsection}{2.1.1. Motivaci\'on}

We know that pretriangulated categories (\cf section \ref{(Co)suspended, triangulated and pretriangulated categories}) are a generalization of abelian and triangulated categories. One of the common features of abelian and triangulated categories is that they are \emph{balanced}\index{category!balanced}, \ie every morphism which is both a monomorphism and an epimorphism is necessarily an isomorphism. That property no longer holds in arbitrary pretriangulated categories, although monomorphisms and epimorphisms in them are easily detected using triangles. Indeed, we have:
\begin{lemma} \label{mono in left-triangulated}
Let $(\cd, \Omega)$ be a left triangulated category. For a morphism $f:M\ra N$ in $\cd$, the following assertions are
equivalent:
\begin{enumerate}[1)]
\item $f$ is a monomorphism.
\item There is a left triangle $\Omega N\ra L\arr{0}M\arr{f}N$.
\end{enumerate}
\end{lemma}
\begin{proof}
Let $f$ be an arbitrary morphism in $\cd$. We know that in the left triangle
\[\Omega N\ra L\arr{g}M\arr{f}N
\]
we have $fg=0$. If $f$ is a monomorphism, this implies $g=0$. Conversely, assume there is a left triangle
\[\Omega N\ra L\arr{0}M\arr{f}N,
\]
and a morphism $h:U\ra M$ such that $fh=0$. Consider the exact sequence
\[\dots\ra\cd(U,L)\arr{0}\cd(U,M)\arr{f^{\we}}\cd(U,N)
\]
Then $h\in\ker(f^{\we})=\im(0)=0$.
\end{proof}

In view of the above lemma, it makes sense to give the following:

\begin{definition}
Let $(\cd,\Omega)$ be a left triangulated category. A morphism $f:M\ra N$ in $\cd$ will be said to be a \emph{strong monomorphism}\index{strong!monomorphism} if there is a left triangle $\Omega N\ra 0\ra M\arr{f}N$ in $\cd$. The dual notion on a right triangulated category is that of \emph{strong epimorphism}\index{strong!epimorphism}. Recall (\cf Definition \ref{definition weakly balanced}) that, if $\cd$ is a pretriangulated category, we shall say that it is \emph{weakly balanced} in case that every morphism which is both a strong monomorphism and strong epimorphism is an isomorphism.
\end{definition}

It is natural to expect that the class of (weakly) balanced pretriangulated categories is an interesting one. The moral goal of this chapter is to show that (weak) balance is a very restrictive condition on stable categories. We show that by considering the stable category (\cf Definition \ref{stable category}) $\ul{\ca}$ of an arbitrary abelian category $\ca$ associated to a full subcategory $\ct$ when $\ct$ is a Serre class or when $\ct$ consists of projective objects (and, by duality, when $\ct$ consists of injective objects). 

\subsection{Outline of the chapter}
\addcontentsline{lot}{subsection}{2.1.2. Esbozo del cap\'itulo}

In Section \ref{Monomorphisms and strong monomorphisms in stable categories}, we characterize monomorphism and strong monomorphisms in
$\ul{\ca}$, for any full additive subcategory $\ct$ closed under direct summands. In Section \ref{When T is a Serre class}, we show
that if $\ct$ is a covariantly finite Serre class in $\ca$, then $\ul{\ca}$ is (weakly) balanced if and only if $\ct$ is a direct
summand  of $\ca$ as an additive category (Corollary \ref{balanced in Serre case}). In the latter sections of the chapter, $\ct$ is
supposed to consist of projective objects, in which case the balance of $\ul{\ca}$ is characterized in Section \ref{Balance
when T consists of projective objects} (Theorem \ref{characterization balanced}) and the weak balance in Section \ref{Weak balance when T consists of projective objects} (Theorem \ref{characterization weakly-balanced}). A byproduct is that if $\ct$ consists of projective-injective objects, then $\ul{\ca}$ is
always balanced. When $\ca=\ch$ has the property that subobjects of projective objects are projective (\eg if $\ch$ is a
hereditary abelian category), these results give that if $\ct$ is covariantly finite then $\ul{\ch}$ is (weakly) balanced if and only if $^\perp\ct$ is closed under subobjects (Proposition \ref{hereditary weakly balanced}). The final section of the chapter gives some examples in categories of modules showing, in
particular, that stable balanced categories need be neither abelian nor triangulated and, also, that weakly balanced stable categories need not be balanced.

\subsection{Notation}
\addcontentsline{lot}{section}{2.1.3. Notaci\'on}

All throughout the chapter $\ca$ will be an abelian category,  $\ct$ will be  a full additive subcategory closed under direct summands, ${\langle\ct\rangle}$ will be the ideal of $\ca$ formed by the morphisms that factors through an object of $\ct$ and $\ul{\ca}:=\ca/{\langle\ct\rangle}$ will be the stable category of $\ca$ associated to $\ct$ (\cf Definition \ref{stable category}). Recall that the image of a morphism $f$ under the canonical quotient functor $\ca\ra\ul{\ca}$ will be denoted by $\ol{f}$. The hypothesis of $\ct$ being closed under direct summands is not strictly necessary for the ideals ${\langle\ct\rangle}$ and ${\langle\add(\ct )\rangle}$ to coincide (remember that $\add(?)$ denotes the closure under finite direct sums and direct summands).  But it simplifies the statements and we will assume it in the sequel.

For the sake of intuition and simplification of proofs, we shall allow ourselves some abuse of terminology concerning abelian categories and shall use expressions of module theory, like ``canonical inclusion $\im(f)\hookrightarrow Y$'' for a morphism $f:M\ra N$, meaning the (unique up to isomorphism) monomorphism appearing in the epi-mono factorization of $f$. For the final section, the terminology concerning rings can be found in \cite{AndersonFuller1992} and \cite{Rotman1979} while that concerning finite dimensional algebras can be found in \cite{Ringel1984} and \cite{AuslanderReitenSmalo1995}.

\section{Monomorphisms and strong monomorphisms in stable categories}\label{Monomorphisms and strong monomorphisms in stable categories}
\addcontentsline{lot}{section}{2.2. Monomorfismos y monomorfismos fuertes en categor\'ias estables}

In this section we undertake the identification of monomorphisms and strong monomorphisms in $\ul{\ca}$. 

\begin{proposition}\label{monos}
For a morphism $f:M\ra N$ in $\ca$ the following assertions are equivalent:
\begin{enumerate}[1)]
\item $\ol{f}$ is a monomorphism in $\ul{\ca}$.
\item For every $p\in\ca(T,N)$ with $T\in\ct$, the parallel to $p$ in the pullback of $f$ and $p$ factors through an object of $\ct$.
\end{enumerate}

When, in addition, $\ct$ is contravariantly finite in $\ca$ the following assertions are also equivalent:

\begin{enumerate}[3)]
\item For some $\ct$-precover $p_{N}:T_{N}\ra N$ the parallel to $p_{N}$ in the pullback of $f$ and $p_{N}$ factors through an object of $\ct$.
\end{enumerate}
\begin{enumerate}[4)]
\item In the left triangulated category $\ul{\ca}$ there is a left triangle of the form
\[\Omega N\ra L\arr{\ol{0}} M\arr{\ol{f}} N
\]
\end{enumerate}
\end{proposition}
\begin{proof}
$1)\Rightarrow 2)$ From the cartesian square
\[\xymatrix{L\ar[r]^{g}\ar[d]^{h} & M\ar[d]^{f} \\
T\ar[r]_{p} & N
}
\]
we have $\ol{f}\ol{g}=\ol{0}$ and so $\ol{g}=\ol{0}$, \ie $g$ factors through an object of $\ct$.

$2)\Rightarrow 1)$ Assume $\ol{f}\ol{u}=\ol{0}$ for some morphism $u:U\ra M$. Then we have a commutative square
\[\xymatrix{U\ar[r]^{u}\ar[d]_{v} & M\ar[d]^{f} \\
T\ar[r]_{p} & N
}
\]
By doing the pullback of $f$ and $p$ we get the following commutative diagram
\[\xymatrix{U\ar@{.>}[dr]_{w}\ar@/^/[drr]^{u}\ar@/_/[ddr]_{v} & & \\
& L\ar[r]^{g}\ar[d]_{h} & M\ar[d]^{f} \\
& T\ar[r]_{p} & N
}
\]
Since, by hypothesis, $\ol{g}=\ol{0}$, and we have $\ol{u}=\ol{g}\ol{w}$, we conclude $\ol{u}=\ol{0}$.

$1)\Leftrightarrow 4)$ and $2)\Rightarrow 3)$ are clear (see Lemma \ref{mono in left-triangulated}). As for $3)\Rightarrow 4)$, use the canonical construction of left triangles (\cf \cite{BeligiannisMarmaridis1994} or section \ref{(Co)suspended, triangulated and pretriangulated categories}).
\end{proof}

The following result identifies strong monomorphisms in stable categories.

\begin{proposition}\label{left zero}
Assume that $\ct$ is contravariantly finite in $\ca$. The following assertions are equivalent for a morphism $f\in\ca(M,N)$:
\begin{enumerate}[1)]
\item $\ol{f}$ is a strong monomorphism in the left triangulated category $\ul{\ca}$.
\item For every (respectively, some) $\ct$-precover $p_{N}:T_{N}\ra N$ the parallel to $p_{N}$ in the pullback of $f$ and $p_{N}$ is a $\ct$-precover of $M$.
\item For every (respectively, some) $\ct$-precover $p_{N}:T_{N}\ra N$, in the cartesian square
\[\xymatrix{T\ar[r]^g\ar[d]_{h} & M\ar[d]^{f} \\
T_{N}\ar[r]_{p_{N}} & N
}
\]
we have $T\in\ct$.
\end{enumerate}
\end{proposition}
\begin{proof}
$2)\Rightarrow 3)$ is clear and $3)\Leftrightarrow 1)$ follows from the construction of left triangles in $\ul{\ca}$ (\cf \cite{BeligiannisMarmaridis1994} or section \ref{(Co)suspended, triangulated and pretriangulated categories}).

$3)\Rightarrow 2)$ Consider the cartesian square of 2) and a morphism $g':T'\ra M$ with $T'\in\ct$. Since $p_{N}$ is a $\ct$-precover,
there exists a morphism $h':T'\ra T_{N}$ such that $fg'=p_{N}h'$. Now, by the universal property of the cartesian square, there exists
a morphism $t:T'\ra T$ such that $gt=g'$.
\end{proof}

\begin{remark}\label{observaciones}
\begin{enumerate}[1)]
\item Applying the duality principle, we obtain corresponding results, dual to  the above,  when considering the right triangulated structure in $\ul{\ca}$ defined by a covariantly finite additive full subcategory $\ct$ of $\ca$ closed under direct summands. We leave the statements to the reader.
\item Under the hypotheses of Proposition ~\ref{left zero}, if $f:M\ra N$ is a morphism in $\ca$ such that $\ol{f}$ is a strong monomorphism  in
$\ul{\ca}$, then $\ol{f}\cong\ol{g}$, for some morphism $g:M'\ra N$ such that $\ker(g)\in\ct$. Indeed, take 
\[g=\left[\begin{array}{cc}f&p_Y\end{array}\right]:M\oplus T_N\ra N,
\] 
where $p_N:T_N\ra N$ is a $\ct$-precover.
\item If, in the situation of $2)$, the subcategory $\ct$ consists of injective objects of $\ca$, then $g$ can be chosen to be a monomorphism in $\ca$. Indeed, according to $2)$, we can assume from the beginning that $\ker(f)\in\ct$. Then the monomorphism $\ker(f)\ra M$ is a section and we can write
\[f=\left[\begin{array}{cc}0&f_{|M'}\end{array}\right]:\ker(f)\oplus M'\ra N,
\] 
for some complement $M'$ of $\ker(f)$ in $M$. Now choose $g=f_{|M'}$
\end{enumerate}
\end{remark}

\section{When $\ct$ is a Serre class}\label{When T is a Serre class}
\addcontentsline{lot}{section}{2.3. Cuando $\ct$ es una clase de Serre}

Recall the following definition:

\begin{definition}\label{Serre class}
A full subcategory $\ct$ of an abelian category $\ca$ is a \emph{Serre class}\index{subcategory!Serre}\index{class!Serre} if, given an exact sequence $0\ra L\ra M\ra
N\ra 0$ in $\ca$, the object $M$ belongs to $\ct$ if and only if $L$ and $N$ belong to $\ct$. 
\end{definition}

All throughout this section, we assume that $\ct$ is a Serre class in $\ca$. The selfdual condition of Serre classes allows to give a dual result of any one that we shall give here.

In Definition \ref{preenvelope} we introduced the notion of \emph{precover}. In the proof of the following lemma we will use the stronger notion of \emph{cover}, which we remind:

\begin{definition}\label{cover}
Let $\cc$ be an additive category and let $\cu$ be a full subcategory of $\cc$ closed under direct sums and direct summands. A morphism $p_{M}:U_{M}\ra M$ is a \emph{$\cu$-cover}\index{cover} of $M$ if it is a $\cu$-precover and any endomorphism $f:U_{M}\ra U_{M}$ is an automorphism provided the following diagram commutes:
\[\xymatrix{U_{M}\ar[r]^{p_{M}}\ar[d]_{f} & M \\
U_{M}\ar[ur]_{p_{M}} &
}
\]
The dual notion is that of \emph{$\cu$-envelope}\index{envelope}. Alternatively, a $\cu$-cover is also called \emph{minimal right $\cu$-approximation}\index{approximation!right!minimal} and a $\cu$-envelope is also called \emph{minimal left $\cu$-approximation}\index{approximation!left!minimal}.
\end{definition}

\begin{lemma} \label{(co- contra-)variantly finite Serre classes}
Let $\ca$ be an abelian category and $\cu$ be a full subcategory closed under quotients and extensions. The following assertions are equivalent:
\begin{enumerate}[1)]
\item $\cu$ is contravariantly finite in $\ca$.
\item The inclusion functor $i:\cu\hookrightarrow\ca$ has a right adjoint.
\item $(\cu ,\cu^\perp )$ is a  torsion pair in $\ca$.
\end{enumerate}
\end{lemma}
\begin{proof}
$3)\Rightarrow 2)$ and $2)\Rightarrow 1)$ are clear. We just need to prove $1)\Rightarrow 3)$. If $p_M:U_M\ra M$ is a $\cu$-precover of $M$, then $\im(p_M)=:u(M)$ belongs to $\cu$ and one easily shows that the inclusion morphism $j_M:u(M)\ra M$ is a $\cu$-cover. Indeed, it is straightforward to check that it is a $\cu$-precover, and if $f:u(M)\ra u(M)$ is an endomorphism such that $j_{M}f=j_{M}$, we have $f=\id_{u(M)}$ since $j_{M}$ is a monomorphism. Therefore, $j_{M}$ is a $\cu$-cover and whence uniquely determined up to isomorphism. We leave to the reader the easy verification that the assignment $M\mapsto u(M)$ extends to a functor $u:\ca\ra\cu$ which is right adjoint to
the inclusion.

Now $\cu^\perp$ is precisely the class of those $F\in\ca$ such that $u(F)=0$ and we just need to prove that $\cu =^\perp (\cu^\perp )$ or, what is enough, that $u(M/u(M))=0$ for all $M\in\ca$. But this is clear for $u(M/u(M))={U}/{u(M)}$ for some subobject $U$ of $M$ containing $u(M)$ and fitting in a short exact sequence $0\ra u(M)\ra U\ra
u(M/u(M))\ra 0$. Since the two outer terms belong to $\cu$ we also have $U\in\cu$. But then the canonical inclusion $U\hookrightarrow M$ factors through $j_M:u(M)\ra M$, which implies that $U\subseteq u(M)$ and, hence, we get $u(M/u(M))=0$ as desired.
\end{proof}

Our next result identifies the monomorphisms and epimorphisms in $\ul{\ca}$.

\begin{proposition} \label{triangles for Serre classes}
The following statements hold for a morphism $f:M\ra N$ in $\ca$:
\begin{enumerate}[1)]
\item $\ol{f}$ is a monomorphism in $\ul{\ca}$ if and only if the kernel $\ker(f)$ of $f$ belongs to $\ct$. In case $\ct$ is contravariantly
finite in $\ca$, that happens if and only if $\ol{f}$ is a strong monomorphism in $\ul{\ca}$.
\item[1')] $\ol{f}$ is an epimorphism in $\ul{\ca}$ if and only if the cokernel $\cok(f)$ of $f$ belongs to $\ct$. In case $\ct$ is covariantly finite in $\ca$, that happens if and
only if $\ol{f}$ is a strong epimorphism in $\ul{\ca}$.
\item If $M$ belongs to $\ct^\perp$, then $\ol{f}$ is an isomorphism in $\ul{\ca}$ if and only if $f$ is a section in $\ca$ whose cokernel $\cok(f)$ belongs to $\ct$.
\end{enumerate}
\end{proposition}
\begin{proof}
$1')$ follows from $1)$ by duality. So, we prove assertion $1)$. If $\ol{f}$ is a monomorphism in $\ul{\ca}$, then we know that we
have a factorization as follows:
\[f^k: \ker(f   )\arr{j}T\arr{v} X
\]
with $T\in\ct$ (notice that $j$ is a monomorphism). Since $\ct$ is closed under subobjects, we conclude that $\ker(f)\in\ct$.

Suppose now that $\ker(f)\in\ct$. According to Proposition ~\ref{monos}, we need to prove that, for every morphism $h:T\ra N$
with $T\in\ct$, the parallel to $h$ in the pullback of $f$ and $h$ factors through an object of $\ct$. But here we have even more,
for if $L$ is the upper left corner of that pullback, we have an exact sequence $0\ra \ker(f)\ra L\arr{h}T$. The fact that $\ct$ is
a Serre class implies that $L\in\ct$. This same argument also proves, with the help of Proposition ~\ref{left zero}, that if
$\ct$ is covariantly finite, then $\ol{f}$ is a strong monomorphism in $\ul{\ca}$. That ends the proof of assertion $1)$.

We next prove (the `only if' part of) assertion $2)$. If $M\in\ct^\perp$ and $\ol{f}$ is an isomorphism in $\ul{\ca}$, then
assertions $1)$ and $1')$ imply that $\ker(f), \cok(f)\in\ct$. But then $\ker(f)\in\ct\cap\ct^\perp$, which implies $\ker(f)=0$. On the other hand, since $\ol{f}$ is an isomorphism there exists a morphism $g:N\ra M$ such that $\id_M-gf$ factors through an object of $\ct$. That is, we
have a factorization $\id_M-gf:M\ra T\ra M$. But, since $M\in\ct^\perp$,  the second morphism of that factorization is
zero, so that $\id_M-gf=0$  and, hence,  $f$ is a section in $\ca$ with $\cok(f)\in\ct$.
\end{proof}

The main result of this section follows now easily:

\begin{corollary} \label{balanced in Serre case}
Let $\ca$ be an abelian category and $\ct$ be a Serre class in $\ca$. The following assertions are equivalent:
\begin{enumerate}[1)]
\item $\ct$ is contravariantly (respectively, covariantly) finite in $\ca$ and $\ul{\ca}$ is balanced.
\item $\ct$ is functorially finite in $\ca$ and $\ul{\ca}$ is weakly balanced.
\item $(\ct,\ct^{\bot})$ is a 2-decomposition of $\ca$.
\end{enumerate}
\end{corollary}
\begin{proof}
Since assertion $2)$ is selfdual we just consider the contravariantly finite version of assertion $1)$.

$2)\Rightarrow 1)$ follows from Proposition ~\ref{triangles for Serre classes}.

$3)\Rightarrow 2)$ For an arbitrary object $M$ of $\ca$ we consider the decomposition $M=M'\oplus M''$ with $M'\in\ct$ and $M''\in\ct^{\bot}$, guaranteed by the hypothesis of 3). It is straightforward to check that the canonical morphisms $M'\ra M$ and $M\ra M'$ are a $\ct$-precover and a $\ct$-preenvelope, respectively. This proves that $\ct$ is functorially finite. Now, the composition of the quotient with the inclusion
\[\ct^{\bot}\hookrightarrow\ca\ra\ul{\ca}
\]
is an additive functor, which is obviously fully faithful. It is also essentially surjective, since any object $M$ of $\ca$ is isomorphic to $M''$ in $\ul{\ca}$. This gives us an equivalence of additive categories
\[\ct^{\bot}\simeq\ul{\ca}.
\]
Therefore, since $\ct^{\bot}$ is abelian, so is $\ul{\ca}$. In particular, $\ul{\ca}$ is balanced.

$1)\Rightarrow 3)$  We first prove that $\ca (F,T)=0$, for all $F\in\ct^\perp$ and $T\in\ct$ or, equivalently, that
$\ct^\perp\subseteq ^\perp\ct$. Indeed, let $f:F\ra T$ be a morphism. Replacing $T$ by $\im(f)$ if necessary, we can
assume that $f$ is an epimorphism in $\ca$ and we need to prove that $T=0$. For that we consider the monomorphism $f^k:\ker(f)\ra F$. By Proposition ~\ref{triangles for Serre classes}, we know that $\ol{f^k}$ is both a monomorphism and an epimorphism in $\ul{\ca}$. Since this is a balanced category, we conclude that $\ol{f^k}$ is an isomorphism. But then Proposition ~\ref{triangles for Serre classes} gives that $f^k$ is a section (with cokernel in
$\ct$). Then $f$ is a retraction, but any section for it must be zero because $T\in\ct$ and $F\in\ct^\perp$. Therefore $T=0$ as desired.

On the other hand, since $(\ct ,\ct^\perp )$ is a torsion pair (\cf Lemma \ref{(co- contra-)variantly finite Serre classes}) we can consider,  for every object $M\in\ca$, the canonical exact sequence 
\[0\ra t(M)\ra M\arr{p_M}M/t(M)\ra 0.
\] 
According to Proposition \ref{triangles for Serre classes},  we know that
$\ol{p}_M$ is both a monomorphism and an epimorphism, whence an isomorphism,  in $\ul{\ca}$. Since $M/t(M)\in\ct^\perp\subseteq
^\perp\ct$ the dual of Proposition \ref{triangles for Serre classes}(2) says that $p_M$ is a retraction in $\ca$. Then we have
an isomorphism $M\cong t(M)\oplus(M/t(M))$ and the proof is finished.
\end{proof}

\section{Balance when $\ct$ consists of projective objects}\label{Balance when T consists of projective objects}
\addcontentsline{lot}{section}{2.4. Equilibrio cuando $\ct$ est\'a formada por proyectivos}

Throughout the rest of the chapter, unless explicitly said
otherwise,   the full subcategory $\ct$ of $\ca$ consists of
projective objects. Notice that we do not assume that $\ca$ has
enough projectives. Dual results of ours can be obtained by
assuming instead that $\ct$ consists of injective objects. We
leave their statement to the reader.

Our next observation is that every epimorphism in $\ul{\ca}$ can be represented by an epimorphism of $\ca$, something that was already noticed  by M. Auslander and M.~Bridger in the (classical) projectively stable category of a module category (\cf \cite{AuslanderBridger1969} and, for related
questions, see also \cite{Kato2005}).

\begin{lemma}\label{reduction to epis}
Let $f\in\ca(M,N)$ be such that $\ol{f}$ is an epimorphism in $\ul{\ca}$. Then the canonical monomorphism $\im(f)\ra N$ yields
a retraction in $\ul{\ca}$ and there exists an epimorphism $f'\in\ca(M',N)$ such that $\ol{f'}\cong\ol{f}$ in $\ul{\ca}$.
\end{lemma}
\begin{proof}
Let $f^c$ be the cokernel of $f$ in $\ca$. From $f^cf=0$ and the fact that $\ol{f}$ is an epimorphism we deduce that $\ol{f^c}=\ol{0}$. Therefore,
there exists a factorization as follows:
\[f^c:N\arr{h}T\arr{p}\cok(f)
\]
with $T\in\ct$.

Now, since $T$ is projective there exists a morphism $g:T\ra N$ such that $f^cg=p$, and so $f^c(gh-\id_{N})=0$. Hence $gh-\id_{N}$
factors through the canonical morphism $\im(f)\ra N$, which is the kernel of $f^c$. Then $\im(f)\ra N$ yieds a retraction in
$\ul{\ca}$. This proves the first assertion.

Of course, the morphism 
\[f':=\left[\begin{array}{cc}f&g\end{array}\right]: M\oplus T\ra N
\] 
is a morphism such that  $\ol{f}'=\ol{f}$ in
$\ul{\ca}$. Let us see that $f'$ is an epimorphism in $\ca$. If we have $\varphi f'=0$, then $\varphi f=0$ and $\varphi g=0$. Hence,
there exists a morphism $u$ such that $uf^c=\varphi$, which implies $uf^cg=0$, \ie $up=0$. Since $p$ is an epimorphism we get $u=0$ and so $\varphi=0$.
\end{proof}

Since our first goal is to characterize when $\ul{\ca}$ is balanced, the above observation suggests to characterize those epimorphisms in $\ca$ which become epimorphisms, monomorphisms or isomorphisms when passing to $\ul{\ca}$. Our next results go in that direction.

\begin{proposition}\label{epimorphisms which become stable-epis}
Let $f\in\ca(M,N)$ be an epimorphism in $\ca$. The following assertions are equivalent:
\begin{enumerate}[1)]
\item $\ol{f}$ is an epimorphism in $\ul{\ca}$.
\item For every morphism $h:M\ra T$ with $T\in\ct$, there is a morphism $\tilde{h}:M\ra h(\ker(f))$ such that $h$ and $\tilde{h}$ coincide
on $\ker(f)$:
\[\xymatrix{\ker(f)\ar[r]^{f^k} & M\ar[r]^f\ar[dr]^{h}\ar@{.>}[d]_{\tilde{h}} & N \\
& h(\ker(f))\ar@{^(->}[r] & T
}
\]

When $\ct$ is also covariantly finite, the above conditions are also equivalent to:

\item For every (respectively, some) $\ct$-preenvelope $\mu^M:M\ra T^M$, there is a morphism $\tilde{\mu}:M\ra \mu^M(\ker(f))$ such that $\mu^M$ and
$\tilde{\mu}$ coincide on $\ker(f)$:
\[\xymatrix{\ker(f)\ar[r]^{f^k} & M\ar[r]^f\ar[dr]^{\mu^M}\ar@{.>}[d]_{\tilde{\mu}} & N \\
& \mu^{M}(\ker(f))\ar@{^(->}[r] & T^M
}
\]
\end{enumerate}
\end{proposition}
\begin{proof}
$1)\Rightarrow 2)$ According to the dual of Proposition ~\ref{monos}, if $\ol{f}$ is an epimorphism then, for every
morphism $h:M\ra T$ to an object  $T\in \ct$, the parallel to $h$ in the pushout of $f$ and $h$ factors through an
object of $\ct$. We consider that pushout:
\[\xymatrix{M\ar[r]^{f}\ar[d]_{h} & N\ar[d]^{v} \\
 T\ar[r]_{u} & L
}\] 
so that we have a factorization $v:N\arr{v_1}T'\arr{v_2}L$, with $T'\in\ct$. Since $u$ is an epimorphism in $\ca$ and $T'$ is projective, we get a morphism $\varphi :T'\ra T$ such that $u\varphi =v_2$. Then one gets $u\varphi v_1=v$. Hence $u(h-\varphi v_1 f)=uh-vf=0$. From that we get a morphism
$\tilde{h} :M\ra \ker(u)$ such that $h-\varphi v_1 f=u^{k}\tilde{h}$, where $u^k$ is the kernel of $u$, and thus it follows easily that $h$ and $\tilde{h}$ coincide on $\ker(f)$. We only have to notice that, since $f$ is an epimorphism in $\ca$, a classical construction of pushouts in abelian categories gives
that $u$ is the cokernel map of the composition $\ker(f)\arr{f^k} M\arr{h} T$. Then $u^k:\ker(u)\ra T$ gets identified with the canonical inclusion $h(\ker(f))\hookrightarrow T$. 

$2)\Rightarrow 1)$ Let $g:N\ra L$ be a morphism in $\ca$ such that $\ol{g}\ol{f}=\ol{0}$ in $\ul{\ca}$. Then we
have a factorization $gf:M\arr{h}T\arr{p}L$, where $T\in\ct$. The hypothesis says that we have a morphism
$\tilde{h}:M\ra h(\ker(f))$ such that $j\tilde{h}f^k=hf^k$, where $j:h(\ker(f))\ra T$ is the monomorphism of the canonical epi-mono factorization:
\[\xymatrix{\ker(f)\ar[rr]^{hf^k}\ar[dr]_{q} && T \\
& h(\ker(f))\ar[ur]_{j} &
}
\]
Then $(h-j\tilde{h})f^k=0$, which implies that $h-j\tilde{h}$ factors through the cokernel of $f^k$, \ie through $f$. Then we have a morphism $\xi :N\ra T$ such that $h-j\tilde{h}=\xi f$ and we get $gf=ph=p(j\tilde{h}+\xi f)$. Since $pjq=phf^k=gff^k=0$ we get that $pj=0$ and so $gf=p\xi f$. Being $f$ an epimorphism, we
conclude that $g=p\xi$,  and hence  $\ol{g}=\ol{0}$.

Suppose now that $\ct$ is covariantly finite.  We only have to prove (the weak version of) $3)\Rightarrow 2)$. Let us fix a $\ct$-preenvelope $\mu^M:M\ra T^M$ for which condition $3)$ holds. If now $h:M\ra T$ is a morphism in $\ca$ with $T\in\ct$, then there is a factorization $h:M\arr{\mu^M}T^M\arr{g}T$.
Then $g$ induces a morphism $\tilde{g}:\mu^M(\ker(f))\ra h(\ker(f))$. The hypothesis says that we have a morphism $\tilde{\mu}:M\ra\mu^M(\ker(f))$ such that $\mu^M$ and $\tilde{\mu}$ coincide on $\ker(f)$. Put now $\tilde{h}=:\tilde{g}\tilde{\mu}:M\ra h(\ker(f))$
and one easily checks that $h$ and $\tilde{h}$ coincide on $\ker(f)$.
\end{proof}

\begin{proposition} \label{epimorphisms which become stable-monos}
Let $f:M\ra N$ be an epimorphism in $\ca$. Then $\ol{f}$ is a monomorphism in $\ul{\ca}$ if and only if the canonical monomorphism $f^k: \ker(f)\ra M$ factors through an object of $\ct$.
\end{proposition}
\begin{proof}
If $\ol{f}$ is a monomorphism then, since $\ol{f}\ol{f^k}=\ol{0}$, we have $\ol{f^k}=\ol{0}$, \ie $f^k$ factors through $\ct$. On the other hand, suppose that the
canonical monomorphism $f^k:\ker(f)\ra M$ factors in the form $f^k:\ker(f)\arr{u}T\arr{q}M$, where $T\in\ct$. According to Proposition \ref{monos}, it is enough to prove that, for every morphism $p:T'\ra N$  from an object $T'\in\ct$,  the parallel to $p$ in the pullback of $f$ and $p$ is a morphism which factors
through an object of $\ct$. We then consider that pullback:
\[\xymatrix{ L\ar[r]^{r}\ar[d]_{g} & M\ar[d]^{f} \\
T'\ar[r]_{p} & N }
\]
Since $f$ is epimorphism it follows that $g$ is also epimorphism, and we have an exact sequence
\[0\ra \ker(f)\ra L\arr{g}T'\ra 0
\]
in $\ca$. Since $T'$ is projective this sequence splits and we can identify $L=\ker(f)\oplus T'$, 
\[g=\left[\begin{array}{cc}0&\id\end{array}\right]:\ker(f)\oplus T'\ra T'
\] 
and 
\[r=\left[\begin{array}{cc}f^k&\xi\end{array}\right]:\ker(f)\oplus T'\ra M,
\] 
where $\xi:T'\ra M$ is a morphism such that $f\xi=p$. But then we have the following factorization of $r$: \[\xymatrix{ \ker(f)\oplus T'\ar[rr]^{\scriptsize{\left[\begin{array}{cc}u&0\\0&\id\end{array}\right]}} &&
T\oplus T'\ar[rr]^{\scriptsize{\left[\begin{array}{cc}q&\xi\end{array}\right]}} && M }
\]
Therefore $r$ factors through the object $T\oplus T'\in\ct$.
\end{proof}

\begin{corollary}\label{epimorphisms which become stable-monepis}
Let $f:M\ra N$ be an epimorphism in $\ca$. The following assertions are equivalent:
\begin{enumerate}[1)]
 \item $\ol{f}$ is both a monomorphism and an epimorphism in $\ul{\ca}$ 
 \item The canonical monomorphism $f^k:\ker(f)\ra M$ factors through an object of $\ct$ and, for every morphism $h:M\ra T$ to an object of $\ct$, there is a morphism $\tilde{h}:M\ra h(\ker(f))$ such that $h$ and $\tilde{h}$ coincide on $\ker(f)$.
\end{enumerate}
\end{corollary}

The following result finishes this series of preliminaries leading
to the characterization of when $\ul{\ca}$ is balanced.

\begin{proposition} \label{epimorphisms which become stable-isos}
Let $f:M\ra N$ be an epimorphism in $\ca$. The following assertions are equivalent:
\begin{enumerate}[1)]
\item $\ol{f}$ is an isomorphism in $\ul{\ca}$. 
\item $f$ is a retraction in $\ca$ with $\ker(f)\in\ct$. 
\item $f$ is a retraction in $\ca$ whose kernel map $f^k:\ker(f)\ra M$ factors through an object of $\ct$.
\end{enumerate}
\end{proposition}
\begin{proof}
$2)\Rightarrow 1)$ and $2)\Rightarrow 3)$ are clear.

$3)\Rightarrow 2)$ If 3) holds then $\ker(f)$ is a direct summand of an object of $\ct$, and our assumption on $\ct$ implies
that $\ker(f)\in\ct$.

$1)\Rightarrow 3)$ Since $\ol{f}$ is an isomorphism it is also a monomorphism and an epimorphism in $\ul{\ca}$, so that Corollary \ref{epimorphisms which become stable-monepis} applies. Let us choose $g:N\ra M$ such that $\ol{g}=\ol{f}^{-1}$. Then $\id_M-gf$ factors in the form
$M\arr{h}T'\arr{v}M$, with $T'\in\ct$. Applying Corollary \ref{epimorphisms which become stable-monepis}, we get a morphism $\tilde{h}:M\ra h(\ker(f))$ such that $j\tilde{h} f^k=h f^k$, where $j:h(\ker(f))\hookrightarrow T'$ is the canonical inclusion (\ie, the monomorphism in the epi-mono factorization of $h f^k$).

Notice now that the restriction of $v h=\id_M-g f$ to $\ker(f)$ is the identity map, in particular $(vh)(\ker(f))=\ker(f)$. If we denote by $\hat{v}$ the  morphism
$h(\ker(f))\ra (vh)(\ker(f))=\ker(f)$ induced by $v$ then, by definition, we have $f^k\hat{v}=v j$. But then
$f^k\hat{v}\tilde{h} f^k= vj\tilde{h} f^k=v h f^k=(\id_M-g f)f^k=f^k$.  Since $f^k$ is a monomorphism we then get that $\hat{v}\tilde{h} f^k$ is the identity on
$\ker(f)$. Therefore $f^k$ is a section which, according to Corollary \ref{epimorphisms which become stable-monepis}, factors through an object of $\ct$.
\end{proof}

We come now to the main result of this section.

\begin{theorem}\label{characterization balanced}
Let $\ca$ be an abelian category and  $\ct$ be a full subcategory consisting of projective objects which is closed under finite direct sums and direct summands. The following assertions are equivalent:
\begin{enumerate}[1)]
\item $\ul{\ca}$ is balanced.
\item If $\mu:K\ra M$ is a non-split monomorphism which factors through an object of $\ct$,  then there exists a morphism $h:M\ra T'$, with $T'\in\ct$,
such that no morphism $\tilde{h}:M\ra (h\mu)(K)$ coincides with $h$ on $K$.
\item If $\mu:T\ra M$ is a non-split monomorphism with $T\in\ct$, then there exists a morphism $h:M\ra T'$, with $T'\in\ct$, such that no morphism $\tilde{h}:M\ra (h\mu)(T)$ coincides with $h$ on $T$.
\item If $f:M\ra N$ is an epimorphism satisfying conditions $i)$ and $ii)$ below, then it is a retraction:
\begin{enumerate}
\item[i)] $f^k:\ker(f)\ra M$ factors through an object of $\ct$.
\item[ii)] For every $h:M\ra T\in\ct$ the canonical epimorphism $\ker(f)\twoheadrightarrow h(\ker(f))$ factors through $f^k:\ker(f)\ra M$.
\end{enumerate}
\item For every non-split epimorphism $f:M\ra N$ with $\ker(f)\in\ct$, there exists a morphism $h:M\ra T'$, with $T'\in\ct$, such that $\ker(f)+\ker(h-gf)$ is strictly contained in $M$ for all morphisms $g:N\ra T'$.
\end{enumerate}
\end{theorem}
\begin{proof}
$4)\Leftrightarrow 2)$ and $2)\Rightarrow 3)$ are clear.

$1)\Leftrightarrow 4)$ is a consequence of Corollary \ref{epimorphisms which become stable-monepis} and Proposition
\ref{epimorphisms which become stable-isos}, bearing in mind that
every epimorphism of $\ul{\ca}$ can be represented by an
epimorphism of $\ca$.

$3)\Rightarrow 2)$ Let $\mu:K\ra M$ be a non-split monomorphism of $\ca$ which admits  a factorization
\[\mu:K\arr{u}T\arr{v}M
\]
with $T\in\ct$. Consider the following pushout diagram $(*)$ with exact rows
\[\xymatrix{0\ar[r] & K\ar[r]^{\mu}\ar[d]^{u} & M\ar[r]\ar[d]^{u'} & N\ar[r]\ar@{=}[d] & 0 \\
0\ar[r] & T\ar[r]^{\mu'} & M'\ar[r] & N\ar[r] & 0 }
\]
We work separately the case in which $\mu '$ does not split and the case in which it  splits. In the first case, assertion $3)$ gives a morphism $h':M'\ra T'$, with $T'\in\ct$, such that no morphism $\tilde{h}':M'\ra h'\mu'(T)$ coincides with $h'$ on $T$. We claim that $h=:h'u'$ satisfies the property needed
in assertion $2)$. Suppose not, so that there exists a morphism $\tilde{h}:M\ra (h\mu)(K)=(h'u'\mu)(K)=(h'\mu 'u)(K)$ which coincides with $h$ on $K$. Since $u(K)\subseteq T$ we get $(h'\mu 'u)(K)\subseteq (h'\mu ')(T)$ and the composition $M\arr{\tilde{h}}(h'\mu 'u)(K)\hookrightarrow (h'\mu ')(T)$ is denoted by $\alpha$. If $\hat{h}':T\ra (h'\mu ')(T)$ denotes the unique morphism such that the composition $T\arr{\hat{h}'}(h'\mu')(T)\hookrightarrow T'$ coincides with $h'\mu '$, then we have the following commutative diagram, which can be completed with $\tilde{h'}$ due to the universal property of pushouts:
\[\xymatrix{K\ar[r]^{\mu}\ar[d]^{u} & M\ar[d]^{u'}\ar@/^/[ddr]^{\alpha} & \\
T\ar[r]^{\mu'}\ar@/_/[drr]_{\hat{h}'} & M'\ar@{.>}[dr]^{\tilde{h'}} & \\
&& (h'\mu')(T) }
\]
But then $\tilde{h}'$ coincides with $h'$ on $T$ due to the equality $\tilde{h}'\mu '=\hat{h}'$. That contradicts the choice of $h'$.

Suppose next that $\mu'$ splits. Then we rewrite the pushout diagram $(*)$ as:
\[\xymatrix{0\ar[r] & K\ar[r]^{\mu}\ar[d]^{u} & M\ar[rr]^{f}\ar[d]^{\scriptsize{\left[\begin{array}{c}h\\f\end{array}\right]}} && N\ar@{=}[d]\ar[r] & 0 \\
0\ar[r] & T\ar[r]_{\scriptsize{\left[\begin{array}{c}\id\\0\end{array}\right]}\ \ \ \ \ } & T\oplus
N\ar[rr]_{\scriptsize{\left[\begin{array}{cc}0&\id\end{array}\right]}} && N\ar[r] & 0 }
\]
We claim that $h$ satisfies the property required by assertion $2)$. Indeed, suppose not, so that  there exists $\tilde{h}:M\ra (h\mu)(K)=u(K)$ that coincides with $h$ on $K$. Since $u$ is a monomorphism the induced morphism $\hat{u}:K\arr{\sim}u(K)$ is an isomorphism, which, in addition, has the property that
$u\hat{u}^{-1}$ is the canonical inclusion $j:u(K)=(h\mu)(K)\hookrightarrow T$. Consider the composition 
\[\alpha:M\arr{\tilde{h}}u(K)\arr{\hat{u}^{-1}}K.
\]
Then $u\alpha\mu=u\hat{u}^{-1}\tilde{h}\mu=j\tilde{h}\mu=h\mu=u$. Since $u$ is a monomorphism we get $\alpha\mu=\id_{K}$, and so $\mu$ split, which is a contradiction.

$3)\Rightarrow 5)$ Let $f:M\ra N$ be an epimorphism as indicated and let $\mu:T:=\ker(f)\ra M$ be its kernel. Choose a morphism
$h:M\ra T'$ as given by assertion $3)$. Suppose that there is a morphism $g:N\ra T'$ such that $T+\ker(h-gf)=X$. Then we have:
$\im(h-gf)=(h-gf)(X)=(h-gf)(T)\subseteq h(T)+(gf)(T)=h(T)$. That gives a morphism $\tilde{h}:M\ra h(T)$ such that $i\tilde{h}=h-gf$, where $i:h(T)\ra T'$ is the canonical inclusion. But then $i\tilde{h}\mu=(h-gf)\mu=h\mu$, which contradicts the choice of $h$.

$5)\Rightarrow 3)$ Let $\mu:T\ra M$ be a non-split monomorphism with $T\in\ct$ and let $f:M\ra N$ be its cokernel. We pick up a
morphism $h:M\ra T'\in\ct$ satisfying the hypothesis of condition $5)$. We will prove that it also satisfies the requirements of condition $3)$. Suppose that there is a morphism $\tilde{h}:M\ra \im(h\mu)$ such that in the diagram 
\[\xymatrix{T\ar[r]^{\mu}\ar[d]_{p} & M\ar[d]^{h}\ar[dl]_{\tilde{h}}\\
\im(h\mu)\ar[r]_{i} & T' }
\]
we have $h\mu=i\tilde{h}\mu$, where $ip=h\mu$ is the canonical epi-mono factorization of $h\mu$. Then $(h-i\tilde{h})\mu=0$ and so there exists a morphism $g:N\ra T'$ such that $i\tilde{h}=h-gf$. Thus $\ker(h-gf)=\ker(i\tilde{h})\cong\ker(\tilde{h})$. We claim that the morphism
\[\left[\begin{array}{cc}\tilde{h}^k&\mu\end{array}\right]: \ker(\tilde{h})\oplus T\ra M
\]
is an epimorphism and that will imply that $\ker(h-gf)+\ker(f)=M$, thus yielding a contradiction. Indeed, from $\tilde{h}\mu=p$ we get that $\tilde{h}$ is an epimorphism, and so it is the cokernel of its kernel. Now, if $a$ is a morphism such that $a\tilde{h}^k=0$ and $a\mu=0$, then there exists a morphism $b$
such that $b\tilde{h}=a$. Hence $0=b\tilde{h}\mu=bp$, which implies $b=0$, and therefore $a=0$.
\end{proof}

\begin{corollary} \label{balanced wrt projective-injective}
Let $\ca$ be an abelian category and let $\ct$ be a full subcategory of $\ca$ consisting of projective-injective objects, closed under finite direct sums and direct summands. Then $\ul{\ca}$ is balanced.
\end{corollary}
\begin{proof}
Condition $3)$ of Theorem ~\ref{characterization balanced} trivially holds.
\end{proof}

\section{Weak balance when $\ct$ consists of projective objects}\label{Weak balance when T consists of projective objects}
\addcontentsline{lot}{section}{2.5. Equilibrio d\'ebil cuando $\ct$ est\'a formada por proyectivos}

In this section the hypotheses on $\ct$ are the same as in section 3, but, in order to have suitable one-sided (pre)triangulated structures on $\ul{\ca}$, further assumptions will be made in each particular result. We already know that strong epimorphisms in $\ul{\ca}$, whenever they make sense, are represented by epimorphisms. In order to characterize the weak balance of $\ul{\ca}$ in case $\ct$ is functorially finite, we will identify first those epimorphisms in $\ca$ which become strong epimorphisms and strong monos in $\ul{\ca}$.

\begin{proposition}\label{objetos cero en right triangles2}
Let $f\in\ca(M,N)$ be an epimorphism and suppose that $\ct$ is covariantly finite in $\ca$. The following assertions are equivalent:
\begin{enumerate}[1)]
\item $\ol{f}$ is a strong epimorphism in the right triangulated category $\ul{\ca}$. 
\item For every (respectively, some) $\ct$-preenvelope $\mu^M:M\ra T^M$, one has that the canonical monomorphism $\mu^M(\ker(f))\ra T^M$ splits. 
\item There is a decomposition $M=T'\oplus M'$, with $T'\in\ct$, and
\[f=\left[\begin{array}{cc}0 & f'\end{array}\right]: T'\oplus M'\ra N
\]
such that $\ca(f',T):\ca(N,T)\arr{\sim}\ca(M',T)$ is an isomorphism for every  $T\in\ct$.
\item There is a decomposition $M=T'\oplus M'$, with $T'\in\ct$, and
\[f=\left[\begin{array}{cc}0 & f'\end{array}\right]: T'\oplus M'\ra N
\]
such that every morphism $h:M'\ra T$ to an object of $\ct$ vanishes on $\ker(f')$.
\end{enumerate}
\end{proposition}
\begin{proof}
$1)\Leftrightarrow 2)$ Let $\mu^M:M\ra T^M$ be any $\ct$-preenvelope. Since $f$ is an epimorphism a canonical construction gives a cocartesian square of the form:
\[\xymatrix{M\ar[r]^{\mu^M}\ar[d]_{f} & T^{M}\ar[d]^{\pi} \\
N\ar[r] & T^{M}/\mu^M(\ker(f))
}
\]
where $\pi$ is the canonical projection. Now, from the dual of Proposition ~\ref{left zero}, we get that condition $1)$ holds if and only if $T^M/\mu^M(\ker(f))\in\ct$. Since $\ct$ consists of projective objects, the latter is equivalent to $2)$.

$3)\Leftrightarrow 4)$ Let $f':M'\ra N$ be an epimorphism in $\ca$. Clearly, the map $\ca (f',T):\ca (N,T)\ra\ca (M',T)$ is an
isomorphism if and only if every morphism $h:M'\ra T$ vanishes on $\ker(f')$. From that the equivalence of assertions $3)$ and $4)$ is obvious.

$4)\Rightarrow 2)$ If we have a decomposition 
\[f=\left[\begin{array}{cc}0&f'\end{array}\right]:M=T'\oplus M'\ra N
\] 
as given by condition $4)$, then $f'$ is necessarily an epimorphism in $\ca$. Let now 
\[\mu^M=\left[\begin{array}{cc}u & v \end{array}\right]:T'\oplus M'\ra T^M
\]
be any $\ct$-preenvelope. One readily sees that $u$ is a section, so that $\mu^M$ can be written as a matrix 
\[\mu^M=\left[\begin{array}{cc}\id&\varphi\\0&\psi\end{array}\right]:T'\oplus M'\ra T'\oplus T''=T^M.
\] 
One then checks that $\psi:M'\ra T''$ is a $\ct$-preenvelope and, in particular, we have $\varphi =\rho\psi$, for some morphism $\rho
:T''\ra T'$. Now we have 
\[\mu^M(\ker(f))=\left[\begin{array}{cc}\id&\rho\psi\\0&\psi\end{array}\right](T'\oplus \ker(f')).
\] 
By hypothesis, $\psi (\ker(f'))=0$, from which it follows that $\mu^M(\ker(f))=T'\oplus 0$ is a direct summand of $T^M$.

$2)\Rightarrow 4)$ Let us fix a $\ct$-preenvelope $\mu^M:M\ra T^M$ satisfying condition $2)$. Then the composition of inclusions
$\mu^M(\ker(f))\hookrightarrow \im(\mu^M)\hookrightarrow T^M$ is a section in $\ct$, which implies that $T'=:\mu^M(\ker(f))$ is also a
direct summand of $\im(\mu^M)$, so that we have a decomposition $\im(\mu^M)=T'\oplus V$. Looking now at the induced epimorphisms
$\tilde{\mu}:M\twoheadrightarrow \im(\mu^M)=T'\oplus V$ and $\ker(f)\twoheadrightarrow \mu^M(\ker(f))=T'$,  and bearing in mind
that $T'$ is a projective object, we can assume without loss of generality that there are decompositions $M=T'\oplus M'$ and
\[f=\left[\begin{array}{cc}0 & f' \end{array}\right]:T'\oplus M'\ra N
\] 
(\ie, $T'\subseteq \ker(f)$). Then $\mu^M$ gets identified with
\[\left[\begin{array}{cc}1 & 0\\ 0 & \mu^{M'}\end{array}\right]:T'\oplus M'\ra T'\oplus T^{M'}=T^{M},
\] 
where $\mu^{M'}:M'\ra T^{M'}$ is a $\ct$-preenvelope such that $\mu^{M'}(\ker(f'))=0$. Since every morphism $h:M'\ra T$, with $T\in\ct$, factors through $\mu^{M'}$ assertion $4)$ follows.
\end{proof}

\begin{proposition} \label{epimorphisms which become strong monos}
Let $f\in\ca (M,N)$ be an epimorphism and suppose that $\ct$ is contravariantly finite in $\ca$. In the left triangulated category $\ul{\ca}$, the morphism $\ol{f}$ is a strong monomorphism if and only if $\ker(f)\in\ct$. 
\end{proposition}
\begin{proof}
If $p_N:T_N\ra N$ is any $\ct$-precover, then the pullback of $f$ and $p_{N}$
\[\xymatrix{L\ar[r]^{g}\ar[d]_{r} & T_N\ar[d]^{p_N} \\
 M\ar[r]_{f} & N
}
\]
gives rise to a short exact sequence $0\ra \ker(f)\ra L\arr{g} T_N\ra 0$ in $\ca$, which splits due to the fact that $T_N$ is projective. Then $\ker(f)$ belongs to $\ct$ if and only if $L$ belongs to $\ct$. Then the result follows from Proposition ~\ref{left zero}.
\end{proof}

\begin{corollary} \label{epimorphisms with zero left-and-right vertex}
Suppose that $\ct$ is functorially finite in $\ca$ and let $f:M\ra N$ be an epimorphism in $\ca$. The following assertions are equivalent:
\begin{enumerate}[1)]
 \item $\ol{f}$ is a strong monomorphism and a strong epimorphism in the pretriangulated category $\ul{\ca}$.
\item $\ker(f)\in\ct$ and, for some $\ct$-preenvelope $\mu^M:M\ra T^M$, the object $\mu^M(\ker(f))$ is a direct summand of $T^M$.
\item $\ker(f)\in\ct$ and there is a decomposition 
\[f=\left[\begin{array}{cc}0&f'\end{array}\right]:M=T'\oplus M'\ra N
\] 
such
that every morphism $h:M'\ra T$ to an object of $\ct$ vanishes on $\ker(f')$
\end{enumerate}
\end{corollary}
\begin{proof}
Direct consequence of Propositions \ref{objetos cero en right triangles2} and \ref{epimorphisms which become strong monos}.
\end{proof}

All these preliminaries lead to the following main result of this section.

\begin{theorem} \label{characterization weakly-balanced}
Let $\ca$ be an abelian category and $\ct\subseteq\ca$ be a functorially finite full subcategory consisting of projective objects and closed under direct summands. Consider the following assertions:
\begin{enumerate}[1)]
\item For every object $T\in\ct$, there is a monomorphism $T\rightarrowtail E$, where $E$ is an injective(-projective) object of $\ca$ which belongs to $\ct$. \item For every $T\in\ct\setminus\{0\}$, there is a non-zero morphism $\varphi:T\ra T'$, with $T'\in\ct$, which factors through an injective object of $\ca$. 
\item The pretriangulated category $\ul{\ca}=\ca/\langle\ct\rangle$ is weakly balanced. 
\item If $f:M\ra N$ is an epimorphism in $\ca$ such that $\ker(f)\in\ct$ and every morphism $h:M\ra T\in\ct$ vanishes on $\ker(f)$, then $f$ is an isomorphism in $\ca$. 
\item If $j: T\ra M$ is a non-zero monomorphism in $\ca$, with $T\in\ct$, then there is a morphism $h:M\ra T'$ such that $h j\neq 0$, for some $T'\in\ct$. 
\item If $N$ is an object of $\ca$ such that the covariant functor
\[\Ext_{\ca}^1(N,?)_{|\ct}:\ct\ra \Mod\Z
\]
is non-zero, then $\Ext_{\ca}^1(N,?)_{|\ct}$ does not contain a non-zero representable subfunctor.
\end{enumerate}
Then $1)\Rightarrow 2)\Rightarrow 3)\Leftrightarrow 4)\Leftrightarrow 5)\Leftrightarrow 6)$ and, in case $\ca$ has enough injectives, all assertions $2)-6)$ are equivalent.
\end{theorem}
\begin{proof}
$1)\Rightarrow 2)$ is clear.

$2)\Rightarrow 4)$ Let $f:M\ra N$ be an epimorphism in $\ca$ satisfying the hypothesis of $(4)$. Suppose that $f$ is not an isomorphism. Then $\ker(f)\in\ct\setminus\{0\}$, and by $(2)$ there exists a non-zero morphism $\varphi:\ker(f)\ra T'$, for some $T'\in\ct$, which factors through an injective object of $\ca$.
Then $\varphi$ extends to $M$ and this contradicts the hypothesis of $4)$.

$3)\Leftrightarrow 4)$ is a direct consequence of Corollary \ref{epimorphisms with zero left-and-right vertex} and Proposition \ref{epimorphisms which become stable-isos}, bearing in mind that strong epimorphisms of $\ul{\ca}$ can be represented by epimorphisms of $\ca$.

$4)\Leftrightarrow 5)$ is clear.

$4)\Rightarrow 6)$ Let $N\in\ca$ be an object such that $\Ext_{\ca}^1(N,?)_{|\ct}:\ct\ra\Mod\Z$ is a non-zero functor. Suppose that we have a monomorphism of functors $\mu:\ct (T,?)\ra\Ext_{\ca}^1(N,?)_{|\ct}$.  We choose the element $\tilde{\mu}\in\Ext_{\ca}^1(N,T)$ corresponding to $\mu$ by Yoneda's lemma. Then $\tilde{\mu}$ represents a short exact sequence $0\ra T\arr{j}M\arr{f}N\ra 0$ in $\ca$. Then $f$ is an epimorphism in $\ca$ such that $\ker(f)\in\ct$. On the other hand, if $h:M\ra T'$ is a morphism to an object of $\ct$, then $\mu_{T'}:\ct (T,T')\ra \Ext_{\ca}^1(N,T')$ maps $hj$ onto the (exact) lower row of the following pushout diagram:
\[\xymatrix{0\ar[r] & T\ar[r]^{j}\ar[d]_{hj} & M\ar[r]\ar[d] & N\ar[r]\ar@{=}[d] & 0 \\
0\ar[r] & T'\ar[r] & M'\ar[r] & N\ar[r] & 0
}
\]
That lower row splits because $hj$ factors through $j$. Therefore $\mu_{T'}(hj)=0$. The monomorphic condition of $\mu$ implies that $hj=0$. Thus, every morphism $h:M\ra T'$, with $T'\in\ct$, vanishes on $T=\ker(f)$. Now assertion $(4)$ says that $f$ is an isomorphism in $\ca$ and, hence, that $T=0$.

$6)\Rightarrow 4)$ Let $f:M\ra N$ be an epimorphism as indicated in assertion $4)$ and consider the associated short exact sequence $0\ra T\arr{j} M\arr{f}N\ra 0$ in $\ca$. By Yoneda's lemma, we get an induced morphism of functors $\mu :\ct (T,?)\ra\Ext_{\ca}^1(N,?)_{|\ct}$ (functors $\ct\ra\Mod\Z$). We claim
that $\mu$ is a monomorphism and then the hypothesis will give $\ct (T,?)=0$, whence $T=0$, which will end the proof of this implication. Let us prove our claim. If $T'\in\ct$ is any object then $\mu_{T'}:\ct (T,T')\ra \Ext_{\ca}^1(N,T')$ maps a morphism $\varphi:T\ra T'$ onto the (exact) lower row of the
following pushout diagram:
\[\xymatrix{0\ar[r] & T\ar[r]^{j}\ar[d]_{\varphi} & M\ar[r]\ar[d] & N\ar[r]\ar@{=}[d] & 0 \\
0\ar[r] & T'\ar[r] & M'\ar[r] & N\ar[r] & 0
}
\]
One has that $\mu_{T'}(\varphi )=0$ if and only if  that lower row splits. But this is equivalent to say that $\varphi$ is the restriction to $T=\ker(f)$ of a morphism $h:M\ra T'$. In that case, the hypothesis on $f$ says that $\varphi =0$. Therefore $\mu$ is a monomorphism of functors as desired.

We finally prove $5)\Rightarrow 2)$ in case $\ca$ has enough injectives. Let us take $0\neq T\in\ct$. There is a monomorphism $j:T\ra E$, with $E$ injective. Now assertion $5)$ says that there exists a morphism $h:E\ra T'$, with $T'\in\ct$, such that $\varphi =hj\neq 0$. That ends the proof.
\end{proof}

\begin{example}
Let $A$ be a right perfect left coherent ring such that the injective envelope of the (regular) right $A$-module $A_{A}$ is projective (\eg if $A$ is a QF3 in the sense of C. M. Ringel and H. Tachikawa \cite{Tachikawa1973}). Since the class $\op{Proj}A$ of projective modules is closed under small products it follows that $\op{Proj}A$ is covariantly finite in $\Mod A$, and hence functorially finite. The stable category of $\Mod A$ module the projective $A$-modules is a pretriangulated category which is weakly balanced due to Theorem \ref{characterization weakly-balanced}.
\end{example}

We shall end the section by characterizing (weak) balance in abelian categories in which subobjects of projective objects are projective. We denote by $\Sub(\ct)$\index{$\Sub(\ct)$} the full subcategory formed by the subobjects of objects in $\ct$. The following is an auxiliary lemma.

\begin{lemma} \label{covariantly-finite in hereditary}
Let $\ch$ be an abelian category in which subobjects of projective objects are projective and let $\ct$ be a full subcategory of $\ch$ consisting of projective objects and closed under direct summands. The following assertions are equivalent:
\begin{enumerate}[1)]
\item $\ct$ is covariantly finite in $\ch$. 
\item $\Sub(\ct)$ is covariantly finite in $\ch$ and every $P\in\Sub(\ct)$ has a $\ct$-preenvelope.  
\item $(^\perp\ct ,\Sub(\ct))$ is a (split) torsion pair in $\ch$ and every $P\in\Sub(\ct)$ has a $\ct$-preenvelope.
\end{enumerate}
\end{lemma}
\begin{proof}
$1)\Rightarrow 2)$ If $\mu^M:M\ra T^M$ is a $\ct$-preenvelope and we take its epi-mono factorization $M\arr{p}P\stackrel{j}{\hookrightarrow}T^M$,
then $p$ is a $\Sub(\ct)$-preenvelope.

$2)\Rightarrow 1)$ If $\lambda:M\ra P$ and $\mu:P\ra T^P$ are a $\Sub(\ct)$-preenvelope and a $\ct$-preenvelope, respectively, then $\mu\lambda$ is a
$\ct$-preenvelope.

$2)\Leftrightarrow 3)$ is a direct consequence of the dual of Lemma \ref{(co- contra-)variantly finite Serre classes}.
\end{proof}

We remind the following definition which we are about to use:

\begin{definition}\label{hereditary torsion pair}
A torsion pair $(\cx ,\cy )$ on an abelian category is called \emph{hereditary}\index{pair!torsion!hereditary}  if the torsion class $\cx$ is closed under subobjects.
\end{definition}

\begin{proposition} \label{hereditary weakly balanced}
Let $\ch$ be an abelian category in which subobjects of projective objects are projective, and let $\ct$ be a covariantly finite full subcategory of $\ch$ consisting of projective objects and closed under direct summands. The following assertions are equivalent: 
\begin{enumerate}[1)]
\item The stable category $\ul{\ch}$ of $\ch$ associated to $\ct$ is  balanced. 
\item  $^\perp\ct$ is closed under subobjects. 
\item The pair $(^\perp\ct,\Sub(\ct))$ is a hereditary (split) torsion pair in $\ch$.
\end{enumerate}
When $\ct$ is contravariantly finite in $\ch$, the above assertions are equivalent to:
\begin{enumerate}[4)]
\item  $\ul{\ch}$ is weakly balanced
\end{enumerate}
When $\ch$ has enough injectives,  assertions $1), 2)$ and $3)$ are also  equivalent to:
\begin{enumerate}[5)]
\item For every object $T\in\ct$, there is a monomorphism $T\ra E$, where $E$ is an injective(-projective) object of $\ch$ which belongs to $\ct$.
\end{enumerate}
\end{proposition}
\begin{proof}
We start by proving that, under anyone  of conditions $1)$, $3)$ or $4)$,  if
\[\lambda:=\left[\begin{array}{c}u\\ v\end{array}\right]:T\ra N\oplus P
\]
is a monomorphism, where $T\in\ct$, $N\in ^\perp\ct$ and $P$ is projective, then $v$ is also a monomorphism. Indeed, since subobjects of projective objects are projective we get that $T_2:=\im(v)$ is projective and the induced epimorphism $\tilde{v}:T\ra T_2$ splits. Then we can decompose $T\cong
T_1\oplus T_2$, where $T_1:=\ker(v)$, and then rewrite $\lambda$ as 
\[\left[\begin{array}{cc}u_1 & u_2 \\ 0 & v_2 \end{array}\right].
\] 
Since $\lambda$ is a monomorphism it follows easily that $u_1$ is a monomorphism. Under conditions 1) or 4), using condition $3)$ of Theorem ~\ref{characterization balanced} or condition $5)$ of Theorem \ref{characterization weakly-balanced},  one readily sees that necessarily $T_1=0$.
Under condition 3), one has that $T_1\in \ct\cap ^\perp\ct$ since the torsion pair $(^\perp\ct,\Sub(\ct))$ is hereditary. But then $T_1=0$ as well.

$1)\Rightarrow 2)$ By Lemma \ref{covariantly-finite in hereditary}, we know that $(^\perp\ct, \Sub(\ct ))$ is a split torsion pair in $\ch$. Let $i:M\ra N$ be a monomorphism, where $N\in ^\perp\ct$, and take any morphism $h:M\ra T$, with $T\in\ct$. We then take the bicartesian square:
\[\xymatrix{M\ar[r]^{i}\ar[d]_{h} & N\ar[d]^{\eta} \\
  T\ar[r]_{\lambda} & V
}
\]
Then $\lambda$ is a monomorphism. We decompose $V=V'\oplus P$, where $V'\in ^\perp\ct$ and $P\in\Sub(\ct)$. According to the
first paragraph of this proof, we can rewrite 
\[\lambda=\left[\begin{array}{c}\gamma \\ j\end{array}\right]:T\ra V'\oplus P,
\]
where $j:T\ra P$ is a monomorphism. On the other hand, since $(^\perp\ct, \Sub(\ct))$ is a torsion pair, we have $\ca(N,P)=0$. Then the second componen of $\eta$ vanishes
\[\eta=\left[\begin{array}{c}u\\ 0 \end{array}\right]:N\ra V'\oplus P.
\] 
The commutativity of the square implies that $jh=0$ and, since $j$ is a
monomorphism, we conclude that $h=0$. Therefore $M\in ^\perp\ct$.

$2)\Rightarrow 3)$ is a direct consequence of Lemma \ref{covariantly-finite in hereditary}.

$3)\Rightarrow 1)$ We shall check assertion $3)$ of Theorem ~\ref{characterization balanced}. Let $\mu:T\ra M$ be a non-split monomorphism in $\ch$ with $0\neq T\in\ct$.  By hypothesis, we can decompose $M=M'\oplus P$, where $M'\in ^\perp\ct$ and $P\in\Sub(\ct)$. Then the first paragraph of
this proof says that we have 
\[\mu=\left[\begin{array}{c}u \\ v\end{array}\right]:T\ra M'\oplus P,
\] 
where $v:T\ra P$ is a monomorphism. But, since $P\in \Sub(\ct)$, we have a monomorphism $w:P\ra T'$, with $T'\in\ct$. Clearly $wv\neq 0$ (it is actually a monomorphism). We then claim that 
\[h:=\left[\begin{array}{cc}0&w\end{array}\right]:M'\oplus P=M\ra T'
\] 
is a morphism such that no morphism $\tilde{h}:M=M'\oplus P\ra (h\mu )(T)=wv(T)$ coincides
with $h$ on $T$. To see that notice that, if it exists, such a $\tilde{h}$ can be written as 
\[\tilde{h}=\left[\begin{array}{cc}0 & \beta \end{array}\right]:M'\oplus P\ra wv(T),
\] 
because $M'\in ^\perp\ct$ and $wv(T)\cong T\in\ct$. If we denote by $i$ the canonical inclusion $wv(T)\hookrightarrow T'$ and by
$\xi:T\arr{\sim}wv(T)=h\mu (T)$ the isomorphism induced by $h\mu$ (or $wv$), then we have 
\[i\xi =h\mu=i\tilde{h}\mu =i\left[\begin{array}{cc}0 & \beta\end{array}\right] \left[\begin{array}{c}u\\v \end{array}\right]=i\beta v.
\] 
Since $i$ is a monomorphism we conclude that $\xi =\beta v$, so that $v$ is a section. But then the morphism $\mu$ is also a section, which contradicts the hypothesis. Therefore $\tilde{h}$ cannot exists and the proof of this implication is finished.

Suppose now that $\ct$ is contravariantly finite in $\ch$. Then $1)\Rightarrow 4)$ clearly holds. The proof of $4)\Rightarrow 2)$ is identical to that of $1)\Rightarrow 2)$, done above.

Finally, suppose  that $\ca$ has enough injectives.

$5)\Rightarrow 1)$ We check condition $3)$ of Theorem \ref{characterization balanced}. Let $\mu:T\ra M$ be a non-split monomorphism. By hypothesis, there is another monomorphism $j:T\ra E$, where $E$ is an injective object of $\ch$ belonging to $\ct$. For simplicity, we view $j$ as
an inclusion. Now $j$ factors through $\mu$, so that we get a morphism $h:M\ra E$ such that $h\mu =j$. If there were a morphism $\tilde{h}:M\ra h\mu (T)=j(T)=T$ agreeing with $h$ on $T$, then $\tilde{h}$ would be a retraction for $\mu$, which is absurd.

$3)\Rightarrow 5)$ Let $0\neq T\in\ct$ be any non-zero object. We then have a monomorphism $\lambda :T\ra E$, where $E$ is an injective object of $\ca$. We decompose $E=E'\oplus P$, with $E'\in ^\perp\ct$ and $P\in \Sub(\ct)$. Notice that then $P$ is injective(-projective) and necessarily belongs to $\ct$. If we write now 
\[\lambda =\left[\begin{array}{c}u \\ v \end{array}\right]:T\ra E'\oplus P,
\] 
then the first paragraph of this proof says that $v$ is a monomorphism. That ends the proof.
\end{proof}

\begin{remark}
On an arbitrary abelian category $\ca$,  N. Yoneda defined the big abelian group  $\Ext_{\ca}^n(M,N)$ of  $n$-extensions (\cf \cite[III.5]{MacLane1975}), for all objects $M,N\in\ca$, which was functorial on both variables.  Due to J. L. Verdier \cite[Ch. III, 3.2.12]{Verdier1996}, $\Ext_\ca^n(M,N)=(\cd\ca)(M,N[n])$, where $\cd\ca$ is the (possibly large) derived category of $\ca$. As a consequence, the long exact sequences of $\Ext$'s hold in every abelian category. In particular, if $\ca =\ch$ is hereditary (namely, $\Ext_{\ch}^2(?,?)=0$), then every subobject of a projective object in $\ch$ is projective. Thus our Proposition \ref{hereditary weakly balanced} applies to hereditary abelian categories.
\end{remark}

\section{Some examples}\label{Some examples}
\addcontentsline{lot}{section}{2.6. Algunos ejemplos}

In this section we illustrate the previous results with some examples in module categories. Unless said otherwise, all rings are associative with unit and modules are right modules. For a given ring $A$, we denote by $\Mod A$ (respectively, $\mod A$) the category of all (respectively, finitely presented) $A$-modules.

\begin{example}
Let $A$ be a right semihereditary left coherent ring. Let $\ct =\proj A$ the full subcategory of the abelian category $\ch=\mod A$ consisting of finitely generated projective modules. By \cite[Corollary 3.11]{DingChen1993}, we know that $\ct$ is covariantly finite in $\ch$. On the other hand, the right semihereditary condition of $A$ gives that in $\ch$ submodules of projective objects are projective, so that we are in the situation
of Proposition \ref{hereditary weakly balanced}. Then the stable category of $\mod A$ associated to $\proj A$ is (weakly) balanced if and only if the class $^\perp\ct$, formed by all the modules $M\in\mod A$ such that $\Hom_A(M,A)=0$, is closed under (finitely generated) submodules. Two particular cases of this situation are those in which $A$ is a commutative principal ideal domain or $A=k\vec{A_{n}}$ is the path algebra
of the Dynkin quiver 
\[\vec{A_{n}}:1\ra 2\ra ...\ra n
\] 
over a field $k$.
\end{example}

It is well-known that the above path algebra $k\vec{A_{n}}$ is isomorphic to $T_n(k)$, the ring of lower triangular $n\times n$ matrices with coefficients in $k$.  When one moves from finitely presented to arbitrary modules, only this example survives.

\begin{example}\label{balancedMorita}
Let $H$ be a right hereditary ring. The following assertions are equivalent:
\begin{enumerate}[1)]
\item The category $\Proj H$ of projective $H$-modules is covariantly finite in $\Mod H$ and the stable category of $\Mod H$ associated to $\Proj H$ is a (weakly) balanced category.
\item $H$ is Morita equivalent to a finite direct product $T_{n_1}(D_1)\times ...\times T_{n_r}(D_r)$ of rings of lower triangular matrices, for some division rings $D_i$ and natural numbers $n_i$
\end{enumerate}
\end{example}
\begin{proof}
$1)\Rightarrow 2)$ Since $\Proj H$ is covariantly finite in $\Mod H$ it is closed under products, so that $H$ is semiprimary (\cf \cite{Chase1960} and \cite{Small1967}). On the other hand,  since assertion $4)$ of Proposition \ref{hereditary weakly balanced} holds, we get that $E(H_{H})$ is projective, so
that $H$ is right QF-3 in the sense of Ringel and Tachikawa (\cf \cite[Proposition 4.1]{Tachikawa1973} ). Then, by \cite[Theorem and Remark 1]{ColbyRutter1968}, one concludes that $H$ is Morita equivalent to the indicated direct product of rings of lower triangular matrix.

$2)\Rightarrow 1)$ There is no loss of generality in assuming  that $H=T_n(D)$, for some natural number $n$ and some division ring $D$. Since $H$ is  (two-sided) Artinian, $\Proj H$ is covariantly finite in $\Mod H$ \cite[Corollary 3.5]{RadaSaorin1998}. On the other hand,
$H$ is a serial ring and the row module $E=\left[\begin{array}{cccc} D& D& \cdots &D\end{array}\right]$ is the unique injective-projective indecomposable $H$-module, which cogenerates all other projective $H$-modules. Then assertion $5)$ of Proposition \ref{hereditary weakly balanced} holds, so that the stable category of $\Mod H$ associated to $\Proj H$ is balanced.
\end{proof}

In the situation of Proposition \ref{hereditary weakly balanced}, whenever $\ct$ is closed under subobjects, one has an equivalence of categories 
\[\ul{\ch}\simeq ^\perp\ct,
\] 
so that $\ul{\ch}$ is an abelian category. That is no longer true if $\ct$ is not closed under subobjects, as the following example shows.

\begin{example} \label{hereditary non-abelian non-triangulated}
Let $H=k\vec{A_{3}}$,  where $\vec{A_{3}}$ is the Dynkin quiver $1\ra 2\ra 3$ and $k$ is a field. Let $e_i$ be the primitive idempotent of $H$ given by the vertex $i\ko 1\leq i\leq 3$ and take $\ch =\mod H$ and $\ct =\add(e_{3}H)$, which is the full subcategory of projective-injective objects of $\mod H$. Then the stable category $\ul{\ch}$ of $\mod H$ associated to $\ct$ is a balanced category which is neither abelian nor triangulated (for its usual pretriangulated structure).
\begin{proof}
Due to Corollary \ref{balanced wrt projective-injective}, the stable category $\ul{\ch}$ is balanced. If 
\[f:e_{2}H\ra S_2=e_{2}H/e_{2}J(H)
\] 
is the canonical projection, then, by Proposition \ref{epimorphisms which become stable-epis}, we have that $\ol{f}$ is not an epimorphism in $\ul{\ch}$. If this latter category were abelian, we would have a morphism $g$ of $\ch$ such that $0\neq\ol{g}=\ol{f}^c$. But such a morphism $g:S_2\ra C$ necessarily factors through the (unique up to multiplication by scalars) irreducible morphism $j:S_{2}\ra E(S_2)=e_{3}H/Soc(e_{3}H)$ starting at $S_2$. Therefore there exists a morphism $h:E(S_2)\ra C$ such that $hj=g$. Notice that, in principle,  we have $C\cong S_2^n\oplus E(S_2)^m$ but, since $\Hom_A(E(S_2),S_2)=0$, we then get that $h=\left[\begin{array}{cc} 0 & h' \end{array}\right]^t:E(S_2)\ra S_2^n\oplus E(S_2)^m$ and hence $g=\left[\begin{array}{cc} 0 & g' \end{array}\right]^t:S_2\ra S_2^n\oplus E(S_2)^m$, with the obvious meaning.  The universal property of the cokernel implies that $n=0$ and $m=1$. In other word, up to multiplication by a nonzero scalars, we would have $\ol{g}=\ol{j}$. But this is absurd for $\ol{j}$ is not an epimorphism in $\ul{\mathcal{H}}$: if 
\[p:E(S_2)\ra e_{3}H/e_{3}J(H)=S_3
\] 
is the canonical projection, then $pj=0$ but $\ol{p}\neq 0$.

Notice also that $\ct^\perp\neq 0$ (for instance $e_{1}H\in\ct^\perp$). Thus, the loop functor $\Omega$, given by the pretriangulated structure of $\ul{\ch}$, vanishes on some non-zero object. Therefore $\ul{\ch}$ is not a triangulated category either.
\end{proof}
\end{example}

Corollary \ref{balanced in Serre case} and Proposition \ref{hereditary weakly balanced} might induce the reader to believe that ``weakly balance'' (when it makes sense) and ``balance'' are synonymous concepts for stable categories. Our final example shows that it is not the case, even when $\ct$ consists of projective objects.

\begin{example} \label{weakly balanced not balanced}
Let $k$ be a field and $A$ be the finite dimensional $k$-algebra given by the quiver
\[\xymatrix{1\ar[d]^{\delta} & 2\ar[l]_{\alpha}\ar[r]^{x} & 5\ar[r]^{y} & 6 \\
                  4\ar[r]_{\gamma} & 3\ar[u]_{\beta} &&
}
\]
with the following set of monomial relations  $R=\{yx,x\beta,\beta\gamma,\gamma\delta\alpha,\delta\alpha\beta\}$. Then
$\proj A$ is functorially finite in $\mod A$. We claim that the stable category $\ul{\mod}A$ of $\mod A$ associated to $\proj A$ is weakly balanced but not
balanced. To see that, first notice that assertion $2)$ of Theorem \ref{characterization weakly-balanced} is equivalent in this case to the property that, for every indecomposable projective $A$-module $P$, the canonical restriction map 
\[\Hom_A(E(P),A)\ra \Hom_A(P,A)
\] 
has a non-zero image. This property is trivially satisfied whenever $E(P)$ is projective. In our situation, that is the case for $P=e_{i}A$, with $i\in\{1,2,3,6\}$. On the other hand,
we have 
\[E(e_{4}A)\cong E(e_{5}A)\cong (e_{4}A\oplus e_{5}A)/N=:E,
\]
where $N$ is the cyclic submodule of $e_{4}A\oplus e_{5}A$ generated by $(-\delta\alpha,x)$. One readily sees that $\Hom_A(E,A)$ is a
$2$-dimensional vector space generated by the morphisms $f:\ol{(a,b)}\mapsto \gamma a$ and $g:\ol{(a,b)}\mapsto yb$. Consider now the morphisms
\[i:e_{4}A\ra E\ko a\mapsto\ol{(a,0)}
\] 
and
\[j:e_{5}A\ra E\ko b\mapsto\ol{(0,b)},
\] 
which satisfy $fi\neq 0$ and $gj\neq 0$. That proves that $\ul{\mod}A$ is weakly balanced. Now we shall see now that condition $3)$ of Theorem \ref{characterization balanced} is not satisfied by the morphism $i:e_{4}A\ra E$. Indeed, by the above paragraph, any morphism $h =E\ra A$ is of the form $h =k_1f+k_2g$, for some $k_1,k_2\in k$. Then it maps $\ol{(a,b)}\mapsto k_1\gamma a+k_2yb$ and, hence, $(hi)(a)=k_1\gamma a$. In case $k_1=0$, clearly $h$ does not satisfy condition $3)$ of Theorem \ref{characterization balanced}. But if $k_1\neq 0$ then $\im(hi)=\im(f)$ and, choosing 
\[\tilde{h}=k_1f:E\ra \im(f)\ko \ol{(a,b)}\mapsto k_1\gamma a,
\] 
one has that $\tilde{h}$ coincides with $h$ on $e_{4}A$. That proves that $\ul{\mod}A$ is not balanced.
\end{example}

\chapter{Split TTF triples on module categories}\label{Split TTF triples on module categories}
\addcontentsline{lot}{chapter}{Cap\'itulo 3. Ternas TTF escindidas en categor\'ias de m\'odulos}

\section{Introduction}
\addcontentsline{lot}{section}{3.1. Introducci\'on}

\subsection{Motivation}
\addcontentsline{lot}{subsection}{3.1.1. Motivaci\'on}

Recall that a \emph{torsion torsionfree(=TTF) triple} $(\cx,\cy,\cz)$ on an abelian category consists of a couple of torsion pairs $(\cx,\cy)$ and $(\cy,\cz)$ (\cf Definition \ref{TTF triple}). This notion was introduced by J. P. Jans \cite{Jans}, who gave a
bijection, for an arbitrary (associative unital) algebra $A$, between TTF triples on the module category $\Mod A$ and idempotent
two-sided ideals of $A$ (\cf Theorem \ref{parametrizando TTF ternas}). J. P. Jans also proved that this bijection restricts to another one
between centrally split TTF triples on $\Mod A$ (\cf Definition \ref{TTF triple}) and central idempotents of $A$. However, the existence of
TTF triples $(\cx,\cy,\cz)$ for which only one of the torsion pairs $(\cx, \cy)$ and $(\cy, \cz)$ splits, which we shall call \emph{one-sided
split}, has been known for a long time (\cf \cite{Teply}) and no classification of them has been available. That is, until now, the
idempotent ideals of $A$ which correspond by J. P. Jans' bijection to those one-sided split TTF triples have not been identified, even
though there were some efforts to classify them (see for instance \cite{Azumaya} and \cite{Ikeyama}). The goal of this
chapter is to present such a classification, thus solving a problem which has been open for almost forty years. 

\subsection{Outline of the chapter}
\addcontentsline{lot}{subsection}{3.1.2. Esbozo del cap\'itulo}

In section ~\ref{Torsion torsionfree triples} we give the definitions and relevant known results concerning TTF triples on a module category.
In section ~\ref{Left split TTF triples over arbitrary rings} we give the classification of left split TTF triples on $\Mod A$ (\cf Theorem \ref{izquierda} and Corollary \ref{clasificacion left split}). For the sake of clarity, and as an intermediate step, we present in section ~\ref{Right split TTF triples over `good' rings} a partial classification of right split TTF triples, which is actually total when the ring $A$ belongs to a class which includes semiperfect and left N\oe therian rings (\cf Theorem \ref{clasificacion right split para buenos} and Corollary \ref{right split para buenos2}). Then, Êin section ~\ref{Right split TTF triples over arbitrary rings}, Êwe classify the right split TTF triples $(\cx,\cy,\cz )$ on $\Mod A$ such that $A_A$ belongs to $\cz$ (Theorem \ref{clasificacion right split}), from which the classification of all right split TTF triples (Corollary \ref{clasification right split2}) follows.

\subsection{Notation}
\addcontentsline{lot}{subsection}{3.1.3. Notaci\'on}

All rings appearing in the chapter are associative with identity and, unless explicitly said otherwise, all modules are right modules, the category of which will be denoted $\Mod A$\index{$\Mod A$}. Two-sided ideals will be simply called ``ideals'', Êwhen there is no risk of confusion. The Jacobson radical of a ring $A$ will be denoted by $J(A)$\index{$J(A)$}. If $M$ is a module over a ring $A$, $S$ is a subset of $A$, $N$ a subset of $M$ and $I$ an ideal of $A$ then 
\begin{enumerate}[1)]
\item $\ann_{M}(S)$\index{$\ann_{M}(S)$} is the denotes the set of elements $m\in M$ such that $ms=0$ for each $s\in S$. 
\item $\ann_{A}(N)$\index{$\ann_{A}(N)$} denotes the set of elements $a\in A$ such that $na=0$ for each $n\in N$. 
\item $\lann_{A}(I)$\index{$\lann_{A}(I)$} denotes the set of elements $a\in A$ such that $ai=0$ for each $i\in I$.
\end{enumerate}
We recall that if $(\cx,\cy)$ is a torsion pair on an additive category $\cd$ (in particular, a torsion pair on an abelian category), we write $x:\cx\ra\cd$ and $y:\cy\ra\cd$ for the inclusion functors, $\tau_{\cx}:\cd\ra\cx$ for the right adjoint to $x$ and $\tau^{\cy}$ for the left adjoint to $y$.

\section{Torsion torsionfree triples}\label{Torsion torsionfree triples}
\addcontentsline{lot}{section}{3.2. Ternas de torsi\'on y libres de torsi\'on}

\begin{lemma}
Let $I$ be an ideal of a ring $A$. Then:
\begin{enumerate}[1)]
\item The class $\cy$ of modules annihilated by $I$ is $\Gen(A/I)$.
\end{enumerate}
If $I$ is idempotent we also have:
\begin{enumerate}[2)]
\item The class $\cx$ of modules $M$ such that $MI=M$ is $\Gen(I)$.
\end{enumerate}
\begin{enumerate}[3)]
\item $\cx$ is the left orthogonal to $\cy$, \ie $\cx=\ ^{\bot}\cy$
\end{enumerate}
\end{lemma}
\begin{proof}
1) If $M$ is generated by $A/I$ it is clear that $MI=0$. Conversely, assume $M$ is a module annihilated by $I$. There exists an epimorphism $A^{(S)}\twoheadrightarrow M$, which fits into the diagram
\[\xymatrix{I^{(S)}\ar[r] & A^{(S)}\ar@{->>}[r]\ar@{->>}[d] & (A/I)^{(S)}\ar[r] & 0 \\ 
& M & &
}
\]
with exact row. Since $MI=0$, then there exists a factorization
\[\xymatrix{I^{(S)}\ar[r] & A^{(S)}\ar@{->>}[r]\ar@{->>}[d] & (A/I)^{(S)}\ar[r]\ar@{.>}[dl] & 0 \\ 
& M & &
}
\]
so that the dotted arrow is an epimorphism, and so $M$ is in $\Gen(A/I)$.

2) If $M$ is in $\Gen(I)$, then, since $I$ is idempotent, it is clear that $MI=M$. Conversely, let $M$ be a module such that $MI=M$. Assume $M=A^{(S)}/N$ for a certain module $N$. Then we have
\[M=MI=(A^{(S)}I+N)/N=(I^{(S)}+N)/N\cong I^{(S)}/(I^{(S)}\cap N),
\] 
\ie $M$ is generated by $I$.

3) If $M$ is in the left orthogonal to $\cy$, then the epimorphism $M\twoheadrightarrow M/MI$ is zero, which implies that $M/MI=0$, \ie $M=MI$. On the other hand, if $f:M\ra N$ is a morphism of $A$-modules such that $M=MI$ and $N$ is annihilated by $I$, then we have $f(M)=f(MI)\subseteq NI=0$, and so $f=0$.
\end{proof}

Associated to the notion of TTF triple we have:
\begin{definition}
Let $A$ be a ring. A class $\cy$ of $A$-modules is a \emph{TTF class}\index{class!TTF} of $\Mod A$ if it is closed under submodules, quotients, extensions and products.
\end{definition}

The following is an alternative proof of J. P. Jans' result \cite[Corollary 2.2]{Jans} (see also \cite[VI.8]{Stenstrom}).

\begin{theorem}\label{parametrizando TTF ternas}
Let $A$ be a ring. There is a one-to-one correspondence between:
\begin{enumerate}[1)]
\item The set of TTF triples on $\Mod A$.
\item The set of TTF classes of $\Mod A$. 
\item The set of idempotent ideals of $A$. 
\end{enumerate}
This correspondence sends the TTF triple $(\cx,\cy,\cz)$ to the ideal $\tau_{\cx}(A)$ and the ideal $I$ to the TTF triple $(\cx,\cy,\cz)$ where $\cx$ is the full subcategory formed by the modules $M$ such that $MI=M$ and $\cy$ is the full subcategory formed by the modules $M$ such that $MI=0$.
In this situation, we have $\cx=\Gen(I)$, $\cy=\Gen(A/I)$, $\tau_{\cx}(M)=MI$ and $\tau_{\cy}(M)=\ann_{M}(I)$.
\end{theorem}
\begin{proof}
If $(\cx,\cy,\cz)$ is a TTF triple on $\Mod A$, then it is clear that $\cy$ is a TTF class. Conversely, if $\cy$ is a TTF class of $\Mod A$, then $(^{\bot}\cy,\cy,\cy^{\bot})$ is a TTF triple (\cf \cite[Proposition VI.2.1, Proposition VI.2.2]{Stenstrom}).

Let $I$ be an ideal of $A$ and take $\cy$ to be the full subcategory of modules $M$ such that $MI=0$. It is clear that $\cy$ is closed under submodules, quotients, small products and small coproducts. If, moreover, $I$ is idempotent, then $\cy$ is closed under extensions, and so it is a TTF class, which implies that it fits in the middle of a TTF triple. For the description of the left orthogonal to $\cy$ apply the lemma above. 

Conversely, let $(\cx,\cy,\cz)$ be a TTF triple on $\Mod A$. Put $I:=\tau_{\cx}(A)$ and consider the short exact sequence
\[0\ra I\hookrightarrow A\twoheadrightarrow A/I\ra 0,
\]
of $\Mod A$ with $I$ in $\cx$ and $A/I$ in $\cy$. Since $\tau_{\cx}$ is a subfunctor of the identity functor $\id_{\Mod A}$ and all the morphisms of $\op{End}_{A}(A)$ are given by multiplications to the left of elements of $A$, then $I$ is a two-sided ideal of $A$. By using the short exact sequence above, one realises that the restriction of $\Hom_{A}(A/I,?)$ to $\cy$ is isomorphic to the identity functor $\id_{\cy}$. This implies that $A/I$ is a compact projective generator of the cocomplete abelian category $\cy$ and so, by P. Gabriel's characterization of module categories among abelian categories \cite{Gabriel1962} we know that $\cy\simeq\Mod A/I$. More directly, since $A/I$ is a generator of $\cy$ and $\cy$ is closed under small coproducts, by \cite[Proposition IV. 6.2]{Stenstrom} the modules of $\cy$ are precisely those which are quotients of a small coproduct of copies of $A/I$, and so, by using the lemma above we have the descriptions of $\cy$ and $\cx$.
\end{proof}

The TTF triple on $\Mod A$ associated to the idempotent ideal $I$ of $A$ can be depicted as follows:
\[\xymatrix{\Gen(A/I)\ar[rr]^y && \Mod A\ar@/_1pc/[ll]_{?\otimes_{A}A/I}\ar@/_-1pc/[ll]^{\Hom_{A}(A/I,?)}\ar[rr]^{?\otimes_{A}I} && \Gen(I)\ar@/_1pc/[ll]_{\Hom_{A}(I,?)}\ar@/_-1pc/[ll]^x
}
\]

The following is also well-known (\cf \cite[Proposition VI.8.5]{Stenstrom} or Corollary \ref{centrally split TTF triples on abelian categories}).

\begin{proposition}
The one-to-one correspondence of Theorem ~\ref{parametrizando TTF ternas} restricts to a one-to-one correspondence between:
\begin{enumerate}[1)]
\item Centrally split TTF triples on $\Mod A$.
\item (Ideals of $A$ generated by) single central idempotents of $A$.
\end{enumerate}
\end{proposition}

Put $\mathfrak{L}$, $\mathfrak{C}$ and $\mathfrak{R}$ for the sets of left, centrally and right split TTF triples on $\Mod A$.

Since there are one-sided split TTF triples which are not centrally split (\cf \cite{Teply}) Êwe should have a diagram of the form:
\[\xymatrix{ & \{\text{TTF triples}\}\simeq\{\text{idempotent ideals}\} &  \\
\mathfrak{L}\simeq ?\ar@{^(->}[ur] & & \mathfrak{R}\simeq ?\ar@{_(->}[ul] \\
& \mathfrak{C}\simeq\{\text{central idempotents}\}\ar@{_(->}[ul]\ar@{^(->}[ur] &
}
\]

The following is the main question Êtackled in the chapter:

\begin{question}
ÊWhat should replace the question marks in the diagram above?
\end{question}
\bigskip

\section{Left split TTF triples over arbitrary rings}\label{Left split TTF triples over arbitrary rings}
\addcontentsline{lot}{section}{3.3. Ternas TTF escindidas por la izquierda sobre anillos arbitrarios}

\begin{definition}\label{triangular matrix ring}
Given rings, $B$ and $C$, and a $B$-$C$-bimodule $M$, we can construct the associated \emph{triangular matrix ring}\index{ring!triangular matrix}\index{triangular matrix ring}:
\[A=\left[\begin{array}{cc}C&0\\M&B\end{array}\right].
\] 
The elements of $A$ are $2\times 2$ matrices
\[\left[\begin{array}{cc}c& 0 \\ m& b\end{array}\right]
\]
with $c\in C\ko m\in M$ and $b\in B$, and it becomes a ring with componentwise addition and matrix multiplication. 
\end{definition}

The category $\Mod A$ admits an explicit description in terms of modules over $C$ and $B$. Indeed, let $\cc_{A}$ be a category whose objects are triples $(L,N;\varphi)$ where $L\in\Mod C\ko N\in\Mod B$ and $\varphi\in\Hom_{C}(N\otimes_{B}M,L)$. The morphisms between two objects $(L,N;\varphi)$ and $(L',N';\varphi')$ are pairs 
\[(\alpha,\beta)\in\Hom_{C}(L,L')\times\Hom_{B}(N,N')
\]
such that the following diagram commutes:
\[\xymatrix{N\otimes_{B}M\ar[r]^{\beta\otimes\id_{M}}\ar[d]^{\varphi} & N'\otimes_{B}M\ar[d]^{\varphi'} \\
L\ar[r]^{\alpha} & L'
}
\]
Componentwise addition of morphisms gives $\cc_{A}$ the structure of an additive category. Now define a functor
\[F:\cc_{A}\ra\Mod A
\]
which takes $(L,N;\varphi)$ to the $A$-module defined by the abelian group $L\coprod N$ with scalar multiplication
\[\left[\begin{array}{cc}l&n\end{array}\right]\cdot \left[\begin{array}{cc}c&0\\m&b\end{array}\right]=\left[\begin{array}{cc}lc+\varphi(n\otimes m)&nb\end{array}\right].
\]
At the level of morphisms, $F$ takes $(\alpha,\beta)$ to the canonical morphism $\alpha\coprod\beta$.

It is well-known (see \cite[Proposition III. 2.2]{AuslanderReitenSmalo1995}) that $F$ is an equivalence, and we shall often identify $\Mod A$ with $\cc_{A}$.
Of course, a similar description of left $A$-module can be made.

 Notice that the assignments $L\mapsto (L,0;0)$ and $N\mapsto (0,N;0)$ give (fully faithful) embeddings $\Mod C\ra\cc_{A}\simeq\Mod A$ and $\Mod B\ra\cc_{A}\simeq\Mod A$. We shall frequently identify $\Mod C$ and $\Mod B$ with their images by these embeddings.

\begin{definition}\label{hereditary sigma injective}
If $C$ is a ring and $M$ is a $C$-module, then $M$ is \emph{hereditary $\Sigma$-injective}\index{injective!hereditary $\Sigma$-} in case every quotient of a direct sum of copies of $M$ is injective.
\end{definition}

We can already give the main result of this section. 

\begin{theorem}\label{izquierda}
Let $A$ be a ring and let $\cy$ be a full subcategory of $\Mod A$. The following assertions are equivalent:
\begin{enumerate}[1)]
\item $(^{\bot}\cy,\cy,\cy^{\bot})$ is a left split TTF triple on $\Mod A$.
\item There exists an idempotent $e\in A$ such that:
\begin{enumerate}[i)]
\item $(1-e)Ae=0$.
\item $NeA$ is a direct summand of $N$ for every $N\in\Mod A$.
\item $\cy$ is formed by the modules $N\in\Mod A$ such that $Ne=0$.
\end{enumerate}
\item The ring $A$ is isomorphic to the triangular matrix ring
\[A\cong\left[\begin{array}{cc}C&0\\M&B\end{array}\right]
\]
where:
\begin{enumerate}[i)]
\item $M$ is a $B$-$C$-bimodule such that $M$ is hereditary $\Sigma$-injective in $\Mod C$.
\item $\cy$ identifies with $\Mod C$.
\end{enumerate}
\end{enumerate}
\end{theorem}
\begin{proof}
$1)\Rightarrow 2)$ Put $(\cx,\cy,\cz):=(^{\bot}\cy,\cy,\cy^{\bot})$. By hypothesis $I=\tau_{\cx}(A)=eA$, for some idempotent $e\in A$, Ê
and so $(1-e)A\cong A/\tau_{\cx}(A)$ belongs to $\cy$. Hence $(1-e)Ae\cong\Hom_{A}(eA,(1-e)A)=0$.
We know that $\tau_{\cx}(N)=NI=NeA$, which is then a direct summand of $N$, for every $N\in\Mod A$. We also have that a module $N\in\Mod A$ belongs to $\cy$ if and only if $Ne=0$.

$2)\Rightarrow 1)$ Of course, $\cy$ is a TTF class. Then it only remains to prove that the torsion pair $(^{\bot}\cy,\cy)$ splits.
First of all, since $eA\in{}^{\bot}\cy$ we get that $\Gen(eA)\subseteq{}^{\bot}\cy$. On the other hand, decompose an arbitrary $N\in{}^{\bot}\cy$ as
$N=NeA\oplus N'$ where $N'e=0$. Hence $N'=N'(1-e)\in\cy$, and so $N'=0$ and Ê$N\in\Gen(eA)$. Then $(^{\bot}\cy,\cy)=(\Gen(eA),\cy)$, Ê
which is a split torsion pair by 2.ii).

$1)=2)\Rightarrow 3)$ Put $(\cx,\cy,\cz):=(^{\bot}\cy,\cy,\cy^{\bot})$. Take $C:=(1-e)A(1-e)\ko B:=eAe$ and $M:=eA(1-e)$. All the conditions of
3) are clearly satisfied except, perhaps, that $M_{C}$ is hereditary $\Sigma$-injective. Let us prove it. In case $eA(1-e)=0$ we are done, so assume that $eA(1-e)\neq 0$. For an arbitrary set $S$ put
\[T:=eA(1-e)A^{(S)}\in\cy\ko D:=eA^{(S)}\in\cx\ko F:=\left(\frac{eA}{eA(1-e)A}\right)^{(S)}\in\cz.
\]
The short exact sequence
\[0\ra T\overset{i}{\hookrightarrow}D\twoheadrightarrow F\ra 0
\]
is not split. Indeed, since $\Hom_{A}(\cx,\cy)=0$ then $\Hom_{A}(D,T)=0$. Take now a non-zero epimorphism
$p:T\twoheadrightarrow E$, and let us prove that $E$ is injective over $C$. By doing the pushout of $p$ and $i$ we get the commutative diagram
\[\xymatrix{0\ar[r] & T\ar[r]^{i}\ar@{->>}[d]_{p} & D\ar[r]\ar@{->>}[d] & F\ar[r]\ar[d]^{\id_{F}} & 0 \\
0\ar[r] & E\ar[r]^{\mu} & V\ar[r]^{\pi} & F\ar[r] & 0 }
\]
with exact rows. Now, for any $T'\in\cy$, there is an exact sequence
\[\Hom_{A}(T',F)\ra\op{Ext}^{1}_{A}(T',E)\ra\op{Ext}^{1}_{A}(T',V).
\]
But $\Hom_{A}(T',F)=0$ because $T'\in\cy$ and $F\in\cz$, and $\op{Ext}^{1}_{A}(T',V)=0$ because $T'\in\cy\ko V\in\cx$, and $(\cx,\cy)$ splits.
So, $\op{Ext}^{1}_{A}(T',E)=0$, proving $E_{C}$ is injective.

$3)\Rightarrow 2)$ We identify $A$ with the triangular matrix ring 
\[\left[\begin{array}{cc}C & 0 \\ M & B\end{array}\right].
\] 
Taking the idempotent
\[e=\left[\begin{array}{cc}0 & 0 \\ 0 & 1\end{array}\right]
\] 
we trivially have $(1-e)Ae=0$, and $\cy\simeq\Mod C$ is formed by the modules $N\in\Mod A$ such that $Ne=0$.
It only remains to prove that $NeA$ is a direct summand of $N$ for each $N\in\Mod A$. Of course, we have $N=NeA+N(1-e)A$.
Note that $NeA\cap N(1-e)A=NeA(1-e)A$. Indeed, one inclusion is obvious. For the converse one, if $x\in NeA\cap N(1-e)A$, then $xe=0$ because $(1-e)Ae=0$, and so $x=x(1-e)\in NeA(1-e)\subseteq NeA(1-e)A$. Thus, $NeA\cap N(1-e)A$ is generated by 
\[eA(1-e)A=\left[\begin{array}{cc}0&0\\M&0\end{array}\right],
\] 
which is hereditary $\varSigma$-injective in $\Mod C$. Then
$NeA\cap N(1-e)A$ is injective in $\Mod C$, and so it induces a decomposition $N(1-e)A=NeA(1-e)A\oplus N'$ in $\Mod C\subseteq\Mod A$.
It follows that $N=N'\oplus NeA$.
\end{proof}

As a direct consequence of the theorem, the classification of left split TTF triples on $\Mod A$ is at hand.

\begin{corollary} \label{clasificacion left split}
Let $A$ be a ring. The one-to-one correspondence of Theorem ~\ref{parametrizando TTF ternas} restricts to a one-to-one correspondence between:
\begin{enumerate}[1)]
\item Left split TTF triples on $\Mod A$.
\item Ideals of $A$ of the form $I=eA$ where $e$ is an idempotent of $A$ such that $eA(1-e)$ is hereditary $\Sigma$-injective as a right $(1-e)A(1-e)$-module.
\end{enumerate}
\end{corollary}
\begin{proof}
If $e\in A$ is an idempotent such that $I=eA$ is a two-sided ideal, then $Ae\subseteq eA$ and, hence, $(1-e)Ae=0$. Now apply (the proof of) Theorem \ref{izquierda}.
\end{proof}

\section{Right split TTF triples over `good' rings}\label{Right split TTF triples over `good' rings}
\addcontentsline{lot}{section}{3.4. Ternas TTF escindidas por la derecha sobre `buenos' anillos}

Recall some definitions of module theory:

\begin{definition}
A ring $A$ is
\begin{enumerate}[1)]
\item \emph{right hereditary}\index{ring!right hereditary} if every right ideal is projective (equivalently, any submodule of a projective right $A$-module is projective).
\item \emph{semilocal}\index{ring!semilocal} if $A/J(A)$ is semisimple.
\item \emph{left semiartinian}\index{ring!left semiartinian} if every non-zero left $A$-module has a non-zero socle.
\item \emph{semiprimary}\index{ring!semiprimary} if it is semilocal and $J(A)$ is nilpotent.
\item \emph{semiperfect}\index{ring!semiperfect} if it is semilocal and idempotents lift module $J(A)$.
\end{enumerate}
\end{definition}

Recall now H. Bass' theorem \cite{Bass1960}:

\begin{theorem}\label{Bass theorem}
The following properties of a ring $A$ are equivalent:
\begin{enumerate}[1)]
\item $A$ is left semiartinian and semilocal.
\item $A$ is left semiartinian and semiperfect.
\item Every flat right $A$-module is projective.
\item Every right $A$-module has a projective cover.
\end{enumerate}
\end{theorem}
\begin{proof}
\emph{Cf.} \cite[Proposition VIII.5.1]{Stenstrom} or \cite[Theorem 28.4]{AndersonFuller1992}.
\end{proof}

The following definition is due to H. Bass \cite{Bass1960}:

\begin{definition}
A ring $A$ is \emph{right perfect}\index{ring!right perfect} if it satisfies any of the equivalent conditions of the above proposition. Dually, one defines \emph{left perfect}\index{ring!left perfect}.
\end{definition}

\begin{proposition}
Every semiprimary ring is right and left perfect.
\end{proposition}
\begin{proof}
\emph{Cf.} \cite[Corollary 28.8]{AndersonFuller1992}.
\end{proof}

In this section and the next one hereditary perfect rings will play an important r\^{o}le. We gather some known properties of them which will be useful:

\begin{proposition}\label{hereditary perfect}
Let $A$ be a right perfect right hereditary ring. Then it is semiprimary, hereditary on both sides and the class of projective
$A$-modules (on either side) is closed under arbitrary products.
\end{proposition}
\begin{proof}
From Ê\cite[Corollary 2 and Theorem 3]{Small1967} if follows that $A$ is semiprimary and hereditary (whence coherent) on both sides. Then apply \cite[Theorem 3.3]{Chase1960}.
\end{proof}

We start with some properties of (right split) TTF triples which will be used in the sequel. The following property is due to G. Azumaya \cite[Theorem 6]{Azumaya}:

\begin{lemma}\label{hereditary igual a pure igual a condicion aritmetica}
Let $(\cx,\cy,\cz)$ be a TTF triple on $\Mod A$ with associated idempotent ideal $I$. The torsion pair $(\cx ,\cy )$ is hereditary if, and only if, $I$ is pure as a left ideal.
\end{lemma}

\begin{proposition}\label{primeras propiedades de los right split}
Let $(\cx,\cy,\cz)$ be a right split TTF triple on $\Mod A$. Then $(\cx,\cy)$ is a hereditary torsion pair and $\cx\subseteq\cz$.
\end{proposition}
\begin{proof}
Let $M\in\cy$ and let $E(M)$ be its injective envelope. Of course, $M\subseteq \tau_{\cy}(E(M))$ and, by hypothesis, $\tau_{\cy}(E(M))$ is a direct summand of $E(M)$. But $M$ is essential in $E(M)$, which implies that $\tau_{\cy}(E(M))=E(M)$\ko \ie $E(M)\in\cy$. We have proved that $\cy$ is closed under injective envelopes. But then \cite[Proposition VI.3.2]{Stenstrom} says that $(\cx,\cy)$ is hereditary, and consequently $\cx\subseteq\cz$ by \cite[Lemma VI.8.3]{Stenstrom}.
\end{proof}

We first want to classify the following type of right split TTF triples.

\begin{proposition} \label{anhadida}
Let $(\cx,\cy,\cz)$ be a right split TTF triple on $\Mod A$ with associated idempotent ideal $I$. The following conditions are equivalent:
\begin{enumerate}[1)]
\item $\cx$ is closed under products.
\item $I=Ae$ for some idempotent $e\in A$.
\item Ê$A/I$ has a projective cover in $\Mod A$.
\item Ê$I$ is finitely generated as a left ideal.
\end{enumerate}
\end{proposition}
\begin{proof}
$1)\Leftrightarrow 2)$ ÊSince $(\cx ,\cy )$ is hereditary (\cf Proposition \ref{primeras propiedades de los right split}), the class $\cx$ is closed
under products if, and only if, $\cx$ is a TTF class. Then Ê\cite[Theorem 3]{Azumaya} applies.

$1)\Leftrightarrow 3)$ ÊIt follows from Ê\cite[Theorem 8]{Azumaya}.

$2)\Rightarrow 4)$ is clear.

$4)\Rightarrow 2)$ Since $I$ is pure submodule of $_{A}A$ (\cf Lemma \ref{hereditary igual a pure igual a condicion aritmetica} and Proposition \ref{primeras propiedades de los right split}) and $_{A}A$ is flat, then \cite[Proposition I.11.1]{Stenstrom} says that $A/I$ is a flat left $A$-module. Now, if $I$ is finitely generated on the left, then $A/I$ is also a finitely presented left $A$-module, and so \cite[Proposition I.11.5]{Stenstrom} states that $A/I$ is a projective left $A$-module. Therefore, the short exact sequence of left $A$-modules
\[0\ra I\hookrightarrow A\twoheadrightarrow A/I\ra 0
\] 
splits. Ê
\end{proof}

\begin{definition}
If $B$ is a ring, aÊ $B$-module $P$ is called \emph{hereditary projective}\index{projective!hereditary}\index{projective!hereditary $\Pi$-} (respectively, \emph{hereditary $\Pi$-projective}) in case every submodule of $P$ (respectively, of a direct product of copies of $P$) is projective. Recall also that a $B$-module $N$ is called \emph{FP-injective}\index{injective!FP-} in case $\Ext^1_B(?,N)$ vanishes on all finitely presented $B$-modules. 
\end{definition}

The following type of modules will appear in our next classification theorem:

\begin{proposition}\label{proposicion de hereditary pi-projective dual}
Let $B$ be a ring and $M$ be a left $B$-module. The following conditions are equivalent:
\begin{enumerate}[1)]
\item For every bimodule structure $_{B}M_{C}$ and every $N\in\Mod C$, the right $B$-module $\Hom_{C}(M,N)$ is hereditary projective.
\item There exists a bimodule structure $_{B}M_{C}$ such that, for every $N\in\Mod C$, the right $B$-module $\Hom_{C}(M,N)$ is hereditary projective.
\item If $C=\op{End}_{B}(M)^{op}$ and $Q$ is the minimal injective cogenerator of $\Mod C$, then $\Hom_{C}(M,Q)$ is hereditary $\Pi$-projective in $\Mod B$.
\item The right $B$-module $M^+=\Hom_{\Z}(M,\mathbf{Q}/\Z)$ is hereditary $\Pi$-projective.
\item $\ann_B(M)=eB$ for some idempotent $e\in B$, the ring $\ol{B}=B/\ann_B(M)$ is hereditary perfect and $M$ is FP-injective as a
left $\overline{B}$-module.
\end{enumerate}
When $B$ is an algebra over a commutative ring $k$, the above assertions are equivalent to:
\begin{enumerate}[6)]
\item If $Q$ is a minimal injective cogenerator of $\Mod k$, then $D(M):=\Hom_{k}(M,Q)$ is hereditary $\Pi$-projective in $\Mod B$.
\end{enumerate}
\end{proposition}
\begin{proof}
$1)\Rightarrow 2)$ Clear.

$1)\Rightarrow 3)$ Let $L$ be a submodule of a product $\Hom_{C}(M,Q)^{I}$ of copies of the right $B$-module $\Hom_{C}(M,Q)$. Since $\Hom_{C}(M,Q)^{I}\cong\Hom_{C}(M,Q^{I})$, we deduce that $L$ is projective

$3)\Rightarrow 4)$ If $E$ is an arbitrary injective cogenerator of $\Mod C$ then we have a section $E\ra Q^{S}$ for some set $S$.
Then $\Hom_{C}(M,E)$ is a direct summand of $\Hom_{C}(M,Q)^{S}$ and hence also $\Hom_{C}(M,E)$ is hereditary $\Pi$-projective. In
particular, this is true for $E=\Hom_{\Z}(C,\mathbf{Q}/\Z)$. By adjunction,
\[\Hom_{C}(M,\Hom_{\Z}(C,\mathbf{Q}/\Z))\cong\Hom_{\Z}(C\otimes_{C}M,\mathbf{Q}/\Z)\cong M^+,
\]
we get that $M^+$ is hereditary $\Pi$-projective.

$4)\Rightarrow 1)$ Let $C$ be a ring and $_{B}M_{C}$ a bimodule structure on $M$. Since $E=\Hom_{\Z}(C,\mathbf{Q}/\Z)$ is an injective cogenerator
of $\Mod C$, every right $C$-module embeds in a direct product of copies of $E$. So, it will be enough to prove that $\Hom_{C}(M,E)$ is hereditary
$\Pi$-projective, which is clear since we have
\[\Hom_{C}(M,E)\cong M^+
\]
by adjunction.

$2)\Rightarrow 4)$ We want to prove that, for any set $I$, every submodule of
\[\Hom_{\Z}(M,\mathbf{Q}/\Z)^{I}\cong\Hom_{C}(M,\Hom_{\Z}(C,\mathbf{Q}/\Z))^{I}\cong\Hom_{C}(M,\Hom_{\Z}(C,\mathbf{Q}/\Z)^{I})
\]
is projective. Take for this $N=\Hom_{\Z}(C,\mathbf{Q}/\Z)$ in 2).

$5)\Rightarrow 4)$ We have $B=eB\oplus (1-e)B$ in $\Mod B$, and thus $\ol{B}\cong(1-e)B$ is a projective right $B$-module. Hence, any projective right $\ol{B}$-module is also projective when regarded as a right $B$-module. Also notice that every submodule of a direct product of
copies of $M^+_B$ is annihilated by $\ann_B(M)$. Then, it suffices to prove that $M^{+}$ is hereditary $\Pi$-projective when regarded as a right $\ol{B}$-module. In fact, since $\ol{B}$ is hereditary perfect, it suffices (\cf Proposition \ref{hereditary perfect}) to prove that $M^{+}_{\ol{B}}$ is projective. Now, $M^{+}_{\ol{B}}$ is flat thanks to \cite[Proposition 1.15]{Skljarenko}, and projectivity follows from H. Bass's theorem \ref{Bass theorem} because $\ol{B}$ is perfect.

$1)=4)\Rightarrow 5)$ Put $C=\op{End}_{B}(M)^{op}$ and let $\varphi: B\ra\op{End}_{C}(M)$ be the canonical ring morphism. By 1),
$\op{End}_{C}(M)$ is hereditary projective as a right $B$-module. Then the short exact sequence in $\Mod B$
\[0\ra\ann_{B}(M)\hookrightarrow B\twoheadrightarrow\im(\varphi)\ra 0
\]
splits, and so $\ann_{B}(M)=eB$, for some idempotent $e\in B$. Now, $\ol{B}_B$ is a submodule of $\op{End}_{C}(M)$, which is hereditary $\Pi$-projective by 3). Then every submodule of a direct product of copies of $\ol{B}_{\ol{B}}$ is projective, which implies that $\ol{B}$ is hereditary perfect by \cite[Theorem 3.3]{Chase1960}. But then, by \cite[Proposition 1.15]{Skljarenko}, the fact that $M^+_{\ol{B}}$ is projective implies that $_{\ol{B}}M$ is FP-injective.

Finally, the equivalence of 6) with $1)-3)$ follows as the equivalence of 4) with $1)-3)$ but replacing $\Z$ by $k$.
\end{proof}

\begin{definition}
If $B$ is a ring, a left $B$-module $M$ satisfying the equivalent conditions of the last proposition will be said to \emph{have hereditary $\Pi$-projective dual}\index{module!have hereditary $\Pi$-projective dual}.
\end{definition}

We can now give the desired partial classification.

\begin{theorem} \label{clasificacion right split para buenos}
Let $A$ be a ring and let $(\cx,\cy,\cz)$ be a TTF triple on $\Mod A$ with associated idempotent Êideal $I$. The following assertions are equivalent:
\begin{enumerate}[1)]
\item $(\cx,\cy,\cz)$ is right split and $I$ is finitely generated as a left ideal.
\item There is an idempotent $e$ of $A$ such that $I=Ae$ and the left $(1-e)A(1-e)$-module $(1-e)Ae$ has hereditary $\Pi$-projective dual.
\item The ring $A$ is isomorphic to a triangular matrix ring 
\[A\cong\left[\begin{array}{cc}C & 0 \\ M & B\end{array}\right]
\]
in such a way that $\cy$ gets identified with $\Mod B$ and $_{B}M$ has hereditary $\Pi$-projective dual.
\item There exist rings $B'\ko H\ko C$, where $H$ is hereditary perfect, Êand bimodules $_{H}M_{C}\ko _{B'}N_{H}$ such that:
\begin{enumerate}[i)]
\item The ring $A$ is isomorphic to the triangular matrix ring
\[A\cong \left[\begin{array}{ccc}C & 0 & 0\\ M & H & 0 \\ 0 & N & B'\end{array}\right].
\]
\item $\cy$ gets identified with the category of right modules over the matrix ring 
\[\left[\begin{array}{cc}H & 0 \\ N & B'\end{array}\right].
\]
\item $_{H}M$ is faithful and FP-injective.
\end{enumerate}
\end{enumerate}
\end{theorem}
\begin{proof}
$2)\Rightarrow 3)$ Since $I=Ae$ is also a right ideal, then $eA\subseteq Ae$ and so $eA(1-e)=0$. Therefore, if $C=eAe\ko M=(1-e)Ae$ and $B=(1-e)A(1-e)$, the map
\[A\ra\left[\begin{array}{cc}C&0\\ M&B\end{array}\right]\ko a\mapsto \left[\begin{array}{cc}eae&0\\ (1-e)ae&(1-e)a(1-e)\end{array}\right]
\]
is a ring isomorphism.

$3)\Rightarrow 2)$ Take $e\in A$ to be the element corresponding to the matrix
\[\left[\begin{array}{cc}1_{C}&0\\ 0&0\end{array}\right].
\]

$3)\Rightarrow 4)$ We identify $A$ with the triangular matrix ring 
\[\left[\begin{array}{cc}C & 0 \\ M & B\end{array}\right].
\]
By Proposition \ref{proposicion de hereditary pi-projective dual}, there exists an idempotent $e'$ of $B$ such that $\ann_{B}(M)=e'B$, Ê
$\ol{B}:=B/\ann_{B}(M)$ is a hereditary perfect ring and $_{\ol{B}}M$ is FP-injective. Notice that, since $e'B$ is a two-sided ideal, then $Be'\subseteq e'B$ and $(1-e')Be'=0$. This implies that $e'Be'=Be'$. We put $B':=e'Be'=Be'\ko H:=(1-e')B(1-e')\cong\ol{B}$ and $N:=e'B(1-e')$, and there is no loss of generality in assuming
\[B=\left[\begin{array}{cc}H & 0 \\ N & B'\end{array}\right]\ko e'=\left[\begin{array}{cc}0 & 0 \\ 0 &
1\end{array}\right],
\]
so that 
\[\ann_{B}(M)=e'B=\left[\begin{array}{cc}0 & 0 \\ N & B'\end{array}\right].
\] 
Representing now the left $B$-module $M$ as a triple
\[(_{H}L,_{B'}N';\varphi:N\otimes_{H}L\ra N')
\] 
we necessarily get $N'=0$. Then, abusing of
notation, Êwe can put $M=\left[\begin{array}{cc}M & 0\end{array}\right]^{t}$. That gives the desired triangularization
\[\left[\begin{array}{cc}C & 0 \\ M & B\end{array}\right]\cong\left[\begin{array}{ccc}C & 0 & 0\\ M & H & 0\\
0 & N & B'\end{array}\right].
\]

\vspace*{0.3cm}
Ê$4)\Rightarrow 3)$ Take $B=\left[\begin{array}{cc}H & 0 \\ N & B'\end{array}\right]$ and $M=\left[\begin{array}{cc}M & 0\end{array}\right]^{^t}$.

\vspace*{0.3cm}

$1)\Rightarrow 2)$ From Proposition \ref{anhadida} we know that $I=Ae$, for some idempotent $e\in A$. Since $I$ is a two-sided ideal, then $eA\subseteq Ae$ and so $eA(1-e)=0$. Therefore, we can identify $A$ with the ring
\[\left[\begin{array}{cc}C & 0 \\ M & B\end{array}\right],
\] 
where $C=eAe\ko M=(1-e)Ae$ and $B=(1-e)A(1-e)$. Then 
\[I=\left[\begin{array}{cc}C & 0 \\ M & 0\end{array}\right]
\] 
and $\cy$ consists of those modules $N\in\Mod A$ annihilated by $I$, \ie such that $N(1-e)=N$.
This latter subcategory gets identified $\Mod B$.

We want to identify now $\tau_{\cy}(N)$ for every $N\in\Mod A$. If 
\[N=(L_{C},N'_{B}; \varphi:N'\otimes_{B}M\ra L)
\] 
then $\tau_{\cy}(N)=
\ann_{N}(I)=\{(0,y)\in N=L\oplus N'\mid \varphi(y\otimes
M)=0\}$. Now, we have that $\varphi(y\otimes M)=0$ if and only if $\varphi^t(y)=0$, where $\varphi^{t}$ comes from the adjunction isomorphism
\[\Hom_{C}(N'\otimes_{B}M,L)\arr{\sim}\Hom_{B}(N',\Hom_{C}(M,L))\ko \varphi\mapsto\varphi^t.
\]
Then $\tau_{\cy}(N)=(0,\ker(\varphi^t), 0)$ and, since it is a direct summand of $N=(L,N';\varphi)$ by
hypothesis, we get that $\ker(\varphi^t)$ is a direct summand of $N'$ in $\Mod B$, for
every $\varphi\in\Hom_{C}(N'\otimes_{B}M,L)$. But then $\ker(\psi)$ is a direct summand of
$N'$ in $\Mod B$ for every $\psi\in\Hom_{B}(N',\Hom_{C}(M,L))$. Since this is valid for arbitrary $L_{C}\ko N'_{B}$, we can take Êfor $N'_{B}$ an arbitrary projective and we get that every $B$-submodule of $\Hom_{C}(M,L)$ is projective. Then by Proposition ~\ref{proposicion de hereditary pi-projective dual} we get that $_{B}M$ has hereditary $\Pi$-projective dual.

$3)\Rightarrow 1)$ We identify $A$ with the ring
\[\left[\begin{array}{cc}C & 0 \\ M & B\end{array}\right].
\] 
Take
\[e=\left[\begin{array}{cc}1 & 0 \\ 0 & 0\end{array}\right] \text{ and } 
I=Ae=\left[\begin{array}{cc}C & 0 \\
M & 0\end{array}\right].
\] 
As we saw before, if $N=(L,N';\varphi)\in\Mod A$, we have the identity $\tau_{\cy}(N)=(0,\ker(\varphi^t);0)$. But $\op{Im}(\varphi^t)$ is a $B$-submodule of $\Hom_{C}(M,L)$, whence projective in $\Mod B$. Thus, $\ker(\varphi^t)$ is a direct summand of $N'$ in $\Mod B$,
and so $\tau_{\cy}(N)$ is a direct summand of $N$ in $\Mod A$.
\end{proof}

\begin{corollary} \label{right split para buenos2}
Let $A$ be a ring. Then the one-to-one correspondence of Theorem ~\ref{parametrizando TTF ternas} restricts to a one-to-one correspondence between:
\begin{enumerate}[1)]
\item Right split TTF triples on $\Mod A$ whose associated idempotent ideal $I$ is finitely generated on the left.
\item Ideals of the form $I=Ae$, where $e$ is an idempotent of $A$ such that the left ${(1-e)A(1-e)}$-module $(1-e)Ae$ has hereditary $\Pi$-projective dual.
\end{enumerate}
In particular, when $A$ satisfies either one of the two following conditions, the class $1)$ above covers all the right split TTF triples on $\Mod A$:
\begin{enumerate}[i)]
\item $A$ is semiperfect.
\item Every idempotent Êideal of $A$ which is pure on the left is also finitely generated on the left (\eg if $A$ is left N\oe therian).
\end{enumerate}
\end{corollary}
\begin{proof}
Using Proposition \ref{anhadida}, under conditions $i)$ or $ii)$ the idempotent Êideal associated to a right split TTF triple on $\Mod A$ is always finitely
generated on the left. With that in mind, the result is a direct consequence of the foregoing theorem.
\end{proof}

\section{Right split TTF triples over arbitrary rings}\label{Right split TTF triples over arbitrary rings}
\addcontentsline{lot}{section}{3.5. Ternas TTF escindidas por la derecha sobre anillos arbitrarios}

Let $(\cx,\cy,\cz)$ be a ÊTTF triple on $\Mod A$ Êwith associated idempotent ideal $I$. 

\begin{lemma}\label{pre right split}
If $\tau_{\cy}(A_A)=(1-\varepsilon )A$, for some idempotent $\varepsilon\in A$, then $\varepsilon A(1-\varepsilon)=0$ and $I\subseteq\varepsilon A\varepsilon$.
\end{lemma}
\begin{proof}
Since $\tau^{\cz}(A)=A/(1-\varepsilon)A=\varepsilon A$, then 
\[0=\Hom_{A}((1-\varepsilon)A,\varepsilon A)\arr{\sim}\varepsilon A(1-\varepsilon).
\] 
Moreover, $0=\tau_{\cy}(A)I=(1-\varepsilon)I$, and so $I\subseteq\varepsilon A=\varepsilon A\varepsilon$.
\end{proof}

\begin{proposition}
The one-to-one correspondence of Proposition ~\ref{parametrizando
TTF ternas} restricts to a one-to-one correspondence between:
\begin{enumerate}[1)]
\item Right split TTF triples on $\Mod A$.
\item Idempotent ideals $I$ of $A$ such that, for some idempotent $\varepsilon\in A$, one has that $\lann_A(I)=(1-\varepsilon )A$ and the
TTF triple on $\Mod\varepsilon A\varepsilon$ associated to $I$ is
right split.
\end{enumerate}
\end{proposition}
\begin{proof}
According to Lemma \ref{pre right split}, Êour goal reduces to prove that if $(\cx,\cy,\cz)$ is a TTF triple such that $\tau_{\cy}(A_A)=(1-\varepsilon )A$ is a direct
summand of $A_A$, then it is right split if, and only if, the TTF triple defined by $I$ in $\Mod\varepsilon A\varepsilon$ is also right split.

\emph{`Only if' part:} Every right $\varepsilon A\varepsilon $-module $M$ can be viewed as a right $A$-module (by defining $M\cdot (1-\varepsilon )A=0$), and
then $M=\tau_{\cy}(M)\oplus F=\ann_{M}(I)\oplus F$ for some $\varepsilon A\varepsilon$-submodule $F$.

\emph{`If' part:} Conversely, suppose that the TTF triple on $\Mod\varepsilon A\varepsilon$ associated to $I$ is right split, and
put $C=\varepsilon A\varepsilon\ko B=(1-\varepsilon )A(1-\varepsilon )$ and $M=(1-\varepsilon)A\varepsilon$. As usual, we can identify $A$ with the triangular matrix ring
\[\left[\begin{array}{cc}C & 0 \\ M & B\end{array}\right],
\] 
and in this case also $I$ with 
\[\left[\begin{array}{cc}I & 0 \\0 & 0\end{array}\right],
\] 
where $I$ is an idempotent Êideal of $C$ such that the TTF triple on $\Mod C$ associated to $I$ is right
split. ÊNotice that, since $\lann_A(I)=(1-\varepsilon )A$, we have $MI=0$.Ê Now, let $M'=(L_{C},N_{B};\varphi)$
be a right $A$-module. Then $\tau_{\cy}(M')=\ann_{M'}(I)=(\ann_{L}(I),N;\tilde{\varphi})$, where $\tilde{\varphi}$ is given by the decomposition of
$\varphi$
\[N\otimes_BM\stackrel{\tilde{\varphi}}{\longrightarrow}\ann_L(I)\stackrel{j}{\hookrightarrow} L.
\] 
Consider now a retraction $p:L\ra\ann_{L}(I)$ in $\Mod C$ for the canonical inclusion $j$, which exists because the
TTF triple induced by $I$ in $\Mod C$ is right split. Then $p\circ\varphi =p\circ j\circ\tilde{\varphi}=1\circ\tilde{\varphi}=\tilde{\varphi}$,
which implies that 
\[\left[\begin{array}{cc}p & 1\end{array}\right]:(L,N;\varphi)=M'\ra \tau_{\cy}(M')=(\ann_{L}(I),N;\tilde{\varphi})
\]
is a morphism of $A$-modules, which is then a retraction for the canonical inclusion $\tau_{\cy}(M')\hookrightarrow M'$.
\end{proof}

In the situation of the above proposition, one has that $\lann_{\varepsilon A\varepsilon}(I)=0$, that is, the TTF triple $(\cx',\cy',\cz')$ in
$\Mod\varepsilon A\varepsilon$ associated to $I$ has the property that $\varepsilon A\varepsilon\in\cz'$. The problem of classifying right split TTF triples gets then reduced to answer the following:

\begin{question} Let $I$ be an idempotent Êideal of a ring $A$ such that $\lann_{A}(I)=0$ (\ie $A\in\cz$ where $(\cx,\cy,\cz)$ is the associated
TTF triple on $\Mod A$). Which conditions on $I$ are equivalent to say that $(\cx,\cy,\cz)$ is right split?
\end{question}

\begin{definition}\label{saturated}
Let $I$ be an idempotent ideal of a ring $A$. Given a Êright $A$-module $M$ and a submodule $N$ , we shall say that $N$ is \emph{$I$-saturated}\index{submodule!$I$-saturated} in $M$ when $xI\subseteq N$, with $x\in M$, implies that $x\in N$. Equivalently, this occurs when $M/N\in\cz$, where $(\cx,\cy,\cz)$ is the TTF triple on $\Mod A$ associated to $I$.
\end{definition}

When $\Lambda$ is a subset of $A$, we shall denote by $\mathcal{M}_{n\times n}(\Lambda)$ the subset of matrices of $\mathcal{M}_{n\times n}(A)$ with entries in $\Lambda$.

\begin{lemma} \label{new-splitting}
Let $I$ be an idempotent ideal of the ring $A$. The following assertions are equivalent:
\begin{enumerate}[1)]
\item For every integer $n>0$ and every $\mathcal{M}_{n\times n}(I)$-saturated right ideal $\mathfrak{a}$ of $\mathcal{M}_{n\times n}(A)$,
there exists Ê$x\in\mathfrak{a}$ such that $(\id_n-x)\mathfrak{a}\subseteq\mathcal{M}_{n\times n}(I)$.
\item For every integer $n>0$ and every $I$-saturated submodule $K$ of $A^{(n)}$, the quotient $\frac{A^{(n)}}{K+I^{(n)}}$ is
projective as a right $\frac{A}{I}$-module.
\end{enumerate}
\end{lemma}
\begin{proof}
Fix any integer $n>0$ and consider the equivalence of categories $F=:\Hom_A(A^{(n)},?):\Mod A\arr{\sim}\Mod \mathcal{M}_{n\times n}(A)$. It establishes a
bijection between $I$-saturated submodules $K$ of $A^{(n)}$ and $\mathcal{M}_{n\times n}(I)$-saturated right ideals $\mathfrak{a}$ of the ring
$\mathcal{M}_{n\times
n}(A)=\Hom_A(A^{(n)},A^{(n)})$. Now if $K$ and $\mathfrak{a}$ correspond by that bijection, one has that
\[F\left(\frac{A^{(n)}}{K+I^{(n)}}\right)\cong\frac{\mathcal{M}_{n\times n}(A)}{\mathfrak{a}+\mathcal{M}_{n\times n}(I)}.
\] 
Assertion 2) is equivalent to say that,
for such a $K$, Êthe canonical projection 
\[\left(\frac{A}{I}\right)^{(n)}\cong\frac{A^{(n)}}{I^{(n)}}\twoheadrightarrow \frac{A^{(n)}}{K+I^{(n)}}
\] 
is a retraction Êin $\Mod A$. The proof is hence reduced to check that condition 1) is equivalent to say that the canonical projection
\[\frac{\mathcal{M}_{n\times n}(A)}{\mathcal{M}_{n\times n}(I)}\twoheadrightarrow\frac{\mathcal{M}_{n\times n}(A)}{\mathfrak{a}+\mathcal{M}_{n\times n}(I)}
\] 
is a retraction in $\Mod \mathcal{M}_{n\times n}(A)$, for every
$\mathcal{M}_{n\times n}(I)$-saturated right ideal
$\mathfrak{a}$ of $\mathcal{M}_{n\times n}(A)$. To do that it is not restrictive to
assume that $n=1$, something that we do from now on in this proof. Then the existence of
an element $x\in\mathfrak{a}$ such that $(1-x)\mathfrak{a}\subseteq I$ is equivalent to
say that there is an element $x\in\mathfrak{a}$ such that $\bar{x}=x+I$ generates
$(\mathfrak{a}+I)/I$ and $\bar{x}^2=\bar{x}$. That is clearly equivalent to say
that the canonical projection 
\[\frac{A}{I}\ra\frac{A}{\mathfrak{a}+I}
\] 
is a retraction.
\end{proof}

\begin{definition}\label{definition right splitting}
An idempotent ideal $I$ of a ring $A$ will be called \emph{right splitting}\index{ideal!right splitting} if it satisfies one (and hence both) of the conditions of Lemma \ref{new-splitting}, is pure as a left ideal and $\lann_{A}(I)=0$.
\end{definition}

\begin{example}
Let $I$ be an idempotent ideal of a ring $A$. If $I$ is pure on the left, $\lann_{A}(I)=0$ and $A/I$ is a semisimple ring, then $I$ is right splitting.
\end{example}

\begin{theorem} \label{clasificacion right split}
Let $A$ be a ring and $(\cx,\cy,\cz)$ be a TTF triple on $\Mod A$ with associated idempotent ideal $I$ and such that $A\in\cz$.
The following assertions are equivalent:
\begin{enumerate}[1)]
\item $(\cx,\cy,\cz)$ is right split.
\item $(\cx,\cy)$ is hereditary and $F/FI$ is a projective right $A/I$-module for all $F\in\cz$.
\item $I$ is right splitting and $A/I$ is a hereditary perfect ring.
\end{enumerate}
\end{theorem}
\begin{proof}
$1)\Rightarrow 2)$ By Proposition ~\ref{primeras propiedades de los right split} we know that $(\cx,\cy)$ is hereditary. Now take
arbitrary modules $T\in\cy\ko F\in\cz$, and apply $\Hom_{A}(?,T)$ to the short exact sequence
\[0\ra FI\hookrightarrow F\ra F/FI\ra 0.
\]
One gets the exact sequence
\[0=\Hom_{A}(FI,T)\ra\Ext^1_{A}(F/FI,T)\ra\Ext^1_{A}(F,T).
\]
The split condition of $(\cy,\cz)$ gives that $\Ext^1_{A}(F,T)=0$, and hence $\Ext^1_{A}(F/FI,T)=0$ for all $T\in\cy=\Mod\frac{A}{I}$.
Then $F/FI$ is a projective right $A/I$-module.

$2)\Rightarrow 1)$ Since $(\cx,\cy)$ is hereditary, it follows that $\tau_{\cy}(M)\cap MI=0$ for every $M\in\Mod A$. So the composition
\[\tau_{\cy}(M)\overset{j}{\hookrightarrow}M\overset{p}{\twoheadrightarrow}M/MI
\]
is a monomorphism with cokernel 
\[\cok(pj)\cong\frac{M}{\tau_{\cy}(M)+MI}\cong\frac{M/\tau_{\cy}(M)}{(M/\tau_{\cy}(M))\cdot I},
\]
which, by hypothesis, is a projective right $A/I$-module. Then $pj$ is a section in ($\Mod\frac{A}{I}$, and so in) $\Mod A$. Thus $j$ is a section in $\Mod A$.

$2)\Rightarrow 3)$ By Lemma ~\ref{hereditary igual a pure igual a condicion aritmetica} we know that $I$ is pure on the left.
Moreover, since $A_{A}\in\cz$ we get $\lann_{A}(I)=0$. On the other hand, the fact that $F/FI$ is projective over $A/I$, for all
$F\in\cz$, implies that if $K<A^{(n)}$ is an $I$-saturated submodule then $A^{(n)}/(K+I^{(n)})$ is projective as a
right $\frac{A}{I}$-module. Then, from Lemma \ref{new-splitting}, we derive that $I$ is right splitting.

Take now any right ideal $\mathfrak{b}/I$ of $A/I$ (notice that we then have $\mathfrak{b}\in\cz$). Since $(\cx,\cy)$ is hereditary
we have $\tau_{\cx}(\mathfrak{b})=\mathfrak{b}\cap\tau_{\cx}(A)$,
\ie $\mathfrak{b}I=\mathfrak{b}\cap I=I$. Thus $\mathfrak{b}/I=\mathfrak{b}/\mathfrak{b}I$ is a projective right $A/I$-module.
That proves that $A/I$ is a right hereditary ring.

Finally, observe that if $\cs$ is any set then $A^{\cs}/A^{\cs}I$ is a projective right $A/I$-module. Now, \cite[Theorem 5.1]{Goodearl} says that $A/I$ is right perfect.

$3)\Rightarrow 2)$ By Lemma ~\ref{hereditary igual a pure igual a condicion aritmetica} we know that $(\cx,\cy)$ is hereditary. Let
us take now $F\in\cz$. If $F$ is finitely generated then $F\cong A^{(n)}/K$, for some $I$-saturated submodule
$K<A^{(n)}$. Then, by Lemma \ref{new-splitting}, we have that $F/FI\cong A^{(n)}/K+I^{(n)}$ is projective as a right
$A/I$-module. In case $F$ is not necessarily finitely generated, then Ê$F=\bigcup F_\alpha$ is the directed union of its finitely
generated submodules, which implies that $F/FI$ is a direct limit of the $F_\alpha/F_\alpha I$. Then $F/FI$ is a direct limit of
projective right $A/I$-modules. Since $A/I$ is right perfect, we conclude that $F/FI$ is projective over $A/I$, for all $F\in\cz$.
\end{proof}

The desired full classification of right split TTF triples on $\Mod A$ is now available:

\begin{corollary} \label{clasification right split2}
Let $A$ be an arbitrary ring. The one-to-one correspondence of Theorem ~\ref{parametrizando TTF ternas} restricts to a one-to-one correspondence between:
\begin{enumerate}[1)]
\item Right split TTF triples on $\Mod A$.
\item Idempotent ideals $I$ such that, for some idempotent $\varepsilon\in A$, $\lann_A(I)=(1-\varepsilon )A$ and $I$ is a right splitting ideal of
$\varepsilon A\varepsilon $ with $\varepsilon A\varepsilon /I$ a hereditary perfect ring.
\end{enumerate}
\end{corollary}

\begin{remark}\label{condicion necesaria}
If in the definition of right splitting ideal we replace the conditions Êof Lemma \ref{new-splitting} by its corresponding ones with $n=1$, then
the correspondent of Theorem \ref{clasificacion right split} is not true. Indeed, consider 
\[A=\left[\begin{array}{ccc}k & 0 \\ M & H \end{array}\right],
\] 
where $k$ is an algebraically closed field, 
\[H=k(\xymatrix{\cdot\ar@<0.75ex>[r]\ar@<-0.75ex>[r] & \cdot})
\] 
is the Kronecker algebra
and $_HM=\tau_H (S)$, where $\tau_H$ is the Auslander-Reiten translation \cite[Chapter VII]{AuslanderReitenSmalo1995} and $S$ is the
simple injective left $H$-module. Notice that $_{H}M$ is faithful. We put 
\[e=\left[\begin{array}{ccc}1 & 0 \\ 0 & 0 \end{array}\right]
\] 
and take
\[I=Ae=\left[\begin{array}{ccc}k & 0 \\ M & 0 \end{array}\right],
\] 
which is clearly pure as a left ideal and satisfies that
$\lann_A(I)=0$. A right ideal $\mathfrak{a}$ of $A$ is represented by a triple $(V,\mathfrak{u};\varphi:\mathfrak{u}\otimes_HM\longrightarrow V)$,
where $V$ is a vector subspace of $k\oplus M$, $\mathfrak{u}$ is a right ideal of $H$ such that $0\oplus\mathfrak{u}M\subseteq V$ and $\varphi$ is the
canonical multiplication map $h\otimes m\mapsto (0,hm)$. One readily sees that $\mathfrak{a}$ is $I$-saturated if, and only
if, $h\in\mathfrak{u}$ whenever $0\oplus hM\subseteq V$. In that case the canonical morphism 
\[H\longrightarrow \Hom_k(M,k\oplus M)
\] 
in $\Mod H$ induces a monomorphism 
\[H/\mathfrak{a}\ra \Hom_k(M,\frac{k\oplus M}{V}).
\] 
But $\Hom_k(M,(k\oplus M)/V)$ is isomorphic to $D(M)^{(r)}$, for some natural number $r$, where $D=\Hom_k(?,k)$ is
the canonical duality. Then $D(M)=\tau_H^{-1}(T)$, where $T=D(S)$ is the simple projective right $H$-module. It is known that all
cyclic submodules of $D(M)$ are projective (see for instance \cite[Section 3.2]{Ringel1984}). From that one easily derives that, in
our case, $H/\mathfrak{a}$ is a projective right $H$-module and, hence, Êcondition 2 of Lemma \ref{new-splitting}
holds for $n=1$. However, according to Theorem \ref{clasificacion right split para buenos}, the TTF triple on $\Mod A$ associated to $I$ is not right split
because $_HM$ is not FP-injective.
\end{remark}

\begin{example}
Let $H,C$ be rings, the first one being hereditary perfect, and $_HM_C$ be a bimodule such that $_HM$ is Êfaithful. The idempotent ideal
\[I\arr{\sim}\left[\begin{array}{cc}C&0\\M&0\end{array}\right]
\]
of
\[A\arr{\sim}\left[\begin{array}{cc}C&0\\M&H\end{array}\right]
\]
is clearly pure on the left and Ê$\lann_A(I)=0$. Combining
Theorem \ref{clasificacion right split para buenos} and Theorem
\ref{clasificacion right split}, one gets that $I$ is right
splitting (\ie, it satisfies the equivalent conditions of Lemma
\ref{new-splitting}) if, and only if, $_HM$ is FP-injective
(equivalently, $_HM$ has hereditary $\Pi$-projective dual). We
leave as an exercise to check it directly by using an argument
similar to that of Remark \ref{condicion necesaria}.
\end{example}

\begin{example}\label{l=c<r}
If $A$ is commutative and we denote by $\mathfrak{L}$, $\mathfrak{C}$ and $\mathfrak{R}$ the sets of left, centrally and right split
ÊTTF triples on $\Mod A$, respectively, then Ê$\mathfrak{L}=\mathfrak{C}\subset\mathfrak{R}$ and the last inclusion may be strict.
ÊIndeed, since all idempotents in $A$ are central, the equality $\mathfrak{L}=\mathfrak{C}$ follows from Corollary \ref{clasificacion left split}. On the other hand, if Ê$k$ is a field and $A$ is the ring formed by the eventually constant sequences of elements of $k$, then the subset $I=k^{(\N)}$ formed by the eventually zero sequences is an idempotent ideal of $A$ which is pure and satisfies that $\lann_A(I)=0$. Moreover, one has $A/I\cong k$ and then condition 3) in TheoremÊ~\ref{clasificacion right split} holds (see the example after the Definition ~\ref{definition right splitting}). The associated TTF triple is then right split but not centrally split.
\end{example}

\chapter{Preliminary results on triangulated categories}\label{Preliminary results}
\addcontentsline{lot}{chapter}{Cap\'itulo 4. Resultados preliminares sobre categor\'ias trianguladas}

\section{Introduction}
\addcontentsline{lot}{section}{4.1. Introducci\'on}

\subsection{Motivation}
\addcontentsline{lot}{subsection}{4.1.1. Motivaci\'on}

We remind results and set terminology for the subsequent chapters. We prove, however, some apparently new results which measure the distance between ``compact'' and ``self-compact'' (or between ``compact'' and ``perfect''). Also, we study in some detail methods of constructing t-structures and some properties of `right bounded derived categories' that we will need for the last chapter. 

\subsection{Outline of the chapter}
\addcontentsline{lot}{subsection}{4.1.2. Esbozo del cap\'itulo}

In section \ref{TTF triples and recollements}, we remind the notion of \emph{triangulated torsion-torsionfree(=TTF) triple} on a triangulated category $\cd$, and prove that triangulated TTF triples on $\cd$ are in `bijection' with \emph{d\'{e}collements} of $\cd$. In section \ref{Generators and infinite devissage}, we study several ways of generating a triangulated category and the relationship between them. In section \ref{(Super)perfectness, compactness and smashing}, we study the crucial notion of \emph{compact} and \emph{(super)perfect} objects, and the r\^{o}le they play in: the several ways of generating a triangulated category, the construction or existence of t-structures, Brown representability theorem,\dots Also, we recall the definition of \emph{smashing subcategory} and the fact that smashing subcategories are in bijection with triangulated TTF triples, for instance when the ambient triangulated category is perfectly generated. In section \ref{B. Keller's Morita theory for derived categories}, we recall some basis of the seminal B.~Keller's Morita theory of derived categories of dg categories. In section \ref{Some constructions of t-structures}, we study general methods of constructing t-structures. Some of these methods require techniques of homotopical algebra which we recall. Finally, in section \ref{The right bounded derived category of a dg category}, we define and study the `right bounded' derived category of a dg category.

\subsection{Notation}
\addcontentsline{lot}{subsection}{4.1.3. Notaci\'on}

Unless otherwise stated, $k$ will be a commutative (associative, unital) ring and every additive category will be assumed to be $k$-linear. We recall that when $\cd$ is a triangulated category, the \emph{shift} or the \emph{translation functor} will be denoted by $?[1]$. When we speak of ``all the shifts'' or ``closed under shifts'' and so on, we will mean ``shifts in both directions'', that is to say, we will refer to the $n$th power $?[n]$\index{$?[n]$} of $?[1]$ for all the integers $n\in\Z$. In case we want to consider another situation (\eg, non-negative shifts $?[n]\ko n\geq 0$) this will be said explicitly.

If $\cq$ is a class of objects of a triangulated category $\cd$:
\begin{enumerate}[1)]
\item $\cq^+$\index{$\cq^+$} will be the class of all non-negative shifts of objects of $\cq$.
\item $\Sum_{\cd}(\cq)$\index{$\Sum_{\cd}(\cq)$}, or $\Sum(\cq)$\index{$\Sum(\cq)$} if $\cd$ is clear, will be the class of all small coproducts of objects of $\cq$. 
\item $\aisle_{\cd}(\cq)$\index{$\aisle_{\cd}(\cq)$}, or $\aisle(\cq)$\index{$\aisle(\cq)$} if $\cd$ is clear, will be the smallest aisle in $\cd$ (\cf Definition \ref{aisle}) containing $\cq$. Notice that $\aisle_{\cd}(\cq)$ might not exist, but, if it exists, it is closed under small coproducts.
\item $\Susp_{\cd}(\cq)$\index{$\Susp_{\cd}(\cq)$}, or $\Susp(\cq)$\index{$\Susp(\cq)$} if $\cd$ is clear, will be the smallest full suspended subcategory of $\cd$ (\cf Definition \ref{definition suspended category}) containing $\cq$ and closed under small coproducts.
\item $\Tria_{\cd}(\cq)$\index{$\Tria_{\cd}(\cq)$}, or $\Tria(\cq)$\index{$\Tria(\cq)$} if $\cd$ is clear, will be the smallest full triangulated subcategory of $\cd$ containing $\cq$ and closed under small coproducts.
\end{enumerate}

If $\cu$ and $\cv$ are two classes of objects of a triangulated category $\cd$, then $\cu*\cv$\index{$\cu*\cv$} is the class of extensions of objects of $\cv$ by objects of $\cu$, \ie the class formed by those objects $M$ occurring in a triangle
\[U\ra M\ra V\ra U[1]
\]
of $\cd$ with $U\in\cu$ and $V\in\cv$. The axiom RT4) of Definition \ref{definition suspended category} implies that the operation $*$ is associative. For each natural number $n\geq 0$, the objects of 
\[\cu^{*n}:=\cu*\overset{\text{n times}}{\dots}*\cu
\] 
are called \emph{$n$-fold extensions}\index{$\cu^{*n}$} or \emph{extensions of length $n$}\index{extensions!$n$-fold}\index{extensions!of length $n$} of objects of $\cu$.

Recall that if $(\cu,\cv[1])$ is a t-structure on $\cd$ we denote by $u:\cu\hookrightarrow\cd$ and $v:\cv\hookrightarrow\cd$ the inclusion functors, by $\tau_{\cu}$ a right adjoint to $u$ and by $\tau^{\cv}$ a left adjoint to $v$.

\section{Triangulated TTF triples and recollements/d\'{e}collements}\label{TTF triples and recollements}
\addcontentsline{lot}{section}{4.2. Ternas TTF trianguladas y (des)aglutinaciones}

Recall the definition of triangulated TTF triple on a triangulated category:

\begin{definition}\label{triangulated TTF triple}
A \emph{triangulated TTF triple} on a triangulated category $\cd$ is a triple $(\cx,\cy,\cz)$ of full subcategories of $\cd$ such that $(\cx,\cy)$ and $(\cy,\cz)$ are t-structures on $\cd$. Notice that, in particular, $\cx\ko \cy$ and $\cz$ are full triangulated subcategories of $\cd$. 
\end{definition}

\begin{definition}\label{recollement}
Let $\cd\ko \cd_{F}$ and $\cd_{U}$ be triangulated categories. Then $\cd$ is a \emph{recollement of $\cd_{F}$ and $\cd_{U}$}\index{recollement}, diagrammatically expressed by
\[ \xymatrix{ \cd_{F}\ar[r]^{i_{*}=i_{!}} & \cd\ar@/_1.5pc/[l]_{i^{*}}\ar@/_-1pc/[l]^{i^{!}}\ar[r]^{j^*=j^!} & 
\cd_{U}\ar@/_1.5pc/[l]_{j_{*}}\ar@/_-1pc/[l]^{j_{!}},
}
\]
if there exist six triangle functors which satisfies the following four conditions:
\begin{enumerate}[R1)]
\item $(i^*,i_{*}=i_{!},i^!)$ and $(j_{!},j^!=j^*,j_{*})$ are adjoint triples.
\item $i^!j_{*}=0$ (and thus $j^!i_{!}=j^*i_{*}=0$ and $i^*j_{!}=0$).
\item $i_{*}\ko j_{!}$ and $j_{*}$ are full embeddings (and thus $i^*i_{*}\cong i^!i_{!}\cong \id_{\cd_{F}}$).
\item For every object $M$ of $\cd$ there exist two triangles,
\[i_{!}i^!M\ra M\ra j_{*}j^*M\ra (i_{!}i^!M)[1]
\]
and
\[j_{!}j^!M\ra M\ra i_{*}i^*M\ra (j_{!}j^!M)[1],
\]
in which the morphisms $i_{!}i^!M\ra M$ etc. are the corresponding adjunction morphisms.
\end{enumerate}
In this case we say that the data 
\[(\cd_{F},\cd_{U}, i^*,i_{*}=i_{!},i^!, j_{!},j^!=j^*,j_{*})
\] 
is a \emph{d\'{e}collement}\index{d\'{e}collement} of $\cd$.
Two d\'{e}collements of $\cd$,
\[ \xymatrix{ \cd_{F}\ar[r]^{i_{*}=i_{!}} & \cd\ar@/_1.5pc/[l]_{i^{*}}\ar@/_-1pc/[l]^{i^{!}}\ar[r]^{j^*=j^!} & 
\cd_{U}\ar@/_1.5pc/[l]_{j_{*}}\ar@/_-1pc/[l]^{j_{!}}
}
\]
and
\[\xymatrix{ \cd'_{F}\ar[r]^{i'_{*}=i'_{!}} & \cd'\ar@/_1.5pc/[l]_{i'^{*}}\ar@/_-1pc/[l]^{i'^{!}}\ar[r]^{j'^*=j'^!} & 
\cd'_{U}\ar@/_1.5pc/[l]_{j'_{*}}\ar@/_-1pc/[l]^{j'_{!}},
}
\]
are \emph{equivalent}\index{d\'{e}collement!equivalent} if 
\[(\im(i_{*}),\im(j_{*}),\im(j_{!}))=(\im(i'_{*}),\im(j'_{*}),\im(j'_{!})),
\] 
where by $\im(i_{*})$ we mean the essential image of $i_{*}$, and analogously with the other functors.
\end{definition}
 
It is well-known that triangulated TTF triples are in bijection with equivalence classes of d\'{e}collements (\cf \cite[1.4.4]{BeilinsonBernsteinDeligne}, \cite[subsection 9.2]{Neeman2001}). For the sake of completeness we recall here how this bijection works. First, we recall \cite[Proposition 1.1.9]{BeilinsonBernsteinDeligne}:

\begin{lemma}\label{truncacion funtorial}
Let $\cd$ be a triangulated category and let $X\arr{u}M\arr{v}Y\arr{z}X[1]$ and $X'\arr{u'}M'\arr{v'}Y'\arr{z'}X'[1]$ be two triangles such that $\cd(X',Y)=\cd(X,Y')=\cd(X,Y'[-1])=0$. Then every morphism $g\in\cd(M,M')$ can be completed uniquely to a morphism of triangles
\[\xymatrix{X\ar[r]^{u}\ar@{.>}[d]^{f} & M\ar[r]^{v}\ar[d]^g & Y\ar[r]^{w}\ar@{.>}[d]^{h} & X[1]\ar@{.>}[d]^{f[1]} \\
X'\ar[r]^{u'} & M'\ar[r]^{v'} & Y'\ar[r]^{w'} & X'[1]
}
\]
\end{lemma}
\begin{proof}
By applying the functor $\cd(X,?)$ to the second triangle we get an isomorphism
\[\cd(X,X')\arr{\sim}\cd(X,M'),
\]
and so there exists a unique $f\in\cd(X,X')$ such that $gu=u'f$. Similarly, by applying the functor $\cd(?,Y')$ to the first triangle we get an isomorphism
\[\cd(Y,Y')\arr{\sim}\cd(M,Y'),
\]
and so there exists a unique $h\in\cd(Y,Y')$ such that $hv=v'f$.
\end{proof}

\begin{proposition}\label{TTF triples are recollements}
If $(\cx,\cy,\cz)$ is a triangulated TTF triple on $\cd$, then there exists a d\'{e}collement of $\cd$ as follows:
\[\xymatrix{ \cy\ar[r]^{y} & \cd\ar@/_-1pc/[l]^{\tau_{\cy}}\ar@/_1.5pc/[l]_{\tau^\cy}\ar[r]^{\tau_{\cx}} & \cx,\ar@/_-1pc/[l]^{x}\ar@/_1.5pc/[l]_{z\tau^{\cz}x}
}
\]
\end{proposition}
\begin{proof}
R1) The only point to check is that $(\tau_{\cx},z\tau^{\cz}x)$ is an adjoint pair. Let us do it: given $M\in\cd$ and $X\in\cx$, by applying $\cd(x\tau_{\cx}(M),?)$ to the triangle 
\[y\tau_{\cy}x(X)\ra xX\ra z\tau^{\cz}x(X)\ra y\tau_{\cy}x(X)[1]
\]
we get an isomorphism
\[\cx(\tau_{\cx}(M),X)\arr{\sim}\cd(x\tau_{\cx}(M),z\tau^{\cz}x(X)).
\]
Similarly, by applying $\cd(?,z\tau^{\cz}x(X))$ to the triangle
\[x\tau_{\cx}(M)\ra M\ra y\tau^{\cy}(M)\ra x\tau_{\cx}(M)[1]
\]
we get an isomorphism
\[\cd(M,z\tau^{\cz}x(X))\arr{\sim}\cd(x\tau_{\cx}(M),z\tau^{\cz}x(X)).
\]

R2) $\tau^{\cy}z\tau^{\cz}x=0$ because $\tau^{\cy}z=0$.

R3) Clearly $y$ and $x$ are full embeddings. Also
\[\xymatrix{ \cx\ar@{^(->}[r]^{x}&\cd\ar[r]^{\tau^{\cz}} & \cz\ar@{^(->}[r]^{z} & \cd
}
\]
is a full embedding since $\tau^{\cz}x:\cx\arr{\sim}\cz$ is a triangle equivalence (\cf Lemma \ref{x igual a z}).

R4) By applying $\cd(?,z\tau^{\cz}x\tau_{\cx}(M))$ to the triangle
\[x\tau_{\cx}M\arr{\delta_{\cx,M}}M\arr{\eta^{\cy}_{M}}y\tau^{\cy}M\ra x\tau_{\cx}(M)[1]
\]
we get an isomorphism
\[\cd(x\tau_{\cx}(M),z\tau^{\cz}x\tau_{\cx}(M))\arr{\sim}\cd(M,z\tau^{\cz}x\tau_{\cx}(M)).
\]
Hence, there exists a unique morphism $\eta$ such that the square
\[\xymatrix{x\tau_{\cx}(M)\ar[d]^{\id}\ar[r]^{\delta_{\cx,M}} & M\ar@{.>}[d]^{\eta} \\
x\tau_{\cx}(M)\ar[r]_{\eta^{\cz}_{x\tau_{\cx}(M)}} & z\tau^{\cz}x\tau_{\cx}(M)
}
\]
commutes. Notice that the mapping cones of $\delta_{\cx,M}$ and $\eta^{\cz}_{x\tau_{\cx}(M)}$ are in $\cy$ and thus, by the axiom RT4 of Definition \ref{definition suspended category}, so is the mapping cone $\cone(\eta)$ of $\eta$. But then, the uniqueness of the torsion triangle associated to the t-structure $(\cx,\cy)$ (\cf Lemma \ref{truncacion funtorial}) implies that $\cone(\eta)\cong y\tau_{\cy}(M)[1]$. Therefore, there exists a triangle
\[y\tau_{\cy}(M)\ra M\ra z\tau^{\cz}x\tau_{\cx}(M)\ra y\tau_{\cy}(M)[1].
\]
\end{proof}

\begin{proposition}\label{recollements are TTF triples}
If
\[\xymatrix{ \cd_{F}\ar[r]^{i_{*}} & \cd\ar@/_-1pc/[l]^{i^!}\ar@/_1.5pc/[l]_{i^*}\ar[r]^{j^*} & \cd_{U}\ar@/_-1pc/[l]^{j_{!}}\ar@/_1.5pc/[l]_{j_{*}}
}
\]
is a d\'{e}collement of $\cd$, then 
\[(j_{!}(\cd_{U}),i_{*}(\cd_{F}),j_{*}(\cd_{U}))
\] 
is a triangulated TTF triple on $\cd$, where by $j_{!}(\cd_{U})$ we mean the essential image of $j_{!}$, and analogously with the other functors. 
\end{proposition}
\begin{proof}
Let us check that  $(j_{!}(\cd_{U}),i_{*}(\cd_{F}))$ is a t-structure. Clearly both $j_{!}(\cd_{U})$ and $i_{!}(\cd_{F})$ are strictly full triangulated categories of $\cd$. Also, we already know that for every $M\in\cd$ there exists a triangle
\[j_{!}j^!M\ra M\ra i_{*}i^*M\ra (j_{!}j^!M)[1],
\]
where $j_{!}j^!M\in j_{!}(\cd_{U})$ and $i_{*}i^*M\in i_{*}(\cd_{F})$. Finally, by using the adjunction $(i^*,i_{*})$ we can do 
\[\cd(j_{!}(\cd_{U}),i_{*}(\cd_{F}))\cong\cd(i^*j_{!}(\cd_{U}),\cd_{F})=\cd(0,\cd_{F})=0.
\]
To prove that $(i_{*}(\cd_{F}),j_{*}(\cd_{U}))$ is a t-structure we procede similarly.
\end{proof}

\section{Generators and infinite d\'{e}vissage}\label{Generators and infinite devissage}
\addcontentsline{lot}{section}{4.3. Generadores y d\'evissage infinito}

\begin{definition}\label{generate triangulated category}
We say that a triangulated category $\cd$ is \emph{generated}\index{category!triangulated!generated by a class of objects} by a class $\cq$ of objects if an object $M$ of $\cd$ is zero whenever
\[\cd(Q[n],M)=0
\]
for every object $Q$ of $\cq$ and every integer $n\in\Z$. In this case, we say that $\cq$ is a \emph{class of generators}\index{category!triangulated!generator of}\index{generator!of a triangulated category} of $\cd$. 
\end{definition}

\begin{definition}\label{infinite devissage}
Recall that given a class $\cq$ of objects of a triangulated category $\cd$, we denote by $\Tria(\cq)$ the smallest full triangulated subcategory of $\cd$ containing $\cq$ and closed under small coproducts. We say that a triangulated category $\cd$ satisfies the \emph{principle of infinite d\'{e}vissage}\index{category!triangulated!principle of infinite d\'{e}vissage} with respect to a class of objects $\cq$ if $\cd=\Tria(\cq)$. 
\end{definition}

\begin{lemma}\label{right orthogonal and devissage}
Let $\cq$ be a class of objects of a triangulated category $\cd$. Then $\Tria(\cq)^{\bot}$ is precisely the class of objects $M$ of $\cd$ which are right orthogonal to all the shifts of objects of $\cq$.
\end{lemma}
\begin{proof}
Let $M$ be an object of $\cd$ such that $\cd(Q[n],M)=0$ for all $Q\in\cq$ and $n\in\Z$. Then, it turns out that the full subcategory $\cc$ of $\Tria(\cq)$ formed by all the objects $N$ such that $\cd(N[n],M)=0$, for every integer $n$, is a full triangulated subcategory of $\cd$ containing $\cq$ and closed under small coproducts. This implies that $\cc=\Tria(\cq)$.
\end{proof}

\begin{definition}\label{exhaustively generated}
We say that a triangulated category $\cd$ is \emph{exhaustively generated}\index{category!triangulated!exhaustively generated} by a set $\cq$ of objects of $\cd$ if the following conditions hold:
\begin{enumerate}[1)]
\item Small coproducts of objects of $\bigcup_{n\geq 0}\Sum(\cq)^{*n}$ exist in $\cd$.
\item For each object $M\in\cd$ there exists an integer $i\in\Z$ and a triangle
\[\coprod_{n\geq 0}Q_{n}\ra\coprod_{n\geq 0}Q_{n}\ra M[i]\ra \coprod_{n\geq 0}Q_{n}[1]
\]
in $\cd$ where for each $Q_{n}$ there exists some $m\geq 0$ such that $Q_{n}\in\Sum(\cq)^{*m}$. 
\end{enumerate}
If in the above setting $\cq=\cp^+$ for some set $\cp$, then we also say that $\cd$ is \emph{exhaustively generated to the left}\index{category!triangulated!exhaustively generated to the left} by $\cp$.
\end{definition}

\begin{lemma}\label{e implies d implies g}
Let $\cq$ be set of objects of a triangulated category $\cd$.
\begin{enumerate}[1)] 
\item If $\cd$ is exhaustively generated by $\cq$, then any full triangulated subcategory of $\cd$ containing the small coproducts of extensions of objects of $\Sum(\cq)$ equals $\cd$. In particular, it satisfies the principle of infinite d\'{e}vissage with respect to $\cq$. 
\item If $\cd$ satisfies the principle of infinite d\'{e}vissage with respect to $\cq$, then it is generated by $\cq$.
\end{enumerate}
\end{lemma}
\begin{proof}
1) Let $\cu$ be a full triangulated subcategory of $\cd$ containing all the small coproducts of objects in $\bigcup_{n\geq 0}\Sum(\cq)^{*n}$. Let $M$ be an object of $\cd$. There exists an integer $i\in\Z$ and a triangle 
\[M'\ra M''\ra M[i]\ra M'[1]
\]
with $M'$ and $M''$ being small coproducts of objects of $\bigcup_{n\geq 0}\Sum(\cq)^{*n}$. Therefore, $M'$ and $M''$ belong to $\cu$, and then so does $M$.

2) Let $M$ be an object of $\cd$ such that $\cd(Q[i],M)=0$ for every $Q\in\cq\ko i\in\Z$. Consider the full subcategory $\cu$ of $\cd$ formed by all the objects $N$ such that $\cd(N[i],M)=0$ for every $i\in\Z$. It is a full triangulated subcategory of $\cd$ containing $\cq$ and closed under small coproducts. Then, $\cu=\cd$. In particular, $\cd(M,M)=0$, \ie $M=0$.
\end{proof}

Conversely, the first part of the following lemma states that under certain hypothesis `generators' implies `infinite d\'{e}vissage'.

\begin{lemma}\label{generators and t-structures}
\begin{enumerate}[1)]
\item Let $\cd$ be a triangulated category and let $\cd'$ be a full triangulated subcategory generated by a class $\cq$ of objects.
If $\Tria(\cq)$ is an aisle in $\cd$ contained in $\cd'$, then $\cd'=\Tria(\cq)$.
\item Let $\cd$ be a triangulated category and let $(\cx,\cy)$ be a t-structure on $\cd$ with triangulated aisle. 
\begin{enumerate}[2.1)]
\item If $\cq$ is a class of generators of $\cd$, then $\tau^{\cy}(\cq)$ is a class of generators of $\cy$.
\item A class $\cq$ of objects of $\cx$ generates $\cx$ if and only if the objects of $\cy$ are precisely those which are right orthogonal to all the shifts of objects of $\cq$.
\end{enumerate}
\end{enumerate}
\end{lemma}
\begin{proof}
1) Indeed, given an object $M$ of $\cd'$ there exists a triangle
\[M'\ra M\ra M''\ra M'[1]
\]
with $M'$ in $\Tria(\cq)$ and $M''$ in $\Tria(\cq)^{\bot}$. Since $M'$ and $M$ are in $\cd'$, then so is $M''$. But $\cd'$ is generated by $\cq$, which implies that $M''=0$ and so $M$ belongs to $\Tria(\cq)$.

2.1) If $Y\in\cy$ is such that $0=\cy(\tau^{\cy}(Q)[n],Y)\cong\cd(Q[n],Y)$ for each $Q\in\cq\ko n\in\Z$, then $Y=0$.

2.2) Assume the class $\cq$ generates $\cx$. If $M\in\cd$ is right orthogonal to all the shifts of objects of $\cq$, then $\tau_{\cx}M$ has the same property and so $\tau_{\cx}M=0$, \ie $M\in\cy$. Conversely, assume that an object $Y\in\cd$ is in $\cy$ if and only if $\cd(Q[n],Y)=0$ for each $Q\in\cq\ko n\in\Z$. In that case, if an object $X$ of $\cx$ is right orthogonal to all the shifts of objects of $\cq$, then $X\in\cx\cap\cy=0$.
\end{proof}

\section{(Super)perfectness, compactness and smashing}\label{(Super)perfectness, compactness and smashing}
\addcontentsline{lot}{section}{4.4. (Super)perfecci\'on, compacidad y aplastamiento}

\begin{lemma}\label{coproducts of triangles}
Let $\cd$ be a triangulated category and let $L_{i}\arr{f_{i}}M_{i}\arr{g_{i}}N_{i}\arr{h_{i}}L_{i}[1]\ko i\in I$, be a family of triangles of $\cd$. If the coproducts $\coprod_{I}L_{i}$ and $\coprod_{I}M_{i}$ exist in $\cd$, then so does the coproduct $\coprod_{I}N_{i}$. Moreover, the induced sequence
\[\coprod_{I}L_{i}\arr{\coprod_{I}f_{i}}\coprod_{I}M_{i}\arr{\coprod_{I}g_{i}}\coprod_{I}N_{i}\arr{\coprod_{I}h_{i}}\coprod_{I}L_{i}[1]
\]
is a triangle of $\cd$.
\end{lemma}
\begin{proof}
Axiom RT4) of triangulated categories guarantees the existence a triangle of $\cd$ starting with $\coprod_{I}f_{i}$:
\[\coprod_{I}L_{i}\arr{\coprod_{I}f_{i}}\coprod_{I}M_{i}\ra N\ra \coprod_{I}L_{i}[1].
\]
Notice that for each index $i\in I$ axiom RT3) implies the existence of a morphism $w_{i}$ making the following diagram commutative:
\[\xymatrix{\coprod_{I}L_{i}\ar[r]^{\coprod_{I}f_{i}} & \coprod_{I}M_{i}\ar[r] & N\ar[r] & \coprod_{I}L_{i}[1] \\
L_{i}\ar[r]^{f_{i}}\ar[u]_{\can} & M_{i}\ar[r]^{g_{i}}\ar[u]_{\can} & N_{i}\ar[r]^{h_{i}}\ar[u]_{w_{i}} & L_{i}\ar[u]_{\can[1]}
}
\]
Given an arbitrary object $U$ of $\cd$, we then get a morphism of long exact sequences
\[\xymatrix{\dots \cd(\coprod_{I}L_{i}[1],U)\ar[r]\ar[d]^{\wr} & \cd(N,U)\ar[r]\ar[d] & \cd(\coprod_{I}M_{i},U)\ar[r]\ar[d]^{\wr} & \cd(\coprod_{I}L_{i},U)\ar[d]^{\wr}  \\
\dots \prod_{I}\cd(L_{i}[1],U)\ar[r] & \prod_{I}\cd(N_{i},U)\ar[r] & \prod_{I}\cd(M_{i},U)\ar[r] & \prod_{I}\cd(L_{i},U)
}
\]
and the five lemma implies that $N=\coprod_{I}N_{i}$.
\end{proof}

\begin{definition}
Let $\cd$ be a triangulated category. A \emph{localization}\index{functor!localization} of $\cd$ is a triple $(L,\alpha,\eta)$ where $(L,\alpha)$ is a triangle endofunctor of $\cd$ and $\eta:\id\ra L$ is a natural transformation such that:
\begin{enumerate}[1)]
\item $L\eta:L\arr{\sim}L^2$ is an isomorphism,
\item $L\eta=\eta L$,
\item $\eta$ commutes with the shift functor, \ie for each $M\in\cd$ the following diagram is commutative:
\[\xymatrix{M[1]\ar[r]^{\eta_{M}[1]}\ar[dr]_{\eta_{M[1]}} & (LM)[1] \\
& L(M[1])\ar[u]_{\alpha_{M}}^{\wr}
}
\] 
 \end{enumerate}
The localization is said to be \emph{smashing}\index{functor!localization!smashing} if $L$ preserves small coproducts.
\end{definition}

The motivation for the term ``smashing'' here can be found, for instance, in K.~Br\"{u}ning's thesis \cite[Proposition 2.4.4]{Bruning2007}.

In the following proposition we consider some standard definitions of \emph{smashing subcategory} and study their interplay in case the ambient triangulated category does not have small coproducts.

\begin{proposition}\label{characterization of smashing}
Let $\cd$ be a triangulated category and let $\cx$ be a full triangulated subcategory of $\cd$. Consider the following assertions:
\begin{enumerate}[1)]
\item $\cx$ is the kernel of a triangle functor $\tau^{\cy}:\cd\ra\cy$ having a right adjoint $y:\cy\ra\cd$ which preserves small coproducts and such that the counit of the adjunction is an isomorphism $\tau^{\cy}y\arr{\sim}\id$ (or, equivalently, $y$ is fully faithful).
\item The inclusion functor $\cx^{\bot}\ra\cd$ preserves small coproducts, has a left adjoint and $\ ^{\bot}(\cx^{\bot})=\cx$.
\item $\cx$ is closed under direct summands and small coproducts, and the quotient functor $q:\cd\ra\cd/\cx$ admits a right adjoint which preserves small coproducts.
\item The inclusion functor $x:\cx\ra\cd$ has a right adjoint $\tau_{\cx}$ which preserves small coproducts.
\item $\cx$ is the kernel of a smashing localization functor $(L,\alpha,\eta)$. 
\item $\cx$ is an aisle in $\cd$ such that $\cx^{\bot}$ is closed under small coproducts.
\end{enumerate}
Then $1)\Leftrightarrow 2)\Leftrightarrow 3)\Rightarrow 4)\Leftrightarrow 5)\Rightarrow 6)$. If $\cd$ has small coproducts then the six assertions are equivalent.
\end{proposition}
\begin{proof}
$1)\Rightarrow 2)$ Since $y$ is fully faithful, we can identify $\cy$ with a full triangulated subcategory of $\cd$, which turns out to be a coaisle in $\cd$ with torsion functor given by $\tau^{\cy}$ and associated aisle given by $\cx$. In particular, $\cx^{\bot}=\cy$ and $^{\bot}\cy=\cx$.

$2)\Rightarrow 1)$ Put $\cy:=\cx^{\bot}$. Since $\cy$ is a triangulated subcategory of $\cd$, the inclusion functor $\cy\ra\cd$ is a triangle functor and then so is its left adjoint, say $\tau^{\cy}$ (\cf \cite[Proposition 1.6]{KellerVossieck1987}). It remains to prove that $\cx$ is precisely the kernel of $\tau^{\cy}$. If $X$ is an object of $\cx$, then $\cy(\tau^{\cy}X,\tau^{\cy}X)\cong\cd(X,y\tau^{\cy}X)=0$, and so $\tau^{\cy}X=0$. Conversely, if $M$ is an object of $\cd$ such that $\tau^{\cy}M=0$, then for each $Y\in\cx^{\bot}$ we have $0=\cy(\tau^{\cy}M,Y)\cong\cd(M,yY)$, that is to say, $M$ belongs to $\ ^{\bot}(\cx^{\bot})=\cx$.

$1)\Rightarrow 3)$ The functor $\tau^{\cy}$ preserves small coproducts, since it has a right adjoint. Therefore, its kernel $\cx$ is closed under direct summands and small coproducts. Since $y$ is fully faithful, we can identify $\cy$ with a coaisle in $\cd$ whose associated aisle is $\cx$. The quotient functor $q$ identifies then with $\tau^{\cy}$.

$3)\Rightarrow 1)$ Take $\cy:=\cd/\cx$. The existence of a right adjoint $y$ for $q$ ensures that $\cy$ is a `true' category, \ie it has small morphisms spaces. Take $\tau^{\cy}:=q$. Since $\cx$ is closed under direct summands, then $\cx$ is precisely the kernel of $\tau^{\cy}$ (\cf \cite[Lemma 2.1.33]{Neeman2001}). Finally, \cite[Lemma 9.1.7]{Neeman2001}, together with Lemma \ref{properties of adjunctions}, proves that the counit $\tau^{\cy}y\ra\id$ is an isomorphism. Alternatively, one can prove that $y$ is fully faithful as follows. Let $\cs'$ be the class formed by those morphisms $f$ of $\cd$ such that $q(f)$ is an isomorphism. Then $q$ factors through the localization (in the sense of P. Gabriel and M. Zisman \cite{GabrielZisman}) $\cd\ra\cd[\cs'^{-1}]$ via a certain functor $q'$:
\[\xymatrix{\cd\ar[d]_{q}\ar[r] & \cd[\cs'^{-1}]\ar[dl]^{q'} \\
\cd/\cx &
}
\]
Thanks to \cite[Proposition I.1.3]{GabrielZisman}, we know that $y$ is fully faithful if and only if $q'$ is an equivalence. For this, we will prove that $q$ is a localization of $\cd$ with respect to $\cs'$. Indeed, by definition $q$ is the localization of $\cd$ with respect to the class $\cs$ formed by those morphisms $f$ of $\cd$ whose mapping cone $\cone(f)$ belongs to $\cx$. Of course, we have that $\cs$ is contained in $\cs'$. Conversely, let $f$ be a morphism such that $q(f)$ is an isomorphism, \ie $\cone(q(f))=0$ or, equivalently, $q(\cone(f))=0$. Now, \cite[Lemma 2.1.33]{Neeman2001} says that $\cone(f)$ is a direct summand of an object of $\cx$, and so it belongs to $\cx$.

$1)\Rightarrow 4)$ Given an arbitrary object $M$ of $\cd$ we consider the triangle
\[M'\ra M\arr{\eta_{M}}y\tau^{\cy}M\ra M'[1],
\]
where $\eta_{M}$ is the unit of the adjunction $(\tau^{\cy},y)$. Since $\tau^{\cy}(\eta_{M})$ is an isomorphism, we get that $M'$ belongs to $\cx$. By using the adjunction $(\tau^{\cy},y)$ we prove that $\cx$ is contained in $^{\bot}\cy$. Then, Lemma \ref{truncacion funtorial} implies that the construction $M\mapsto M'$ underlies a functor, denoted by $\tau_{\cx}:\cd\ra\cx$, which is easily seen to be right adjoint to the inclusion. Consider now a family $M_{i}\ko i\in I$, of objects of $\cd$ such that its coproduct $\coprod_{I}M_{i}$ exists. Since $\tau^{\cy}$ and $y$ preserve small coproducts, then the coproduct $\coprod_{I}y\tau^{\cy}M_{i}$ also exists. Finally, Lemma \ref{coproducts of triangles} implies the existence of the coproduct $\coprod_{I}\tau_{\cx}M_{i}$ and the canonical isomorphism $\coprod_{I}\tau_{\cx}M_{i}\arr{\sim}\tau_{\cx}\coprod_{I}M_{i}$.

$4)\Rightarrow 5)$ Consider the t-structure $(\cx,\cy)$ on $\cd$ associated to the aisle $\cx$, and take $L$ to be the composition
\[ L:\cd\arr{\tau^{\cy}}\cy\arr{y}\cd.
\]
It is clear that $\cx$ is the kernel of $L$. It remains to prove that $L$ is a smashing localization functor. Take $\eta:\id\ra L$ to be the unit of the adjunction $(\tau^{\cy},y)$. From the basic properties of t-structures one deduces that $L$ is a localization functor. Moreover, it preserves small coproducts. Indeed, let $M_{i}\ko i\in I$, be a family of objects of $\cd$ such that its coproduct $\coprod_{I}M_{i}$ exists. Since $\tau_{\cx}$ preserves small coproducts, we have that $\coprod_{I}x\tau_{\cx}M_{i}$ exists and the canonical morphism $\coprod_{I}x\tau_{\cx}M_{i}\ra x\tau_{\cx}\coprod_{I}M_{i}$ is an isomorphism. Now, Lemma \ref{coproducts of triangles} ensures that the coproduct $\coprod_{I}y\tau^{\cy}M_{i}$ exists and that
\[\coprod_{I}x\tau_{\cx}M_{i}\ra\coprod_{I}M_{i}\ra \coprod_{I}y\tau^{\cy}M_{i}\ra \coprod_{I}x\tau_{\cx}M_{i}[1]
\]
is a triangle of $\cd$. Then, the canonical morphism $\coprod_{I}y\tau^{\cy}M_{i}\ra y\tau^{\cy}\coprod_{I}M_{i}$ is an isomorphism.

$5)\Rightarrow 4)$ Given $M\in\cd$ we consider the triangle
\[M'\ra M\arr{\eta_{M}}LM\ra M'[1].
\]
Since $L(\eta_{M})$ is an isomorphism, then $L(M')=0$. Notice that $\cd(M',LN)=0$ for every $N\in\cd$. Indeed, it suffices to check that the map $\cd(LM,LN)\ra\cd(M,LN)$ induced by $\eta_{M}$ is surjective, which follows from the commutative square
\[\xymatrix{M\ar[rr]^{\eta_{M}}\ar[d]^{f} && LM\ar[d]^{Lf} \\
LN\ar[rr]^{\sim}_{\eta_{LN}=L(\eta_{N})} && L^{2}N
}
\]
valid for every $f$. The fact that $M'$ is left orthogonal to all the objects of the form $LN$ implies, via Lemma \ref{truncacion funtorial}, that the construction $M\mapsto M'$ induces a functor $\tau_{\cx}:\cd\ra\cx$. Let us check that this functor is right adjoint to the inclusion $x:\cx\ra\cd$. For this, it suffices to prove that $\cd(X,LM)=0$ for every $X\in\cx\ko M\in\cd$. Take $f\in\cd(X,LM)$. The commutative square
\[\xymatrix{X\ar[rr]^{\eta_{X}}\ar[d]^{f} && LX=0\ar[d]^{Lf} \\
LM\ar[rr]^{\sim}_{\eta_{LM}=L(\eta_{M})} && L^{2}M
}
\]
proves that $f=0$. Finally, we have to prove that $\tau_{\cx}$ preserves small coproducts. Let $M_{i}\ko i\in I$, be a family of objects of $\cd$ such that the coproduct $\coprod_{I}M_{i}$ exists. Since $L$ is smashing, the coproduct $\coprod_{I}LM_{i}$ exists and the canonical morphism $\coprod_{I}LM_{i}\ra L\coprod_{I}M_{i}$ is an isomorphism. Then, Lemma \ref{coproducts of triangles} says that $\coprod_{I}x\tau_{\cx}M_{i}$ exists and that the induced sequence
\[\coprod_{I}x\tau_{\cx}M_{i}\ra \coprod_{I}M_{i}\ra \coprod_{I}LM_{i}\ra \coprod_{I}\tau_{\cx}M_{i}[1]
\]
is a triangle. This implies that the coproduct $\coprod_{I}\tau_{\cx}M_{i}$ exists in $\cx$ and that the canonical morphism $\coprod_{I}\tau_{\cx}M_{i}\ra\tau_{\cx}\coprod_{I}M_{i}$ is an isomorphism.

$4)\Rightarrow 6)$ Notice that $\cx^{\bot}$ is the kernel of $\tau_{\cx}$. Let $Y_{i}\ko i\in I$, be a family of objects of $\cx^{\bot}$ whose coproduct exists in $\cd$. Then $\tau_{\cx}\coprod_{I}Y_{i}\cong\coprod_{I}\tau_{\cx}Y_{i}=0$, and so $\coprod_{I}Y_{i}$ belongs to $\cx^{\bot}$.

Assume from now on that $\cd$ has small coproducts.

$6)\Rightarrow 2)$ We just have to prove that the inclusion functor $\cx^{\bot}\ra\cd$ preserves small coproducts. Let $Y_{i}\ko i\in I$, be a family of objects of $\cx^{\bot}$ and consider its coproduct $\coprod_{I}Y_{i}$ in $\cd$. Since $\cx^{\bot}$ is closed under small coproducts, then $\coprod_{I}Y_{i}$ belongs to $\cx^{\bot}$, and so it is the coproduct in $\cx^{\bot}$ of the family $Y_{i}\ko i\in I$.
\end{proof}

\begin{definition}
Let $\cd$ be a triangulated category with small coproducts. A full triangulated subcategory of $\cd$ is \emph{smashing}\index{subcategory!smashing} if it satisfies the (equivalent) conditions of Proposition \ref{characterization of smashing}.
\end{definition}

\begin{definition}\label{special objects}
An object $P$ of a triangulated category $\cd$ is \emph{perfect}\index{object!perfect} (respectively, \emph{superperfect}\index{object!superperfect}) if for every countable (respectively, small) family of morphisms $M_{i}\ra N_{i}\ko i\in I$, of $\cd$ such that the coproducts $\coprod_{I}M_{i}$ and $\coprod_{I}N_{i}$ exist, the induced map
\[\cd(P,\coprod_{I}M_{i})\ra\cd(P,\coprod_{I}N_{i})
\]
is surjective provided every map
\[\cd(P,M_{i})\ra\cd(P,N_{i})\ko i\in I
\]
is surjective. Particular cases of superperfect objects are \emph{compact}\index{object!compact} objects, \ie objects $P$ such that $\cd(P,?)$ preserves small coproducts.
\end{definition}

\begin{remark}\label{superperfect closed coproducts}
The class of (super)perfect objects is closed under small coproducts, whereas the subclass of compact objects is not. Indeed, any algebra $A$ is compact in its derived category $\cd A$, but this is not the case for the free $A$-module $A^{(I)}$ over an infinite set $I$ (\cf the proofs of \cite[Lemma 2.1, Proposition 6.3]{Rickard1989}).
\end{remark}

\begin{definition}\label{compactly generated}
A triangulated category with small coproducts is \emph{perfectly}\index{category!triangulated! perfectly generated} (respectively, \emph{superperfectly}\index{category!triangulated! superperfectly generated}, \emph{compactly}\index{category!triangulated! compactly generated}) \emph{generated} if it is generated by a set of perfect (respectively, superperfect, compact) objects. A triangulated TTF triple $(\cx,\cy,\cz)$ on a triangulated category with small coproducts is \emph{perfectly}\index{TTF triple!triangulated!perfectly generated} (respectively, \emph{superperfectly}\index{TTF triple!triangulated!superperfectly generated}, \emph{compactly}\index{TTF triple!triangulated!compactly generated}) \emph{generated} if so is $\cx$ as a triangulated category. 
\end{definition}

Let us prove a very useful lemma under condition 2) (slightly reformulated) of Proposition \ref{characterization of smashing}:

\begin{lemma}\label{truncating special objects and from local to global via smashing}
Let $(\cx,\cy)$ be a t-structure on a triangulated category $\cd$. Assume that $\cx$ is triangulated and that the inclusion functor $y:\cy\hookrightarrow\cd$ preserves small coproducts. The following assertions hold:
\begin{enumerate}[1)]
\item If $M$ is a perfect (respectively, superperfect, compact) object of $\cd$, then $\tau^{\cy}M$ is a perfect (respectively, superperfect, compact) object of $\cy$.
\item If $M$ is a perfect (respectively, superperfect, compact) object of $\cx$, then $M$ is perfect (respectively, superperfect, compact) regarded as an object of $\cd$.
\end{enumerate}
\end{lemma}
\begin{proof}
1) Assume $M$ is compact in $\cd$ and let $Y_{i}\ko i\in I$, be a family of objects of $\cy$ whose coproduct exists in $\cy$. The following commutative diagram finishes the proof
\[\xymatrix{\coprod_{i\in I}\cy(\tau^{\cy}M,Y_{i})\ar[d]^{\can}\ar[r]^{\sim} & \coprod_{i\in I}\cd(M,Y_{i})\ar[d]_{\wr}^{\can} \\
\cy(\tau^{\cy}M,\coprod_{i\in I}Y_{i})\ar[r]^{\sim} &\cd(M,\coprod_{i\in I}Y_{i})
}
\] 
In case $M$ is (super)perfect we proceed similarly.

2) For simplicity, we will prove the statement for a compact object. The case of a (super)perfect object is proved similarly.
Let $M_{i}\ko i\in I$, be a family of objects of $\cd$ whose coproduct $\coprod_{I}M_{i}$ exists in $\cd$. Since both $\tau^{\cy}$ and $y$ preserve small coproducts, the coproduct $\coprod_{I}y\tau^{\cy}M_{i}$ exists in $\cd$ and the canonical morphism $\coprod_{I}y\tau^{\cy}M_{i}\ra y\tau^{\cy}\coprod_{I}M_{i}$ is an isomorphism. Lemma \ref{coproducts of triangles} says that the coproduct $\coprod_{I}x\tau_{\cx}M_{i}$ exists in $\cd$ and that the sequence 
\[\coprod_{I}x\tau_{\cx}M_{i}\ra\coprod_{I}M_{i}\ra\coprod_{I}y\tau^{\cy}M_{i}\ra(\coprod_{I}x\tau_{\cx}M_{i})[1]
\]
is a triangle of $\cd$. Moreover, we have a canonical isomorphism $\coprod_{I}x\tau_{\cx}M_{i}\arr{\sim}x\tau_{\cx}\coprod_{I}M_{i}$. Since $\coprod_{I}y\tau^{\cy}M_{i}\arr{\sim}y\tau^{\cy}\coprod_{I}M_{i}\in\cx^{\bot}$, then we have a commutative square
\[\xymatrix{\coprod_{I}\cx(M,\tau_{\cx}M_{i})\ar[r]^{\sim}\ar[d]_{\wr}^{\can} & \coprod_{I}\cd(M,M_{i})\ar[d]^{\can} \\
\cx(M,\coprod_{I}\tau_{\cx}M_{i})\ar[r]^{\sim} & \cd(M,\coprod_{I}M_{i})
}
\]
which proves that $M$ is compact in $\cd$.
\end{proof}

Now we introduce a crucial construction which formally imitates the construction of the direct limit in an abelian category.

\begin{definition}\label{Milnor colimit}
Let $\cd$ be a triangulated category and let
\[M_{0}\arr{f_{0}}M_{1}\arr{f_{1}}M_{2}\arr{f_{2}}\dots
\]
be a sequence of morphisms of $\cd$ such that the coproduct $\coprod_{n\geq 0}M_{n}$ exists in $\cd$. The \emph{Milnor colimit}\index{Milnor!colimit} of this sequence, denoted by $\Mcolim M_{n}$, is given, up to non-unique isomorphism, by the triangle
\[\coprod_{n\geq 0}M_{n}\arr{\id-\sigma}\coprod_{n\geq 0}M_{n}\arr{\pi} \Mcolim M_{n}\ra\coprod_{n\geq 0}M_{n}[1],
\]
where the morphism $\sigma$ has components
\[M_{n}\arr{f_{n}} M_{n+1}\arr{\can}\coprod_{p\geq 0}M_{p}.
\]
The above triangle is the \emph{Milnor triangle}\index{Milnor!triangle} (\cf \cite{Milnor1962, Keller1998b}) associated to the sequence $f_{n}\ko n\geq 0$. The notion of Milnor colimit has appeared in the literature under the name of \emph{homotopy colimit}\index{colimit!homotopy} (\cf \cite[Definition 2.1]{BokstedtNeeman1993}, \cite[Definition 1.6.4]{Neeman2001}) and \emph{homotopy limit}\index{limit!homotopy} (\cf \cite[subsection 5.1]{Keller1994a}). However, we think it is better to keep this terminology for the notions appearing in the theory of derivators \cite{Maltsiniotis2001, Maltsiniotis2005, CisinskiNeeman2005} and in the theory of model categories \cite{Hirschhorn2003}.
\end{definition}

Now let us prove that compact objects transform Milnor colimits into true direct limits:

\begin{lemma}\label{compact and Milnor}
Let $\cd$ be a triangulated category, $P$ a compact object of $\cd$ and
\[M_{0}\arr{f_{0}}M_{1}\arr{f_{1}}M_{2}\arr{f_{2}}\dots
\]
a sequence of morphisms of $\cd$ such that the coproduct $\coprod_{n\geq 0}M_{n}$ exists in $\cd$. Then we have a natural isomorphism of short exact sequences
\[\xymatrix{0\ar[r] & \coprod_{n\geq 0}\cd(P,M_{n})\ar[d]^{\wr}\ar[r]^{\id-s} &  \coprod_{n\geq 0}\cd(P,M_{n})\ar[d]^{\wr}\ar[r] & \lid\cd(P,M_{n})\ar[r]\ar[d]^{\wr} & 0 \\
0\ar[r] & \cd(P,\coprod_{n\geq 0}M_{n})\ar[r]^{(\id-\sigma)^{\we}} & \cd(P,\coprod_{n\geq 0}M_{n})\ar[r] & \cd(P,\Mcolim M_{n})\ar[r] & 0
}
\]
In particular, for every morphism $g\in\cd(P,\Mcolim M_{n})$ there exists a natural number $N\geq 0$ and a morphism $g_{N}\in\cd(P,M_{N})$ such that $g$ factors through $g_{N}$ via the $N$th component $\pi_{N}$ of $\pi:\coprod_{n\geq 0}M_{n}\ra\Mcolim M_{n}$:
\[\xymatrix{P\ar[rr]^{g}\ar[dr]_{g_{N}} && \Mcolim M_{n} \\
& M_{N}\ar[ur]_{\pi_{N}}
}
\]
\end{lemma}
\begin{proof}
Compactness of $P$ implies the existence of a natural isomorphism
\[\xymatrix{ \coprod_{n\geq 0}\cd(P,M_{n}[i])\ar[d]^{\wr}\ar[rr]^{\id-s(i)} &&\coprod_{n\geq 0}\cd(P,M_{n}[i])\ar[d]^{\wr} \\
\cd(P,\coprod_{n\geq 0}M_{n}[i])\ar[rr]^{(\id-\sigma[i])^{\we}} && \cd(P,\coprod_{n\geq 0}M_{n}[i])
}
\]
for all $i\in\Z$, where $s(i)$ is the morphism with components
\[\cd(P,M_{n}[i])\arr{(f_{n}[i])^{\we}}\cd(P,M_{n+1}[i])\ra\coprod_{m\geq 0}\cd(P,M_{m}[i]).
\]
Since $\id-s(i)$ is injective, we get the following commutative diagram
\[\xymatrix{0\ar[r] & \coprod_{n\geq 0}\cd(P,M_{n})\ar[d]^{\wr}\ar[r]^{\id-s} &  \coprod_{n\geq 0}\cd(P,M_{n})\ar[d]^{\wr}\ar[r] & \lid\cd(P,M_{n})\ar[r] & 0 \\
0\ar[r] & \cd(P,\coprod_{n\geq 0}M_{n})\ar[r]^{(\id-\sigma)^{\we}} & \cd(P,\coprod_{n\geq 0}M_{n})\ar[r] & \cd(P,\Mcolim M_{n})\ar[r] & 0
}
\]
with exact rows. The universal property of the cokernel $\lid\cd(P,M_{n})$ of $\id-s$ induces the desired isomorphism. Finally, given a morphism $g:P\ra\Mcolim M_{n}$, we know that there exists a family $\{g_{n}\}_{n\geq 0}\in\coprod_{n\geq 0}\cd(P,M_{n})$ mapped to $g$. That is to say, if $N$ is the greatest index such that $g_{N}\neq 0$, then 
\[g=\sum_{n=0}^{N}\pi_{n}g_{n}=\pi_{N}(f_{N-1}\dots f_{0}g_{0}+ f_{N-1}\dots f_{1}g_{1}+\dots +g_{N}).
\]
\end{proof}

\begin{definition}\label{satisfying BRT}
Let $\cd$ be a triangulated category. A contravariant functor $H:\cd\ra\Mod k$ is \emph{cohomological}\index{functor!cohomological} if it maps a triangle
\[L\arr{f}M\arr{g}N
\]
to an exact sequence
\[H(N)\arr{H(g)}H(M)\arr{H(f)}H(L).
\]
We say that $\cd$ \emph{satisfies the Brown's representability theorem}\index{category!triangulated!satisfies the Brown's representability theorem} if every cohomological functor taking small coproducts to small products is representable.
\end{definition}

The following very useful lemma is known as ``The adjoint functor argument''\index{adjoint functor argument}:
\begin{lemma}\label{adjoint functor argument}
If $\cd$ is a triangulated category and $\cd'$ a full triangulated subcategory of $\cd$ such that the inclusion functor $\cd'\hookrightarrow \cd$ preserves small coproducts and $\cd'$ satisfies the Brown representability theorem, then $\cd'$ is an aisle in $\cd$.
\end{lemma}
\begin{proof}
If 
\[\iota: \cd'\hookrightarrow\cd
\]
is the inclusion functor, for an object $M$ of $\cd$ we define the functor
\[H(?):=\cd(\iota(?),M):\cd'\ra\Mod k,
\]
which is cohomological and takes small coproducts to small products. Then, by assumption this functor is represented by an object, say $\tau(M)\in\cd'$. By Yoneda's lemma it turns out that the map $M\mapsto\tau(M)$ underlies a functor $\tau:\cd\ra\cd'$ which is right adjoint to $\iota$. Thus $\cd'$ is an aisle in $\cd$.
\end{proof}

Now we present the so-called ``Brown's representability theorem for cohomology'', which is a seminal result in the theory of triangulated categories . One of the main ingredients of the proof is the algebraic analogue of the topological \emph{iterated attaching of cells}. This theorem was independently proved by B.~Keller \cite[Theorem 5.2]{Keller1994a} and A.~Neeman \cite{Neeman1992, Neeman1996} for the case of compactly generated triangulated categories. Then, A. Neeman \cite[Theorem 8.3.3]{Neeman2001} proved the theorem for the more general case of $\aleph_{1}$-\emph{perfectly generated} triangulated categories (\cf \cite[Definition 8.1.2]{Neeman2001}), motivated by V. Voevodsky's work on motivic cohomology. From this, he deduced the theorem for the important class of \emph{well-generated} triangulated categories (\cf \cite[Definition 1.15 and Definition 8.1.6]{Neeman2001}, see also the characterization due to H. Krause \cite{Krause2001}), which are, in particular, perfectly generated. Finally, H. Krause \cite[Theorem A]{Krause2002} gave a short proof for the case of perfectly generated triangulated categories. This theorem deals with the representability of certain functors. From it, via the adjoint functor argument, one can deduce the existence of a t-structure which implies that,  under certain hypotheses, `generators' implies `exhaustivity'. We present now H. Krause's version.

\begin{theorem}\label{Krause on perfects}
Let $\cd$ be a triangulated category with small coproducts and let $\cp$ be a set of objects of $\cd$ which are perfect in $\Tria(\cp)$. Then, $\Tria(\cp)$ satisfies the Brown representability theorem (which implies that it is an aisle in $\cd$) and every object of $\Tria(\cp)$ is the Milnor colimit of a sequence
\[P_{0}\arr{f_{0}}P_{1}\arr{f_{1}}P_{2}\arr{f_{2}}\dots,
\]
where $P_{n}$ is an $n$th extension of small coproducts of shifts of objects of $\cp$. In particular, if $\cp$ generates $\cd$ then $\Tria(\cp)=\cd$.
\end{theorem}
\begin{proof}
Let $H:\Tria(\cp)\ra\Mod\Z$ be a cohomological functor which takes small coproducts to small products. It is proved in \cite[Theorem A]{Krause2002} that there exists a sequence
\[P_{0}\arr{f_{0}}P_{1}\arr{f_{1}}P_{2}\arr{f_{2}}\dots,
\]
such that
\begin{enumerate}[$\bullet$]
\item each $P_{n}$ is an $n$th extension of small coproducts of shifts of objects of $\cp$, 
\item if we denote by $P_{\infty}$ the Milnor colimit of the sequence, there exists an isomorphism of functors
\[\pi_{\infty}:\Tria(\cp)(?,P_{\infty})\arr{\sim} H.
\]
\end{enumerate}
This proves that $\Tria(\cp)$ satisfies the Brown's representability theorem. Moreover, if $M$ is an object of $\Tria(\cp)$, we can take $H=\Tria(\cp)(?,M)$ and then by Yoneda's lemma we get an isomorphism $P_{\infty}\arr{\sim}M$. This gives a description of all the objects of $\Tria(\cp)$ which shows that it is exhaustively generated by all the shifts of objects of $\cp$. The last statement follows from Lemma \ref{generators and t-structures}.
\end{proof}

\begin{proposition}\label{smashing are TTF}
If $\cd$ is a perfectly generated triangulated category, then the map
\[\cx\mapsto (\cx,\cx^{\bot},\cx^{\bot\bot})
\]
is a one-to-one correspondence between smashing subcategories of $\cd$ and triangulated TTF triples on $\cd$. 
\end{proposition}
\begin{proof}
Indeed, if $(\cx,\cy,\cz)$ is a triangulated TTF triple, then $\cx$ is a smashing subcategory since $\cy$ being an aisle is always closed under small coproducts. Conversely, if $\cx$ is a smashing subcategory, then $(\cx,\cy)$ is a t-structure on $\cd$. But now, by using the Lemma \ref{generators and t-structures} and Lemma \ref{truncating special objects and from local to global via smashing} we have that $\tau^{\cy}$ takes the set of perfect generators of $\cd$ to a set of perfect generators $\cy$. Therefore, $\cy$ is a perfectly generated triangulated category closed under small coproducts in $\cd$, and by Theorem \ref{Krause on perfects} together with the adjoint functor argument (see Lemma \ref{adjoint functor argument}) we conclude that $\cy$ is an aisle.
\end{proof}

We recall in the following result some of B.~Keller's elegant techniques \cite[Lemma 4.2]{Keller1994a} to recognize fully faithful triangle functors or triangle equivalences:

\begin{lemma}\label{detecting triangle equivalences}
Let $\cd$ and $\cd'$ be triangulated categories, $F:\cd\ra\cd'$ a triangle functor which commutes with small coproducts and $\cq$ a set of compact objects of $\cd$. 
\begin{enumerate}[1)]
\item Assume $\cd$ satisfies the principle of infinite d\'{e}vissage with respect to $\cq$. If $FQ$ is compact in $\cd'$ for each $Q\in\cq$ and $F$ induces an isomorphism
\[\cd(Q,P[n])\arr{\sim}\cd'(FQ,FP[n])
\]
for each $P\ko Q\in\cq$ and $n\in\Z$, then $F$ is fully faithful.
\item Assume $\cd$ is exhaustively generated by $\cq$. Then $F$ is a triangle equivalence if and only if
\begin{enumerate}[2.1)]
\item $FP$ is compact in $\cd'$ for each $P\in\cq$.
\item $F$ induces an isomorphism
\[\cd(Q,P[n])\arr{\sim}\cd'(FQ,FP[n])
\]
for each $P\ko Q\in\cq$ and $n\in\Z$.
\item $\cd'$ is exhaustively generated by the objects $FP\ko P\in\cq$.
\end{enumerate}
\end{enumerate}
\end{lemma}
\begin{proof}
1) Let $\cu$ be the full subcategory of $\cd$ formed by the objects $N$ such that $F$ induces and isomorphism
\[\cd(Q[n],N)\arr{\sim}\cd'(FQ[n],FN)
\]
for each $Q\in\cq\ko n\in\Z$. It is a full triangulated subcategory of $\cd$ closed under small coproducts and containing $\cq$. Hence $\cu=\cd$. Fix now an object $N\in\cd$ and consider the full subcategory $\cv$ of $\cd$ formed by the objects $M$ such that $F$ induces an isomorphism
\[\cd(M,N[n])\arr{\sim}\cd'(FM,FN[n])
\]
for each $n\in\Z$. Again, it is a full triangulated subcategory of $\cd$ closed under small coproducts and containing $\cq$ and so $\cv=\cd$.

2) It is clear that $F$ satisfies 2.1)--2.3) provided it is a triangle equivalence. Conversely, thanks to 1), we know that 2.1) and 2.2) imply that $F$ is fully faithful. Hence, its essential image is a full triangulated subcategory of $\cd'$ containing the objects $FQ\ko Q\in\cq$ and closed under small coproducts of objects of $\bigcup_{n\geq 0}\Sum(\{FQ\}_{Q\in\cq})^{*n}$. Therefore, condition 2.3), together with Lemma \ref{e implies d implies g}, implies that $F$ is essentially surjective.
\end{proof}

\begin{proposition}
Let $\cd$ be a triangulated category with small coproducts and let $P$ be an object of $\cd$. The following conditions are equivalent:
\begin{enumerate}[1)]
\item $P$ is compact in $\cd$.
\item $P$ satisfies:
\begin{enumerate}[2.1)]
\item $P$ is perfect in $\cd$.
\item $P$ is compact in the full subcategory $\mbox{Sum}(\{P[n]\}_{n\in\Z})$ of $\cd$ formed by small coproducts of shifts of $P$.
\item $\Tria(P)^{\bot}$ is closed under small coproducts.
\end{enumerate}
\item $P$ satisfies:
\begin{enumerate}[3.1)]
\item $P$ is compact in $\Tria(P)$.
\item $\Tria(P)^{\bot}$ is closed under small coproducts.
\end{enumerate}
\item $P$ satisfies:
\begin{enumerate}[4.1)]
\item $P$ is superperfect in $\cd$.
\item $P$ is compact in $\Sum(\{P[n]\}_{n\in\Z})$.
\end{enumerate}
\end{enumerate}
\end{proposition}
\begin{proof}
Of course, assertion 1) implies any of the others.

$2)\Rightarrow 1)$ If $P$ is perfect in $\cd$, by Theorem A of \cite{Krause2002} and the adjoint functor argument (\cf Lemma \ref{adjoint functor argument}) we know that $\cx=\Tria(P)$ is an aisle in $\cd$. Moreover, assumption 2.3) says that $\cx$ is a smashing subcategory. Thanks to Lemma \ref{truncating special objects and from local to global via smashing}, it suffices to prove that $P$ is compact in $\cx$. For this, we will use the following facts:
\begin{enumerate}[a)]
\item Every object of $\cx$ is the Milnor colimit of a sequence 
\[X_{0}\arr{f_{0}}X_{1}\arr{f_{1}}X_{2}\arr{f_{2}}\dots
\]
where $X_{0}$ as well as each mapping cone $\cone(f_{n})$ is a small coproducts of shifts of $P$.
\item Thanks to the proof of \cite[Theorem A]{Krause2002} (\cf also the proof of \cite[Theorem 2.2]{Souto2004}), if $\{X_{n}, f_{n}\}_{n\geq 0}$ is a direct system as in a) we know that the canonical morphism
\[\lid\cd(P,X_{n})\ra\cd(P,\Mcolim X_{n})
\]
is an isomorphism.
\item Hypothesis 2.2) implies that, for any fixed natural number $m\geq 0$, the functor $\cd(P,?)$ preserves small coproducts of objects of $\mbox{Sum}(\{P[n]\}_{n\in\Z})^{*m}$.
\end{enumerate}
Let $\Mcolim X^{i}_{n}\ko i\in I$, be an arbitrary family of objects of $\cx$. Then the canonical morphism
\[\coprod_{i\in I}\cx(P,\Mcolim X^{i}_{n})\ra\cx(P,\coprod_{i\in I}\Mcolim X^{i}_{n})
\]
is the composition of the following canonical isomorphisms
\begin{align}
\coprod_{i\in I}\cx(P,\Mcolim X^{i}_{n})\cong\coprod_{i\in I}\lid\cx(P,X^{i}_{n})\cong\lid\coprod_{i\in I}\cx(P,X^{i}_{n})\cong \nonumber \\
\cong\lid\cx(P,\coprod_{i\in I}X^{i}_{n})\cong \cx(P,\Mcolim \coprod_{i\in I}X^{i}_{n})\cong\cx(P,\coprod_{i\in I}\Mcolim X^{i}_{n}) \nonumber
\end{align}
Hence, $P$ is compact in the smashing subcategory $\cx$.

$3)\Rightarrow 1)$ By the adjoint functor argument we know that $\Tria(P)$ is an aisle in $\cd$, and condition 3.2) ensures that, moreover, it is a smashing subcategory. Hence Lemma \ref{truncating special objects and from local to global via smashing} finishes the proof.

$4)\Rightarrow 1)$ Of course, 4) implies 2), and 2) implies 1). However, there exists a shorter proof pointed out by B.~Keller: let $M_{i}\ko i\in I$, be a family of objects of $\cd$ and take, for each index $i\in I$, an object $Q_{i}\in\Sum(\{P[n]\}_{n\in\Z})$ together with a morphism $Q_{i}\ra M_{i}$ such that the induced morphism
\[\cd(P,Q_{i})\ra\cd(P,M_{i})
\]
is surjective. Since $P$ is superperfect, this implies that the morphism 
\[\cd(P,\coprod_{I}Q_{i})\ra\cd(P,\coprod_{I}M_{i}).
\]
is surjective. Now consider the commutative square
\[\xymatrix{\cd(P,\coprod_{I}Q_{i})\ar[r] &\cd(P,\coprod_{I}M_{i}) \\
\coprod_{I}\cd(P,Q_{i})\ar[u]_{\can}^{\wr}\ar[r] &\coprod_{I}\cd(P,M_{i})\ar[u]_{\can}
}
\]
The first vertical arrow is an isomorphism by assumption 4.2, and the horizontal arrows are surjections. Then, the second vertical arrow is surjective. But, of course, it is also injective, and so bijective.
\end{proof}

\section{B. Keller's Morita theory for derived categories}\label{B. Keller's Morita theory for derived categories}
\addcontentsline{lot}{section}{4.5. La teor\'ia de Morita para categor\'ias derivadas de B. Keller}

Let $\ca$ be a small dg category (\cf \cite{Keller1994a, Keller2006b}). It was proved by B.~Keller \cite{Keller1994a} that its derived category $\cd\ca$\index{$\cd\ca$} is a triangulated category compactly generated by the \emph{representable modules}\index{module!representable} $A^{\we}:=\ca(?,A)$ induced by the objects $A$ of $\ca$. Conversely, he also proved \cite[Theorem 4.3]{Keller1994a} that every algebraic triangulated category with small coproducts and with a set $\cp$ of compact generators is the derived category of a certain dg category whose set of objects is equipotent to $\cp$. 

The proof of Theorem 4.3 of \cite{Keller1994a} has two parts. In the first part, he proves that every algebraic triangulated category admits an enhancement, \ie comes from an exact dg category. 

\begin{definition}\label{exact dg}
According with \cite{Keller1999, Keller2006b}, we say that a dg category $\ca'$ is \emph{exact}\index{category!exact dg} if the image of the (fully faithful) Yoneda functor
\[\Zy 0\ca'\ra \cc\ca'\ko M\mapsto M^{\we}:=\ca'(?,M)
\]
is stable under shifts and extensions (in the sense of the exact structure on $\cc\ca$ in which the conflations are the degreewise split short exact sequences).
\end{definition}

Equivalently (\cf Lemma \cite[2.1]{Keller1999}), for all objects $M\ko N$ of $\ca$ and all integers $n$, the object $M^{\we}[n]$ is isomorphic to $M[n]^{\we}$ and the cone $\cone(f^{\we})$ over a morphism $f^{\we}: M^{\we}\ra N^{\we}$ is isomorphic to $\cone(f)^{\we}$ for unique objects $M[n]$ and $\cone(f)$ of $\Zy 0(\ca)$.

 If $\ca'$ is an exact dg category, then $\Zy 0\ca'$ becomes a Frobenius category and $\ul{\Zy 0\ca'}=\H 0\ca'$ is a full triangulated subcategory of $\ch\ca'$. B.~Keller has shown \cite[Example 2.2.c)]{Keller1999} that if $\cc$ is a Frobenius category with class of conflations $\ce$, then $\ul{\cc}=\H 0\ca'$ for the exact dg category $\ca'$ formed by the acyclic complexes with $\ce$-projective-injective components over $\cc$.

In the second part, he proves the following. 

\begin{proposition}
Let $\ca'$ be an exact dg category such that the associated triangulated category $\H 0\ca'$ is compactly generated by a set $\cb$ of objects. Notice that $\cb$ inherits a structure of (small) dg category, regarded as a full subcategory of $\ca'$. Then, the map
\[M\mapsto M^{\we}_{\ \ |_{\cb}}:=\ca'(?,M)_{|_{\cb}}
\]
induces a triangle equivalence
\[\H 0\ca'\arr{\sim}\cd\cb.
\]
\end{proposition}

The dg category associated to the Frobenius category in the first part of the proof of \cite[Theorem 4.3]{Keller1994a} is not very handy. However, many times in practice we are already like in the second step of the proof, which allows us a better choice of the dg category. In what follows, we will recall how this better choice can be made.

Recall that for a dg category $\ca$ there exists an exact dg category $\cc_{dg}\ca$\index{$\cc_{dg}\ca$} \cite{Keller2006b} such that the category of right dg $\ca$-modules $\cc\ca$\index{$\cc\ca$} is $\Zy 0(\cc_{dg}\ca)$ and the category of dg modules up to homotopy $\ch\ca$ is $\H 0(\cc_{dg}\ca)$. Recall also that there is a triangulated TTF triple $(\ch_{\bf{p}}\ca\ko \cn\ko\ch_{\bf{i}}\ca)$ on $\ch\ca$ such that $\cn$ is the full subcategory of acyclic objects. $\ch_{\bf{p}}\ca$\index{$\ch_{\bf{p}}\ca$} is the category of the \emph{$\ch$-projective}\index{module!$\ch$-projective} or \emph{cofibrant}\index{module!cofibrant} modules and $\ch_{\bf{i}}\ca$\index{$\ch_{\bf{i}}\ca$} is the category of the \emph{$\ch$-injective}\index{module!$\ch$-injective} or \emph{fibrant}\index{module!fibrant} modules. Notice that the composition
\[\ch_{\bf{i}}\ca\hookrightarrow\ch\ca\ra\ch\ca/\cn=\cd\ca
\]
of the projection with the inclusion is a triangle equivalence. A quasi-inverse will be denoted by
\[\bf{i}:\cd\ca\arr{\sim}\ch_{\bf{i}}\ca
\]
and will be called the \emph{$\ch$-injective resolution}\index{functor!$\ch$-injective resolution}\index{$\bf{i}$} (or \emph{fibrant resolution}\index{functor!fibrant resolution}) functor. Similarly, we have the \emph{$\ch$-projective}\index{functor!$\ch$-projective resolution}\index{$\bf{p}$} (or \emph{cofibrant resolution}\index{functor!cofibrant resolution}) functor
\[\bf{p}:\cd\ca\arr{\sim}\ch_{\bf{p}}\ca.
\]
Let $\cp$ be a set of objects of $\cd\ca$ and define $\cb$ as the dg subcategory of $\cc_{dg}\ca$ formed by the $\ch$-injective resolutions $\textbf{i}P$ of the modules $P$ of $\cp$. Then we have a dg $\cb$-$\ca$-bimodule $X$, defined by
\[X(A,B):=B(A)
\]
for $A$ in $\ca$ and for $B$ in $\cb$, and a pair $(?\otimes_{\cb}X,\ch om_{\ca}(X,?))$ of adjoint dg functors
\[\xymatrix{\cc_{dg}\ca\ar@<1ex>[d]^{\ch om_{\ca}(X,?)} \\
\cc_{dg}\cb\ar@<1ex>[u]^{?\otimes_{\cb}X}
}
\]
For instance, $\ch om_{\ca}(X,?)$\index{$\ch om_{\ca}(X,?)$} is defined by $\ch om_{\ca}(X,M):=(\cc_{dg}\ca)(?,M)_{|_{\cb}}$ for $M$ in $\cc_{dg}\ca$. These functors induce a pair of adjoint triangle functors between the corresponding \emph{categories up to homotopy}\index{category!up to homotopy} \cite{Keller1994a, Keller2006b}
\[\xymatrix{\ch\ca\ar@<1ex>[d]^{\ch om_{\ca}(X,?)} \\
\ch\cb\ar@<1ex>[u]^{?\otimes_{\cb}X}.
}
\]
The \emph{total right derived functor}\index{functor!total right derived} $\textbf{R}\ch om_{\ca}(X,?)$\index{$\textbf{R}\ch om_{\ca}(X,?)$} is the composition
\[\cd\ca\arr{\bf i}\ch_{\bf i}\ca\hookrightarrow\ch\ca\arr{\ch om_{\ca}(X,?)}\ch\cb\ra\cd\cb,
\]
and the \emph{total left derived functor}\index{functor!total left derived} $?\otimes^{\textbf{L}}_{\cb}X$\index{$?\otimes^{\textbf{L}}_{\cb}X$} is the composition
\[\cd\cb\arr{\bf p}\ch_{\bf p}\ca\hookrightarrow\ch\cb\arr{?\otimes_{\cb}X}\ch\ca\ra\cd\ca.
\]
They form a pair of adjoint triangle functors at the level of derived categories
\[\xymatrix{\cd\ca\ar@<1ex>[d]^{\textbf{R}\ch om_{\ca}(X,?)} \\
\cd\cb\ar@<1ex>[u]^{?\otimes_{\cb}^\textbf{L}X}
}
\]

The following is an easy consequence of Proposition \ref{B. Keller's Morita theory for derived categories}.

\begin{corollary}\label{remark Keller}
Assume that the objects of $\cp$ are compact in the full triangulated subcategory $\Tria(\cp)$ of $\cd\ca$. Then:
\begin{enumerate}[1)]
\item the functors $(?\otimes_{\cb}^\textbf{L}X,\textbf{R}\ch om_{\ca}(X,?))$ induce mutually quasi-inverse triangle equivalences
\[\xymatrix{\Tria(\cp)\ar@<1ex>[rr]^{\textbf{R}\ch om_{\ca}(X,?)} && \cd\cb\ar@<1ex>[ll]^{?\otimes_{\cb}^\textbf{L}X}
}
\]
which gives a bijection between the objects of $\cp$ and the representable dg $\cb$-modules,
\item $\Tria(\cp)$ is an aisle in $\cd\ca$ and, for each right dg $\ca$-module $M$, the triangle of $\cd\ca$ associated to the t-structure $(\Tria(\cp),\Tria(\cp)^{^\bot})$ is 
\[\textbf{R}\ch om_{\ca}(X,M)\otimes^{\bf L}_{\cb}X\arr{\delta_{M}}M\ra M'\ra (\textbf{R}\ch om_{\ca}(X,M)\otimes^{\bf L}_{\cb}X)[1],
\]
where $\delta$ is the counit of the adjunction $(?\otimes_{\cb}^\textbf{L}X,\textbf{R}\ch om_{\ca}(X,?))$.
\end{enumerate}
\end{corollary}
\begin{proof}
1) Since the $\ch$-injective resolution functor $\bf i:\cd\ca\arr{\sim}\ch_{\bf i}\ca$ is a triangle equivalence, it induces a triangle equivalence between $\Tria(\cp)$ and a certain full triangulated subcategory of $\ch_{\bf i}\ca$. This subcategory is algebraic, \ie the stable category $\ul{\cc}$ of a certain Frobenius category $\cc$, because it is a subcategory of $\ch\ca$.  Since $\ch\ca$ is the $\H 0$-category of the exact dg category $\cc_{dg}\ca$, then $\ul{\cc}$ is the $\H 0$-category of an exact dg subcategory of $\cc_{dg}\ca$. Moreover,  $\ul{\cc}$ is compactly generated by the objects of $\cb$. Then, by Proposition \ref{B. Keller's Morita theory for derived categories} the restriction of $\ch om_{\ca}(X,?)$ to $\ul{\cc}$ induces a triangle equivalence
\[\ul{\cc}\arr{\ch om_{\ca}(X,?)}\ch\cb\ra\cd\cb.
\]
The picture is:
\[\xymatrix{ & \H 0(\cc_{dg}\ca)=\ch\ca & \\
\cd\ca\ar[r]^{\bf i}_{\sim} & \ch_{\bf i}\ca\ar@{^(->}[u] & \\
\Tria(\cp)\ar[r]_{\sim}\ar@{^(->}[u] & \ul{\cc}\ar[r]_{\sim}\ar@{^(->}[u] & \cd\cb
}
\]
The composition of the bottom arrows is the restriction of $\textbf{R}\ch om_{\ca}(X,?)$ to $\Tria(\cp)$. Notice that by the adjoint functor argument (\cf Lemma \ref{adjoint functor argument}) $\Tria(\cp)$ is an aisle, and so $\Tria(\cp)=\ ^{\bot}(\Tria(\cp)^{\bot})$. Finally, by adjunction the image of $?\otimes_{\cb}^\textbf{L}X$ is in $\ ^{\bot}(\Tria(\cp)^{\bot})=\Tria(\cp)$. Alternatively, we can also prove that the image of $?\otimes_{\cb}^\textbf{L}X$ is contained in $\Tria(\cp)$ by using Lemma \ref{e implies d implies g}, and then use Theorem \ref{Krause on perfects} and Lemma \ref{detecting triangle equivalences} to show that
\[?\otimes^{\L}_{\cb}X:\cd\cb\ra\Tria(\cp)
\]
is an equivalence.

2) Since $?\otimes^\textbf{L}_{\cb}X:\cd\cb\ra\cd\ca$ is fully faithful, the unit $\eta$ of the adjunction $(?\otimes_{\cb}^\textbf{L}X,\textbf{R}\ch om_{\ca}(X,?))$ is an isomorphism, thanks to Lemma \ref{properties of adjunctions}. Then, when we apply the functor $?\otimes_{\cb}^\textbf{L}X\circ \textbf{R}\ch om_{\ca}(X,?)$ to the counit $\delta$ we get an isomorphism. This shows that for each module $M$ in $\cd\ca$, the triangle 
\[\textbf{R}\ch om_{\ca}(X,M)\otimes^{\bf L}_{\cb}X\arr{\delta_{M}}M\ra M'\ra (\textbf{R}\ch om_{\ca}(X,M)\otimes^{\bf L}_{\cb}X)[1]
\]
of $\cd\ca$ satisfies that $\textbf{R}\ch om_{\ca}(X,M')\otimes^{\bf L}_{\cb}X=0$. Since $?\otimes^{\bf L}_{\cb}X$ is fully faithful, then $\textbf{R}\ch om_{\ca}(X,M')=0$. That is to say,
\[\textbf{R}\ch om_{\ca}(X,M')(P)=(\cc_{dg}\ca)(P,\textbf{i}M')
\]
is acyclic for each object $P$ in $\cp$. But then, we have
\[\H n(\cc_{dg}\ca)(P,\textbf{i}M')=(\ch\ca)(P,\textbf{i}M'[n])\cong(\cd\ca)(P,M'[n])=0
\]
for each object $P$ of $\cp$ and each integer $n\in\Z$. This implies, by infinite d\'{e}vissage, that $M'$ belongs to the coaisle $\Tria(\cp)^{\bot}$ of $\Tria(\cp)$.
\end{proof}

\begin{remark}
This result generalizes \cite[Theorem 1.6]{Jorgensen2006} and \cite[Theorem 2.1]{DwyerGreenlees}. Indeed, if $\ca$ is the dg category associated to the dg $k$-algebra $A$ and the set $\cp$ has only one element $P$, then $\cb$ is the dg category associated to the dg algebra $B:=(\cc_{dg}A)(\textbf{i}P,\textbf{i}P)$, the dg $\cb$-$\ca$-bimodule $X$ corresponds to the dg $B$-$A$-bimodule $\textbf{i}P$ and the triangle equivalence
\[\textbf{R}\ch om_{\ca}(X,?):\Tria(P)\arr{\sim}\cd B
\]
is given by $\textbf{R}\Hom_{A}(\textbf{i}P,?)$. 
\end{remark}

\section{Some constructions of t-structures}\label{Some constructions of t-structures}
\addcontentsline{lot}{section}{4.6. Algunas construcciones de t-estructuras}

Theorem \ref{Krause on perfects} shows a way of constructing t-structures. Indeed, it says that a set of `special' objects of a triangulated category with small coproducts gives rise to a triangulated aisle. Now, we will show that sometimes we can weaken the conditions on the set of objects if we increase the requirements on the ambient triangulated category. This will be done by assuming an underlying model. First, we need some previous terminology concerning homotopical algebra \cite{Hirschhorn2003, Hovey1999}.

\begin{definition}\label{kappa filtered}
Let $\kappa$ be a cardinal. An ordinal $\lambda$ is \emph{$\kappa$-filtered}\index{ordinal!$\kappa$-filtered} if it is a limit ordinal and, if $\cs$ is a subset of $\lambda$ with cardinality $|\cs|\leq\kappa$, then  its supreme satisfies $\mbox{sup}(\cs)<\lambda$.
\end{definition}

\begin{definition}\label{lambda sequence}
Let $\cc$ be a category with small colimits and let $\lambda$ be an ordinal (regarded as a category). A \emph{$\lambda$-sequence}\index{sequence!$\lambda$-} in $\cc$ is a colimit-preserving functor $X:\lambda\ra\cc$, depicted by
\[X_{0}\ra X_{1}\ra\dots\ra X_{\beta}\ra\dots
\]
Since $X$ preserves colimits, for all limit ordinals $\gamma<\lambda$, the induced map
\[\text{colim}_{\beta<\gamma}X_{\beta}\ra X_{\gamma}
\]
is an isomorphism. We say that the map
\[X_{0}\ra\text{colim}_{\beta<\lambda}X_{\beta}
\]
is the \emph{composition}\index{composition of a $\lambda$-sequence}\index{sequence!$\lambda$-!composition} of the $\lambda$-sequence $X$. 
\end{definition}

\begin{definition}\label{smallness}
Let $\cc$ be a category with small colimits, $\ci$ a class of morphisms of $\cc$, $M$ an object of $\cc$ and $\kappa$ a cardinal. We say that $M$ is \emph{$\kappa$-small relative to $\ci$}\index{object!$\kappa$-small relative} if, for all $\kappa$-filtered ordinals $\lambda$ and all $\lambda$-sequences $X:\lambda\ra\cc$ such that each map $X_{\beta}\ra X_{\beta+1}$ is in $\ci$ with $\beta+1<\lambda$, the canonical map of sets
\[\text{colim}_{\beta<\lambda}\cc(M,X_{\beta})\ra\cc(M,\text{colim}_{\beta<\lambda}X_{\beta})
\]
is a bijection. We say that $M$ is \emph{small relative to $\ci$}\index{object!small relative} if it is $\kappa$-small relative to $\ci$ for some $\kappa$. We say that $M$ is \emph{small}\index{object!small} if it is small relative to the class of all the morphisms of $\cc$.
\end{definition}

\begin{definition}\label{cell inj proj}
Let $\cc$ be a category with small colimits and let $\ci$ be a set of morphisms in $\cc$. A \emph{relative $\ci$-cell complex}\index{relative $\ci$-cell complex}\index{complex!relative $\ci$-cell} is a transfinite composition of pushouts of elements of $\ci$. That is, if $f:M\ra N$ is a relative $\ci$-cell complex, then there is an ordinal $\lambda$ and a $\lambda$-sequence $X:\lambda\ra\cc$ such that $f$ is the composition of $X$ and such that, for each $\beta$ with $\beta+1<\lambda$, there is a cocartesian square as follows,
\[\xymatrix{Y_{\beta}\ar[r]\ar[d]^{g_{\beta}} & X_{\beta}\ar[d] \\
Z_{\beta}\ar[r] & X_{\beta+1}
}
\]
with $g_{\beta}$ in $\ci$. The class of relative $\ci$-cell complexes is denoted by $\ci$-cell\index{$\ci$-cell}. An object $M$ is an \emph{$\ci$-cell complex}\index{complex!$\ci$-cell} if the map from the initial object $0\ra M$ is a relative $\ci$-cell complex. We say that a morphism $f$ is \emph{$\ci$-injective}\index{morphism!$\ci$-injective} if it has the right lifting property with respect to every morphism in $\ci$, \ie if for every commutative diagram
\[\xymatrix{X\ar[r]^x\ar[d]_g & M\ar[d]^{f} \\
Y\ar[r]^y & N
}
\]
with $g$ in $\ci$ there exists a lifting $h:Y\ra M$ such that $hg=x$ and $fh=y$. The class of $\ci$-injective morphisms is denoted $\ci$-inj\index{$\ci$-inj}. Given a class $\ci'$ of morphisms of $\cc$, a morphism $g$ is \emph{$\ci'$-projective}\index{morphism!$\ci'$-projective} if it has the left lifting property with respect to every morphism in $\ci'$, \ie if for every commutative diagram
\[\xymatrix{X\ar[r]^x\ar[d]_g & M\ar[d]^{f} \\
Y\ar[r]^y & N
}
\]
with $f$ in $\ci'$ there exists a lifting $h:Y\ra M$ such that $hg=x$ and $fh=y$. The class of $\ci'$-projective morphisms is denoted $\ci'$-proj\index{$\ci'$-proj}. A morphism is an \emph{$\ci$-cofibration}\index{morphism!$\ci$-cofibration} if it is in ($\ci$-inj)-proj. The class of $\ci$-cofibrant morphisms is denoted $\ci$-cof\index{$\ci$-cof}. 
\end{definition}

\begin{lemma}\label{cell is cof}
With the notation of Definition \ref{cell inj proj}, we have that $\ci$-cell $\subseteq\ci$-cof, \ie every relative $\ci$-cell complex is an $\ci$-cofibration.
\end{lemma}
\begin{proof}
\emph{Cf.} \cite[Lemma 2.1.10]{Hovey1999}.
\end{proof}

\begin{notation}\label{I Q}
Let $\cc$ be a Frobenius category with small colimits. For a set $\cq$ of objects of $\cc$, we define $\ci_{\cq}$\index{$\ci_{\cq}$} to be the set formed by the inflations $i_{Q}:Q\ra IQ$ where $Q$ runs through $\cq$. 
\end{notation}

The following theorem appears in \cite{Nicolas2007b} and shows that under mild smallness assumptions the existence of certain aisles is guaranteed. We include its proof here for the sake of completeness.

\begin{theorem}\label{localizations via small object argument}
Let $\cc$ be a Frobenius category with small colimits and let $\cq$ be a set of objects of $\cc$ closed under shifts (in the stable category $\ul{\cc}$) and such that its objects are small relative to $\ci_{\cq}$-cell. Then:
\begin{enumerate}[1)]
\item $\Tria(\cq)$ is an aisle in $\ul{\cc}$ with coaisle given by the class $\cn=\cq^{\bot_{\ul{\cc}}}$. 
\item The objects of $\Tria(\cq)$ are precisely those isomorphic to an $\ci_{\cq}$-cell complex.
\end{enumerate}
\end{theorem}
\begin{proof}
1) By infinite d\'{e}vissage, it is always true that $\Tria(\cq)^{\bot_{\ul{\cc}}}=\cq^{\bot_{\ul{\cc}}}=:\cn$. Now, we have to prove that for each object $M$ there exists a triangle
\[M\ra M'\ra M''\ra M[1]
\]
in $\ul{\cc}$ with $M'$ in $\cn$ and $M''$ in $\Tria(\cq)$. Now, given $M$, thanks to the small object argument \cite[Theorem 2.1.14]{Hovey1999} we have a relative $\ci_{\cq}$-cell complex $f:M\ra M'$ such that the morphism $M'\ra 0$ is $\ci_{\cq}$-injective, \ie $M'$ belongs to $\cn$. Now we use that fact that $f$ is an $\ci_{\cq}$-cofibration (\cf Lemma \ref{cell is cof}). This implies that $f$ is an inflation. Indeed, consider a lifting in the following diagram
\[\xymatrix{M\ar[d]_{f}\ar[r]^{i_{M}} & IM\ar[d] \\
M'\ar[r]\ar@{.>}[ur]_{g} & 0
}
\]
This proves that $f$ is a monomorphism in $\cc$ and in the ambient abelian category (\cf Remark \ref{Keller Quillen}), since $i_{M}$ is a monomorphism in that abelian category. Consider now the following commutative diagram in the ambient abelian category
\[\xymatrix{M\ar[r]^{f}\ar[d]_{\id_{M}} & M'\ar[r]^{f^c}\ar[d]^{g} & \cok f\ar[d]^{e} \\
M\ar[r]_{i_{M}} & IM\ar[r]_{p_{M}} & SM
}
\]
by using the cokernel $f^c$ of $f$. Since the right square is cartesian, then $f^c$ is a deflation. Since $f$ is a monomorphism in the ambient abelian category, then $f$ is the kernel of $f^c$ and so $f$ is an inflation. In particular, there exists a conflation $M\arr{f}M'\ra M''$ and hence a triangle $M\arr{\ol{f}}M'\ra M''\ra M[1]$ in $\ul{\cc}$. Since $f$ is a relative $\ci_{\cq}$-cell complex and $\ci_{\cq}$-cell is closed under pushouts, the cokernel $M''$ is an $\ci_{\cq}$-cell complex and, by Lemma \ref{Icell in triaq} below, it belongs to $\Tria(\cq)$.

2) The lemma below proves that every $\ci_{\cq}$-cell complex is in $\Tria(\cq)$. Conversely, given $M$ in $\Tria(\cq)$ we can form the triangle
\[M[-1]\ra N\ra C\ra M 
\]
with $N$ in $\cn$ and $C$ an $\ci_{\cq}$-cell complex, as we have shown in the proof of 1). But, since $(\Tria(\cq),\cn)$ is a t-structure and $M$ is in $\Tria(\cq)$, the morphism $M[-1]\ra N$ vanishes. This implies that $C\cong N\oplus M$. Therefore, $N$ belongs both to $\cn$ and to $\Tria(\cq)$, that is to say, $N=0$ and so $C\cong M$.
\end{proof}

\begin{lemma}\label{Icell in triaq}
Under the hypotheses of Theorem \ref{localizations via small object argument}, each $\ci_{\cq}$-cell complex belongs to $\Tria(\cq)$.
\end{lemma}
\begin{proof}
\noindent \emph{First step: Let $\lambda$ be an ordinal. If we have
a direct system of conflations
\[
\eps_\alpha: 0 \to X_\alpha \to Y_\alpha \to Z_\alpha \to 0 \ko
\alpha<\lambda \ko
\]
such that the structure morphisms $Z_\alpha \to Z_{\beta}$ are
inflations for all $\alpha<\beta<\lambda$, then the colimit of the
system is a conflation.} Indeed, it suffices to check that for each
injective $I$, the sequence of abelian groups
\[
0 \to \cc(\colim Z_\alpha, I) \to \cc(\colim Y_\alpha, I) \to
\cc(\colim X_\alpha, I) \to 0
\]
is exact. This follows from the Mittag-Leffler criterion \cite[0$_{\text{III}}$, 13.1]{EGAIIIa} since the
maps
\[
\cc(Z_\beta, I) \to \cc(Z_\alpha, I)
\]
are surjective for all $\alpha<\beta<\lambda$.

\noindent \emph{Second step: If we have an acyclic complex of $\cc$
\[
\ldots \to X^p \to X^{p-1} \to \ldots \to X^0 \to Y \to 0\ko
\]
then $Y$ belongs to the smallest full triangulated subcategory of $\ul{\cc}$ containing the $X^p$ and stable under countable coproducts.} Indeed, by Lemma~6.1 of \cite{Keller1990}, for each complex $K$ over $\cc$, there is a triangle
\[
\mathbf{a} K \to K \to \mathbf{i} K \to S \mathbf{a} K
\]
of $\ch(\cc)$ such that $\mathbf{i}K$ has injective components and $\mathbf{a}K$ is the colimit (in the category of complexes) of a countable sequence of componentwise split monomorphisms of acyclic complexes. The functor $\mathbf{i}K$ is the left adjoint of the inclusion into $\ch(\cc)$ of the full subcategory of complexes with injective components. Thus, it commutes with small coproducts. The composed functor
\[
F: \ch(\cc) \to \ul{\cc} \ko K \mapsto Z^0(\mathbf{a} K)
\]
is a triangle functor which commutes with small coproducts and extends the projection $\cc\to \ul{\cc}$ from $\cc$ to $\ch(\cc)$. Moreover, it vanishes on acyclic complexes. Thus, it maps the truncated complex
\[
X'=(\ldots \to X^p \to X^{p+1} \to \ldots \to X^0 \to 0 \to \ldots)
\]
to an object isomorphic to $Y$ in $\ul{\cc}$. Since $F$ commutes with small coproducts, it suffices to show that the complex $X'$ is in the smallest triangulated subcategory of $\ch(\cc)$ containing the $X^p$ and stable under countable coproducts. This holds thanks to Milnor triangle 
\[
\coprod X^{\geq p} \to \coprod X^{\geq q} \to X' \to S\coprod
X^{\geq p} \ko
\]
where $X^{\geq p}$ is the subcomplex
\[
0 \to X^p \to X^{p+1} \to \ldots \to X^0 \to 0
\]
and the leftmost morphism has the components
\[
\xymatrix{X^{\geq p} \ar[r]^-{[\id, -i]^t} & X^{\geq p}\oplus
X^{\geq p+1} \ar[r] & \coprod X^{\geq q} } \ko
\]
where $i$ is the inclusion $X^{\geq p} \to X^{\geq p+1}$.

\noindent \emph{Third step: The claim.} Let $X$ be an $\ci_{\cq}$-cell complex. Then there is an ordinal $\lambda$ and a direct system
$X_\alpha$, $\alpha\leq\lambda$, such that we have $X=X_\lambda$ and
\begin{itemize}
\item[-] $X_0=0$,
\item[-] for all $\alpha<\lambda$, the morphism $X_\alpha\to
X_{\alpha+1}$ is an inflation with cokernel in $\cq$,
\item[-] for all limit ordinals $\beta\leq \lambda$, we have
$X_\beta = \colim_{\alpha<\beta} X_\alpha$.
\end{itemize}
We will show by induction on $\beta\leq \lambda$ that $X_\beta$ belongs to
$\Tria(\cq)$. This is clear for $X_0$. Moreover, if $X_\alpha$ is in
$\Tria(\cq)$, so is $X_{\alpha+1}$. So let us assume that $\beta$ is
a limit ordinal and $X_\alpha$ belongs to $\Tria(\cq)$ for each
$\alpha<\beta$. We wish to show that $X_\beta$ belongs to
$\Tria(\cq)$. Let $\Fun(\beta,\cc)$ be the category of functors $\beta \to
\cc$. The evaluation
\[
\Fun(\beta,\cc)\to \cc \ko Y \to Y_\alpha
\]
admits a left adjoint denoted by $Z\mapsto Z\ten \alpha$. For each
$Y\in\Fun(\beta,\cc)$, the morphism
\[
\coprod_{\alpha<\beta} Y_\alpha\ten\alpha \to Y
\]
is a pointwise split epimorphism. By splicing exact sequences of the form
\[
0 \to Y' \to \coprod_{\alpha<\beta} Y_\alpha\ten\alpha \to Y \to 0
\]
we construct a complex
\[
\ldots \to X^{p} \to X^{p-1} \to \ldots \to X^0 \to X \to 0
\]
which is acyclic for the pointwise split exact structure on
$\Fun(\beta,\cc)$ and such that each $X^p$ is a small coproduct of objects
$Y\ten \alpha$, $Y\in \Tria(\cq)$, $\alpha<\beta$. By the first
step, the colimit $C$ of the above complex is still acyclic.
Moreover, the components of $C$ are small coproducts of objects
\[
\colim_\gamma (Y\ten \alpha)(\gamma) = Y
\]
belonging to $\Tria(\cq)$. Thus, each component of $C$ belongs to
$\Tria(\cq)$. Now, the claim follows from the second step.
\end{proof}

The following is the `dg generalization' of \cite[Lemma 2.3.2]{Hovey1999} and shows that the smallness assumption of Theorem \ref{localizations via small object argument} is completely innocuous.

\begin{lemma}\label{smallness of dg modules}
Let $\ca$ be a small dg category and let $M$ be a right dg $\ca$-module. Consider the cardinal 
\[\kappa:=|\bigcup_{A\in\ca}MA|\cdot\left(|\bigcup_{A\in\ca}MA|+|\bigcup_{A, A'\in\ca}\ca(A,A')|+1\right).
\]
Then, $M$ is $\kappa$-small. In particular, every dg $\ca$-module is small.
\end{lemma}
\begin{proof}
Let $\lambda$ be a $\kappa$-filtered ordinal and let $X:\lambda\ra\cc\ca$ be a $\lambda$-sequence of right $\ca$-modules. Let us prove that the natural map
\[\underset{\beta<\lambda}{\text{colim}}(\cc\ca)(M,X_{\beta})\ra(\cc\ca)(M,\underset{\beta<\lambda}{\text{colim}}X_{\beta})
\]
is an isomorphism. Indeed, call $X:=\underset{\beta<\lambda}{\text{colim}}X_{\beta}$ and let $F\in(\cc\ca)(M,X_{\beta})$ and $F'\in(\cc\ca)(M,X_{\beta'})$ be two morphisms such that $u_{\beta}F:M\ra X$ and $u_{\beta'}F':M\ra X$ are the same morphism, with $u_{\beta}:X_{\beta}\ra X$ being the structural morphism. Consider, for each $A\in\ca$ and each $m\in MA$ an ordinal $\beta(A,m)\in\lambda$ such that $(FA)(m)=(FA')(m)$ in $X_{\beta(A,m)}$. The set $\cs$ formed by the ordinals $\beta(A,m)$, for each $A\in\ca$ and $m\in MA$, is a subset of $\lambda$ satisfying $|\cs|\leq\kappa$, and so its supremum $\gamma$ satisfies $\gamma<\lambda$. Then, the two compositions 
\[\tilde{F}=(M\arr{F}X_{\beta}\ra\dots\ra X_{\gamma})
\]
and
\[\tilde{F'}=(M\arr{F'}X_{\beta'}\ra\dots\ra X_{\gamma})
\]
agree, $\tilde{F}$ is equivalent fo $F$ and $\tilde{F'}$ is equivalent to $F'$. Let us prove that the map is surjective. Given $F\in(\cc\ca)(M,X)$, for each $A\in\ca$ and $m\in MA$ we consider $\beta(A,m)\in\lambda$ such that $(FA)(m)$ is in the image of $u_{\beta(A,m)}:X_{\beta(A,m)}\ra X$. The set $\cs$ formed by the ordinals $\beta(A,m)$, for each $A\in\ca$ and $m\in MA$, is a subset of $\lambda$ satisfying $|\cs|\leq\kappa$, and so its supremum $\alpha$ satisfies $\alpha<\lambda$. This allows a factorization of $F$ of the form
\[M\arr{G}X_{\alpha}\arr{u_{\alpha}}X,
\]
with $(GA)(m):=(FA)(m)$. Of course, $G$ may not be a morphism of right dg $\ca$-modules. Anyway, since for each $A\in\ca$ and $m\ko n\in MA$ we have
\[(u_{\alpha}A)((GA)(m+n))=(u_{\alpha}A)((GA)(m)+(GA)(n)),
\]
there exists an ordinal $\beta(A,m,n)\in\lambda$ such that $(GA)(m+n)=(GA)(m)+(GA)(n)$ in $X_{\beta(A,m,n)}$. Similarly, 
\begin{enumerate}[-]
\item for each $A\in\ca$ and $m\in MA$ there exists an ordinal $\beta(A,m)$ such that the identity $(GA)(dm)=d(GA)(m)$ holds in $X_{\beta(A,m)}$,
\item for each $A\ko A'\in\ca\ko m\in MA$ and $a\in\ca(A,A')$ there exists an ordinal $\beta(A,A',m,a)\in\lambda$ such that the identity $(GA')((Ma)(m))=(X_{\alpha}a)((Ga)(m))$ holds in $X_{\beta(A,A',m,a)}$.
\end{enumerate}
Then, the set $\cs$ formed by all the ordinals $\beta(A,m)\ko \beta(A,m,n)$ and $\beta(A,A',m,a)$ as before is a subset of $\lambda$ such that $|\cs|\leq\kappa$, and so its supremum $\gamma$ is $\gamma<\lambda$. Therefore, the composition
\[M\arr{G}X_{\alpha}\ra\dots\ra X_{\gamma}
\]
is a morphism of right dg $\ca$-modules and $F$ factors through it:
\[\xymatrix{M\ar[rrrr]^{F}\ar[dr]_{G} &&&& \underset{\beta<\lambda}{\text{colim}}X_{\beta}\\
& X_{\alpha}\ar[r]\ar[urrr]^{u_{\alpha}}&\dots\ar[r]&X_{\gamma}\ar[ur]_{u_{\gamma}}&
}
\]
\end{proof}


\begin{corollary}\label{derived categories are aisled}
Let $\ca$ be a small dg category. Then, for any set $\cq$ of objects of $\cd\ca$ we have that $\Tria(\cq)$ is an aisle in $\cd\ca$. 
\end{corollary}
\begin{proof}
The derived category $\cd\ca$ is triangle equivalent to the subcategory $\ch_{\textbf{p}}\ca$ of $\ch\ca$ formed by the $\ch$-projective modules. Then, it suffices to check that $\ch_{\textbf{p}}\ca$ is aisled. Let $\cq$ be a set of objects of $\ch_{\textbf{p}}\ca$. Since $\ch_{\textbf{p}}\ca$ is closed under small coproducts, then $\Tria_{\ch_{\textbf{p}}\ca}(\cq)=\Tria_{\ch\ca}(\cq)$. Thanks to Theorem \ref{localizations via small object argument} and Lemma \ref{smallness of dg modules}, we know that $\Tria_{\ch\ca}(\cq)$ is an aisle in $\ch\ca$. Then, if $M$ is an object of $\ch_{\textbf{p}}\ca$, there exists a triangle of $\ch\ca$,
\[M'\ra M\ra M''\ra M'[1],
\]
with $M'\in\Tria_{\ch\ca}(\cq)=\Tria_{\ch_{\textbf{p}}\ca}(\cq)$ and $M''\in\Tria_{\ch\ca}(\cq)^{\bot}$. In particular, $M'\in\ch_{\textbf{p}}\ca$. This implies that $M''$ belongs to $\ch_{\textbf{p}}\ca$, and so $M''\in\Tria_{\ch_{\textbf{p}}\ca}(\cq)^{\bot_{\ch_{\textbf{p}}\ca}}$. That is to say, $\Tria_{\ch_{\textbf{p}}\ca}(\cq)$ is an aisle in $\ch_{\textbf{p}}\ca$.
\end{proof}

\begin{remark}
Corollary \ref{derived categories are aisled} is a generalization of the second part of Corollary \ref{remark Keller} and it is both a generalization and an explanation of \cite[Proposition 4.5]{AlonsoJeremiasSouto2000}. 
\end{remark}

\begin{definition}\label{aisle generated}
Recall that, if $\cq$ is a set of objects of a triangulated category $\cd$, we write $\aisle(\cq)$ for the smallest aisle in $\cd$ containing $\cq$. An aisle $\cu$ in $\cd$ is \emph{generated}\index{aisle!generated} by $\cq$ if $\cu=\aisle(\cq)$.
\end{definition}

Some sets generate non-proper aisles:

\begin{example}
Let $\cd$ be a triangulated category and let $\cq$ be a class of generators of $\cd$ closed under shifts. Then $\aisle_{\cd}(\cq)=\cd$. More generally, if $\cu$ is an aisle in $\cd$ containing $\cq$, then $\cu^{\bot}=0$, and so $\cu=\cd$. 
\end{example}

Notice that $\aisle(\cq)$ might not exist in general, but if it does then it is closed under small coproducts. Anyway, we know that $\Susp(\cq)$, the smallest full suspended subcategory containing $\cq$ and closed under small coproducts, always exists. Let us state some basic results on $\Susp(\cq)$:

\begin{lemma}\label{susp is aisle}
Let $\cq$ be a class of objects of a triangulated category $\cd$. Then:
\begin{enumerate}[1)]
\item $\Susp(\cq)^{\bot}$ consists of those objects of $\cd$ which are right orthogonal to all the non-negative shifts of objects of $\cq$, \ie $\Susp(\cq)^{\bot}=(\cq^+)^{\bot}$.
\item If $\Susp(\cq)$ is an aisle in $\cd$, then $\Susp(\cq)=\aisle(\cq)$.
\end{enumerate}
\end{lemma}
\begin{proof}
1) is proved by infinite d\'{e}vissage. Consider an object $N\in\cd$ be such that
\[\cd(Q[n],N)=0
\]
for each $Q\in\cq$ and $n\geq 0$. Let $\cu$ be the full subcategory of $\Susp(\cq)$ formed by the objects $M$ such that $\cd(M[n],N)=0$ for each $n\geq 0$. We have that $\cu$ is a full suspended subcategory of $\cd$ closed under small coproducts, containing $\cq$ and contained in $\Susp(\cq)$. Then $\cu=\Susp(\cq)$.

2) Let $\cu$ be an aisle in $\cd$ containing $\cq$ and contained in $\Susp(\cq)$. In particular, $\cu$ is closed under small coproducts, and the minimality of $\Susp(\cq)$ implies $\cu=\Susp(\cq)$.
\end{proof}

Lemma \ref{susp is aisle} suggests a way of generating aisles in $\cd$ from a set of objects: it suffices to prove that the smallest full suspended subcategory of $\cd$ containing this set and closed under small coproducts is an aisle. For this, one can adapt the procedures used in Theorem \ref{Krause on perfects} and Theorem \ref{localizations via small object argument}. Indeed, by conveniently adapting H. Krause's proof of \cite[Theorem A]{Krause2002}, M. Jos\'{e} Souto Salorio proved in \cite[Theorem 2.2]{Souto2004} the following:

\begin{example}\label{M. Jose}
If $\cd$ is a triangulated category with small coproducts and $\cp$ a set of perfect objects of $\cd$, then: 
\begin{enumerate}[1)]
\item $\Susp(\cp)$ is an aisle in $\cd$.
\item The objects of $\Susp(\cp)$ are precisely the Milnor colimits of sequences
\[P_{0}\arr{f_{0}}P_{1}\arr{f_{1}}P_{2}\arr{f_{2}}\dots
\]
where $P_{0}$, as well as each mapping cone $\cone(f_{n})\ko n\geq 0$, is a small coproduct of non-negative shifts of objects of $\cp$.
\end{enumerate}
Notice that, in this case, $\bigcup_{n\in\Z}\Susp(\cp)[n]$ is a triangulated category (in fact, a full triangulated subcategory of $\cd$, \cf Lemma \ref{union is triangulated subcategory}) exhaustively generated to the left by $\cp$.
\end{example}

\section{The right bounded derived category of a dg category}\label{The right bounded derived category of a dg category}
\addcontentsline{lot}{section}{4.7. La categor\'ia derivada acotada por la derecha de una categor\'ia dg}

By using Example \ref{M. Jose} we get the following:

\begin{example}\label{almost canonical t-structure}
If $\ca$ is a small dg category, then $\cd^{\leq 0}\ca:=\Susp(\{A^{\we}\}_{A\in\ca})$\index{$\cd^{\leq 0}\ca$} is an aisle in $\cd\ca$ whose coaisle, denoted by $\cd^{>0}\ca$\index{$\cd^{>0}\ca$}, consists of those modules $M$ with cohomology concentrated in positive degrees, \ie $\H nM(A)=0$ for each $A\in\ca$ and $n\leq 0$. For each integer $n\in\Z$ we put\index{$\cd^{\leq n}\ca$}\index{$\cd^{>n}\ca$}
\[\cd^{\leq n}\ca:=\cd^{\leq 0}\ca[-n]
\]
and
\[\cd^{>n}\ca:=\cd^{>0}\ca[-n],
\]
and denote by $\tau^{\leq n}$\index{$\tau^{\leq n}$} and $\tau^{>n}$\index{$\tau^{>n}$} the torsion and torsionfree functors corresponding to $(\cd^{\leq n}\ca, \cd^{>n}\ca)$.
\end{example}

The following lemma ensures that, in case $\ca$ has cohomology concentrated in non-positive degrees, then $\cd^{\leq n}\ca$ admits a familiar description in terms of cohomology.

\begin{lemma}\label{looking for coherence}
Let $\ca$ be a dg category with cohomology concentrated in degrees $(-\infty,m]$ for some integer $m\in\Z$. For a dg $\ca$-module $M$ we consider the following assertions:
\begin{enumerate}[1)]
\item $M\in\cd^{\leq s}\ca$.
\item $\H iM(A)=0$ for each integer $i>m+s$ an every object $A$ of $\ca$.
\end{enumerate}
Then $1)\Rightarrow 2)$ and, in case $m=0$, we also have $2)\Rightarrow 1)$.
\end{lemma}
\begin{proof}
$1)\Rightarrow 2)$ Since $M[s]$ belongs to $\Susp(\{A^{\we}\}_{A\in\ca})$, there exists a triangle in $\cd\ca$
\[\coprod_{n\geq 0}P_{n}\ra\coprod_{n\geq 0}P_{n}\ra M[s]\ra \coprod_{n\geq 0}P_{n}[1]
\]
with $P_{n}\in\Sum(\{A^{\we}\}^{+}_{A\in\ca})^{*n}$ for each $n\geq 0$. Then, for each $A\in\ca$ we get the long exact sequence of cohomology
\[\dots\ra\coprod_{n\geq 0}\H iP_{n}(A)\ra\H {i+s}M(A)\ra\coprod_{n\geq 0}\H {i+1}P_{n}(A)\ra\dots
\]
with $\H iP_{n}(A)\cong(\cd\ca)(A^{\we},P_{n}[i])=0$ for each $i>m$.

$2)\Rightarrow 1)$ Consider the triangle in $\cd\ca$
\[M'\ra M\ra M''\ra M'[1]
\]
with $M'\in\cd^{\leq s}\ca$ and $M''\in(\cd^{\leq s}\ca)^{\bot}$. In particular, $\H {i}M''(A)=0$ for each $A\in\ca$ and each $i\leq s$. The aim is to prove that $\H iM''(A)=0$ for each $A\in\ca$ and each $i\in\Z$. Thus, consider the long exact sequence of cohomology
\[\dots\ra\H iM(A)\ra\H iM''(A)\ra\H {i+1}M'(A)\ra\dots
\]
By using 1) and the extra assumption on $\ca$, we have that $\H {i}M'(A)=0$ for each $i\geq s+2$ and, by hypothesis, $\H iM(A)=0$ for each $i\geq s+1$. This implies that $\H iM''(A)=0$ for each $i\geq s+1$.
\end{proof}

Lemma \ref{looking for coherence} justifies the following:

\begin{definition}\label{canonical t-structure}
Let $\ca$ be a small dg category. The t-structure associated to $\Susp(\{A^{\we}\}_{A\in\ca})$ is said to be the \emph{canonical t-structure}\index{t-structure!canonical} on $\cd\ca$. 
\end{definition}

\begin{definition}\label{right bounded}
If $\ca$ is a dg category, we will write\index{$\cd^-\ca$}
\[\cd^-\ca:=\bigcup_{n\in\Z}\cd^{\leq n}\ca,
\]
and we will refer to $\cd^-\ca$ as the \emph{right bounded derived category}\index{category!right bounded derived} of $\ca$.
\end{definition}

\begin{remark}
Notice that $\cd^-\ca$ is not closed under small coproducts in $\cd\ca$. Indeed, given $A\in\ca$, the coproduct $\coprod_{n\in\Z}A^{\we}[n]$ does not belong to $\cd^-\ca$. Notice also that, thanks to Example \ref{M. Jose}, $\cd^-\ca$ is exhaustively generated to the left by the free $\ca$-modules $A^{\we}\ko A\in\ca$.
\end{remark}

The following are general properties of subcategories constructed like the right bounded derived category of Definition \ref{right bounded}.

\begin{lemma}\label{union is triangulated subcategory}
Let $\cq$ be a class of objects of a triangulated category $\cd$. If $\aisle(\cq)$ exists, then $\cd':=\bigcup_{n\in\Z}\aisle(\cq)[n]$ is a full triangulated subcategory of $\cd$.
\end{lemma}
\begin{proof}
It is clear that $\cd'$ is closed under shifts. Now, let
\[L\ra M\ra N\ra L[1]
\]
be a triangle of $\cd$ with $L\ko N\in\cd'$. Then, there exists integers $l\ko n\in\Z$ such that $L\in\aisle(\cq)[l]$ and $N\in\aisle(\cq)[n]$. If $l\leq n$, then $\aisle(\cq)[n]\subseteq\aisle(\cq)[l]$ and so $N\in\aisle(\cq)[l]$. Since $\aisle(\cq)[l]$ is closed under extensions, then $M\in\aisle(\cq)[l]$ and so $M\in\cd'$.
\end{proof}

\begin{lemma}\label{preservation of coproducts and compacity}
Let $\cd$ be a triangulated category with small coproducts and let $\cq$ be a set of generators of $\cd$ such that $\Susp(\cq)$ is an aisle. The following assertions hold for the full triangulated subcategory $\cd':=\bigcup_{n\in\Z}\aisle(\cq)[n]$:
\begin{enumerate}[1)]
\item The inclusion functor $\iota:\cd'\hookrightarrow\cd$ preserves small coproducts.
\item If $\Susp(\cq)^{\bot}$ is closed under small coproducts, then an object $P$ of $\cd'$ is compact (respectively, perfect, superperfect) if and only if it is compact (respectively, perfect, superperfect) in $\cd$.
\end{enumerate}
\end{lemma}
\begin{proof}
1) Let $D'_{i}\ko i\in I$, be a family of objects of $\cd'$ whose coproduct exists in $\cd'$. We write $\coprod_{i\in I}D'_{i}$ for the coproduct in $\cd$, $D'$ for the coproduct in $\cd'$ and $v_{i}:D'_{i}\ra D'$ for the canonical morphisms. For simplicity, put $\aisle(\cq)[k]=\cu_{k}$. Therefore, we have a chain
\[\dots\subseteq\cu_{k+1}\subseteq\cu_{k}\subseteq\cu_{k-1}\subseteq\dots\subseteq\cd'
\]
of aisles in $\cd$ whose union is $\cd'$. 

\emph{Claim: If $m\ko n\in\Z$ are integers such that $D'\in\cu_{n}$ and $D'_{i}\in\cu_{m}\setminus\cu_{m+1}$ for some $i\in I$, then $n\leq m$.} Indeed, fix such an $i$ and assume $n>m$ and consider the triangle
\[\tau_{\cu_{n}}(D'_{i})\ra D'_{i}\arr{f}\tau^{\cu_{n}^{\bot}}(D'_{i})\ra \tau_{\cu_{n}}(D'_{i})[1].
\]
Since the two first vertices of this triangle belong to $\cd'$, then so does $\tau^{\cu_{n}^{\bot}}(D'_{i})$. Hence, by using the universal property of the coproduct, we have that $f$ induces a morphism
\[\tilde{f}:D'\ra \tau^{\cu_{n}^{\bot}}(D'_{i})
\]
such that
\[\tilde{f}v_{j}=\begin{cases}f & \text{ if }j=i, \\
0 & \text{ otherwise.}\end{cases}
\]
Since $D'\in\cu_{n}$, then $\tilde{f}=0$ and so $f=0$. Therefore, $D'_{i}$ is a direct summand of $\tau_{\cu_{n}}(D'_{i})$. This implies that $D'_{i}$ belongs to $\cu_{n}$, and so it belongs to $\cu_{m+1}$, which is a contradiction.

Consider the following two situations:

\emph{First situation:} For each $i\in I$ we have $D'_{i}\in\bigcap_{k\in\Z}\cu_{k}$. Since aisles are closed under small coproducts, this implies that the coproduct $\coprod_{i\in I}D'_{i}$ belongs to $\bigcap_{k\in\Z}\cu_{k}$, and so to $\cd'$. Hence $D'\cong\coprod_{i\in I}D'_{i}$.

\emph{Second situation:} There exists $j\in I$ such that $D'_{j}\in\cu_{m}\setminus\cu_{m+1}$. Given $i\in I$, put $m_{i}$ for the maximum of the set of those integers $k\in\Z$ such that $D'_{i}\in\cu_{k}$. Put $m_{i}=\infty$ if $D'_{i}\in\bigcap_{k\in\Z}\cu_{k}$. Thanks to the claim, we know that, in any case, $m_{i}\geq n$ for each $i\in I$. Then $D'_{i}\in\cu_{n}$ for every $i\in I$, and so $\coprod_{i\in I}D'_{i}\in\cu_{n}$. Again, this implies $\coprod_{i\in I}D'_{i}\cong D'$.

2) Assertion 1) implies that if $P\in\cd'$ is compact in $\cd$ then it is also compact in $\cd'$. Conversely, let $P\in\cd'$ be compact in $\cd'$ and fix an integer $n\in\Z$ such that $P\in\cu_{n}$. If $D_{i}\ko i\in I$, is a family of objects of $\cd$, then we have isomorphisms
\[\cd(P,D_{i})\cong\cu_{n}(P,\tau_{\cu_{n}}(D_{i}))=\cd'(P,\tau_{\cu_{n}}(D_{i}))
\]
for each $i\in I$, and
\[\cd(P,\coprod_{i\in I}D_{i})\cong\cu_{n}(P,\tau_{\cu_{n}}(\coprod_{i\in I}D_{i}))=\cd'(P,\tau_{\cu_{n}}(\coprod_{i\in I}D_{i})).
\]
By hypothesis, $\cu^{\bot}_{n}$ is closed under small coproducts. This is equivalent to the fact that $\tau_{\cu_{n}}$ preserves small coproducts, and so we have a canonical isomorphism
\[\coprod_{i\in I}\tau_{\cu_{n}}(D_{i})\arr{\sim}\tau_{\cu_{n}}(\coprod_{i\in I}D_{i}).
\]
Finally, we have the commutative diagram
\[\xymatrix{\coprod_{i\in I}\cd(P,D_{i})\ar[r]_{\sim\ \ \ \ \ \ }\ar[dd]^{\can} & \coprod_{i\in I}\cd'(P,\tau_{\cu_{n}}(D_{i}))\ar[d]_{\wr}^{\can} \\
&\cd'(P,\coprod_{i\in I}\tau_{\cu_{n}}(D_{i}))\ar[d]_{\wr}^{\can} \\
\cd(P,\coprod_{i\in I}D_{i})\ar[r]_{\sim\ \ \ \ \ \ } & \cd'(P,\tau_{\cu_{n}}(\coprod_{i\in I}D_{i}))
}
\]
where the morphisms `can' are the canonical ones. This proves that $P$ is compact in $\cd$. The case of $P$ being (super)perfect follows similarly using adjunction.
\end{proof}

\chapter{Triangulated TTF triples on triangulated categories}\label{TTF triples on triangulated categories}
\addcontentsline{lot}{chapter}{Cap\'itulo 5. Ternas TTF trianguladas en categor\'ias trianguladas}

\section{Introduction}
\addcontentsline{lot}{section}{5.1. Introducci\'on}

\subsection{Motivations}
\addcontentsline{lot}{section}{5.1.1 Motivaci\'on}

One of the aims of this chapter is to give a (probably naive) parametrization of all the \emph{d\'{e}collements} (\cf Definition \ref{recollement}) of triangulated categories of a certain, general and interesting enough, type. Nevertheless, this parametrization, together with B.~Keller's Morita theory for derived categories \cite{Keller1994a}, already generalizes some results of \cite{DwyerGreenlees, Jorgensen2006} and offers an `unbounded' version of S.~K\"{o}nig's theorem \cite{Konig1991} on d\'{e}collements of right bounded derived categories of algebras. In Chapter \ref{Restriction of TTF triples in triangulated categories} we will study the problem of descending the parametrization from unbounded to right bounded derived categories, and the corresponding lifting.

The following facts suggest that, in the case of derived categories, a more sophisticated parametrization is possible:
\begin{enumerate}[1)]
\item D\'{e}collements `are' precisely triangulated TTF triples (\cf section \ref{TTF triples and recollements} of Chapter \ref{Preliminary results}).
\item TTF triples on module categories are well understood and a tangible parametrization of them was given by J. P. Jans \cite{Jans}. 
\item A natural proof of J. P. Jans' theorem (\cf Theorem \ref{parametrizando TTF ternas}) uses P. Gabriel's characterization of module categories among abelian categories \cite{Gabriel1962}, which is at the basis of Morita theory of module categories.
\item The `triangulated' analogue of P. Gabriel's result was proved by B.~Keller, who developed a Morita theory for derived categories of dg categories \cite{Keller1994a} (and latter in \cite{Keller1996, Keller1998b},\dots).
\item Some results (\cf \cite[Theorem 4.3]{Keller1994a}, \cite[Theorem 5.2]{Porta2007}) suggest that derived categories of dg categories play the r\^{o}le, in the theory of triangulated categories, that module categories play in the theory of abelian categories. 
\end{enumerate}
Then, another aim of this chapter is to give a touchable parametrization of triangulated TTF triples on derived categories of dg categories by using B.~Keller's theory, and to elucidate the links with H. Krause's parametrization \cite{Krause2000} of smashing subcategories of compactly generated triangulated categories. For this, we use a generalization of the notion of \emph{homological epimorphism} of algebras due to W. Geigle and H. Lenzing \cite{GeigleLenzing}. Homological epimorphisms appear as (stably flat) universal localizations in the work of P. M. Cohn \cite{Cohn}, A. H. Schofield \cite{Schofield}, A. Neeman and A. Ranicki \cite{NeemanRanickiI}, \dots Recently, H. Krause has studied \cite{Krause2005} the link between homological epimorphisms of algebras and: the \emph{chain map lifting problem}, the \emph{generalized smashing conjecture} and the existence of long exact sequences in algebraic K-theory. Homological epimorphisms also appear in the work of L. Angeleri H\"{u}gel and J. S\'{a}nchez \cite{AngeleriSanchez2007} in the construction of tilting modules from ring epimorphisms.

\subsection{Outline of the chapter}
\addcontentsline{lot}{section}{5.1.2. Esbozo del cap\'itulo}

In section \ref{Parametrization}, we introduce for a class of objects of a triangulated category the property of being \emph{recollement-defining} (subsection \ref{Recollement-defining classes}) and prove how to find recollement-defining sets in \emph{aisled} triangulated categories (subsection \ref{Recollement-defining sets in aisled categories}) and in perfectly generated triangulated categories (subsection \ref{Recollement-defining sets in perfectly generated triangulated categories}). In subsection \ref{Parametrization of TTF triples on triangulated categories}, recollement-defining sets enable us to parametrize all the triangulated TTF triples on a triangulated category with a set of generators, and all the d\'{e}collements of a `good' triangulated category in terms of compactly generated triangulated categories. In section \ref{Homological epimorphisms of dg categories}, we define the notion of \emph{homological epimorphisms of dg categories}, generalizing the homological epimorphisms of algebras of W. Geigle and H. Lenzing \cite{GeigleLenzing}. We easily prove that this kind of morphisms always induce a triangulated TTF triple, which allows us to give several examples of recollements of unbounded derived categories of algebras which were already known for right bounded derived categories (\cf S.~K\"{o}nig's paper \cite{Konig1991}). Conversely, we prove that every triangulated TTF triple on the derived category of a $k$-flat dg category $\ca$ is induced by a homological epimorphism starting in $\ca$. This correspondence between triangulated TTF triples and homological epimorphisms keeps a lot of similitudes with the one accomplished by J. P. Jans \cite{Jans} for module categories. In section \ref{Parametrization for derived categories}, we state a parametrization of all the smashing subcategories of a compactly generated algebraic triangulated category which uses the main results of subsection \ref{Parametrization of TTF triples on triangulated categories} and section \ref{Homological epimorphisms of dg categories}. Finally, in section \ref{Idempotent two-sided ideals}, we an\-a\-lyse how idempotent two-sided ideals of the category $\cd^c$ of compact objects appear in the description of triangulated TTF triples on (or smashing subcategories of) a compactly generated triangulated category $\cd$. In Theorem \ref{from ideals to devissage wrt hc} we prove that such an idempotent two-sided ideal, which is moreover stable under shifts, induces a nicely described triangulated TTF triple. This, together with assertion 2') of H.~Krause's \cite[Theorem 4.2]{Krause2000}, gives a short proof of a result (Theorem \ref{our result}) in the spirit of H.~Krause's bijection \cite[Corollary 12.5 and Corollary 12.6]{Krause2005} between smashing subcategories and special idempotent two-sided ideals. As a consequence (\cf Corollary \ref{gsc}), we get the following weak version of the \emph{generalized smashing conjecture}: any smashing subcategory of a compactly generated triangulated category satisfies the principle of infinite d\'{e}vissage with respect to a set of Milnor colimits of compact objects. Another consequence (\cf Corollary \ref{Krause algebraic}) is that, when $\cd$ is algebraic, we recover precisely H.~Krause's bijection. We think it is worth to mention that, in the algebraic case, assertion 2') of \cite[Theorem 4.2]{Krause2000} admits a short proof (\cf Proposition \ref{for the idempotency}) based on the `omnipresence' of homological epimorphisms of dg categories.

\section{General parametrization}\label{Parametrization}
\addcontentsline{lot}{section}{5.2. Parametrizaci\'on general}

\subsection{Recollement-defining classes}\label{Recollement-defining classes}
\addcontentsline{lot}{section}{5.2.1. Clases definidoras de aglutinaciones}

\begin{definition}\label{recollement defining}
Let $\cd$ be a triangulated category. A class $\cp$ of objects of $\cd$ is \emph{recollement-defining}\index{class!recollement-defining} in $\cd$ if the class $\cy$ of objects which are right orthogonal to all the shifts of objects of $\cp$ is both an aisle and a coaisle in $\cd$. 
\end{definition}

Notice that, in this case, one has that the triangulated category $\ ^{\bot}\cy$ is generated by $\cp$ thanks to the Lemma \ref{generators and t-structures}.

In the following subsections, we will show how to weaken the conditions imposed to a set in order to be recollement-defining in some particular frameworks.

\subsection{Recollement-defining sets in aisled categories}\label{Recollement-defining sets in aisled categories}
\addcontentsline{lot}{section}{5.2.2. Conjuntos definidores de aglutinaciones en categor\'ias aladas}

\begin{definition}\label{aisled triangulated category}
A triangulated category $\cd$ is \emph{aisled}\index{category!triangulated!aisled} if it has a set of generators, small coproducts and for every set $\cq$ of objects of $\cd$ we have that $\Tria(\cq)$ is an aisle in $\cd$.
\end{definition}

\begin{lemma}\label{rds in aisled}
For a set $\cp$ of objects of an aisled triangulated category $\cd$ the following assertions are equivalent:
\begin{enumerate}[1)]
\item $\cp$ is a recollement-defining set.
\item The class $\cy$ of objects of $\cd$ which are right orthogonal to all the shifts of objects of $\cp$ is closed under small coproducts.
\end{enumerate}
In this case $\Tria(\cp)=\ ^{\bot}\cy$.
\end{lemma}
\begin{proof}
$1)\Rightarrow 2)$ By hypothesis, we have that $\cy$ is a triangulated aisle in $\cd$, and so $\cy=^{\bot}\cz$ for some full triangulated subcategory $\cz$ of $\cd$. This implies that $\cy$ is closed under small coproducts.

$2)\Rightarrow 1)$ Since $\cd$ is aisled, $\Tria(\cp)$ is an aisle in $\cd$. By infinite d\'{e}vissage, we have that its coaisle is precisely $\cy$. If $\cg$ is a set of generators of $\cd$, then by using Lemma \ref{generators and t-structures} we know that $\tau^{\cy}(\cg)$ is a set of generators of $\cy$. Notice that $\Tria(\tau^{\cy}(\cg))$ is an aisle in $\cd$ contained in $\cy$. Hence, by Lemma \ref{generators and t-structures} $\Tria(\tau^{\cy}(\cg))=\cy.$ This proves that $\cy$ is an aisle in $\cd$.
\end{proof}

\begin{remark}
It is worth noting that, thanks to Lemma \ref{generators and t-structures}, if $\cd$ is an aisled triangulated category and $\cd'$ is a full triangulated subcategory of $\cd$ closed under small coproducts and generated by a set of objects $\cq$, then $\cd'=\Tria(\cq)$. That is to say: roughly speaking, ``to be generated by a set'' and ``to satisfy the principle of infinite d\'{e}vissage with respect to a set'' are the same in aisled triangulated categories.
\end{remark}

\subsection{Recollement-defining sets in perfectly generated triangulated categories}\label{Recollement-defining sets in perfectly generated triangulated categories}
\addcontentsline{lot}{section}{5.2.3. Conjuntos definidores de aglutinaciones en categor\'ias perfectamente generadas}

\begin{lemma}\label{recollement defining in perfectly generated}
Let $\cd$ be a perfectly generated triangulated category, let $\cp$ be a set of objects of $\cd$ and let $\cy$ be the class of objects of $\cd$ which are right orthogonal to all the shifts of objects of $\cp$. The following assertions are equivalent:
\begin{enumerate}[1)]
\item $\cp$ is recollement-defining.
\item $\cy$ is a coaisle in $\cd$ closed under small coproducts.
\end{enumerate}
If $\cp$ consists of perfect objects, the statements above are also equivalent to:
\begin{enumerate}[3)]
\item $\cy$ is closed under small coproducts.
\end{enumerate}
In this last case $\Tria(\cp)=\ ^{\bot}\cy$.
\end{lemma}
\begin{proof}
$1)\Rightarrow 2)$ Since $\cy$ is an aisle, then it is closed under small coproducts.

$2)\Rightarrow 1)$ Since $\ ^{\bot}\cy$ is a smashing subcategory in a perfectly generated triangulated category, thanks to Proposition \ref{smashing are TTF} we know that its coaisle $\cy$ is also an aisle.

$3)\Rightarrow 2)$ Theorem \ref{Krause on perfects} implies that $\Tria(\cp)$ is an aisle in $\cd$,  whose coaisle is $\cy$ by infinite d\'{e}vissage.
\end{proof}

\begin{remark}\label{superperfect are recollement defining}
Notice that any set $\cp$ of superperfect (\eg compact) objects of a perfectly generated triangulated category satisfies condition 3) of the proposition above, and so it is recollement-defining.
\end{remark}

\subsection{Parametrization of triangulated TTF triples}\label{Parametrization of TTF triples on triangulated categories}
\addcontentsline{lot}{subsection}{5.2.4. Parametrizaci\'on de las ternas TTF trianguladas}

\begin{proposition}\label{naive general parametrization}
Let $\cd$ be a triangulated category with a set of generators. Consider the map which takes a set $\cp$ of objects of $\cd$ to the triple
\[(\ ^{\bot}\cy,\cy,\cy^{\bot})
\]
of subcategories of $\cd$, where $\cy$ is formed by those objects which are right orthogonal to all the shifts of objects of $\cp$. The following assertions hold:
\begin{enumerate}[1)]
\item This map defines a surjection from the class of all recollement-defining sets onto the class of all the triangulated TTF triples on $\cd$.
\item If $\cd$ is aisled, then this map induces a surjection from the class of all objects $P$ such that $\{P[n]\}_{n\in\Z}^{\bot}$ is closed under small coproducts onto the class of all triangulated TTF triples on $\cd$.
\item If $\cd$ is perfectly generated, then this map induces surjections from
\begin{enumerate}[3.1)]
\item the class of perfect objects $P$ such that $\{P[n]\}_{n\in\Z}^{\bot}$ is closed under small coproducts onto the class of all perfectly generated triangulated TTF triples.
\item the class of superperfect objects onto the class of superperfectly generated triangulated TTF triples.
\item the class of sets of compact objects onto the class of compactly generated triangulated TTF triples.
\end{enumerate}
\end{enumerate}
\end{proposition}
\begin{proof}
1) Let $(\cx,\cy,\cz)$ be a triangulated TTF triple on $\cd$ and let $\cg$ be a set of generators of $\cd$. Since $\tau^{\cz}(\cg)$ is a set of generators of $\cz$ (\cf Lemma \ref{generators and t-structures}), and the composition $\cz\arr{z}\cd\arr{\tau_{\cx}}\cx$ is a triangle equivalence (\cf Lemma \ref{x igual a z}), then $\tau_{\cx}z\tau^{\cz}(\cg)$ is a set of generators of $\cx$. But then, by using Lemma \ref{generators and t-structures} we know that $\cy$ is the set of objects which are right orthogonal to all the shifts of objects of $\tau_{\cx}z\tau^{\cz}(\cg)$, which proves both that $\tau_{\cx}z\tau^{\cz}(\cg)$ is a recollement-defining set and that the triangulated TTF triple comes from a recollement-defining set.

2) We use Lemma \ref{Recollement-defining sets in aisled categories} and the fact that if $\cp$ is a recollement-defining set, then $\{\coprod_{P\in \cp}P\}$ is also a recollement-defining set which is sent to the same triangulated TTF triple onto which $\cp$ was sent, because
\[\{(\coprod_{P\in \cp}P)[n]\}_{n\in\Z}^{\bot}=\{P[n]\}^{\bot}_{n\in\Z\ko P\in\cp}
\]

3) For 3.1) and 3.2) we use the idea of the proof of 2) together with the fact that the class of (super)perfect objects is closed under small coproducts. For 3.3) we use Lemma \ref{truncating special objects and from local to global via smashing}.
\end{proof}

\begin{definition}
An object $M$ of a triangulated category $\cd$ is called \emph{exceptional}\index{object!exceptional} if it has no self-extensions, \ie $\cd(M,M[n])=0$ for each integer $n\neq 0$.
\end{definition}

The following corollary generalizes \cite[Theorem 3.3]{Jorgensen2006}, and also \cite[Theorem 2.16]{Heider2007} via \cite[Theorem 2.9]{Heider2007}.

\begin{corollary}\label{recollements unbounded derived}
Let $\cd$ be a triangulated category which is either perfectly generated or aisled. The following assertions are equivalent:
\begin{enumerate}[1)]
\item $\cd$ is a recollement of triangulated categories generated by a single compact (and exceptional) object.
\item There are (exceptional) objects $P$ and $Q$ of $\cd$ such that:
\begin{enumerate}[2.1)]
\item $P$ is compact.
\item $Q$ is compact in $\Tria(Q)$.
\item $\cd(P[n],Q)=0$ for each $n\in\Z$.
\item $\{P, Q\}$ generates $\cd$.
\end{enumerate}
\item There is a compact (and exceptional) object $P$ such that $\Tria(P)^{\bot}$ is generated by a compact (and exceptional) object in $\Tria(P)^{\bot}$.
\end{enumerate}
In case $\cd$ is compactly generated by a single object we can add:
\begin{enumerate}[4)]
\item There is a compact (and exceptional) object $P$ (such that $\Tria(P)^{\bot}$ is generated by an exceptional compact object).
\end{enumerate}
In case $\cd$ is algebraic we can add:
\begin{enumerate}[5)]
\item $\cd$ is a recollement of derived categories of dg algebras (concentrated in degree $0$, \ie ordinary algebras).
\end{enumerate}
\end{corollary}
\begin{proof}
$1)\Rightarrow 2)$ If $(\cx,\cy,\cz)$ is the triangulated TTF triple corresponding to the recollement of 1), we can take $P$ to be a compact generator of $\cx$ and $Q$ to be a compact generator of $\cy$. Of course, conditions 2.2), 2.3) and 2.4) are satisfied. Finally, by Lemma \ref{truncating special objects and from local to global via smashing}, $P$ is compact in $\cd$.

$2)\Rightarrow 1)$ Thanks to Theorem \ref{Krause on perfects}, we know that $\Tria(P)$ is an aisle. Moreover, since the associated coaisle consists of those objects which are right orthogonal to all the shifts of $P$, then we have that $\Tria(P)$ is a smashing subcategory of $\cd$. Conditions 2.3) and 2.4) say that $Q$ generates $\Tria(P)^{\bot}$. Moreover, $\Tria(Q)$ is contained in $\Tria(P)^{\bot}$ and condition 2.2) ensures, via Theorem \ref{Krause on perfects}, that $\Tria(Q)$ is an aisle in $\cd$. Now Lemma \ref{generators and t-structures} implies that $\Tria(P)^{\bot}=\Tria(Q)$, and so $Q$ is a compact generator of $\Tria(P)^{\bot}$. Finally, by using either that $\cd$ is perfectly generated (together with Lemma \ref{Recollement-defining sets in perfectly generated triangulated categories}) or that $\cd$ is aisled (together with Lemma \ref{Recollement-defining sets in aisled categories}) we have that $(\Tria(P), \Tria(P)^{\bot},\Tria(P)^{\bot\bot})$ is a triangulated TTF triple.

$2)\Rightarrow 3)$ Thanks to condition 2.2), it suffices to prove $\Tria(Q)=\Tria(P)^{\bot}$. Condition 2.3) implies that $Q\in\Tria(P)^{\bot}$. Since $P$ is compact, $\Tria(P)^{\bot}$ is closed under small coproducts, and so $\Tria(Q)\subseteq\Tria(P)^{\bot}$. Now, condition 2.2) says, via Theorem \ref{Krause on perfects}, that $\Tria(Q)$ is an aisle in $\cd$. Since condition 2.4) states that $Q$ generates $\Tria(P)^{\bot}$, then Lemma \ref{generators and t-structures} implies that $\Tria(Q)=\Tria(P)^{\bot}$.

$3)\Rightarrow 2)$ If $Q$ is an object of $\Tria(P)^{\bot}$, then condition 2.3) holds. If, moreover, $Q$ generates $\Tria(P)^{\bot}$, then condition 2.4) also holds. Notice that, since $P$ is compact, then $\Tria(P)^{\bot}$ is closed under small coproducts. In particular, $\Tria(Q)\subseteq\Tria(P)^{\bot}$ and so, if $Q$ is compact in $\Tria(P)^{\bot}$, then so is in $\Tria(Q)$.  

$3)\Rightarrow 4)$ is clear.

$4)\Rightarrow 3)$ Theorem \ref{Krause on perfects} implies that $\Tria(P)$ is an aisle. Since $P$ is compact, $\Tria(P)^{\bot}$ is closed under small coproducts, and thanks to Lemma \ref{generators and t-structures} and Lemma \ref{truncating special objects and from local to global via smashing} we know that $\Tria(P)^{\bot}$ is generated by a compact object in $\Tria(P)^{\bot}$.

$5)\Rightarrow 1)$ is clear since the derived category $\cd A$ os a dg algebra $A$ is generated by the algebra regarded as a right dg $A$-module with its regular structure.

$1)\Rightarrow 5)$ We use that, by \cite[Theorem 4.3]{Keller1994a}, an algebraic triangulated category compactly generated by a single object is triangle equivalent to the derived category of a dg algebra (\cf subsection \ref{B. Keller's Morita theory for derived categories}).

To deal with the case of exceptional compact objects, one uses that an algebraic triangulated category compactly generated by an exceptional object is triangle equivalent to the derived category of an ordinary algebra (\cf for instance \cite[Theorem 8.3.3]{Keller1998b}).
\end{proof}

Thanks to Corollary \ref{B. Keller's Morita theory for derived categories}, the dg algebras announced in 5) can be chosen to be particularly nice in case $\cd$ is the derived category $\cd A$ of a dg algebra $A$. Indeed, if $P$ and $Q$ are like in 2) and $(\cx,\cy,\cz)=(\Tria(P),\Tria(Q),\Tria(Q)^{\bot})$, then the picture is
\[\xymatrix{\cd B\ar@<1ex>[rr]^{?\otimes^{\bf L}_{B}\textbf{i}Q} && \Tria(Q)\ar@<1ex>[ll]^{\textbf{R}\Hom_{A}(\textbf{i}Q,?)}\ar[r]^{y} & \cd A\ar@/_1pc/[l]_{\tau^{\cy}}\ar@/_-1pc/[l]^{\tau_{\cy}}\ar[r]^{\tau_{\cx}} & \Tria(P)\ar@/_-1pc/[l]^{x}\ar@/_1pc/[l]_{z\tau^{\cz}x}\ar@<1ex>[rr]^{\textbf{R}\Hom_{A}(\textbf{i}P,?)} && \cd C\ar@<1ex>[ll]^{?\otimes^{\bf L}_{C}\textbf{i}P}
}
\]
where $B$ is the dg algebra $(\cc_{dg}A)(\textbf{i}Q,\textbf{i}Q)$, $C$ is the dg algebra $(\cc_{dg}A)(\textbf{i}P,\textbf{i}P)$ and $Q$ is a compact generator of $\Tria(P)^{\bot}=\Tria(Q)$. 
\bigskip

\begin{example}\label{EjemploKonig1}\cite[Example 8]{Konig1991}
Let $k$ be a field and let $A=k(Q,R)$ be the finite dimensional $k$-algebra associated to the quiver
\[\xymatrix{Q=(1\ar@<1ex>[r]^{\alpha} & 2\ar@<1ex>[l]^{\beta})
}
\]
with relations $R=(\alpha\beta\alpha)$. One checks that $P:=P_{2}=e_{2}A$ and $Q:=S_{1}=e_{1}A/e_{1}\mbox{rad}A$ are exceptional objects of $\cd A$ and satisfy conditions $2.1)-2.4)$ of the corollary above. Now, since the dg algebra $(\cc_{dg}A)(\textbf{i}P_{2},\textbf{i}P_{2})$ has cohomology concentrated in degree $0$ and isomorphic, as an algebra, to $C:=\op{End}_{A}(P_{2})$, its derived category is triangle equivalent to $\cd C$. Similarly, the derived category of $(\cc_{dg}A)(\textbf{i}S_{1},\textbf{i}S_{1})$ is triangle equivalent to the derived category of $B:=\op{End}_{A}(S_{1})$.
By applying Corollary \ref{Parametrization of TTF triples on triangulated categories}, we know that there exists a d\'{e}collement
\[\xymatrix{\cd B\ar[r]^{i_{*}} & \cd A\ar@/_1pc/[l]\ar@/_-1pc/[l]\ar[r] & \cd C\ar@/_1pc/[l]\ar@/_-1pc/[l]^{j_{!}}
}
\]
with $i_{*}B=S_{1}$ and $j_{!}C=P_{2}$. 
\end{example}

\section{Homological epimorphisms of dg categories}\label{Homological epimorphisms of dg categories}
\addcontentsline{lot}{section}{5.3. Epimorfismos homol\'ogicos de categor\'ias dg}

Let $F:\ca\ra\cb$ be a dg functor and suppose that the corresponding restriction along $F$\index{$F^{*}$}
\[F^{*}:\cd\cb\ra\cd\ca
\]
is fully faithful. Let $U$ be the $\ca$-$\cb$-bimodule defined by $U(B,A):=\cb(B,FA)$ and let $V$ be the $\cb$-$\ca$-bimodule defined by $V(A,B):=\cb(FA,B)$. Since the functor $F^{*}$ admits a left adjoint
\[?\otimes_{\ca}^{\bf L}U:\cd\ca\ra\cd\cb
\]
and a right adjoint
\[\textbf{R}\ch om_{\ca}(V,?):\cd\ca\ra\cd\cb,
\]
then the essential image $\cy$ of $F^{*}$ is both a triangulated aisle and coaisle in $\cd\ca$. This shows that there exists a triangulated TTF triple on $\cd\ca$ whose central class $\cy$ is triangle equivalent to $\cd\cb$. 

In fact, by using Lemma \ref{generators and t-structures}, Lemma \ref{truncating special objects and from local to global via smashing} and B.~Keller's Morita theory for derived categories (\cf subsection \ref{B. Keller's Morita theory for derived categories}), we know that the central class of any triangulated TTF triple on the derived category $\cd\ca$ of a dg category is always triangle equivalent to the derived category $\cd\cb$ of a certain dg category. We will prove in this section that, up to replacing $\ca$ by a quasi-equivalent \cite{Tabuada2005a} dg category, the new dg category $\cb$ can be chosen so as to be linked to $\ca$ by a morphism $F:\ca\ra\cb$ of dg categories whose corresponding restriction $F^{*}:\cd\cb\ra\cd\ca$ is fully faithful. 

Let us show an `arithmetical' characterization of this kind of morphisms. For this, notice that a morphism $F:\ca\ra\cb$ of dg categories also induces a restriction of the form $\cd(\cb^{op}\otimes_{k}\cb)\ra\cd(\ca^{op}\otimes_{k}\ca)$, still denoted by $F^{*}$, and that $F$ can be viewed as a morphism $F:\ca\ra F^{*}(\cb)$ in $\cd(\ca^{op}\otimes_{k}\ca)$. Let 
\[X\ra\ca\arr{F}F^{*}\cb\ra X[1]
\] 
be the triangle of $\cd(\ca^{op}\otimes_{k}\ca)$ induced by $F$.

\begin{lemma}\label{characterization he}
Let $F:\ca\ra\cb$ be a dg functor between small dg categories. The following statements are equivalent:
\begin{enumerate}[1)]
\item $F^{*}:\cd\cb\ra\cd\ca$ is fully faithful.
\item The counit $\delta$ of the adjunction $(?\otimes^{\bf L}_{\ca}U,F^{*})$ is an isomorphism.
\item The counit $\delta_{B^{\we}}:F^{*}(B^{\we})\otimes^{\bf L}_{\ca}U\ra B^{\we}$ is an isomorphism for each object $B$ of $\cb$.
\item $F$ satisfies the following:
\begin{enumerate}[4.1)]
\item The modules $(FA)^{\we}\ko A\in\ca$ form a set of compact generators of $\cd\cb$.
\item $X(?,A)\otimes^{\bf L}_{\ca}U=0$ for each object $A$ of $\ca$.
\end{enumerate}
\item $F$ satisfies the following:
\begin{enumerate}[5.1)]
\item The modules $(FA)^{\we}\ko A\in\ca$ form a set of compact generators of $\cd\cb$.
\item The class $\cy$ of modules $M\in\cd\ca$ such that $(\cd\ca)(X(?,A)[n],M)=0$, for all $A\in\ca$ and $n\in\Z$, is closed under small coproducts and $(F^{*}\cb)(?,A)\in\cy$ for all $A\in\ca$. 
\end{enumerate}
\end{enumerate}
\end{lemma}
\begin{proof}
The equivalence $1)\Leftrightarrow 2)$ is proved in Lemma \ref{properties of adjunctions}.

$3)\Rightarrow 2)$ can be deduce from the chain of implications $3)\Rightarrow 4)\Rightarrow 2)$ below. However, a much shorter proof is possible by using that $\cd\cb$ satisfies the principle of infinite d\'{e}vissage with respect to the $B^{\we}\ko B\in\cb$ and that the modules $N$ with invertible $\delta_{N}$ form a strictly full triangulated subcategory of $\cd\cb$ closed under small coproducts and containing the $B^{\we}\ko B\in\cb$. 

$3)\Rightarrow 4)$. It is easy to show that if $F^{*}$ is fully faithful, then the objects $(FA)^{\we}\ko A\in\ca$ form a set of (compact) generators. Indeed, if
\[(\cd\cb)((FA)^{\we},N[i])=\H i(N(FA))=\H i((F^{*}N)(A))=0
\]
for each $i\in\Z\ko A\in\ca$, then $F^{*}N=0$, \ie $F^{*}(\id_{N})=0$, and so $N=0$. Now, for an object $A\in\ca$ we get the triangle in $\cd\ca$
\[X(?,A)\ra A^{\we}\arr{F}F^{*}((FA)^{\we})\ra X(?,A)[1].
\]
If we apply $?\otimes^{\bf L}_{\ca}U$ then we get the triangle
\[X(?,A)\otimes^{\bf L}_{\ca}U\ra (FA)^{\we}\ra F^{*}((FA)^{\we})\otimes^{\bf L}_{\ca}U\ra X(?,A)\otimes^{\bf L}_{\ca}U[1]
\]
of $\cd\cb$, which gives the triangle
\[F^{*}(X(?,A)\otimes^{\bf L}_{\ca}U)\ra F^{*}((FA)^{\we})\ra F^{*}(F^{*}((FA)^{\we})\otimes^{\bf L}_{\ca}U)\ra F^{*}(X(?,A)\otimes^{\bf L}_{\ca}U)[1]
\]
of $\cd\ca$. The morphism $F^{*}((FA)^{\we})\ra F^{*}(F^{*}((FA)^{\we})\otimes^{\bf L}_{\ca}U)$ is induced by the unit of the adjunction $(?\otimes^{\bf L}_{\ca}U,F^{*})$, it is the right inverse of $F^{*}(\delta_{(FA)^{\we}})$, and it is an isomorphism if and only if $F^{*}(X(?,A)\otimes^{\bf L}_{\ca}U)=0$. Since $F^{*}$ is fully faithful, then it reflects both isomorphisms and zero objects. Hence, we have that $\delta_{(FA)^{\we}}$ is an isomorphism for every $A\in\ca$ if and only if condition 4.2) holds. 

$4)\Rightarrow 2)$ Condition 4.1) implies that $F^{*}$ reflects both zero objects and isomorphisms. Indeed, thanks to Remark \ref{triangulated categories are weakly balanced}, it suffices to prove that $F^{*}$ reflects zero objects. Now, if $F^{*}N=0$, then
\[\H i((F^{*}N)(A))=(\cd\cb)((FA)^{\we},N[i])=0
\]
for each $i\in\Z\ko A\in\ca$, and so $N=0$. Once we know that $F^{*}$ reflects zero objects and isomorphisms, we can use the argument of the implication $3)\Rightarrow 4)$ to prove that condition 4.2) implies that $\delta_{(FA)^{\we}}$ is an isomorphism for each $A\in\ca$. But condition 4.1) guarantees that the modules $N$ with invertible $\delta_{N}$ form a strictly full triangulated subcategory of $\cd\cb$ closed under small coproducts and containing a set of compact generators. Thus, by infinite d\'{e}vissage, we have that 2) holds.

$4)\Rightarrow 5)$ We know that the essential image $\cy$ of $F^{*}$ is the middle class of a triangulated TTF triple $(\cx,\cy,\cz)$ on $\cd\ca$. Notice that $?\otimes^{\bf L}_{\ca}U$ is left adjoint to $F^{*}$ and so it `is' the torsionfree functor $\tau^{\cy}$ associated to $(\cx,\cy)$. Since $X(?,A)\otimes^{\bf L}_{\ca}U=0$ for each $A\in\ca$, then $\Tria(\{X(?,A)\}_{A\in\ca})\subseteq\cx$. Since $A^{\we}\otimes^{\bf L}_{\ca}X=X(?,A)$ is in $\cx$ and the functor $?\otimes^{\L}_{\ca}X$ preserves small coproducts, then by infinite d\'{e}vissage we have that the essential image of $?\otimes^{\bf L}_{\ca}X$ is contained in $\cx$. Hence, the triangle 
\[X\ra\ca\arr{F}F^{*}\cb\ra X[1]
\]
of $\cd(\ca^{op}\otimes_{k}\ca)$ induces for each $M\in\cd\ca$ a triangle
\[M\otimes^{\bf L}_{\ca}X\ra M\ra F^{*}(M\otimes^{\bf L}_{\ca}U)\ra (M\otimes^{\bf L}_{\ca}X)[1]
\]
of $\cd\ca$ with $M\otimes^{\bf L}_{\ca}X\in\cx$ and $F^{*}(M\otimes^{\bf L}_{\ca}U)\in\cy$. This proves that $\tau_{\cx}(?)=?\otimes^{\bf L}_{\ca}X$, and thus $\Tria(\{X(?,A)\}_{A\in\ca})=\cx$.

$5)\Rightarrow 4)$ We want to prove
\[(\cd\cb)(X(?,A)\otimes^{\bf L}_{\ca}U,X(?,A)\otimes^{\bf L}_{\ca}U)=0\ko
\]
for each $A\in\ca$, that is to say
\[(\cd\ca)(X(?,A),F^{*}(X(?,A)\otimes^{\bf L}_{\ca}U))=0
\]
for each $A\in\ca$ or, equivalently, $F^{*}(X(?,A)\otimes^{\bf L}_{\ca}U)\in\cy$ for each $A\in\ca$. But in fact, $F^{*}(M\otimes^{\bf L}_{\ca}U)\in\cy$ for every $M\in\cd\ca$, as can be proved by infinite d\'{e}vissage since $\cy$ is closed under small coproducts, $F^{*}$ and $?\otimes^{\bf L}_{\ca}U$ preserve small coproducts and $F^{*}(A^{\we}\otimes^{\bf L}_{\ca}U)=F^{*}\cb(?,A)\in\cy$ for each $A\in\ca$.
\end{proof}

\begin{definition}\label{homological epimorphism}
A dg functor $F:\ca\ra\cb$ between small dg categories is a \emph{homological epimorphism}\index{homological epimorphism} if it satisfies the (equivalent) conditions of Lemma \ref{characterization he}. 
\end{definition}

From the proof of Lemma \ref{characterization he} we have that the d\'{e}collement associated to the triangulated TTF induced by a homological epimorphism $F$ is of the form
\[\xymatrix{\cd\cb\ar[rr]^{F^{*}} && \cd\ca\ar@/_1pc/[ll]_{?\otimes^{\bf L}_{\ca}U}\ar@/_-1pc/[ll]^{\textbf{R}\ch om_{\ca}(V,?)}\ar[rr]^{\tau_{\cx}} && \cx\ar@/_1pc/[ll]\ar@/_-1pc/[ll]^x
}
\]
where $x$ is the inclusion functor, $\tau_{\cx}(?)=?\otimes^{\bf L}_{\ca}X$ and $\cx=\Tria(\{X(?,A)\}_{A\in\ca})$. Compare this situation with the one of Theorem \ref{parametrizando TTF ternas}.

\begin{remark}
Our notion of `homological epimorphism of dg categories' is a generalization of the notion of `homological epimorphism of algebras' due to W. Geigle and H. Lenzing \cite{GeigleLenzing}. Indeed, a morphism of algebras $f:A\ra B$ is a \emph{homological epimorphism} if it satisfies:
\begin{enumerate}[1)]
\item the multiplication $B\otimes_{A}B\ra B$ is bijective, 
\item $\op{Tor}^{A}_{i}(B_{A},\ _{A}B)=0$ for every $i\geq 1$. 
\end{enumerate}
But this is equivalent to require that the `multiplication' $B\otimes_{A}^{\bf L}B\ra B$ is an isomorphism in $\cd B$, which is precisely condition 3) of Lemma \ref{characterization he} when we view $A$ and $B$ as dg categories with one object and concentrated in degree $0$ . Hence, our lemma recovers and adds some handy characterizations of homological epimorphisms of algebras. Recently, D. Pauksztello has studied homological epimorphism of dg algebras \cite{Pauksztello2007}.
\end{remark}

The following are particular cases of homological epimorphisms:

\begin{example}\label{quotients which are he}
Let $I$ be a two-sided ideal of an algebra $A$. The following statements are equivalent:
\begin{enumerate}[1)]
\item The canonical projection $A\ra A/I$ is a homological epimorphism.
\item $\op{Tor}^{A}_{i}(I, A/I)=0$ for every $i\geq 0$.
\item The class of complexes $Y\in\cd A$ such that $(\cd A)(I[n],Y)=0$ for every $n\in\Z$ is closed under small coproducts and $\Ext^{i}_{A}(I,A/I)=0$ for every $i\geq 0$.
\item The class of complexes $Y\in\cd A$ such that $(\cd A)(I[n],Y)=0$ for every $n\in\Z$ is closed under small coproducts and $\Ext^{i}_{A}(A/I,A/I)=0$ for every $i\geq 1$.
\end{enumerate}
In this case the d\'{e}collement associated to the triangulated TTF induced by the homological epimorphism $\pi:A\ra A/I$ is of the form
\[\xymatrix{\cd(A/I)\ar[rr]^{\pi_{*}} && \cd A\ar@/_1pc/[ll]_{?\otimes^{\L}_{A}A/I}\ar@/_-1pc/[ll]^{\R Hom_{A}(A/I,?)}\ar[rr]^{?\otimes^{\L}_{A}I} && \Tria(I)\ar@/_1.1pc/[ll]\ar@/_-1pc/[ll]^x
}
\]
The first part of conditions 3) and 4) are always satisfied if $I_{A}$ is compact in $\cd A$, \ie quasi-isomorphic to a bounded complex of finitely generated projective $A$-modules. Note that condition 2) is precisely condition 4) of Lemma \ref{characterization he}. Also, notice that from 
\[0=\op{Tor}^{A}_{0}(I,A/I)=I\otimes_{A}A/I=I/I^2
\] 
follows that $I$ is idempotent. Conversely, if $I$ is idempotent and projective as a right $A$-module (\eg $I=A(1-e)A$ where $e\in A$ is an idempotent such that $eA(1-e)=0$), then condition 2) is satisfied.
\end{example}

The example above contains the unbounded versions of the d\'{e}collements of Corollary 11, Corollary 12 and Corollary 15 of \cite{Konig1991}. 
In Example \ref{restriction in examples} of Chapter \ref{Restriction of TTF triples in triangulated categories} we will show that they restrict to d\'{e}collements of the right bounded derived category $\cd^-A$.  

\begin{example}
Let $j:A\ra B$ be an injective morphism of algebras, which we view as an inclusion. The following statements are equivalent:
\begin{enumerate}[1)]
\item $j$ is a homological epimorphism.
\item $\op{Tor}^{A}_{i}(B/A,B)=0$ for every $i\geq 0$.
\item The class of complexes $Y\in\cd A$ such that $(\cd A)((B/A)[n],Y)=0$ for every $n\in\Z$ is closed under small coproducts and $\Ext^{i}_{A}(B/A,B)=0$ for every $i\geq 0$.
\end{enumerate}
\end{example}

\begin{definition}\label{k flat}
A dg category $\ca$ is \emph{k-flat}\index{category!dg!$k$-flat} if the functor 
\[?\otimes_{k}\ca(A,A'):\cc k\ra\cc k
\]
preserves acyclic complexes of $k$-modules for every $A, A'\in\ca$. Notice that it is always the case if $k$ is a field, since in that case the acyclic complexes of $k$-modules are precisely the contractible ones. 
\end{definition}

\begin{theorem}\label{TTF are he}
Let $\ca$ be a $k$-flat dg category. For every triangulated TTF triple $(\cx,\cy,\cz)$ on $\cd\ca$ there exists a homological epimorphism $F:\ca\ra\cb$ (bijective on objects) such that the essential image of the restriction of scalars functor $F^{*}:\cd\cb\ra\cd\ca$ is $\cy$.
\end{theorem}
\begin{proof}
For each $A\in\ca$ we consider a fixed triangle
\[X_{A}\ra A^{\we}\arr{\varphi_{A}} Y_{A}\ra X[1]
\]
in $\cd\ca$ with $X_{A}\in\cx$ and $Y_{A}\in\cy$. Assume that each $Y_{A}$ is $\ch$-injective. Let $\cc$ be the dg category given by the full subcategory of $\cc_{dg}\ca$ formed by the objects $Y_{A}\ko A\in\ca$. These objects define a $\cc$-$\ca$--bimodule $Y$ as follows:
\[(\cc^{op}\otimes_{k}\ca)^{op}\ra\cc_{dg}k\ko (Y_{A'},A)\mapsto Y(A,Y_{A'}):=Y_{A'}(A),
\]
and the morphisms $\varphi_{A}\ko A\in\ca$ induce a morphism of right dg $\ca$-modules
\[\varphi_{A}: A^{\we}\ra Y(?,Y_{A}).
\]
Let $\xi:Y\ra Y'$ be an $\ch$-injective resolution of $Y$ in $\ch(\cc^{op}\otimes_{k}\ca)$, and let $\cb'$ be the dg category given by the full subcategory of $\cc_{dg}(\cc^{op})$ formed by the objects $Y'(A,?), A\in\ca$. Consider the functor
\[\rho:\ca^{op}\ra\cb',
\]
which takes the object $A$ to $Y'(A,?)$ and the morphism $f\in\ca^{op}(A,A')$ to the morphism $\rho(f)$ defined by
\[\rho(f)(C):Y'(A,C)\ra Y'(A',C)\ko y\mapsto (-1)^{|f||y|}Y'(f,C)(y)
\]
for a homogeneous $f$ of degree $|f|$ and a homogeneous $y$ of degree $|y|$. Since $Y'$ is a $\cc$-$\ca$--bimodule, then the functor $\rho$ is a dg functor. It induces a dg functor between the corresponding opposite dg categories
\[F:\ca\ra\cb'^{op}=:\cb.
\]
Notice that, for each $A'\in\ca$, the functor $(\cc_{dg}\ca)(?,Y_{A'}):\cc\ca\ra\cc k$ induces a triangle functor between the corresponding categories up to homotopy, $(\cc_{dg}\ca)(?,Y_{A'}):\ch\ca\ra\ch k$. Moreover, since $Y_{A'}$ is $\ch$-injective, then this functor induces a triangle functor between the corresponding derived categories
\[(\cc_{dg}\ca)(?,Y_{A'}):\cd\ca\ra\cd k.
\]
When all these functors are applied to the triangles considered above, then we get a family of quasi-isomorphism of complexes of $k$-modules
\[\Psi_{A,A'}:\cc(Y_{A},Y_{A'})\ra Y(A,Y_{A'}),
\]
for each $A, A'\in\ca$. This family underlies a quasi-isomorphism of left dg $\cc$-modules $\Psi_{A,?}:\cc(Y_{A},?)\ra Y(A,?)$ for each $A\in\ca$. Hence, we have a family of quasi-isomorphisms of left dg $\cc$-modules
\[\xi_{A,?}\Psi_{A,?}:\cc(Y_{A},?)\ra Y'(A,?),
\]
for $A\in\ca$. Notice that, since $\ca$ is $k$-flat, for each $A'\in\ca$ the corresponding restriction from $\cc$-$\ca$-bimodules to left dg $\cc$-modules preserves $\ch$-injectives. Indeed, its left adjoint is $?\otimes_{k}\ca$ and preserves acyclic modules. Then for each $A'\in\ca$ the triangle functor
\[\cc_{dg}(\cc^{op})(?,Y'(A',?)):\ch(\cc^{op})\ra\ch k
\]
preserves acyclic modules, and so it induces a triangle functor 
\[\cc_{dg}(\cc^{op})(?,Y'(A',?)):\cd(\cc^{op})\ra\cd k.
\]
When applied to the quasi-isomorphisms of left dg $\cc$-modules $\xi_{A,?}\Psi_{A,?}$, for $A\in\ca$, it gives us quasi-isomorphisms of complexes
\[\cb'(Y'(A,?), Y'(A',?))\ra \cc_{dg}(\cc^{op})((Y_{A})^{\we},Y'(A',?)),
\]
and so quasi-isomorphisms of complexes
\[\varepsilon_{}: \cb'(Y'(A,?),Y'(A',?))\ra Y'(A',Y_{A}),
\]
for each $A,A'\in\ca$. It induces a quasi-isomorphism of right dg $\ca$-modules
\[\varepsilon: F^{*}\cb(?,Y'(A,?))\ra Y'(?,Y_{A})
\]
and we get an isomorphism of triangles
\[\xymatrix{X_{A}\ar[r] & A^{\we}\ar[rr]^{\xi_{?,Y_{A}}\varphi_{A}} && Y'(?,Y_{A})\ar[r] & X[1] \\
X_{A}\ar[r]\ar[u]^{\id} & A^{\we}\ar[rr]^{F}\ar[u]^{\id} && F^{*}\cb(?,Y'(A,?))\ar[u]^{\varepsilon}\ar[r] & X[1]\ar[u]^{\id}
}
\]
Notice that from this it follows that $\tau^{\cy}(A^{\we})=F^{*}\cb(?,Y'(A,?))$ for each $A\in\ca$. 
By infinite d\'{e}vissage of $\cd\cb$ with respect to the set of objects $B^{\we}, B\in\cb$, the functor $F^{*}:\cd\cb\ra\cd\ca$ has its image in $\cy$. We can assume $X_{A}$ to be the restriction $X(?,A)$ of an $\ca$-$\ca$--bimodule $X$ coming from the triangle induced by $F$ in $\cd(\ca^{op}\otimes_{k}\ca)$. Now, by adjunction one proves that the functor $?\otimes^{\bf L}_{\ca}U:\cd\ca\ra\cd\cb$ vanishes on the objects of $\cx$, where $U$ is the $\ca$-$\cb$--bimodule defined by $U(B,A):=\cb(B,FA)$. In particular, it vanishes on $X(?,A)$ and by Lemma \ref{characterization he} we have proved that $F^{*}$ is a homological epimorphism.

Thanks to Lemma \ref{truncating special objects and from local to global via smashing} we know that the objects $F^{*}(B^{\we})\ko B\in\cb$ form a set of compact generators of $\cy$, and so, since $F^{*}$ is fully faithful, the essential image of $F^{*}$ is precisely $\cy$. 

\end{proof}

Notice that, if $F:\ca\ra\cb$ is a homological epimorphism of dg categories, the bimodule $X$ in the triangle
\[X\ra\ca\arr{F}F^{*}\cb\ra X[1]
\]
of $\cd(\ca^{op}\otimes_{k}\ca)$ has a certain `derived idempotency' expressed by condition 4) of Lemma \ref{characterization he}, and plays the same r\^{o}le as the idempotent two-sided ideal in the theory of TTF triples on module categories (\cf Theorem \ref{parametrizando TTF ternas}). 

\begin{remark}
Theorem \ref{TTF are he} is related to subsection 5.4 of K. Br\"{u}ning's Ph. D. thesis \cite{Bruning2007} and section 9 of B. Huber's Ph. D. thesis \cite{Huber2007}. In particular, it generalizes \cite[Corollary 5.4.9]{Bruning2007} and \cite[Corollary 9.8]{Huber2007}. It is also related to B. To\"{e}n's lifting, in the homotopy category of small dg categories up to quasi-equivalences, of \emph{quasi-functors} or \emph{quasi-representable} dg bimodules to genuine dg functors (\cf the proof of \cite[Lemma 4.3]{Toen2005}). \emph{Cf.} also \cite[Lemma 3.2]{Keller1999}.
\end{remark}
\bigskip

\section{Parametrization for derived categories}\label{Parametrization for derived categories}
\addcontentsline{lot}{section}{5.4. Parametrizaci\'on para categor\'ias derivadas}

\begin{definition}\label{equivalent he}
Two homological epimorphisms of dg categories, $F:\ca\ra\cb$ and $F':\ca\ra\cb'$, are \emph{equivalent}\index{homological epimorphism!equivalent} if the essential images of the corresponding restriction functors, $F^{*}:\cd\cb\ra\cd\ca$ and $F'^{*}:\cd\cb'\ra\cd\ca$, are the same subcategory of $\cd\ca$. 
\end{definition}

Thanks to Theorem \ref{TTF are he}, we know that every homological epimorphism $F$ starting in a $k$-flat small dg category is equivalent to another one $F'$ which is bijective on objects. Unfortunately, the path from $F$ to $F'$ is indirect. Nevertheless, there exists a direct way of proving that every homological epimorphism (without any $k$-flatness assumption) is equivalent to another one $F'$ such that $\H0 F'$ is essentially surjective.

\begin{lemma}
Every homological epimorphism $F:\ca\ra\cb$ is equivalent to a homological epimorphism $F':\ca\ra\cb'$ such that $\H 0F':\H 0\ca\ra\H 0\cb'$ is essentially surjective.
\end{lemma}
\begin{proof}
Let $\cb'$ be the full subcategory of $\cb$ formed by the objects $B'$ such that $B'^{\we}\cong(FA)^{\we}$ in $\cd\cb$ (equivalently, $B'\cong FA$ in $\H 0\cb$) for some $A\in\ca$. Put $F':\ca\ra\cb'$ for the dg functor induced by $F$. Now, the restriction $j^{*}:\cd\cb\ra\cd\cb'$ along the inclusion $j:\cb'\ra\cb$ is a triangle equivalence. Indeed, thanks to condition 4) of Lemma \ref{characterization he} we know that the modules $(jB')^{\we}\ko B'\in\cb'$ form a set of compact generators of $\cd\cb$. Moreover, we have
\begin{align}
\H n\cb(B',B'')\cong(\cd\cb)((jB')^{\we},(jB'')^{\we}[n])\underset{\sim}{\arr{j^{*}}}(\cd\cb')(j^{*}(jB')^{\we},j^{*}(jB'')^{\we})\cong \nonumber \\
\cong (\cd\cb)(B'^{\we},B''^{\we}[n])\cong\H n\cb(B',B''). \nonumber
\end{align}
Then Lemma \ref{detecting triangle equivalences} implies that $j^{*}$ is a triangle equivalence. Therefore, the commutative triangle
\[\xymatrix{\cd\cb\ar[dr]^{F^{*}}\ar[d]_{j^{*}}^{\wr} & \\
\cd\cb'\ar[r]_{F'^{*}} & \cd\ca 
}
\]
finishes the proof.
\end{proof}

\begin{lemma}\label{k flat up to quasi-equivalence}
Every compactly generated $k$-linear algebraic triangulated category (whose set of compact generators has cardinality $\lambda$) is triangle equivalent to the derived category of a small $k$-flat dg category (whose set of objects has cardinality $\lambda$).
\end{lemma}
\begin{proof}
Indeed, by B.~Keller's theorem \cite[Theorem 4.3]{Keller1994a} (\cf subsection \ref{B. Keller's Morita theory for derived categories}) we know that it is triangle equivalent to the derived category of a small $k$-linear dg category $\ca$ satisfying this cardinality condition. If $\ca$ is not $k$-flat, then we can consider a cofibrant replacement $F:\ca'\ra\ca$ of $\ca$ in the model structure of the category of small $k$-linear dg categories constructed by G. Tabuada in \cite{Tabuada2005a}, which can be taken to be the identity on objects \cite[Proposition 2.3]{Toen2005}. In particular, $F$ is a quasi-equivalence \cite[subsection 2.3]{Keller2006b}, and so the restriction along $F$ induces a triangle equivalence between the corresponding derived categories \cite[Lemma 3.10]{Keller2006b}. Since $\ca'$ is cofibrant, by \cite[Proposition 2.3]{Toen2005} for all objects $A\ko A'$ in $\ca$ the complex $\ca'(A,A')$ is cofibrant in the category $\cc k$ of complexes over $k$ endowed with its projective model structure \cite[Theorem 2.3.11]{Hovey1999}. This implies that $\ca'(A,A')$ is $\ch$-projective. Finally, since the functor $?\otimes_{k}k$ preserves acyclic complexes and the $\ch$-projective complexes satisfy the principle of infinite d\'{e}vissage with respect to $k$, then $?\otimes_{k}\ca'(A,A')$ preserves acyclic complexes.
\end{proof}

Notice that smashing subcategories of a compactly generated triangulated category $\cd$ form a set. Indeed, \cite[Theorem 5.3]{Keller1994a} implies that the full subcategory $\cd^c$ formed by the compact objects is skeletally small, and we know that smashing subcategories of $\cd$ are in bijection with certain ideals of $\cd^c$ thanks to \cite[Corollary 12.5]{Krause2005}. Now we will give several descriptions of this set in the algebraic case.

\begin{theorem}\label{several descriptions}
Let $\cd$ be a compactly generated algebraic triangulated category, and let $\ca$ be a $k$-flat dg category whose derived category is triangle equivalent to $\cd$. There exists a bijection between:
\begin{enumerate}[1)]
\item Smashing subcategories $\cx$ of $\cd$.
\item Triangulated TTF triples $(\cx,\cy,\cz)$ on $\cd$.
\item (Equivalence classes of) d\'{e}collements of $\cd$.
\item Equivalence classes of homological epimorphisms of dg categories of the form $F:\ca\ra\cb$ (which can be taken to be bijective on objects).
\end{enumerate}
Moreover, if we denote by $\cs$ any of the given (equipotent) sets, then there exists a surjective map $\cR\ra\cs$, where $\cR$ is the class of objects $P\in\cd$ such that $\{P[n]\}^{\bot}_{n\in\Z}$ is closed under small coproducts.
\end{theorem}
\begin{proof}
The bijection between 1) and 2) was proved in Proposition \ref{smashing are TTF}, and the bijection between 2) and 3) was proved in Proposition \ref{TTF triples are recollements} and Proposition \ref{recollements are TTF triples}. The map from 2) to 4) is given by Theorem \ref{TTF are he}, and the map from 4) to 2) was constructed at the beginning of section \ref{Homological epimorphisms of dg categories}. The surjective map $\cR\ra\cs$ follows from Proposition \ref{naive general parametrization}.
\end{proof}

\section{Idempotent two-sided ideals}\label{Idempotent two-sided ideals}
\addcontentsline{lot}{section}{5.5. Ideales bil\'ateros idempotentes}

\subsection{Generalized smashing conjecture: a short survey}\label{Short survey}
\addcontentsline{lot}{subsection}{5.5.1. Conjetura aplastante generalizada: una inspecci\'on breve}

The \emph{generalized smashing conjecture}\index{generalized smashing conjecture} is a generalization to arbitrary compactly generated triangulated categories of a conjecture due to D.~Ravenel \cite[1.33]{Ravenel1984} and, originally, A.~K.~Bousfield \cite[3.4]{Bousfield1979}. It predicts the following:
\bigskip

\noindent Every smashing subcategory of a compactly generated triangulated category satisfies the principle of infinite d\'{e}vissage with respect to a set of compact objects.
\bigskip

However, this conjecture was disproved by B.~Keller in \cite{Keller1994}. More precisely, he proved the following: if $A$ is an algebra and $I$ a two-sided ideal of $A$ contained in the Jacobson radical of $A$ and such that the projection $A\ra A/I$ is a homological epimorphism, then the smashing subcategory $\Tria(I)$ of $\cd A$ contains no non-zero compact object of $\cd A$.

In \cite{Krause2000}, H. Krause gives the definition of being `generated by a class of morphisms':

\begin{definition}
Let $\cd$ be a triangulated category, $\cx$ a strictly full triangulated subcategory of $\cd$ closed under small coproducts and $\ci$ a class of morphisms of $\cd$. We say that $\cx$ is \emph{generated}\index{category!triangulated!generated by a class of morphisms} by $\ci$ if $\cx$ is the smallest full triangulated subcategory of $\cd$ closed under small coproducts and such that every morphism in $\ci$ factors through some object of $\cx$.
\end{definition}

This is a generalization of the notion of infinite d\'{e}vissage (\cf Definition \ref{infinite devissage}), at least when the existence of certain countable coproducts is guaranteed. Indeed:

\begin{lemma}
Let $\cd$ and $\cx$ be as before, let $\cq$ be a class of objects of $\cd$ and let $\ci$ be the class of identity morphisms $\id_{Q}$ where $Q$ runs through $\cq$. Assume that $\cd$ has countable coproducts. Then, every morphism of $\ci$ factors through an object of $\cx$ if and only if every object of $\cq$ is in $\cx$. In particular, $\cx$ satisfies the principle of infinite d\'{e}vissage with respect to $\cq$ if and only if $\cx$ is generated by $\ci$. 
\end{lemma}
\begin{proof}
Let $Q$ be an object of $\cq$, and assume that the identity morphism $\id_{Q}$ factors through an object $X$ of $\cx$:
\[\xymatrix{Q\ar[dr]_{f}\ar[rr]^{\id_{Q}} & & Q \\
&X\ar[ur]_{g}&
}
\]
We have to prove that $Q$ belongs to $\cx$. For this, notice that both $fg=e$ and $\id-e$ are idempotent morphisms. Thanks to \cite[Proposition 1.6.8]{Neeman2001}, we know that there exist morphisms $f':P\ra X$ and $g':X\ra P$ such that $\id-e=f'g'$ and $g'f'=\id$, where $P$ can be taken to be the Milnor colimit of the sequence
\[X\arr{\id-e}X\arr{\id-e}X\ra\dots
\]
Then, the compositions
\[\left[\begin{array}{cc}f&f'\end{array}\right]\left[\begin{array}{c}g\\ g'\end{array}\right]:X\ra Q\oplus P\ra X
\]
and
\[\left[\begin{array}{c}g\\ g'\end{array}\right]\left[\begin{array}{cc}f&f'\end{array}\right]:Q\oplus P\ra X\ra Q\oplus P
\]
are the corresponding identity morphisms. Indeed, $gf'=g(1-e)f'=g(1-fg)f'=(g-g)f'=0$ and $gf'=g'ef=g'(1-f'g')f=(g'-g')f=0$. Thus, $X\cong Q\oplus P$ and in the triangle of $\cd$
\[Q\arr{\scriptsize{\left[\begin{array}{cc}\id \\ 0\end{array}\right]}} Q\oplus P\arr{\scriptsize{\left[\begin{array}{c}0 \\ \id\end{array}\right]^t}} P\arr{0} Q[1]
\]
both $P$ and $Q\oplus P$ belongs to $\cx$. This implies that $Q$ belongs to $\cx$.
\end{proof}

This allows H. Krause the following reformulation of the generalized smashing conjecture: 
\bigskip

\noindent Every smashing subcategory of a compactly generated triangulated category is generated by a set of identity morphisms between compact objects.
\bigskip

This reformulation enables at least two weak versions of the conjecture, one of which was proved by H. Krause \cite[Corollary 4.7]{Krause2000}:
\emph{Every smashing subcategory of a compactly generated triangulated category is generated by a set of morphisms between compact objects.}

In Corollary \ref{gsc} below we prove `the other' weak version of the conjecture, which substitutes ``morphisms between compact objects'' by ``identity morphisms'' of Milnor colimits of compact objects.

\subsection{From ideals to smashing subcategories}\label{From ideals to smashing subcategories}
\addcontentsline{lot}{subsection}{5.5.2. De ideales a subcategor\'ias aplastantes}

If $\cd$ is a triangulated category, we denote by $\cd^c$\index{$\cd^c$} the full subcategory of $\cd$ formed by the compact objects and by $\cm or(\cd^c)$\index{$\cm or(\cd^c)$} the class of morphisms of $\cd^c$. If $\cd$ is the derived category $\cd\ca$ of a dg category $\ca$, then we also write $\cd^c\ca$\index{$\cd^c\ca$} to refer to $\cd^c$.

We write $\cp(\cm or(\cd^c))$\index{$\cp(\cm or(\cd^c))$} for the large complete lattice of subsets of $\cm or(\cd^c)$, where the order is given by the inclusion. Notice that if $\cd$ is compactly generated and we form the lattice with a skeleton of $\cd^c$ instead of $\cd^c$ itself, then we get a small lattice. Also, we write $\cp(\cd)$\index{$\cp(\cd)$} for the large complete lattice of classes of objects of $\cd$, where the order here is also given by the inclusion. Consider the following Galois connection (\cf \cite[section III.8]{Stenstrom} for the definition and basic properties of Galois connections):
\[\xymatrix{\cp(\cm or(\cd^c))\ar@<1ex>[rr]^{?^{\bot}} && \cp(\cd)\ar@<1ex>[ll]^{\cm or(\cd^c)^{?}},
}
\]
where given $\ci\in\cp(\cm or(\cd^c))$ we define $\ci^{\bot}$\index{$\ci^{\bot}$} to be the class of objects $Y$ of $\cd$ such that $\cd(f,Y)=0$ for every morphism $f$ of $\ci$, and given $\cy\in\cp(\cd)$ we define $\cm or(\cd^c)^{\cy}$\index{$\cm or(\cd^c)^{\cy}$} to be the class of morphisms $f$ of $\cd^c$ such that $\cd(f,Y)=0$ for every $Y$ in $\cy$.

\begin{definition}\label{closed ideals}
According to the usual terminology in the theory of Galois connections, we say that a subset $\ci$ of $\cm or(\cd^c)$ is \emph{closed}\index{closed!subset of morphisms} if $\ci=\cm or(\cd^c)^{\ci^{\bot}}$, and a class $\cy$ of objects of $\cd$ is \emph{closed}\index{closed!class of objects} if $\cy=(\cm or(\cd^c)^{\cy})^{\bot}$.
\end{definition}

One of the basic properties of Galois connections says that:

\begin{lemma}\label{basic properties Galois}
The following descriptions of the closed elements of the Galois connection hold:
\begin{enumerate}[1)]
\item A subset $\ci$ of $\cm or(\cd^c)$ is closed if and only if $\ci=\cm or(\cd^c)^{\cy}$ for some class $\cy$ of objects of $\cd$.
\item A class $\cy$ of objects of $\cd$ is closed if and only if $\cy=\ci^{\bot}$ for some set $\ci$ of morphisms of $\cd^c$.
\end{enumerate}
\end{lemma}

If $\ci$ is a subset of $\cm or(\cd^c)$, we write $\ci[1]$\index{$\ci[1]$} for the subset formed by all the morphisms of the form $f[1]$ with $f$ in $\ci$.

\begin{lemma}
The Galois connection above induces a bijection between the class formed by the closed ideals $\ci$ of $\cd^c$ such that $\ci[1]=\ci$ and the class formed by the full triangulated subcategories of $\cd$ closed under small coproducts which are closed for the Galois connection.
\end{lemma}
\begin{proof}
It is well-known that any Galois connection induces a pair of mutually inverse anti-isomorphisms between the corresponding complete lattices of closed elements. Notice that, in our situation, the closed elements of $\cp(\cm or(\cd^c))$ are ideals of the additive category $\cd^c$, and the closed elements of $\cp(\cd)$ are classes closed under small coproducts and extensions. Moreover, if $\ci$ is an ideal of $\cd^c$ such that $\ci[1]=\ci$, then $\ci^{\bot}$ is closed under shifts. Conversely, if $\cy$ is a full triangulated subcategory of $\cd$, then $\cm or(\cd^c)^{\cy}[1]=\cm or(\cd^c)^{\cy}$.
\end{proof}

\begin{definition}\label{saturated ideal}
A two-sided ideal $\ci$ of a triangulated category $\cd$ is \emph{saturated}\index{ideal!saturated ... of a triangulated category} if whenever there exists a triangle
\[P'\arr{u} P\arr{v} P''\ra P'[1]
\]
in $\cd$ and a morphism $f\in\cd(P,Q)$ with $fu\ko v\in\ci$, then $f\in\ci$.
\end{definition}

\begin{lemma}\label{closed are saturated}
Let $\cd$ be a triangulated category. Every closed two-sided ideal $\ci$ of $\cd^c$ is saturated. 
\end{lemma}
\begin{proof}
Indeed, let $\ci=\cm or(\cd^c)^{\cy}$ be an ideal of $\cd^c$,
\[P'\arr{u} P\arr{v} P''\ra P'[1]
\]
a triangle in $\cd^c$ with $v\in\ci$ and $f\in\cd^c(P,Q)$ a morphism such that $fu\in\ci$. We have to prove that $f\in\ci$, \ie that $gf=0$ for each morphism $g\in\cd(Q,Y)$ with $Y\in\cy$. But, since $fu\in\ci$, then $gfu=0$, and so there exists a morphism $h\in\cd(P'',Y)$ such that $hv=gf$:
\[\xymatrix{P'\ar[r]^{u}\ar[ddr]_{0} & P\ar[r]^{v}\ar[d]^f & P''\ar[r]^{w}\ar@{.>}[ddl]^{h} & P'[1] \\
& Q\ar[d]^g && \\
& Y &&
}
\]
Now, since $v\in\ci$, then $hv=0$, \ie $gf=0$. 
\end{proof}

In the following result we explain how a certain two-sided ideal induces a nice smashing subcategory or, equivalently, a triangulated TTF triple.

\begin{theorem}\label{from ideals to devissage wrt hc}
Let $\cd$ be a compactly generated triangulated category and let $\ci$ be an idempotent two-sided ideal of $\cd^c$ with $\ci[1]=\ci$. There exists a triangulated TTF triple $(\cx,\cy,\cz)$ on $\cd$ such that:
\begin{enumerate}[1)]
\item $\cx=\Tria(\cp)$, for a certain set $\cp$ of Milnor colimits of sequences of morphisms of $\ci$.
\item $\cy=\ci^{\bot}$.
\item A morphism of $\cd^c$ belongs to $\ci$ if and only if it factors through an object of $\cp$.
\end{enumerate}
\end{theorem}
\begin{proof}
Theorem 5.3 of \cite{Keller1994a} implies that $\cd^c$ is skeletally small. Fix a small skeleton of $\cd^c$ closed under shifts. Let $\cp$ be the set of Milnor colimits of all the sequences of morphisms of $\ci$ between objects of the fixed skeleton. Remark that $\cp$ is closed under shifts. Put $(\cx,\cy,\cz):=(\Tria(\cp),\Tria(\cp)^{\bot},\Tria(\cp)^{\bot\bot})$. 

\emph{First step: $\Tria(\cp)^{\bot}=\ci^{\bot}$.} Notice that $\Tria(\cp)^{\bot}$ is the class of those objects which are right orthogonal to all the objects of $\cp$. Hence, for the inclusion $\ci^{\bot}\subseteq\Tria(\cp)^{\bot}$, it suffices to prove $\cd(P,Y)=0$ for each $P\in\cp$ and $Y\in\ci^{\bot}$. For this, let
\[P_{0}\arr{f_{0}}P_{1}\arr{f_{1}}P_{2}\ra\dots
\]
be a sequence of morphisms in $\ci$ between objects of the fixed skeleton and consider the corresponding Milnor triangle
\[\coprod_{i\geq 0}P_{i}\arr{\id-\sigma}\coprod_{i\geq 0}P_{i}\ra P\ra\coprod_{i\geq 0}P_{i}[1].
\]
Since $\cd(\sigma,Y[n])=0$ for each $n\in\Z$, we get a long exact sequence
\[\dots \arr{\id}\cd(\coprod_{i\geq 0}P_{i}[1],Y)\ra\cd(P,Y)\ra\cd(\coprod_{i\geq 0}P_{i},Y)\arr{\id}\dots 
\]
which proves that $\cd(P,Y)=0$. Conversely, let $Y\in\Tria(\cp)^{\bot}$ and consider a morphism $f:P\ra P'$ of $\ci$. Put $P_{0}:=P$ and, by using the idempotency of $\ci$, consider a factorization of $f$
\[\xymatrix{P_{0}\ar[rr]^f\ar[dr]_{f_{0}} && P' \\
& P_{1}\ar[ur]_{g_{1}}
}
\]
with $f_{0}\ko g_{1}\in\ci$. We can consider a similar factorization for $g_{1}$, and proceeding inductively we can produce a sequence of morphisms of $\ci$
\[P_{0}\arr{f_{0}} P_{1}\arr{f_{1}} P_{2}\arr{f_{2}}\dots
\]
together with morphisms $g_{i}:P_{i}\ra P'$ of $\ci$ satisfying $g_{i}f_{i-1}=g_{i-1}\ko i\geq 1$ with $g_{0}:=f$. 
\[\xymatrix{P=P_{0}\ar[r]^{f_{0}}\ar[dr]_{f=g_{0}} & P_{1}\ar[r]^{f_{1}}\ar[d]^{g_{1}} & P_{2}\ar[r]^{f_{2}}\ar[dl]^{g_{2}} & \dots \\
& P' &&
}
\]
This induces a factorization of $f$
\[f:P\ra \Mcolim P_{n}\ra P'
\]
Since $\Mcolim P_{n}$ is isomorphic to an object of $\cp$, then $(\cd\ca)(f,Y)=0$. This proves that $\Tria(\cp)^{\bot}=\ci^{\bot}$.

\emph{Second step: $(\cx,\cy,\cz)$ is a triangulated TTF triple.} Thanks to Proposition \ref{smashing are TTF} it suffices to prove that $\cx$ is a smashing subcategory of $\cd$, \ie that $\cx$ is an aisle in $\cd$ and $\cy$ is closed under small coproducts. The fact that $\cx$ is an aisle follows from \cite[Corollary 3.12]{Porta2007}. The fact that $\cy$ is closed under small coproducts follows from the first step, since the morphisms of $\ci$ are morphisms between \emph{compact} objects.

\emph{Third step: part 3)}. Notice that in the proof of the inclusion $\Tria(\cp)^{\bot}\subseteq\ci^{\bot}$ we have showed that every morphism of $\ci$ factors through an object of $\cp$. The converse is also true. Indeed, let $g:Q'\ra Q$ be a morphism between compact objects factoring through an object $P$ of $\cp$:
\[\xymatrix{& Q'\ar[d]^{g}\ar@{.>}[dl]_{h} \\
P\ar[r] & Q
}
\]
 By definition of $\cp$, we have that $P$ is the Milnor colimit of a sequence of morphisms of $\ci$:
\[P=\Mcolim(P_{0}\arr{f_{0}}P_{1}\arr{f_{1}}P_{2}\ra\dots)
\]
Now compactness of $Q'$ implies (see Lemma \ref{compact and Milnor}) that $h$ factors through a certain $P_{n}$:
\[\xymatrix{&& Q'\ar[d]^{g}\ar@{.>}[dll]_{h_{n}} \\
P_{n}\ar[r]_{\pi_{n}}& P\ar[r] & Q
}
\]
where $\pi_{n}$ is the $n$th component of the morphism $\pi$ appearing in the Milnor triangle defining $P$ (see Definition \ref{Milnor colimit}). One of the properties satisfied by the components of $\pi$ is the identity: $\pi_{m+1}f_{m}=\pi_{m}$ for each $m\geq 0$. This gives the following factorization for $g$:
\[\xymatrix{&&& Q'\ar[d]^{g}\ar@{.>}[dlll]_{h_{n}} \\
P_{n}\ar[r]_{f_{n}}&P_{n+1}\ar[r]_{\pi_{n+1}}& P\ar[r] & Q
}
\]
Since $f_{n}$ belongs to the ideal $\ci$, so does $g$. 
\end{proof}

\begin{remark}
Notice that if $\ci=\cm or(\cd^c)$ then $\cx=\cd$. Indeed, in this case $\ci$ contains all the identity morphisms $\id_{P}$ of compact objects $P$. In particular, since $P$ is the Milnor colimit of the sequence
\[P\arr{\id_{P}}P\arr{\id_{P}}P\ra\dots,
\]
we have that $\cp$ contains all the compact objects. Therefore $\Tria(\cp)^{\bot}=\{0\}$ and so $\Tria(\cp)=\cd$. 
\end{remark}

\subsection{Stretching a filtration}\label{Stretching a filtration}
\addcontentsline{lot}{subsection}{5.5.3. Estirando una filtraci\'on}

The results of this subsection are to be used in the proof of Proposition \ref{for the idempotency} and Corollary \ref{Krause algebraic}. However, we believe that they are interesting by themselves. Throughout this subsection $\ca$ will be a small dg category.

We define $\cs$ to be the set of right dg $\ca$-modules $S$ admitting a finite filtration
\[0=S_{-1}\subset S_{0}\subset S_{1}\subset\dots S_{n}=S
\]
in $\cc\ca$ such that
\begin{enumerate}[1)]
\item the inclusion morphism $S_{p-1}\subset S_{p}$ is an inflation for each $0\leq p\leq n$,
\item the factor $S_{p}/S_{p-1}$ is isomorphic in $\cc\ca$ to a \emph{relatively free module of finite type}\index{module!relatively free ... of finite type} (\ie, it is a finite direct sum of modules of the form $A^{\we}[i]\ko A\in\ca\ko i\in\Z$) for each $0\leq p\leq n$.
\end{enumerate}

\begin{lemma}\label{good properties of s}
\begin{enumerate}[1)]
\item The compact objects of $\cd\ca$ are precisely the direct summands of objects of $\cs$.
\item For each $S\in\cs$ the functor $(\cd\ca)(S,q?):\cc\ca\ra\Mod k$ preserves direct limits, where $q:\cc\ca\ra\cd\ca$ is the canonical localization functor.
\item For each $S\in\cs$ the functor $(\cc\ca)(S,?):\cc\ca\ra\Mod k$ preserves direct limits. 
\end{enumerate}
\end{lemma}
\begin{proof}
1) We know by \cite[Theorem 5.3]{Keller1994a} that any compact object $P$ of $\cd\ca$ is a direct summand of a finite extension of objects of the form $A^{\we}[n]\ko A\in\ca\ko n\in\Z$. By using that every triangle of $\cd\ca$ is isomorphic to a triangle coming from a conflation of $\cc\ca$, we have that this kind of finite extensions are objects isomorphic in $\cd\ca$ to objects of $\cs$. 

2) Let $I$ be a directed set. We regard it as a category and denote by $\Fun(I,?)$ the category of functors starting in $I$. We want to prove that, for every object $S\in\cs$, the following square is commutative:
\[\xymatrix{\Fun(I,\cc\ca)\ar[r]^{\lid}\ar[d]_{(I,(\cd\ca)(S,q?))} & \cc\ca\ar[d]^{(\cd\ca)(S,q?)}\\
\Fun(I,\Mod k)\ar[r]^{\lid}&\Mod k
}
\]
\emph{First step:}
We first prove it for $S=A^{\we}\ko A\in\ca$. That will prove assertion 2) for relatively free dg $\ca$-modules of finite type.
For an object $A$ of $\ca$ we consider the dg functor $F:k\ra\ca$, with $F(k)=A$ and $F(1_{k})=\id_{A}$. The restriction
\[F^{*}:\cc\ca\ra\cc k\ko M\mapsto M(A)
\]
along $F$ admits a right adjoint, and so it preserves colimits. Consider the commutative diagram
\[\xymatrix{\Fun(I,\cc\ca)\ar[rr]^{\lid}\ar[d]^{(I,F^{*})}\ar@/_3pc/[dd]_{(I,(\cd\ca)(A^{\we},q?))} && \cc\ca\ar[d]_{F^{*}}\ar@/^3pc/[dd]^{(\cd\ca)(A^{\we},q?)} \\
\Fun(I,\cc k)\ar[rr]^{\lid}\ar[d]^{(I,\H 0)} && \cc k\ar[d]_{\H 0}  \\
\Fun(I,\Mod k)\ar[rr]_{\lid}&&\Mod k
}
\]
Since $F^{*}$ preserves colimits, the top rectangle commutes. Since 
\[\lid:\Fun(I,\Mod k)\ra\Mod k
\] 
is exact, the bottom rectangle commutes. Indeed,
\begin{align}
\lid\H 0M_{i}=\lid\cok(\Bo 0M_{i}\ra \Zy 0M_{i})=\cok(\lid(\Bo 0M_{i}\ra \Zy 0M_{i}))= \nonumber \\
=\cok(\lid\Bo 0M_{i}\ra\lid\Zy 0M_{i})=\cok(\Bo 0(\lid M_{i})\ra\Zy 0(\lid M_{i}))=\H 0\lid M_{i}. \nonumber
\end{align}

\emph{Second step:} Let $S'\ra S\ra S''$ be a conflation of $\cc\ca$ such that $S'$ and $S''$ (and hence all their shifts) satisfy property 2), and let $M\in\Fun(I,\cc\ca)$ be a direct system.
By using the cohomological functors $(\cd\ca)(?,M_{i})\ko (\cd\ca)(?,\lid M_{i})$ and the fact that $\lid:\Fun(I,\Mod k)\ra\Mod k$ is exact, we have a morphism of long exact sequences
\[\xymatrix{\dots\ar[d] & \dots\ar[d] \\
\lid(\cd\ca)(S'', M_{i})\ar[r]^{\sim}\ar[d] & (\cd\ca)(S'', \lid M_{i})\ar[d] \\
\lid(\cd\ca)(S, M_{i})\ar[r]\ar[d] & (\cd\ca)(S, \lid M_{i})\ar[d] \\
\lid(\cd\ca)(S', M_{i})\ar[r]^{\sim}\ar[d] & (\cd\ca)(S', \lid M_{i})\ar[d] \\
\dots & \dots
}
\]
And so the central map is an isomorphism thanks to the five lemma.

3) By using the same techniques as in 2), one proves that relatively free dg modules of finite type have the required property. Now, let $S'\arr{j}S\arr{p}S''$ be a conflation such that $S'$ satisfies the required property and $S''$ is relatively free of finite type. For each $N\in\cc\ca$, this conflation gives an exact sequence
\[0\ra(\cc\ca)(S'',N)\arr{p^{\che}}(\cc\ca)(S,N)\arr{j^{\che}}(\cc\ca)(S',N).
\]
Let $u\in(\cc\ca)(S''[-1],S')$ be a morphism whose mapping cone is $\cone(u)=S$. By using the triangle 
\[S''[-1]\arr{u}S'\arr{j}S\arr{p}S''
\]
and the canonical localization $q:\cc\ca\ra\cd\ca$ we can fit the sequence above in the following commutative diagram
\[\xymatrix{0\ar[d] & \\
(\cc\ca)(S'',N)\ar[d]^{p^{\che}} & \\
(\cc\ca)(S,N)\ar[d]^{j^{\che}}\ar@{->>}[r]^{q} & (\cd\ca)(S,N)\ar[d]^{j^{\che}} \\
(\cc\ca)(S',N)\ar[d]^{\phi}\ar@{->>}[r]^{q} & (\cd\ca)(S',N)\ar[d]^{u^{\che}} \\
(\cd\ca)(S''[-1],N)\ar[r]^{\id} & (\cd\ca)(S''[-1],N)
}
\]
The aim is to prove that the left-hand column is exact. Indeed, $\varphi j^{\che}=0$ since $u^{\che}j^{\che}=0$.
It only remains to prove that $\ker\varphi$ is contained in $\im(j^{\che})$. For this, let $f\in(\cc\ca)(S',N)$ be such that $\varphi(f)=fu$ vanishes in 
\[(\cd\ca)(S''[-1],N)\underset{\sim}{\leftarrow}(\ch\ca)(S''[-1],N),
\] 
\ie $fu$ is null-homotopic. By considering the triangle above, one sees that $f$ factor through $j$ in $\ch\ca$, \ie there exists $g\in(\cc\ca)(S,N)$ such that $f-gj$ is null-homotopic. Then, there exists $h\in(\cc\ca)(IS',N)$ such that $f=gj+hi_{S'}$, where $i_{S'}:S'\ra IS'$ is an inflation to an injective module (for the graded-wise split exact structure of $\cc\ca$). Since $j$ is an inflation, then there exists $g'$ such that $g'j=i_{S'}$. Therefore, $f=gj+hi_{S'}=(g+hg')j$ belongs to $\im(j^{\che})$.

Finally, let $M\in\op{Fun}(I,\cc\ca)$ be a direct system. Using assertion 2) and the hypothesis on $S'$ and $S''$, we have a morphism of exact sequences:
\[\xymatrix{0\ar[d] & 0\ar[d] \\
\lid(\cc\ca)(S'',M_{i})\ar[r]^{\sim}\ar[d] & (\cc\ca)(S'',\lid M_{i})\ar[d] \\
\lid(\cc\ca)(S,M_{i})\ar[r]\ar[d] & (\cc\ca)(S,\lid M_{i})\ar[d] \\
\lid(\cc\ca)(S',M_{i})\ar[r]^{\sim}\ar[d] & (\cc\ca)(S',\lid M_{i})\ar[d] \\
\lid(\cd\ca)(S''[-1],M_{i})\ar[r]^{\sim} & (\cd\ca)(S''[-1],\lid M_{i})
}
\]
where the horizontal arrows are the natural ones. Hence, the second horizontal arrow is an isomorphism.
\end{proof}

\begin{proposition}\label{stretching}
Every $\ch$-projective right dg $\ca$-module is, up to isomorphism in $\ch\ca$, the colimit in $\cc\ca$ of a direct system of submodules $S_{i}\ko i\in I$ such that:
\begin{enumerate}[1)]
\item $S_{i}\in\cs$ for each $i\in I$,
\item for each $i\leq j$ the morphism $\mu^{i}_{j}:S_{i}\ra S_{j}$ is an inflation.
\end{enumerate}
\end{proposition}
\begin{proof}
\emph{First step:} Assume that an object $P$ is the colimit of a direct system $P_{t}\ko t\in T$ of subobjects such that the structure morphisms $\mu^r_{t}: P_{r}\ra P_{t}$ are inflations and each $P_{t}$ is the colimit of a direct system $S_{(t,i)}\ko i\in I_{t}$ of subobjects satisfying the following property $(*)$:
\begin{enumerate}[1)]
\item $S_{(t,i)}\in\cs$ for each $i\in I_{t}$,
\item for each $i$ the morphism $\mu_{(t,i)}: S_{(t,i)}\ra P_{t}$ is an inflation.
\end{enumerate}
Notice that 2) implies that for each $i\leq j$ the structure morphism $\mu^{i}_{j}:S_{(t,i)}\ra S_{(t,j)}$ is an inflation.

Put $I:=\bigcup_{t\in T}(\{t\}\times I_{t})$ and define the following preorder: $(r,i)\leq (t,j)$ if $r\leq t$ and we have a factorization as follows
\[\xymatrix{S_{(r,i)}\ar@{.>}[rr]^{\mu^{(r,i)}_{(t,j)}}\ar[d]_{\mu_{(r,i)}} && S_{(t,j)}\ar[d]^{\mu_{(t,j)}} \\
P_{r}\ar[rr]_{\mu^{r}_{t}} && P_{t}
}
\]
Notice that $\mu^{(r,i)}_{(t,j)}$ is an inflation. To prove that $I$ is a directed preordered set take $(r,i)\ko (t,j)$ in $I$ and let $s\in T$ with $r\ko t\leq s$. Since $S_{(r,i)}$ and $S_{(t,j)}$ are in $\cs$, thanks to Lemma \ref{good properties of s} we know that the compositions $S_{(r,i)}\ra P_{r}\ra P_{s}$ and $S_{(t,j)}\ra P_{t}\ra P_{s}$ factors through $S_{(s,i_{r})}\ra P_{s}$ and $S_{(s,j_{t})}\ra P_{s}$. Now, if $i_{r}\ko j_{t}\leq k$ we have that $(r,i)\ko (t,j)\leq (s,k)$. Now, take the quotient set $I':=I/\sim$, where $(r,i)\sim (r',i')$ when $(r,i)\leq (r',i')\leq (r,i)$. Notice that in this case $r=r'$ and the inflations $\mu^{(r,i)}_{(r,i')}$ and $\mu^{(r,i')}_{(r,i)}$ are mutually inverse. Thus, $S_{(t,i)}\ko [(t,i)]\in I'$ is a direct system of subobjects of $P$ which are in $\cs$ and whose colimit is easily seen to be $P$.

\emph{Second step:} Let $P$ be an $\ch$-projective module. The proof of \cite[Theorem 3.1]{Keller1994a} shows that, up to replacing $P$ by a module isomorphic to it in $\ch\ca$, we can consider a filtration
\[0=P_{-1}\subset P_{0}\subset P_{1}\subset\dots P_{n}\subset P_{n+1}\dots\subset P\ko n\in\N
\]
such that
\begin{enumerate}[1)]
\item $P$ is the union of the $P_{n}\ko n\in\N$,
\item the inclusion morphism $P_{n-1}\subset P_{n}$ is an inflation for each $n\in\N$,
\item the factor $P_{n}/P_{n-1}$ is isomorphic in $\cc\ca$ to a \emph{relatively free module} (\ie, a direct sum of modules of the form $A^{\we}[i]\ko A\in\ca\ko i\in\Z$) for each $n\in\N$.
\end{enumerate}
Thanks to the first step, it suffices to prove that each $P_{n}$ is the colimit of a direct system of subobjects satisfying $(*)$. We will prove it inductively.

\emph{Third step:} Notice that $P_{n}=\cone(f)$ for a morphism $f:L\ra P_{n-1}$, where $L=\bigoplus_{I}A^{\we}_{i}[n_{i}]\ko A_{i}\in\ca\ko n_{i}\in\Z$. By hypothesis of induction, $P_{n-1}=\lid S_{j}$ where $S_{j}\ko j\in J$ is a direct system satisfying $(*)$. Let $\cF\cp(I)$ be the set of finite subsets of $I$, and put $L_{F}=\bigoplus_{i\in F}A^{\we}_{i}[n_{i}]$ for $F\in\cF\cp(I)$. Notice that $\cF\cp(I)$ is a directed set with the inclusion, and that $L=\lid L_{F}$. Consider the set $\Omega$ of pairs $(F,j)\in\cF\cp(I)\times J$ such that there exists a morphism $f_{(F,j)}$ making the following diagram commutative:
\[\xymatrix{L_{F}\ar[r]^{f_{(F,j)}}\ar[d]_{u_{F}} & S_{j}\ar[d]^{\mu_{j}} \\
L\ar[r]_{f} & P_{n-1}
}
\]
$\Omega$ is a directed set with the order: $(F,j)\leq (F',j')$ if and only if $F\subseteq F'$ and $j\leq j'$. Let $\mu^j_{j'}:S_{j}\ra S_{j'}$ and $u^F_{F'}:L_{F}\ra L_{F'}$ be the structure morphisms of the direct systems $S_{j}\ko j\in J$ and $L_{F}\ko F\in\cF\cp(I)$. Then one can check that $\cone(f_{(F,j)})\ko (F,j)\in\Omega$ is a direct system of modules, with structure morphisms
\[\left[\begin{array}{cc}\mu^j_{j'}& 0 \\ 0 & u^F_{F'}[1]\end{array}\right]:\cone(f_{(F,j)})\ra\cone(f_{(F',j')}),
\]
and whose colimit is $\cone(f)=P_{n}$ via the morphisms
\[\left[\begin{array}{cc}\mu_{j}& 0 \\ 0 & u_{F}[1]\end{array}\right]:\cone(f_{(F,j)})\ra\cone(f)=P_{n},
\]
which are inflations. Finally, notice that, since there exists a conflation
\[S_{j}\ra\cone(f_{(F,j)})\ra L_{F}[1]
\]
with $S_{j}\in\cs$ and $L_{F}[1]$ relatively free of finite type, then each $\cone(f_{(F,j)})$ is in $\cs$.
\end{proof}

\subsection{From smashing subcategories to ideals}\label{From smashing subcategories to ideals}
\addcontentsline{lot}{subsection}{5.5.4. De subcategor\'ias aplastantes a ideales}

Let $\cd$ be a compactly generated triangulated category. 
H. Krause gives in \cite[Corollary 12.5]{Krause2005} a bijection between the class of \emph{exact} (\cf \cite[Definition 8.1]{Krause2005}) two-sided ideals $\ci$ of $\cd^c$ and the class of smashing subcategories $\cx$ of $\cd$. This bijection sends an exact ideal $\ci$ to the smashing subcategory $\mbox{Filt}(\ci)$ 
formed by those objects $X$ of $\cd$ such that every morphism $P\ra X$, with $P$ compact, factors through some morphism of $\ci$. Conversely, the bijection sends a smashing subcategory $\cx$ to the ideal $\ci_{\cx}$ formed by those morphisms of $\cd^c$ which factor through an object of $\cx$.

The definition of \emph{exact ideal} is quite abstract, since it appeals to the existence of a so-called \emph{cohomological quotient functor} (\cf \cite[Definition 4.1]{Krause2005}). However, after \cite[Theorem 11.1 and Theorem 12.1]{Krause2005}, a nice characterization of the notion of exact ideal is available in \cite[Corollary 12.6]{Krause2005}. Namely, a two-sided ideal $\ci$ of $\cd^c$ is exact if and only if it is idempotent, saturated and satisfies $\ci[1]=\ci$. On the other hand, thanks to Proposition \ref{smashing are TTF}, we know that we can replace in this bijection ``smashing subcategories $\cx$ of $\cd$'' by ``triangulated TTF triples $(\cx,\cy,\cz)$ on $\cd$''. After all these remarks, we have that:
\begin{enumerate}[1)]
\item A saturated idempotent two-sided ideal $\ci[1]=\ci$ of $\cd^c$ is mapped to 
\[(\mbox{Filt}(\ci),\mbox{Filt}(\ci)^{\bot},\mbox{Filt}(\ci)^{\bot\bot}).
\]
\item A triangulated TTF triple $(\cx,\cy,\cz)$ on $\cd$ is mapped to $\ci_{\cx}$.
\end{enumerate}

Now we use:

\begin{lemma}\label{equivalent definitions ideals}
Let $(\cx,\cy[1])$ be a t-structure on a triangulated category $\cd$. For a morphism $f\in\cd(M,N)$ the following statements are equivalent:
\begin{enumerate}[1)]
\item $f$ factors through an object of $\cx$.
\item $\cd(f,Y)=0$ for every object $Y$ of $\cy$.
\end{enumerate}
\end{lemma}
\begin{proof}
$1)\Rightarrow 2)$ Use that $\cy=\cx^{\bot}$.

$2)\Rightarrow 1)$ Consider the triangle
\[x\tau_{\cx}N\ra N\ra y\tau^{\cy}N\ra(x\tau_{\cx}N)[1].
\]
Since the composition $M\arr{f}N\ra y\tau^{\cy}N$ vanishes, then $f$ factors through $x\tau_{\cx}N\ra N$.
\end{proof}

Thanks to Lemma \ref{equivalent definitions ideals}, we have that $\ci_{\cx}=\cm or(\cd^c)^{\cy}$. Finally, thanks to \cite[Theorem 12.1]{Krause2005}, we have that $\mbox{Filt}(\ci)^{\bot}=\ci^{\bot}$. Hence, one can reformulate H.~Krause's bijection as follows:

\begin{theorem}\label{Krause result}
If $\cd$ is a compactly generated triangulated category, the maps 
\[\ci\mapsto(^{\bot}(\ci^{\bot}),\ci^{\bot}, \ci^{\bot\bot})
\]
and
\[(\cx,\cy,\cz)\mapsto\cm or(\cd^c)^{\cy}
\] 
define a bijection between the set of all saturated idempotent two-sided ideals $\ci$ of $\cd^c$ such that $\ci[1]=\ci$ and the set of all the triangulated TTF triples on $\cd$.
\end{theorem}

Let us deduce a bijection in the same spirit directly from Theorem \ref{from ideals to devissage wrt hc} and the following H. Krause's result:

\begin{proposition}\label{for the idempotency}
Let $\cx$ be a smashing subcategory of a compactly generated triangulated category $\cd$. If $P$ is a compact object of $\cd$, $M$ is an object of $\cx$ and $f:P\ra M$ is a morphism in $\cd$, then there exists a factorization in $\cd$
\[\xymatrix{P\ar[rr]^{f}\ar[dr]_{u} && M \\
& P'\ar[ur]_{v} &
}
\]
with $P'$ compact and $u$ factoring through an object of $\cx$.
\end{proposition}
\begin{proof}
This statement is assertion 2') of \cite[Theorem 4.2]{Krause2000}. We will give here the proof for the algebraic setting by using that, in this case, every smashing subcategory is induced by a homological epimorphism. First, recall that the smashing subcategory $\cx$ fits into a triangulated TTF triple $(\cx,\cy,\cz)$ on $\cd$. Let $\ca$ be a small dg category whose derived category $\cd\ca$ is triangle equivalent to $\cd$.
In the proof of Lemma \ref{k flat up to quasi-equivalence} it was shown that, up to quasi-equivalence, we can assume that $\ca$ is $k$-flat. Also, thanks to Lemma \ref{good properties of s} we can assume that $P$ belongs to $\cs$. Let $F:\ca\ra\cb$ be a homological epimorphism of dg categories associated (\cf Theorem \ref{TTF are he}) to the triangulated TTF triple $(\cx,\cy,\cz)$, and fix a triangle
\[X\arr{\alpha}\ca\arr{F}F^{*}\cb\ra X[1]
\]
in $\cd(\ca^{op}\otimes_{k}\ca)$. Assume $M$ is $\ch$-projective, and let $S_{i}\ko i\in I$, be a direct system of submodules of $M$ as in Proposition \ref{stretching} so that $M=\lid S_{i}$. Then, we get a direct system $S_{i}\otimes_{\ca}X\ko i\in I$ such that $\lid (S_{i}\otimes_{\ca}X)\cong M\otimes_{\ca}X$. Since $M\in\cx$, for each $i\in I$ we have a commutative square
\[\xymatrix{(\cd\ca)(P,S_{i}\otimes_{\ca}X)\ar[r]\ar[d] &(\cd\ca)(P,M\otimes_{\ca}X)\ar[d]^{\wr} \\
(\cd\ca)(P,S_{i})\ar[r] &(\cd\ca)(P,M)
}
\]
where the horizontal arrows are induced by the morphisms $\mu_{i}:S_{i}\ra M$ associated to the colimit and the vertical arrows are induced by the compositions 
\[S_{i}\otimes_{\ca}X\overset{\id\otimes\alpha}{\longrightarrow}S_{i}\otimes_{\ca}\ca\underset{\sim}{\longrightarrow}S_{i}
\] 
and 
\[M\otimes_{\ca}X\underset{\sim}{\overset{\id\otimes\alpha}{\longrightarrow}}M\otimes_{\ca}\ca\underset{\sim}{\longrightarrow}M.
\] 
According to Lemma \ref{good properties of s}, there exists an index $i\in I$ such that $f$ comes from a morphism $(\cd\ca)(P,S_{i}\otimes_{\ca}X)$ via the square above. Then, $f$ factors in $\cd\ca$ as
\[f:P\ra S_{i}\otimes_{\ca}X\ra S_{i}\arr{\mu_{i}}M
\]
and the result follows from the fact that $S_{i}\in\cs$ and $S_{i}\otimes_{\ca}X\in\cx$ (see the comments immediately after Definition \ref{homological epimorphism}).
\end{proof}

Here we have the promised bijection:

\begin{theorem}\label{our result}
If $\cd$ is a compactly generated triangulated category, the maps 
\[(\cx,\cy,\cz)\mapsto\cm or(\cd^c)^{\cy}
\] 
and
\[\ci\mapsto(^{\bot}(\ci^{\bot}),\ci^{\bot}, \ci^{\bot\bot})
\]
define a bijection between the set of all the triangulated TTF triples on $\cd$ and the set of all closed idempotent two-sided ideals $\ci$ of $\cd^c$ such that $\ci[1]=\ci$.
\end{theorem}
\begin{proof}
Let $(\cx,\cy,\cz)$ be a triangulated TTF triple on the category $\cd$ and put $\ci:=\cm or(\cd^c)^{\cy}$. Bearing in mind the Galois connection of subsection \ref{From ideals to smashing subcategories} it only remains to prove that $\ci$ is idempotent and $\ci^{\bot}=\cy$.

Thanks to Lemma \ref{equivalent definitions ideals}, we have that $\ci$ is precisely the class of all morphisms of $\cd^c$ which factor through an object of $\cx$. Then, by using Proposition \ref{for the idempotency} we prove that $\ci$ is idempotent. Finally, let us check that $\ci^{\bot}=\cy$. Of course, it is clear that  $\cy$ is contained in $\ci^{\bot}$. Conversely, let $M$ be an object of $\ci^{\bot}$ and consider the triangle 
\[x\tau_{\cx}M\ra M\ra y\tau^{\cy}M\ra x\tau_{\cx}M[1].
\] 
Since both $M$ and $y\tau^{\cy}M$ belongs to $\ci^{\bot}$, then $x\tau_{\cx}M$ belongs to $\cx\cap\ci^{\bot}$. Since the compact objects generate $\cd$, if $x\tau_{\cx}M\neq 0$ there exists a non-zero morphism 
\[f:P\ra x\tau_{\cx}M
\] 
for some compact object $P$. Thanks to Proposition \ref{for the idempotency}, we know that $f$ admits a factorization $f=vu$ through a compact object with $u$ in $\ci$, and so $f=0$. This contradiction implies $x\tau_{\cx}M=0$ and thus $M\in\cy$.
\end{proof}

\begin{remark}
When $\cd$ is the derived category $\cd\ca$ of a small dg category $\ca$, the idempotency of the ideals of the above theorem reflects the `derived idempotency' of the $\ca$-$\ca$-bimodule $X$ appearing in the characterization of homological epimorphisms (\cf Lemma \ref{characterization he}). 
\end{remark}

The following is our announced weak version of the generalized smashing conjecture:

\begin{corollary}\label{gsc}
Let $\cd$ be a compactly generated triangulated category. Every smashing subcategory of $\cd$ satisfies the principle of infinite d\'evissage with respect to a set of Milnor colimits of compact objects.
\end{corollary}
\begin{proof}
Let $\cx$ be a smashing subcategory of $\cd$. Put $\cy:=\cx^{\bot}$. Thanks to Proposition \ref{smashing are TTF} and Theorem \ref{our result}, we know that $\cy=\ci^{\bot}$ for a certain idempotent two-sided ideal $\ci$ of $\cd^c$ such that $\ci[1]=\ci$. Then Theorem \ref{from ideals to devissage wrt hc} says that $\cx=\Tria(\cp)$, where $\cp$ is a set of Milnor colimits of sequences of morphisms of $\ci$. 
\end{proof}

For compactly generated \emph{algebraic} triangulated categories Theorem \ref{our result} gives an alternative proof of H. Krause's bijection (as stated in Theorem \ref{Krause result}).

\begin{corollary}\label{Krause algebraic}
Let $\cd$ be a compactly generated algebraic triangulated category. The maps 
\[(\cx,\cy,\cz)\mapsto\cm or(\cd^c\ca)^{\cy}
\] 
and
\[\ci\mapsto(^{\bot}(\ci^{\bot}),\ci^{\bot}, \ci^{\bot\bot})
\]
define a bijection between the set of all the triangulated TTF triples on $\cd$ and the set of all saturated idempotent two-sided ideals $\ci$ of $\cd^c$ such that $\ci[1]=\ci$. 
\end{corollary}
\begin{proof}
Thanks to Corollary \ref{gsc} and Lemma \ref{closed are saturated} we just have to prove that, in this case, every saturated idempotent two-sided ideal $\ci$ of $\cd^c$ such that $\ci=\ci[1]$ is closed. For this, let $\ci$ be a saturated idempotent two-sided ideal of $\cd^c$ with $\ci[1]=\ci$. Consider the triangulated TTF triple $(\cx,\cy,\cz):=(^{\bot}(\ci^{\bot}),\ci^{\bot},\ci^{\bot\bot})$ associated to $\ci$ in Theorem \ref{from ideals to devissage wrt hc}, and let $\cp$ be a `set' of Milnor colimits of sequences of morphisms of $\ci$ as in the proof of that theorem (in particular, closed under shifts). Recall that $\Tria(\cp)=\cx$. Put $\ci':=\cm or(\cd^c)^{\cy}$. Of course, $\ci\subseteq\ci'$. The aim is to prove the converse inclusion, which would imply that $\ci$ is closed thanks to Lemma \ref{basic properties Galois}. Let $f:Q'\ra Q$ be a morphism of $\ci'$ and consider the triangle
\[x\tau_{\cx}Q\arr{\delta}Q\arr{\eta}y\tau^{\cy}Q\ra x\tau_{\cx}Q[1]. 
\]
Since $\eta f=0$, then $f$ factors through $\delta$ via the following dotted arrow:
\[\xymatrix{ & Q'\ar[d]^{f}\ar@{.>}[dl]\ar[dr]^{0} & & \\
x\tau_{\cx}Q\ar[r]_{\delta} & Q\ar[r]_{\eta} & y\tau^{\cy}Q\ar[r] & x\tau_{\cx}Q[1]
}
\]
Theorem 4.3 of \cite{Keller1994a} says that we can assume that $\cd$ is the derived category $\cd\ca$ of a small dg category $\ca$. Also, Lemma \ref{good properties of s} enables us to assume that $Q'$ belongs to the set $\cs$ described in subsection \ref{Stretching a filtration}. Thanks to Theorem \ref{localizations via small object argument} we know that $x\tau_{\cx}Q$ can be taken to be, in the category of dg $\ca$-modules, the direct limit of a certain $\lambda$-sequence $X:\lambda\ra\cc\ca$ such that:
\begin{enumerate}[-]
\item $X_{0}=0$,
\item for all $\alpha<\lambda$, the morphism $X_{\alpha}\ra X_{\alpha+1}$ is an inflation with cokernel in $\cp$.
\end{enumerate}
Then Lemma \ref{good properties of s} implies that $f$ factors through a certain $X_{\alpha}$:
\[\xymatrix{ && Q'\ar[d]^{f}\ar@{.>}[dll]\ar[dr]^{0} & & \\
X_{\alpha}\ar[r]_{\can\ \ }&x\tau_{\cx}Q\ar[r]_{\delta} & Q\ar[r]_{\eta} & y\tau^{\cy}Q\ar[r] & x\tau_{\cx}Q[1]
}
\]
Let us prove, by transfinite induction on $\alpha$, that if a morphism of $\ci'$ factors through some $X_{\alpha}$ then it belongs to $\ci$. 

\emph{First step:} For $\alpha=0$ it is clear since $X_{0}=0$. 

\emph{Second step: Assume that every morphism of $\ci'$ factoring through $X_{\alpha}$ belongs to $\ci$.} Let $f$ be a morphism of $\ci'$ factoring through $X_{\alpha+1}$:
\[\xymatrix{Q'\ar[rr]^{f}\ar[dr]_{u} && Q \\
&X_{\alpha+1}\ar[ur]_{v}&
}
\]
In $\cd$ we have a triangle
\[X_{\alpha}\arr{a}X_{\alpha+1}\arr{b} P\arr{\gamma} X_{\alpha}[1]
\]
with $P\in\cp$, \ie $P$ is the Milnor colimit of a sequence
\[P_{0}\arr{g_{0}}P_{1}\arr{g_{1}}P_{2}\ra\dots
\]
of morphisms of $\ci$. By using Lemma \ref{good properties of s} we get a commutative diagram
\[\xymatrix{& Q'\ar[r]^{f}\ar[d]_{u}\ar@/_5pc/[dd]_{w} & Q & \\
X_{\alpha}\ar[r]^{a} & X_{\alpha+1}\ar[r]^{b}\ar[ur]^{v} & P\ar[r]^{\gamma} & X_{\alpha}[1] \\
& P_{t}\ar[ur]^{\pi_{t}}\ar[r]_{g_{t}} & P_{t+1}\ar[u]_{\pi_{t+1}}&
}
\]
where $\pi_{n}$ is the $n$th component of the morphism $\pi$ appearing in the definition of Milnor colimit. Consider the following commutative diagram in which the rows are triangles
\[\xymatrix{X_{\alpha}\ar[r]^{a_{t}}\ar@{=}[d] & M_{t}\ar[r]^{b_{t}}\ar[d]^{\psi} & P_{t}\ar[r]^{\gamma\pi_{t}}\ar[d]^{g_{t}} & X_{\alpha}[1]\ar@{=}[d] \\
X_{\alpha}\ar[r]^{a_{t+1}}\ar@{=}[d] & M_{t+1}\ar[r]^{b_{t+1}}\ar[d]^{\phi} & P_{t+1}\ar[r]^{\gamma\pi_{t+1}}\ar[d]^{\pi_{t+1}} & X_{\alpha}[1]\ar@{=}[d] \\
X_{\alpha}\ar[r]^{a} & X_{\alpha+1}\ar[r]^{b} & P\ar[r]^{\gamma} & X_{\alpha}[1]
}
\]
Since $\gamma\pi_{t}w=\gamma bu=0$, then $w=b_{t}w'$ for some morphism $w':Q'\ra M_{t}$. Therefore, $bu=\pi_{t}w=\pi_{t}b_{t}w'=b\phi\psi w'$, that is to say, $b(u-\phi\psi w')=0$ and so $u-\phi\psi w'=a\xi=\phi\psi a_{t}\xi$ for a certain morphism $\xi:Q'\ra X_{\alpha}$. Hence, $u=\phi\psi(w'+a_{t}\xi)$. Put $u':=\psi(w'+a_{t}\xi)$. We get the following commutative diagram
\[\xymatrix{& Q'\ar[r]^{f}\ar[d]_{u'} & Q & \\
X_{\alpha}\ar[r]_{a_{t+1}} & M_{t+1}\ar[r]_{b_{t+1}}\ar[ur]_{v\phi} & P_{t+1}\ar[r]_{\gamma\pi_{t+1}} & X_{\alpha}[1]
}
\]
Notice that $b_{t+1}u'=b_{t+1}\psi w'+b_{t+1}\psi a_{t}\xi=b_{t+1}\psi w'=g_{t}b_{t}w'$ belongs to $\ci$ since so does $g_{t}$. Apply the octahedron axiom
\[\xymatrix{Q'\ar[r]^{u'}\ar@{=}[d] & M_{t+1}\ar[d]^{b_{t+1}}\ar[r] & L\ar[r]\ar[d] & Q'[1]\ar@{=}[d] \\
Q'\ar[r] & P_{t+1}\ar[d]\ar[r] & N[1]\ar[d]^{n[1]}\ar[r]^{c[1]} & Q'[1]\ar[d]^{u'[1]} \\
& X_{\alpha}[1]\ar@{=}[r]\ar[d]_{-a_{t+1}[1]} & X_{\alpha}[1]\ar[d]\ar[r]_{-a_{t+1}[1]} & M_{t+1}[1] \\
& M_{t+1}[1]\ar[r] & L[1] &
}
\]
and consider the diagram
\[\xymatrix{N\ar[r]^{-c}&Q'\ar[r]^{b_{t+1}u'}\ar[d]^{f} & P_{t+1}\ar[r] & N[1] \\
&Q&&
}
\]
The morphism $-fc=-v\phi u'c=v\phi a_{t+1}n$ belongs to $\ci'$ (since so does $f$) and factors through $X_{\alpha}$. The hypothesis of induction implies that $-fc$ belongs to $\ci$. Since $b_{t+1}u'$ also belongs to $\ci$ and $\ci$ is saturated, then $f$ belongs to $\ci$.

\emph{Third step: Assume $\alpha$ is a limit ordinal and that every morphism of $\ci'$ factoring through $X_{\beta}$ with $\beta<\alpha$ belongs to $\ci$.} Then we have $X_{\alpha}=\lid_{\beta<\alpha}X_{\beta}$ and Lemma \ref{good properties of s} ensures that we have a factorization
\[\xymatrix{&& Q'\ar[d]^{f}\ar@{.>}[dl]\ar@{.>}[dll] \\
X_{\beta}\ar[r]_{\can\ \ \ \ } & \lid_{\beta<\alpha}X_{\beta}\ar[r] & Q
}
\]
The hypothesis of induction implies that $f$ belongs to $\ci$.
\end{proof}

\chapter{Restriction of triangulated TTF triples}\label{Restriction of TTF triples in triangulated categories}
\addcontentsline{lot}{chapter}{Cap\'itulo 6. Restricci\'on de ternas TTF trianguladas}

\section{Introduction}
\addcontentsline{lot}{section}{6.1. Introducci\'on}

\subsection{Motivations}
\addcontentsline{lot}{subsection}{6.1.1. Motivaci\'on}

The idea here is to make an `unbounded' approach to S.~K\"{o}nig's work \cite{Konig1991}. Namely, once the problem of characterizing triangulated TTF triples on compactly generated triangulated categories (in particular, unbounded derived categories of dg categories) was solved in Chapter \ref{TTF triples on triangulated categories}, we study the problem of descend.

\subsection{Outline of the chapter}
\addcontentsline{lot}{subsection}{6.1.2. Esbozo del cap\'itulo}

In section \ref{Restriction of TTF triples}, we characterize in several situations when a triangulated TTF triple on a full triangulated subcategory $\cd'$ of a triangulated category $\cd$ is the restriction of a triangulated TTF triple on $\cd$. Special attention is paid to the case in which $\cd'$ is a kind of `right bounded' triangulated subcategory of $\cd$. In this situation, we also characterize when a triangulated TTF triple on $\cd$ restricts to $\cd'$. In section \ref{Recollements of right bounded derived categories}, we definitely concentrate on the \emph{right bounded derived category} of a dg category. We use `unbounded' methods and the results of section \ref{Restriction of TTF triples} to characterize when such a right bounded derived category is a recollement of two others right bounded derived categories. The characterizations we get are similar to the one obtained by S.~K\"{o}nig in \cite{Konig1991}.

\section{Restriction of triangulated TTF triples}\label{Restriction of TTF triples}
\addcontentsline{lot}{section}{6.2. Restricci\'on de ternas TTF trianguladas}

\begin{definition}\label{definition of restriction}
Let $\cd$ be a triangulated category and let $\cd'$ be a full triangulated subcategory of $\cd$. We say that a triangulated TTF triple $(\cx,\cy,\cz)$ on $\cd$
\emph{restricts to}\index{TTF triple!triangulated!restricts} or \emph{is a lifting of}\index{TTF triple!triangulated!lifts} a triangulated TTF triple $(\cx',\cy',\cz')$ on $\cd'$ if we have
\[(\cx\cap\cd',\cy\cap\cd',\cz\cap\cd')=(\cx',\cy',\cz').
\]
In this case, we say that $(\cx',\cy',\cz')$ \emph{lifts to} or \emph{is the restriction of}  $(\cx,\cy,\cz)$.
\end{definition}

\begin{proposition}\label{restriction}
Let $\cd$ be a triangulated category with small coproducts and let $\cd'$ be a full triangulated subcategory which contains a set of generators of $\cd$. For a triangulated TTF triple $(\cx',\cy',\cz')$ on $\cd'$ the following assertions are equivalent:
\begin{enumerate}[1)]
\item $(\cx',\cy',\cz')$ is the restriction of a triangulated TTF triple on $\cd$.
\item There is a set $\cp$ of objects of $\cx'$ such that:
\begin{enumerate}[2.1)]
\item $\cp$ is recollement-defining in $\cd$.
\item If an object of $\cd$ is right orthogonal to all the shifts of objects of $\cp$, then it is right orthogonal to all the objects of $\cx'$.
\end{enumerate}
\item The objects of $\cx'$ form a recollement-defining class of $\cd$.
\end{enumerate}
\end{proposition}
\begin{proof}
$1)\Rightarrow 2)$ Let $(\cx,\cy,\cz)$ be a triangulated TTF triple on $\cd$ which restricts to $(\cx',\cy',\cz')$, and let $\cq$ be a set of generators of $\cd$ contained in $\cd'$. Notice that, for each object $Q\in\cq$, the torsion triangle associated to $(\cy,\cz)$ can be taken to be
\[\tau_{\cy'}(Q)\ra Q\ra\tau^{\cz'}(Q)\ra\tau_{\cy'}(Q)[1].
\]
Hence, by Lemma \ref{generators and t-structures}, we know that $\tau^{\cz'}(\cq)$ is a set of generators of $\cz$. Remind that the composition
\[\cz\arr{z}\cd\arr{\tau_{\cx}}\cx
\]
is a triangle equivalence (\cf Lemma \ref{x igual a z}), and so $\cp:=\tau_{\cx}(\tau^{\cz'}(\cq))$ is a set of generators of $\cx$. But, since $\tau^{\cz'}(\cq)$ is contained in $\cd'$, then $\cp$ is in fact $\cp=\tau_{\cx'}(\tau^{\cz'}(\cq))$, which is contained in $\cx'$. Now, Lemma \ref{generators and t-structures} implies that $\cy$ is the set of objects of $\cd$ which are right orthogonal to all the shifts of objects of $\cp$, and so $\cp$ is recollement-defining in $\cd$. Finally, the inclusions $\cy\subseteq\cx^{\bot}\subseteq\cx'^{\bot}$ prove 2.2).

$2)\Rightarrow 3)$ Part 2.2) implies that $\cx'^{\bot_{\cd}}$ is the class of all objects which are right orthogonal to all the shifts of objects of $\cp$, and so, thanks to part 2.1), it is both an aisle and a coaisle in $\cd$.

$3)\Rightarrow 1)$ Consider $(\cx,\cy,\cz):=(^{\bot}(\cx'^{\bot}),\cx'^{\bot},\cx'^{\bot\bot})$, with orthogonals taken in $\cd$, which is a triangulated TTF triple on $\cd$. Since $(\cx',\cy',\cz')$ is a triangulated TTF triple on $\cd'$, then we have $\cy'=\cx'^{\bot}\cap\cd'=\cy\cap\cd'$. We next prove $\cx'=\cx\cap\cd'$. The inclusion $\subseteq$ is clear. Conversely, let $X\in\cx\cap\cd'$ and consider the triangle
\[\tau_{\cx'}(X)\ra X\ra\tau^{\cy'}(X)\ra\tau_{\cx'}(X)[1].
\]
Its two terms on the left belong to $\cx$. Then $\tau^{\cy'}(X)\in \cx\cap\cy'\subseteq\cx\cap\cy=\{0\}$ and so $X\in\cx'$. Now, we have the following inclusions 
\[\cz\cap\cd'=\cy^{\bot}\cap\cd'\subseteq\cy'^{\bot}\cap\cd'=\cz'.
\] 
Finally, let $\cq$ be the set of generators of $\cd$ contained in $\cd'$. Lemma \ref{generators and t-structures} implies that $\tau^{\cy}(\cq)^{\bot}=\cz$. Also, notice that $\tau^{\cy}(\cq)\subseteq \cy\cap\cd'=\cy'$. Therefore,
\[\cz'=\cy'^{\bot}\cap\cd'\subseteq \tau^{\cy}(\cq)^{\bot}\cap\cd'=\cz\cap\cd'.
\]
\end{proof}

\begin{corollary}\label{bijection between restricts and are restricted}
Under the hypotheses of Proposition \ref{restriction}, the map
\[(\cx,\cy,\cz)\mapsto(\cx\cap\cd',\cy\cap\cd',\cz\cap\cd')
\]
defines a bijection between:
\begin{enumerate}[1)]
\item Triangulated TTF triples on $\cd$ which restricts to triangulated TTF triples on $\cd'$.
\item Triangulated TTF triples on $\cd'$ which are restriction of triangulated TTF triples on $\cd$.
\end{enumerate}
\end{corollary}
\begin{proof}
Of course, the map is surjective. Now, let $\cq\subseteq\cd'$ a set of generators of $\cd$ and let $(\cx,\cy,\cz)$ be a triangulated TTF triple such that $(\cx',\cy',\cz')=(\cx\cap\cd',\cy\cap\cd',\cz\cap\cd')$ is a triangulated TTF triple on $\cd'$. Then, the proof of Proposition \ref{restriction} shows that $\cy$ is precisely the class of objects of $\cd$ which are right orthogonal to all the shifts of objects of $\tau_{\cx'}(\tau^{\cz'}(\cq))$. This implies the injectivity.
\end{proof}

\begin{corollary}\label{restriction in the two cases}
Under the hypotheses of Proposition \ref{restriction}:
\begin{enumerate}[1)]
\item If $\cd$ is an aisled triangulated category, then a triangulated TTF triple $(\cx',\cy',\cz')$ on $\cd'$ is the restriction of a triangulated TTF triple on $\cd$ if and only if there exists a set $\cp$ of objects of $\cx'$ such that:
\begin{enumerate}[1.1)]
\item The class $\cy$ of objects of $\cd$ which are right orthogonal to all the shifts of objects of $\cp$ is closed under small coproducts.
\end{enumerate}
\begin{enumerate}[1.2)]
\item $\cy$ is contained in $\cx'^{\bot_{\cd}}$.
\end{enumerate}
\item If $\cd$ is compactly generated, then every triangulated TTF triple on $\cd^c$ is the restriction of a  triangulated TTF triple on $\cd$.
\end{enumerate}
\end{corollary}
\begin{proof}
1) If $(\cx',\cy',\cz')$ admits a lifting, then we know there exists such a set $\cp$ thanks to Proposition \ref{restriction}. Conversely, assume the existence of a set $\cp$ satisfying 1.1) and 1.2) is guaranteed. According to Proposition \ref{restriction} we just have to prove that $\cp$ is recollement-defining, which holds thanks to Lemma \ref{rds in aisled}.

2) Let $(\cx',\cy',\cz')$ be a  triangulated TTF triple on $\cd^c$. A skeleton $\cp$ of $\cx'$ clearly satisfies condition 2.2) of Proposition \ref{restriction} and, thanks to Remark \ref{superperfect are recollement defining}, we know that it also satisfies condition 2.1). 
\end{proof}

Recall (\cf Lemma \ref{susp is aisle} and Example \ref{M. Jose}) that if $\cq$ is a set of perfect objects of a triangulated category $\cd$ with small coproducts, the existence of the smallest aisle in $\cd$ containing $\cq$, denoted by $\aisle(\cq)$, is guaranteed.

\begin{proposition}
Let $\cd$ be a triangulated category with small coproducts and perfectly generated by a set $\cq$ such that the coaisle $\aisle(\cq)^{\bot}$ is closed under small coproducts. Put $\cd':=\bigcup_{n\in\Z}\aisle(\cq)[n]$ and let $(\cx',\cy',\cz')$ be a  triangulated TTF triple on $\cd'$. The following assertions are equivalent:
\begin{enumerate}[1)]
\item $(\cx',\cy',\cz')$ is the restriction of a  triangulated TTF triple on $\cd$.
\item $\cx'$ is exhaustively generated to the left by a set $\cp$ of objects such that $\Tria_{\cd}(\cp)$ is an aisle in $\cd$ whose associated coaisle is closed under small coproducts.
\end{enumerate}
\end{proposition}
\begin{proof}
$2)\Rightarrow 1)$ We will apply condition 2) of Proposition \ref{restriction} to the set $\cp$. Since $\Tria_{\cd}(\cp)$ is an aisle in $\cd$, $\Tria_{\cd}(\cp)^{\bot_{\cd}}$ is closed under small coproducts and $\cd$ is perfectly generated, then $\cp$ is recollement-defining in $\cd$ thanks to Lemma \ref{recollement defining in perfectly generated}. This proves condition 2.1). Let us prove condition 2.2). First notice that, thanks to Lemma \ref{preservation of coproducts and compacity}, we know that the inclusion functor $\cx'\hookrightarrow\cd$ preserves small coproducts for it is the composition of the coproduct-preserving inclusions $\cx'\hookrightarrow\cd'\hookrightarrow\cd$. Now, for every object $X\in\cx'$ there exists an integer $i\in\Z$ such that $X$ fits into a triangle of $\cx'$ (and so of $\cd$)
\[\coprod_{n\geq 0}P_{n}[i]\ra \coprod_{n\geq 0}P_{n}[i]\ra X\ra \coprod_{n\geq 0}P_{n}[i+1]
\]
with $P_{n}\in\Sum(\cp^+)^{*n}$ for each $n\geq 0$. Let $M$ be an object of $\cd$ which is right orthogonal to all the shifts of objects of $\cp$. Then, we get a long exact sequence
\[\dots\ra\cd(\coprod_{n\geq 0}P_{n}[i+1],M)\ra\cd(X,M)\ra\cd(\coprod_{n\geq 0}P_{n}[i],M)\ra\dots
\]
in which
\[\cd(\coprod_{n\geq 0}P_{n}[i+1],M)=\cd(\coprod_{n\geq 0}P_{n}[i],M)=0,
\]
and so $\cd(X,M)=0$.

$1)\Rightarrow 2)$ In the proof of Proposition \ref{restriction}, it was shown that $\cp:=(\tau_{\cx'}z'\tau^{\cz'})(\cq)$ is recollement-defining in $\cd$, and so $\Tria_{\cd}(\cp)$ is an aisle in $\cd$ whose associated coaisle is closed under small coproducts. It remains to prove that $\cx'$ is exhaustively generated to the left by $\cp$. Given $X\in\cx'$, there exists an integer $i\in\Z$ such that $X[i]\in\aisle_{\cd}(\cq)$, that is to say, $X[i]\in\Susp_{\cd}(\cq)$ (\cf Lemma \ref{susp is aisle} together with Example \ref{M. Jose}). Hence, thanks to Example \ref{M. Jose}, $X[i]$ occurs in a triangle
\[\coprod_{n\geq 0}Q_{n}\ra\coprod_{n\geq 0}Q_{n}\ra X[i]\ra\coprod_{n\geq 0}Q_{n}[1]
\]
with $Q_{n}\in\Sum(\cq^+)^{*n}$ for each $n\geq 0$. Notice that the above triangle lives in $\cd'$, and so it makes sense to apply the triangle functor $\tau_{\cx'}z'\tau^{\cz'}$. Since this functor is isomorphic to $\tau_{\cx'}$, it preserves small coproducts and we get a triangle
\[\coprod_{n\geq 0}(\tau_{\cx'}z'\tau^{\cz'})(Q_{n})\ra \coprod_{n\geq 0}(\tau_{\cx'}z'\tau^{\cz'})(Q_{n})\ra X[i]\ra (\coprod_{n\geq 0}(\tau_{\cx'}z'\tau^{\cz'})(Q_{n}))[1]
\]
in $\cx'$ with $\tau_{\cx'}(\tau^{\cz'}Q_{n})\in\Sum(\cp^+)^{*n}$ for each $n\geq 0$.
\end{proof}

As a consequence we deduce that, sometimes, an intrinsic property of a triangulated TTF triple on $\cd'$, which is independent of the ambient category $\cd$, already guarantees that it lifts to a triangulated TTF triple on $\cd$. Indeed,

\begin{corollary}\label{easy restriction}
Let $\cd$ be a triangulated category with small coproducts and perfectly generated by a set $\cq$ such that the coaisle $\aisle(\cq)^{\bot}$ is closed under small coproducts. Put $\cd':=\bigcup_{n\in\Z}\aisle(\cq)[n]$ and let $(\cx',\cy',\cz')$ be a triangulated TTF triple on $\cd'$ such that $\cx'$ is exhaustively generated to the left by a set $\cp$ whose objects are superperfect in $\cx'$. Then, $(\cx',\cy',\cz')$ is the restriction of a triangulated TTF triple on $\cd$.
\end{corollary}
\begin{proof}
By using the techniques of the proof of Lemma \ref{truncating special objects and from local to global via smashing}, we prove that the objects of $\cp$ are superperfect in $\cd'$. Now, Lemma \ref{preservation of coproducts and compacity} implies that they are also superperfect in $\cd$, and thanks to Theorem \ref{Krause on perfects} and Remark \ref{superperfect are recollement defining} we know that $\Tria_{\cd}(\cp)$ is an aisle in $\cd$ whose coaisle is closed under small coproducts.
\end{proof}

The following lemma is well-known (\cf \cite[paragraph 1.3.19]{BeilinsonBernsteinDeligne}):

\begin{lemma}\label{general restriction of t-structures}
Let $(\cu,\cv[1])$ be a t-structure on a triangulated category $\cd$, and let $\cd'$ be a strictly full triangulated subcategory of $\cd$. The following assertions are equivalent:
\begin{enumerate}[1)]
\item $(\cd'\cap\cu,\cd'\cap\cv[1])$ is a t-structure on $\cd'$.
\item $(u\tau_{\cu})(\cd')\subseteq\cd'$.
\end{enumerate}
\end{lemma}
\begin{proof}
$1)\Rightarrow 2)$ Put $(\cu',\cv'):=(\cd'\cap\cu, \cd'\cap\cv)$, and let $\iota:\cd'\hookrightarrow\cd$ be the inclusion functor. For each object $N$ of $\cd'$ we have a triangle in $\cd$
\[u\tau_{\cu}N\ra N\ra v\tau^{\cv}N\ra (u\tau_{\cu}N)[1],
\]
which is isomorphic to
\[\iota u'\tau_{\cu'}N\ra N\ra \iota v'\tau^{\cv'}N\ra (\iota u'\tau_{\cu'}N)[1]
\]
thanks to Lemma \ref{truncacion funtorial}. Therefore $u\tau_{\cu}N$ belongs to $\cd'$.

$2)\Rightarrow 1)$ Given an object $N$ of $\cd'$, in the triangle
\[u\tau_{\cu}N\ra N\ra v\tau^{\cv}N\ra (u\tau_{\cu}N)[1]
\]
we have that both $u\tau_{\cu}N$ and $N$ belongs to $\cd'$. Thus, $v\tau^{\cv}N$ belongs to $\cd'$.
\end{proof}

We present now a very particular situation in which condition 2) of Lemma \ref{general restriction of t-structures} can be improved.

\begin{proposition}\label{particular restriction of t-structures}
Let $\cd$ be a triangulated category with small coproducts and perfectly generated by a set $\cq$ such that the coaisle $\aisle(\cq)^{\bot}$ is closed under small coproducts. Put $\cd':=\bigcup_{n\in\Z}\aisle(\cq)[n]$ and let $(\cu,\cv)$ be a t-structure on $\cd$ such that $\cu$ is triangulated and $\cv$ is closed under coproducts. The following assertions are equivalent:
\begin{enumerate}[1)]
\item $(\cd'\cap\cu,\cd'\cap\cv)$ is a t-structure on $\cd'$.
\item $(u\tau_{\cu})(\cq)\subseteq\aisle(\cq)[n]$ for some integer $n$.
\end{enumerate}
\end{proposition}
\begin{proof}
$1)\Rightarrow 2)$ Thanks to Lemma \ref{general restriction of t-structures}, it suffices to prove that $(u\tau_{\cu})(\cd')\subseteq\cd'$ implies condition 2) of the proposition. Since $N:=\coprod_{Q\in\cq}Q$ belongs to $\aisle(\cq)$, there exists an integer $n$ such that $u\tau_{\cu}N$ belongs to $\aisle(\cq)[n]$. Now notice that for each $Q\in\cq$ we have that $u\tau_{\cu}Q$ is a direct summand of $u\tau_{\cu}N$. But since $\aisle(\cq)[n]$ is closed under Milnor colimits in $\cd$ then it is also closed under direct summands. This implies that $u\tau_{\cu}Q$ belongs to $\aisle(\cq)[n]$.

$2)\Rightarrow 1)$ Thanks to Lemma \ref{general restriction of t-structures}, it suffices to prove the inclusion $(u\tau_{\cu})(\cd')\subseteq\cd'$. Let $N$ be an object of $\cd'$ and fix an integer $i$ such that $N[i]$ belongs to $\aisle(Q)$. By Example \ref{M. Jose} we have that $N[i]$ is the Milnor colimit of a sequence
\[M_{0}\arr{f_{0}}M_{1}\arr{f_{1}}M_{2}\arr{f_{2}}\dots
\]
where $M_{n}\in\Sum(\cq^+)^{*n}$ for each $n\geq 0$. Now, since $u\tau_{\cu}$ commutes with small coproducts, by applying it to the corresponding Milnor triangle we get that $u\tau_{\cu}N[i]$ belongs to $\cd'$, and then so does $u\tau_{\cu}N$.
\end{proof}

\begin{remark}
If in Proposition \ref{particular restriction of t-structures} $\cq$ is a finite set, then one can replace condition 2) by: $(u\tau_{\cu})(\cq)\subseteq\cd'$.
\end{remark}

\begin{corollary}\label{conditions to restrict}
Under the hypotheses of Proposition \ref{particular restriction of t-structures}, the following assertions are equivalent for a triangulated TTF triple $(\cx,\cy,\cz)$ on $\cd$:
\begin{enumerate}[1)]
\item $(\cd'\cap\cx,\cd'\cap\cy,\cd'\cap\cz)$ is a triangulated TTF triple on $\cd'$.
\item The following conditions hold:
\begin{enumerate}[2.1)]
\item $(x\tau_{\cx})(\cq)\subseteq\aisle(\cq)[n]$ for some integer $n$.
\item $(y\tau_{\cy})(\cd')\subseteq\cd'$.
\end{enumerate}
\end{enumerate}
\end{corollary}

\begin{example}\label{restriction in examples}
Let $I$ be a two-sided ideal of a $k$-algebra $A$, and assume the projection $\pi:A\ra A/I$ is a homological epimorphism. We know (\cf Example \ref{quotients which are he}) that in this case $\cd A$ is a recollement of $\cd(A/I)$ and $\Tria_{\cd A}(I)$:
\[\xymatrix{\cd(A/I)\ar[rr]^{\pi_{*}} && \cd A\ar@/_1.2pc/[ll]_{?\otimes^{\L}_{A}A/I}\ar@/_-1pc/[ll]^{\R Hom_{A}(A/I,?)}\ar[rr]^{?\otimes^{\L}_{A}I} && \Tria_{\cd A}(I)\ar@/_1.2pc/[ll]\ar@/_-1pc/[ll]^x
}
\]
Notice that, in case $I$ is compact in $\Tria_{\cd A}(I)$, Corollary \ref{remark Keller} implies that $\cd A$ is a recollement of $\cd(A/I)$ and $\cd C$, where $C:=(\cc_{dg}A)(\textbf{i}I,\textbf{i}I)$. Thanks to Corollary \ref{conditions to restrict}, we know that the associated triangulated TTF triple restricts to $\cd^-A$ if and only if the following conditions hold:
\begin{enumerate}[1)]
\item $\tau_{\cx}(A)=A\otimes^{\bf L}_{A}I\cong I$ belongs to $\cd^-A$,
\item $\textbf{R}\Hom_{A}(A/I,M)$ belongs to $\cd^-(A/I)$ for each $M$ in $\cd^-A$.
\end{enumerate}
Of course, the first condition always holds. The second condition holds if $A/I$ has finite projective dimension regarded as a right $A$-module or, equivalently, $I$ has finite projective dimension regarded as a right $A$-module. Assume then that $I_{A}$ has finite projective dimension and also that it is compact in $\Tria_{\cd A}(I)$. In this case the mutually quasi-inverse triangle equivalences
\[\xymatrix{\Tria_{\cd A}(I)\ar@<1ex>[r]& \cd C\ar@<1ex>[l]
}
\]
guaranteed by Corollary \ref{remark Keller} restrict to mutually quasi-inverse triangle equivalences
\[\xymatrix{\Tria_{\cd A}(I)\cap\cd^-A\ar@<1ex>[r] & \cd^-C.\ar@<1ex>[l]
}
\]
Therefore, $\cd^-A$ is a recollement of $\cd^-(A/I)$ and $\cd^-C$. This example contains as particular cases the d\'{e}collements of Corollary 11, Corollary 12 and Corollary 15 of \cite{Konig1991}, and describes functors appearing in those d\'{e}collements as restrictions of derived functors.
\end{example}

\section{Recollements of right bounded derived categories}\label{Recollements of right bounded derived categories}
\addcontentsline{lot}{section}{6.3. Aglutinaciones de categor\'ias derivadas acotadas por la derecha}

All through this section the appearing dg categories are small.

\begin{definition}\label{fibrant replacement for sets}
Let $\ca$ be a dg category with associated exact dg category $\cc_{dg}\ca$ (\cf \cite{Keller2006b} or subsection \ref{B. Keller's Morita theory for derived categories}). A \emph{fibrant replacement}\index{fibrant!replacemente of a set of dg modules} of a set $\cp$ of objects of the derived category $\cd\ca$ is a full subcategory $\cb$ of $\cc_{dg}\ca$ formed by the fibrant (or $\ch$-injective) replacements $\textbf{i}P$ of the modules $P$ of $\cp$. 
\end{definition}

Recall from subsection \ref{B. Keller's Morita theory for derived categories} that, under the conditions of Definition \ref{fibrant replacement for sets}, we have a dg $\cb$-$\ca$-bimodule $X$ defined by
\[X(A,B):=B(A)
\]
for $A$ in $\ca$ and $B$ in $\cb$. This gives rise to a funtor
\[\ch om_{\ca}(X,?):\cc_{dg}\ca\ra\cc_{dg}\cb
\]
which induces triangle functors
\[\ch om_{\ca}(X,?):\ch\ca\ra\ch\cb
\]
and
\[\R\ch om_{\ca}(X,?):\cd\ca\ra\cd\cb.
\]

\begin{definition}\label{dually right bounded}
Under the conditions above, we say that:
\begin{enumerate}[1)]
\item $\cp$ is \emph{right bounded}\index{bounded!right} if $\cp\subseteq\cd^{\leq n}\ca$ for some $n\in\Z$.
\item $\cp$ is \emph{dually right bounded}\index{bounded!dually right} if the functor 
\[\R\ch om_{\ca}(X,?):\cd\ca\ra\cd\cb
\]
sends an object of $\cd^-\ca$ to an object of $\cd^-\cb$.
\end{enumerate}
\end{definition}

\emph{A priori}, the notion of ``dually right bounded'' depends on the fibrant replacement of $\cp$, however this is not drastic. Indeed, 

\begin{lemma}
Assume we are under the conditions above. Consider two fibrant replacements $\cb\ko \cb'$ of $\cp$, and let $X\ko X'$ be the corresponding dg bimodules. Then there exists a triangle equivalence $F:\cd\cb'\arr{\sim}\cd\cb$ making the following diagram
\[\xymatrix{\cd\ca\ar[rr]^{\R\ch om_{\ca}(X,?)}\ar[d]_{\R\ch om_{\ca}(X',?)} && \cd\cb \\
\cd\cb'\ar[urr]_{F} &&
}
\] 
commutative up to triangle equivalence. Moreover, $F$ induces a bijection between the corresponding representable modules.
\end{lemma}
\begin{proof}
Our proof is based on Corollary \ref{remark Keller} but, for the sake of simplicity, we will use the following notation: $H_{X}:=\R\ch om_{\ca}(X,?)\ko T_{X}:=?\otimes_{\cb}^{\bf L}X$ and $\delta^{X}$ for the counit of the adjunction $(T_{X},H_{X})$. Similarly for $X'$. Take $F$ to be the composition
\[\xymatrix{\cd\cb'\ar[r]^{T_{X'}\ \ \ \ \ }&\Tria_{\cd\ca}(\cp)\ar[r]^{\ \ \ \ \ H_{X}}&\cd\cb
}
\]
The fact that $F$ is a triangle equivalence and induces a bijection between the corresponding representable modules follows from Corollary \ref{remark Keller}. Let us check that $F\circ H_{X'}$ is isomorphic to $H_{X}$. Given an object $M$ of $\cd\cb'$ we know, thanks to Corollary \ref{remark Keller} and Lemma \ref{truncacion funtorial}, that there exists a unique isomorphism of triangles
\[\xymatrix{T_{X'}H_{X'}(M)\ar[r]^{\ \ \ \ \delta^{X'}_{M}}\ar[d]_{\alpha_{M}}^{\wr} & M\ar[r]\ar[d]^{\id_{M}} & \cone(\delta^{X'}_{M})\ar[r]\ar[d]^{\wr} &  T_{X'}H_{X'}(M)[1]\ar[d]_{\alpha_{M}[1]}^{\wr}\\
T_{X}H_{X}(M)\ar[r]^{\ \ \ \ \delta^{X}_{M}} & M\ar[r] & \cone(\delta^{X}_{M})\ar[r] & T_{X}H_{X}(M)[1]
}
\]
extending the identity $\id_{M}$. Now, the composition
\[\xymatrix{FH_{X'}(M)=H_{X}T_{X'}H_{X'}(M)\ar[rr]^{\ \ \ H_{X}(\alpha_{M})}_{\sim} && H_{X}T_{X}H_{X}(M)\ar[rr]^{H_{X}(\delta^{X}_{M})}_{\sim} && H_{X}M
}
\]
is an isomorphism which induces an isomorphism between $F\circ H_{X'}$ and $H_{X}$.
\end{proof}

\begin{lemma}\label{characterization of dually right bounded}
Let $\ca$ be a dg category and $\cp$ a set of objects of $\cd^-\ca$ . Assume that there exists an integer $m$ such that for every two objects $P$ and $P'$ of $\cp$ we have $(\cd\ca)(P,P'[i])=0$ for $i>m$. Consider the following assertions:
\begin{enumerate}[1)]
\item $\cp$ is dually right bounded.
\item For each object $M$ of $\cd^-\ca$ there exists an integer $s_{M}$ such that $(\cd\ca)(P,M[i])=0$ for every $P\in\cp$ and every $i>m+s_{M}$.
\end{enumerate}
Then 1) implies 2) and, if $m=0$, we also have that 2) implies 1).
\end{lemma}
\begin{proof}
Let $\cb$ be a fibrant replacement of the set $\cp$. Notice that the assumption on the set $\cp$ is equivalent to say that $\cb$ has cohomology concentrated in degrees $(-\infty,m]$. Let $X$ be the associated $\ca$-$\cb$-bimodule. Assertion 1) says that for each $M\in\cd^-\ca$ there exists an integer $s_{M}$ such that 
\[(\cc_{dg}\ca)(?,\textbf{i}M)_{|_{\cb}}\in\cd^{\leq s_{M}}\cb.
\]
Now, thanks to Lemma \ref{looking for coherence}, this implies that $(\cd\ca)(P,M[i])=0$ for every $P\in\cp$ and every $i>m+s_{M}$. As stated in Lemma \ref{looking for coherence}, in case $m=0$ we can go backward in the proof.
\end{proof}

The following is a kind of `right bounded' version of Corollary \ref{remark Keller}:

\begin{lemma}\label{Morita theory for right bounded}
Let $\ca$ be a dg category and let $\cp$ be a set of objects of $\cd^-\ca$ such that:
\begin{enumerate}[a)]
\item it is right bounded,
\item its objects are compact in $\Tria_{\cd\ca}(\cp)\cap\cd^-\ca$,
\item $\Tria_{\cd\ca}(\cp)\cap\cd^-\ca$ is exhaustively generated to the left by $\cp$.
\end{enumerate}
Let $\cb$ be a fibrant replacement of $\cp$ and let $X$ be the associated $\cb$-$\ca$-bimodule. Then, the functor $?\otimes^{\L}_{\cb}X:\cd\cb\ra\cd\ca$ induces a triangle equivalence
\[?\otimes^{\L}_{\cb}X:\cd^-\cb\arr{\sim}\Tria_{\cd\ca}(\cp)\cap\cd^-\ca,
\]
and the following assertions are equivalent:
\begin{enumerate}[1)]
\item $\Tria_{\cd\ca}(\cp)\cap\cd^-\ca$ is an aisle in $\cd^-\ca$. 
\item $\cp$ is dually right bounded.
\end{enumerate}
\end{lemma}
\begin{proof}
\emph{First step: The triangle functor
\[?\otimes^{\L}_{\cb}X:\cd\cb\ra\cd\ca
\]
induces a triangle functor
\[?\otimes^{\L}_{\cb}X:\cd^-\cb\ra\Tria(\cp)\cap\cd^-\ca.
\]
}
Let $\cu$ be the full subcategory of $\cd^-\cb$ formed by those $N$ such that $N\otimes^{\L}_{\cb}X\in\Tria(\cp)\cap\cd^-\ca$. It is a full triangulated subcategory of $\cd^-\cb$. Notice that, if $B=\textbf{i}P$ is the object of $\cb$ corresponding to a certain $P\in\cp$, then
\[B^{\we}\otimes^{\L}_{\cb}X\cong\textbf{i}P\cong P\in\Tria(\cp)\cap\cd^-\ca.
\]
This proves that $\cu$ contains the representable dg $\cb$-modules $B^{\we}$. It also proves that, since $\Tria(\cp)\cap\cd^-\ca$ is closed under small coproducts of finite extensions of objects $\Sum(\cp^+)$, then $\cu$ is closed under small coproducts of finite extensions of objects of $\Sum(\{B^{\we}\}_{B\in\cb}^+)$.  Finally, Lemma \ref{e implies d implies g} implies that $\cu=\cd^-\cb$.

\emph{Second step: The functor $?\otimes^{\L}_{\cb}X:\cd^-\cb\ra\Tria(\cp)\cap\cd^-\ca$ is a triangle equivalence.}

This is proved by using Lemma \ref{detecting triangle equivalences}. Indeed, if $B=\textbf{i}P$ is the object of $\cb$ corresponding to $P\in\cp$, we have seen already that $B^{\we}\otimes^{\L}_{\cb}X\cong P$, which is compact in $\Tria(\cp)\cap\cd^-\ca$ by hypothesis. Also, if $B=\textbf{i}P$ and $B'=\textbf{i}P'$ are objects of $\cb$, we have
\begin{align}
(\cd\cb)(B^{\we},B'^{\we}[n])\arr{\sim}\H n\cb(B,B')= \nonumber \\
=(\ch\ca)(\textbf{i}P,\textbf{i}P'[n])\arr{\sim}(\cd\ca)(B^{\we}\otimes_{\cb}^{\L}X,B'^{\we}[n]\otimes_{\cb}^{\L}X). \nonumber
\end{align}
Finally, by hypothesis, $\Tria(\cp)\cap\cd^-\ca$ is exhaustively generated to the left by the objects $B^{\we}\otimes^{\L}_{\cb}X\cong\textbf{i}P\cong P$.

\emph{Third step: Thanks to the second step, 1) holds if and only if the functor $?\otimes^{\L}_{\cb}X:\cd^-\cb\ra\cd^-\ca$ has a right adjoint. Let us prove that this happens if and only if $\cp$ is dually right bounded.}

The `if' part is clear. Conversely, let $G:\cd^-\ca\ra\cd^-\cb$ be a right adjoint to $?\otimes^{\L}_{\cb}X$. For simplicity, put $\R\ch om_{\ca}(X,?)=H_{X}$. Consider the diagram
\[\xymatrix{& \cd^-\cb\ar@{^(->}[d]^{\iota_{\cb}}\ar@<1ex>[rr]^{?\otimes^{\L}_{\cb}X} && \cd^-\ca\ar@{^(->}[d]^{\iota_{\ca}}\ar@<1ex>[ll]^{G} \\
\cd^{\leq n}\cb\ar@{^(->}[ur]\ar@<1ex>[r]^{\iota} & \cd\cb\ar@<1ex>[l]^{\tau^{\leq n}}\ar@<1ex>[rr]^{?\otimes^{\L}_{\cb}X} && \cd\ca\ar@<1ex>[ll]^{H_{X}}
}
\]
where $n$ is any integer and $\iota$ is the inclusion functor. We have that
\[\tau^{\leq n}\circ \iota_{\cb}\circ G\cong\tau^{\leq n}\circ H_{X}\circ\iota_{\ca}
\]
since these two compositions are right adjoint to $?\otimes^{\L}_{\cb}X\circ\iota$. Let the $M$ be an object of $\cd^-\ca$ and fix an integer $n$ such that $GM\in\cd^{\leq n}\cb$. Then, we get
\[\tau^{\leq n}H_{X}(M)\cong\tau^{\leq n}G(M)\cong\tau^{\leq n+i}G(M)\cong\tau^{\leq n+i}H_{X}(M)
\]
for each $i\geq 0$. This implies that $\tau^{>n}(H_{X}M)\in\cd^{>n+i}\cb$ for each $i\geq 0$. In particular,
\[\H j(\tau^{>n}H_{X}(M))=0
\]
for every $j\in\Z$, that is to say, $\tau^{>n}(H_{X}(M))=0$. Thus, $H_{X}(M)\in\cd^{\leq n}\cb$.
\end{proof}

The following is a kind of `right bounded' version (with several objects) of Corollary \ref{recollements unbounded derived}

\begin{proposition}\label{parametrization right bounded recollements}
Let $\ca$ be a dg category. The following assertions are equivalent:
\begin{enumerate}[1)]
\item $\cd^-\ca$ is a recollement of $\cd^-\cb$ and $\cd^-\cc$, for certain dg categories $\cb$ and $\cc$.
\item There exist sets $\cp\ko \cq$ in $\cd^-\ca$ such that:
\begin{enumerate}[2.1)]
\item $\cp$ and $\cq$ are right bounded.
\item $\cp$ and $\cq$ are dually right bounded.
\item $\Tria(\cp)\cap\cd^-\ca$ is exhaustively generated to the left by $\cp$ and the objects of $\cp$ are compact in $\cd\ca$.
\item $\Tria(\cq)\cap\cd^-\ca$ is exhaustively generated to the left by $\cq$ and the objects of $\cq$ are compact in $\Tria(\cq)\cap\cd^-\ca$.
\item $(\cd\ca)(P[i],Q)=0$ for each $P\in\cp\ko Q\in\cq$ and $i\in\Z$.
\item $\cp\cup\cq$ generates $\cd\ca$.
\end{enumerate}
\end{enumerate}
\end{proposition}
\begin{proof}
$1)\Rightarrow 2)$ Consider the d\'{e}collement
\[ \xymatrix{ \cd^-\cb\ar[r]^{i_{*}=i_{!}} & \cd^-\ca\ar@/_1.2pc/[l]_{i^{*}}\ar@/_-1pc/[l]^{i^{!}}\ar[r]^{j^*=j^!} & 
\cd^-\cc\ar@/_1.2pc/[l]_{j_{*}}\ar@/_-1pc/[l]^{j_{!}},
}
\]
and let $(\cx',\cy',\cz')$ be the corresponding triangulated TTF triple on $\cd^-\ca$. Let $\cp$ be the set formed by all the objects $j_{!}(C^{\we})\ko C\in\cc$ and let $\cq$ be the set formed by all the objects $i_{*}(B^{\we})\ko B\in\cb$.

2.1) Notice that the coproduct $\coprod_{C\in\cc}C^{\we}$ lives in $\cd^-\cc$ and, since $j_{!}:\cd^-\cc\arr{\sim}\cx'$ is a triangle equivalence, then there exists in $\cd^-\ca$ the coproduct $\coprod_{P\in\cp}P$. Now, the claim in the proof of Lemma \ref{preservation of coproducts and compacity} implies that $\cp$ is right bounded. Similarly for $\cq$.

2.3) Since $j_{!}:\cd^-\cc\arr{\sim}\cx'$ is a triangle equivalence, then $\cx'$ is exhaustively generated to the left by the set $\cp$, whose objects are compact in $\cx'$. Then Corollary \ref{easy restriction} says that $(\cx',\cy',\cz')$ is the restriction of a triangulated TTF triple $(\cx,\cy,\cz)$ on $\cd\ca$. Moreover, $\cx=\Tria(\cp)$ and so $\cx'=\Tria(\cp)\cap\cd^-\ca$. This proves that $\Tria(\cp)\cap\cd^-\ca$ is exhaustively generated to the left by $\cp$. With the techniques of the proof of Lemma \ref{truncating special objects and from local to global via smashing}, we can prove that the objects of $\cp$ are compact in $\cd^-\ca$, and then Lemma \ref{preservation of coproducts and compacity} implies that they are also compact in $\cd\ca$.

2.4) Since $i_{*}:\cd^-\cb\arr{\sim}\cy'$ is a triangle equivalence, then $\cy'$ is exhaustively generated to the left by the set $\cq$, whose objects are compact in $\cy'$. From the proof of 2.3) we know that
\[\cy'=\Tria(\cp)^{\bot}\cap\cd^-\ca.
\] 
Of course, $\cq$ is contained in $\cy'$ and so $\Tria(\cq)\cap\cd^-\ca$ is contained in $\cy'$. Notice that $\Tria(\cq)\cap\cd^-\ca$ is a full triangulated subcategory of $\cy'$ containing $\cq$ and closed under small coproducts of objects of $\bigcup_{n\geq 0}\Sum(\cq^+)^{*n}$. This implies (\cf Lemma \ref{e implies d implies g}) that $\Tria(\cq)\cap\cd^-\ca=\cy'$.

2.2) From the proof of 2.3), we know that $\Tria(\cp)\cap\cd^-\ca$ is an aisle in $\cd^-\ca$. Then Lemma \ref{Morita theory for right bounded} implies that $\cp$ is dually right bounded. Similarly for $\cq$.

2.5) and 2.6) follow from the fact that $(\Tria(\cp),\Tria(\cq))$ is a t-structure on $\cd\ca$.

$2)\Rightarrow 1)$ Since the objects of $\cp$ are compact in $\cd\ca$, then (\cf Lemma \ref{rds in aisled} or Lemma \ref{recollement defining in perfectly generated}) we know that
\[(\Tria(\cp),\Tria(\cp)^{\bot},\Tria(\cp)^{\bot\bot})
\]
is a triangulated TTF triple on $\cd\ca$. From conditions 2.5) and 2.6) we deduce that $\cq$ generates $\Tria(\cp)^{\bot}$. Moreover, since $\Tria(\cp)^{\bot}$ is closed under small coproducts, then $\Tria(\cq)$ is contained in $\Tria(\cp)^{\bot}$. The fact that $\Tria(\cq)$ is an aisle in $\cd\ca$ (\cf Corollary \ref{derived categories are aisled}) together with Lemma \ref{generators and t-structures} implies that $\Tria(\cp)^{\bot}=\Tria(\cq)$. Lemma \ref{Morita theory for right bounded} implies that $\Tria(\cp)\cap\cd^-\ca$ and $\Tria(\cq)\cap\cd^-\ca$ are aisles in $\cd^-\ca$. Given $M\in\cd^-\ca$, consider the triangle
\[M'\ra M\ra M''\ra M'[1]
\]
in $\cd^-\ca$ with $M'\in\Tria(\cp)\cap\cd^-\ca$ and $M''\in(\Tria(\cp)\cap\cd^-\ca)^{\bot}$. In particular, $M'\in\Tria(\cp)$ and $M''\in\Tria(\cp)^{\bot}=\Tria(\cq)$. This proves that 
\[(\Tria(\cp)\cap\cd^-\ca,\Tria(\cq)\cap\cd^-\ca)
\] 
is a t-structure on $\cd^-\ca$. Similarly, 
\[(\Tria(\cq)\cap\cd^-\ca,\Tria(\cq)^{\bot}\cap\cd^-\ca)
\] 
is a t-structure on $\cd^-\ca$. These t-structures together form a triangulated TTF triple $(\cx',\cy',\cz')$ on $\cd^-\ca$. Finally, Lemma \ref{Morita theory for right bounded} implies that $\cx'\cong\cd^-\cc$ (for a cofibrant replacement $\cc$ of $\cp$) and $\cy'\cong\cd^-\cb$ (for a cofibrant replacement $\cb$ of $\cq$).
\end{proof}

We will prove in Corollary \ref{parametrization right bounded hn recollements} below that conditions of assertion 2 in Proposition \ref{parametrization right bounded recollements} can be weakened under certain extra hypotheses. But first we need some preliminary results. The following one is a `right bounded' version of the proof of B.~Keller's theorem \cite[Theorem 5.2]{Keller1994a}:

\begin{proposition}\label{g implies e to the left}
Let $\cp$ be a set of objects of a triangulated category $\cd$ such that
\begin{enumerate}[1)]
\item the objects of $\cp$ are compact in $\Tria(\cp)$,
\item $\cd(P,P'[i])=0$ for each $P\ko P'\in\cp$ and $i\geq 1$,
\item small coproducts of finite extensions of $\Sum(\cp^+)$ exist in $\cd$,
\item for each $M\in\cd$ there exists $k_{M}\in\Z$ such that $\cd(P[n],M)=0$ if $n<k_{M}$. 
\end{enumerate}
Then $\Tria(\cp)$ is an aisle in $\cd$ exhaustively generated to the left by $\cp$. In particular, if $\cp$ generates $\cd$, then $\Tria(\cp)=\cd$.
\end{proposition}
\begin{proof}
Let $M\in\cd$. We know that if $\cd(P[n],M)\neq 0$ for some $P\in\cp$, then $n\geq k_{M}$. Since $\cp$ is a set, there exists an object 
\[P_{0}\in\Sum(\cp^+[k_{M}])
\] 
and a morphism $\pi_{0}:P_{0}\ra M$ inducing a surjection
\[\pi_{0}^{\we}:\cd(P[n],P_{0})\ra\cd(P[n],M)
\]
for each $P\in\cp\ko n\in\Z$. Indeed, one can take
\[P_{0}:=\coprod_{P\in\cp\ko n\geq k_{M}}P[n]^{(\cd(P[n],M))}.
\]
Now, we will inductively construct a commutative diagram
\[\xymatrix{P_{0}\ar[r]^{f_{0}}\ar[dr]_{\pi_{0}} & P_{1}\ar[r]^{f_{1}}\ar[d]^{\pi_{1}} & \dots\ar[r] & P_{q}\ar[r]^{f_{q}}\ar[dll]^{\pi_{q}} & \dots\ko q\geq 0 \\
& M &&&
}
\]
such that:
\begin{enumerate}[a)]
\item $P_{q}\in\Sum(\cp^+[k_{M}])^{*q}$,
\item $\pi_{q}$ induces a surjection 
\[\pi_{q}^{\we}:\cd(P[n],P_{q})\ra\cd(P[n],M)
\]
for each $P\in\cp\ko n\in\Z$.
\end{enumerate}
Suppose for some $q\geq 0$ we have constructed $P_{q}$ and $\pi_{q}$. Consider the triangle
\[C_{q}\arr{\alpha_{q}}P_{q}\arr{\pi_{q}}M\ra C_{q}[1]
\]
induced by $\pi_{q}$. By applying $\cd(P[n],?)$ we get a long exact sequence
\[\dots\ra\cd(P[n+1],M)\ra\cd(P[n],C_{q})\ra\cd(P[n],P_{q})\ra\dots
\]
If $\cd(P[n],C_{q})\neq 0$, then either $\cd(P[n+1],M)\neq 0$ or $\cd(P[n],P_{q})\neq 0$. In the first case, we would have $n\geq k_{M}-1$. In the second case we would have $\cd(P[n],P'[m])\neq 0$ for some $P'\in\cp\ko m\geq k_{M}$, and so $n\geq m\geq k_{M}$. Therefore, $\cd(P[n],C_{q})\neq 0$ implies $n\geq k_{M}-1$. This allows us to take
\[Z_{q}\in\Sum(\cp^+[k_{M}-1])
\]
together with a morphism $\beta_{q}:Z_{q}\ra C_{q}$ inducing a surjection
\[\beta_{q}^{\we}:\cd(P[n],Z_{q})\ra \cd(P[n],C_{q})
\]
for each $P\in\cp\ko n\in\Z$. Define $f_{q}$ by the triangle
\[Z_{q}\arr{\alpha_{q}\beta_{q}}P_{q}\arr{f_{q}}P_{q+1}\ra Z_{q}[1]
\]
Since $\pi_{q}\alpha_{q}=0$, there exists $\pi_{q+1}:P_{q+1}\ra M$ such that $\pi_{q+1}f_{q}=\pi_{q}$. Notice that, since
\[Z_{q}[1]\in\Sum(\cp^+[k_{M}]),
\]
then
\[ P_{q+1}\in\Sum(\cp^+[k_{M}])^{*(q+1)}.
\]
Also, the surjectivity required for $\pi^{\we}_{q+1}$ follows from the surjectivity guaranteed for $\pi^{\we}_{q}$. Define $P_{\infty}$ to be the Milnor colimit of the sequence $f_{q}\ko q\geq 0$:
\[\coprod_{q\geq 0}P_{q}\arr{\phi}\coprod_{q\geq 0}P_{q}\arr{\psi} P_{\infty}\ra \coprod_{q\geq 0}P_{q}[1].
\]
Consider the morphism 
\[\theta=\scriptsize{\left[\begin{array}{ccc}\pi_{0}&\pi_{1}&\dots\end{array}\right]}:\coprod_{q\geq 0}P_{q}\ra M.
\] 
Since $\pi_{q+1}f_{q}=\pi_{q}$ for every $q\geq 0$, we have $\theta\phi=0$, which induces a morphism $\pi_{\infty}:P_{\infty}\ra M$ such that $\pi_{\infty}\psi=\theta$. If we prove that $\pi_{\infty}$ induces an isomorphism
\[\pi^{\we}_{\infty}:\cd(P[n],P_{\infty})\arr{\sim}\cd(P[n],M)
\] 
for every $P\in\cp\ko n\in\Z$, then we would have
\[\cd(P[n],\cone(\pi_{\infty}))=0
\]
for every $P\in\cp\ko n\in\Z$, that is to say
\[\cone(\pi_{\infty})\in\Tria(\cp)^{\bot}.
\]
Therefore, we would have proved that $\Tria(\cp)$ is an aisle in $\cd$. Also, if $M\in\Tria(\cp)$, in the triangle
\[P_{\infty}\arr{\pi_{\infty}}M\ra\cone(\pi_{\infty})\ra P_{\infty}[1]
\]
we would have that $P_{\infty}\ko M\in\Tria(\cp)$, which implies
\[\cone(\pi_{\infty})\in\Tria(\cp).
\]
Therefore, $\cone(\pi_{\infty})=0$ and so $\pi_{\infty}$ is an isomorphism. Thus, we would have proved that for every object of $\Tria(\cp)$ there exists an integer $k_{M}$ and a triangle
\[\coprod_{q\geq 0}P_{q}\ra\coprod_{q\geq 0}P_{q}\ra M[-k_{M}]\ra\coprod_{q\geq 0}P_{q}[1]
\]
with $P_{q}\in\Sum(\cp^+)^{*q}\ko q\geq 0$. In particular, we would have that $\Tria(\cp)$ is exhaustively generated to the left by $\cp$.

Let us prove the bijectivity of $\pi^{\we}_{\infty}$. The surjectivity follows from the identity $\pi^{\we}_{\infty}\psi^{\we}=\theta^{\we}$ and the fact that $\theta^{\we}$ is surjective (thanks to the surjectivity of the $\pi_{q}^{\we}\ko q\geq 0$ and the  compactness of the $P\in\cp$). Now consider the commutative diagram
\[\xymatrix{\coprod_{q\geq 0}\cd(P[n],P_{q})\ar[r]^{\phi^{\we}} & \coprod_{q\geq 0}\cd(P[n],P_{q})\ar[r]^{\psi^{\we}}\ar[dr]_{\theta^{\we}} & \cd(P[n],P_{\infty})\ar[r]\ar[d]^{\pi^{\we}_{\infty}} & 0 \\
&& \cd(P[n],M) &
}
\]
The map $\psi^{\we}$ is surjective since the map
\[\phi[1]^{\we}:\coprod_{q\geq 0}\cd(P[n],P_{q}[1])\ra \coprod_{q\geq 0}\cd(P[n],P_{q}[1])
\]
is injective. If we prove that the kernel of $\theta^{\we}$ is contained in the image of $\phi^{\we}$, then we would have the injectivity of $\pi^{\we}_{\infty}$ by an easy diagram chase. Let
\[g=\left[\begin{array}{ccc}g_{0} & g_{1} & \dots\end{array}\right]^{t}:P[n]\ra\coprod_{q\geq 0}P_{q}
\]
such that
\[\left[\begin{array}{ccc}\pi_{0} & \pi_{1} & \dots\end{array}\right]\left[\begin{array}{ccc}g_{0}&g_{1}&\dots\end{array}\right]^{t}=\pi_{0}g_{0}+\pi_{1}g_{1}+\dots=0.
\]
Notice that there exists an $s\geq 0$ such that $g_{s+1}=g_{s+2}=\dots=0$. Then 
\[\pi_{0}g_{0}+\dots +\pi_{s}g_{s}=0
\] 
implies 
\[\pi_{s}(f_{s-1}\dots f_{0}g_{0}+ f_{s-1}\dots f_{1}g_{1}+\dots +g_{s})=0
\] 
and so the morphism
\[f_{s-1}\dots f_{0}g_{0}+ f_{s-1}\dots f_{1}g_{1}+\dots +g_{s}
\] 
factors through $\alpha_{s}$:
\[f_{s-1}\dots f_{0}g_{0}+ f_{s-1}\dots f_{1}g_{1}+\dots +g_{s}=\alpha_{s}\gamma_{s}:P[n]\ra C_{s}\ra P_{s}.
\]
By construction of $Z_{s}$ we have that $\gamma_{s}$ factors through $\beta_{s}$, and so 
\[f_{s-1}\dots f_{0}g_{0}+ f_{s-1}\dots f_{1}g_{1}+\dots +g_{s}=\alpha_{s}\beta_{s}\xi_{s}.
\] 
This implies
\[f_{s}\dots f_{0}g_{0}+ f_{s}\dots f_{1}g_{1}+\dots +f_{s}g_{s}=f_{s}\alpha_{s}\beta_{s}\xi_{s}=0,
\]
since $f_{s}\alpha_{s}\beta_{s}=0$ by construction of $f_{s}$. Therefore, the morphism
\[h:P[n]\ra\coprod_{q\geq 0}P_{q}
\]
with non-vanishing components
\[P[n]\ra P_{r}\ra\coprod_{q\geq 0}P_{q}
\]
induced by
\[g_{r}+\dots + f_{r-1}\dots f_{1}g_{1}+f_{r-1}\dots f_{0}g_{0}: P[n]\ra P_{r}
\]
with $0\leq r\leq s$, satisfies $\varphi^{\we}(h)=g$.
\end{proof}

\begin{corollary}\label{right bounded compactly generated implies exhaustive}
Let $\ca$ be a dg category and let $\cp$ be a set of objects of $\cd^-\ca$ such that:
\begin{enumerate}[1)]
\item it is both right bounded and dually right bounded,
\item its objects are compact in $\Tria(\cp)\cap\cd^-\ca$,
\item $(\cd\ca)(P,P'[i])=0$ for each $P\ko P'\in\cp$ and $i\geq 1$.
\end{enumerate}
Then $\Tria(\cp)\cap\cd^-\ca$ is exhaustively generated to the left by $\cp$.
\end{corollary}
\begin{proof}
Put $\cd:=\Tria(\cp)\cap\cd^-\ca$. Since $\cp$ is right bounded, then $\cp$ is contained in $\cd$ and for each integer $k$ small coproducts of finite extensions of $\Sum(\cp^+[k])$ are in $\cd$. Also, since $\cp$ is right bounded, for each $M\in\cd$ there exists an integer $k_{M}$ such that $\cd(P[n],M)=0$ for each $P\in\cp$ and $n<k_{M}$. Therefore, we can apply Proposition \ref{g implies e to the left}.
\end{proof}

\begin{corollary}\label{parametrization right bounded hn recollements}
Let $\ca$ be a dg category. The following assertions are equivalent:
\begin{enumerate}[1)]
\item $\cd^-\ca$ is a recollement of $\cd^-\cb$ and $\cd^-\cc$, for certain dg categories $\cb$ and $\cc$ with cohomology concentrated in non-positive degrees.
\item There exist sets $\cp\ko \cq$ in $\cd^-\ca$ such that:
\begin{enumerate}[2.1)]
\item $\cp$ and $\cq$ are right bounded.
\item $\cp$ and $\cq$ are dually right bounded.
\item The objects of $\cp$ are compact in $\cd\ca$ and satisfy 
\[(\cd\ca)(P,P'[i])=0
\] 
for all $P\ko P'\in\cp$ and $i\geq 1$.
\item The objects of $\cq$ are compact in $\Tria(\cq)\cap\cd^-\ca$ and satisfy 
\[(\cd\ca)(Q,Q'[i])=0
\] 
for all $Q\ko Q'\in\cq$ and $i\geq 1$.
\item $(\cd\ca)(P[i],Q)=0$ for each $P\in\cp\ko Q\in\cq$ and $i\in\Z$.
\item $\cp\cup\cq$ generates $\cd\ca$.
\end{enumerate}
\end{enumerate}
\end{corollary}
\begin{proof}
$1)\Rightarrow 2)$ Is similar to the corresponding implication in Proposition \ref{parametrization right bounded recollements}. The fact that the dg categories $\cb$ and $\cc$ have cohomology concentrated in non-positive degrees is reflected in the fact that 
\[(\cd\ca)(P,P'[i])=(\cd\ca)(Q,Q'[i])=0
\]
for each $P\ko P'\in\cp\ko Q\ko Q'\in\cq$ and $i\geq 1$.

$2)\Rightarrow 1)$ Thanks to Corollary \ref{right bounded compactly generated implies exhaustive}, conditions 2.3 and 2.4 of Proposition \ref{parametrization right bounded recollements} are satisfied. Therefore, this Proposition (and its proof) ensures that $\cd^-\ca$ is a recollement of $\cd^-\cb$ and $\cd^-\cc$, where $\cb$ is a fibrant replacement of $\cq$ and $\cc$ is a fibrant replacement of $\cp$. Finally, the fact that
\[(\cd\ca)(P,P'[i])=(\cd\ca)(Q,Q'[i])=0
\]
for each $P\ko P'\in\cp\ko Q\ko Q'\in\cq$ and $i\geq 1$ implies that $\cb$ and $\cc$ have cohomology concentrated in non-positive degrees.
\end{proof}

\begin{definition}
Let $A$ be an ordinary algebra. If $M\in\cc A$ is a complex of $A$-modules, the \emph{graded support}\index{graded support} of $M$ is the set of
integers $i\in\Z$ such that $M^i\neq 0$. In case $M$ is a bounded complex, we consider $w(M)=sup\{i\in\Z\mid M^i\neq 0\}-inf\{i\in\Z\mid M^i\neq 0\}+1$ and call it the \emph{width}\index{width}\index{$w(M)$} of $M$. Suppose now that $P\in\cc A$ is a bounded complex of projective $A$-modules, so that $P$ is a dually right bounded object of $\cd^-A$ (\cf Lemma \ref{characterization of dually right bounded}),  and $M\in\cd^-A$ is any object of the right bounded derived category. Unless $M\in\Tria_{\cd A}(P)^\perp$, there is a well-defined integer $k_M=inf\{n\in\Z\mid (\cd A)(P[n],M)\neq 0\}$\index{$k_{M}$}.
\end{definition}

\begin{lemma} \label{route to Konig2}
Let $P$ be a bounded complex of projective $A$-modules such that $(\cd A)(P,P[i])=0$, for all $i>0$, and the canonical
morphism $(\cd A)(P,P[i])^{(\Lambda)}\ra(\cd A)(P,P[i]^{(\Lambda )})$ is an isomorphism, for every integer $i$ and every set $\Lambda$.
Let $M$ be an object of $\Tria_{\cd A}(P)\cap\cd^-A$. There exists a sequence of inflations $0=P_{-1}\ra P_0\ra P_1\ra ...$ in $\cc A$, whose colimit is denoted by $P_{\infty}$, satisfying the following properties:
\begin{enumerate}[1)]
\item $P_\infty$ is isomorphic to $M$ in $\cd A$.
\item $P_n/P_{n-1}$ belongs to $\Sum(\{P\}^+[k_{M}+n])$, for each $n\geq 0$. 
\item If $n\geq w(P)-k_M$ then the graded supports of $P$ and $P_\infty/P_n$ are disjoint.
\end{enumerate}
\end{lemma}
\begin{proof}
Imitating the proof of Proposition \ref{g implies e to the left}, we shall construct a filtration satisfying conditions 2) and 3), leaving for the last moment the verification of condition 1). 

\emph{First step: condition 2).} Note that in the proof of that proposition, we start with $P_0\in\Sum(P[i]:$ $i\geq k_M)$ and then, at each step, $P_{q+1}$ appears in a triangle
\[Z_q\arr{\alpha_q\beta_q}P_q\arr{f_q}P_{q+1}\ra Z_q[1],
\]
where $Z_q$ is a coproduct of shifts $P[i]$, with $i\geq k_{M}-1$. Working in $\cc A$ and bearing in mind that $Z_q$ is $\ch$-projective (it is a right bounded complex of projective $A$-modules), we can assume without loss of generality that $f_q$ is the mapping cone of a chain map $Z_q\ra P_q$ and, as a consequence, that $f_q$ is an inflation in $\cc A$ appearing in a conflation
\[P_q\arr{f_q}P_{q+1}\ra Z_q[1], 
\]
where $Z_q[1]$ is a coproduct in $\cc A$ of shifts $P[i]$, $i\geq k_M$. We shall prove by induction on $q\geq 0$ that one can choose 
$Z_q[1]\in\Sum(\{P\}^+[q+1+k_{M}])$ or, equivalently, that $Z_q\in\Sum(\{P\}^+[q+k_{M}])$. Since $Z_q$ is defined via a $\Sum(\{P[i]\}_{i\in\Z})$-precover
$Z_q\arr{\beta_q}C_q$, our task reduces to prove that $(\cd A)(P[i],C_q)\neq 0$ implies $i\geq q+k_M$. We leave as an exercise checking that for $q=0$. Provided it is true for $q-1$, we apply the homological functor $(\cd A)(P[i],?)$ to the triangle
\[Z_{q-1}\arr{\beta_{q-1}}C_{q-1}\arr{u_{q-1}}C_q\ra Z_{q-1}[1]
\]
and, bearing in mind that $(\cd A)(P[i],Z_{q-1})\ra(\cd A)(P[i],C_{q-1})$ is surjective, we get that $(\cd A)(P[i],C_{q})\ra(\cd A)(P[i],Z_{q-1}[1])$ is injective. As a consequence, the inequality $(\cd A)(P[i],C_{q})\neq 0$ implies that $(\cd A)(P[i],Z_{q-1}[1])\neq 0$ and the induction hypothesis guarantees that $Z_{q-1}$ is a coproduct of shifts $P[j]$, with $j\geq q-1+k_M$. Then $(\cd A)(P[i],C_{q})\neq 0$ implies that $0\neq(\cd A)(P[i],P[j+1])=(\cd A)(P,P[j+1-i])$,
for some $j\geq q-1+k_M$. Then $i\geq q+k_M$ as desired. In conclusion, we can view the map $f_q:P_q\ra P_{q+1}$ as an inflation in $\cc A$ whose cokernel is isomorphic in $\cc A$ to a coproduct of shifts $P[i]$, with $i\geq q+1+k_M$. 


\emph{Second step: condition 3).} If now $n\geq 0$ is any natural number, then $P_\infty/P_n$ admits a filtration
\[0=P_n/P_n\ra P_{n+1}/P_n\ra ... 
\]
in $\cc A$, where the quotient of two consecutive factors is a coproduct of shifts $P[i]$, with $i\geq n+k_M$. If $n\geq w(P)-k_M$, then any such index $i$ satisfies $i\geq w(P)$ and then the graded supports of $P$ and $P[i]$ are disjoint. As a result the graded supports of $P$ and $P_\infty/P_n$ are disjoint
whenever $n\geq w(P)-k_M$.

\emph{Third step: condition 1).} Finally, in order to prove condition 1), notice that the argument in the final part of the proof of Proposition \ref{g implies e to the left} can be repeated, as soon as we are able to prove that the canonical morphism $\coprod_{n\geq 0}(\cd A)(P[i],P_n)\ra(\cd A)(P[i],\coprod_{n\geq 0}P_n)$ is an isomorphism, for every integer $i\in\Z$. It is not difficult to reduce that to the case in which $i=0$. For that we fix $n\geq w(P)-k_M$ large enough so that also the graded supports of $P[1]$ and $P_\infty /P_n$ are disjoint. Then we get a conflation in $\cc A$
\[\left(\coprod_{k\leq n}P_k\right)\oplus\left(\coprod_{k>n} P_n \right)\ra\coprod_{k\geq 0}P_k\ra\coprod_{k>n}P_k/P_n.
\]
That conflation of $\cc A$ gives rise to the corresponding triangle of $\cd A$. But the right term in the above conflation has a graded support which is
disjoint with those of $P$ and  $P[1]$. That implies that 
\[(\cd A)(P,\coprod_{k>n}P_k/P_n)=0=(\cd A)(P,\coprod_{k>n}P_k/P_n[-1])
\]
and also 
\[\coprod_{k>n}(\cd A)(P,P_k/P_n)=0=\coprod_{k>n}(\cd A)(P,P_k/P_n[-1]).
\] 
We then get a commutative diagram with horizontal isomorphisms:
\[\xymatrix{\left(\coprod_{k\leq n}(\cd A)(P,P_{k})\right)\oplus\left(\coprod_{k>n}(\cd A)(P,P_{n})\right)\ar[d]^{\can}\ar[rr]^{\ \ \ \ \ \ \ \ \ \sim} && \coprod_{k\geq 0}(\cd A)(P,P_{k})\ar[d]^{\can} \\
\left((\cd A)(P,\coprod_{k\leq n}P_{k})\right)\oplus\left((\cd A)(P,\coprod_{k>n}P_{n})\right)\ar[rr]^{\ \ \ \ \ \ \ \ \ \sim} && (\cd A)(P,\coprod_{k\geq 0}P_{k})
}
\]
The proof will be finished if we are able to prove, for any fixed natural number $n$, that $(\cd A)(P[i],?)$ preserves small coproducts of objects in $\Sum(\{P\}^+)^{*n}$ for every $i\in\Z$. Let us prove it. From the hypotheses on $P$ and the fact if $i>w(P)$ then $(\cd A)(P,P[i]^{(\Lambda )})=0$ for every set $\Lambda$, one readily sees that, for every integer $m$ and every family of exponent sets $(\Lambda_i)_{i\geq m}$, the canonical morphism
$\coprod_{i\geq m}(\cd A)(P,P[i])^{(\Lambda_i)}\ra(\cd A)(P,\coprod_{i\geq m}P[i]^{(\Lambda_i)})$ is an isomorphism. Our goal is then attained for $n=0$ and an easy induction argument gets the job done for every $n\geq 0$.
\end{proof}

Now S.~K\"{o}nig's theorem \cite[Theorem 1]{Konig1991} `follows' from our results.

\begin{theorem} \label{Konig}
Let $A$, $B$ and $C$ be ordinary algebras. The following assertions are equivalent: 
\begin{enumerate}[1)]
\item $\cd^-A$ is a recollement of $\cd^-C$ and $\cd^-B$. 
\item There are two objects $P\ko Q\in\cd^-A$ satisfying the following properties:
\begin{enumerate}[2.1)]
\item There are isomorphisms of algebras $C\cong(\cd A)(P,P)$ and $B\cong(\cd A)(Q,Q)$. 
\item $P$ is exceptional and isomorphic in $\cd A$ to a bounded complex of finitely generated projective $A$-modules. 
\item $(\cd A)(Q,Q^{(\Lambda )}[i])=0$ for every set $\Lambda$ and every $i\in\Z\setminus\{0\}$, the canonical map $(\cd A)(Q,Q)^{(\Lambda)}\ra(\cd A)(Q,Q^{(\Lambda )})$ is an isomorphism, and $Q$ is isomorphic in $\cd A$ to a bounded complex of projective $A$-modules. 
\item $(\cd A)(P,Q[i])=0$ for all $i\in\Z$. 
\item $P\oplus Q$ generates $\cd A$.
\end{enumerate}
\end{enumerate}
\end{theorem}
\begin{proof}
$1)\Rightarrow 2)$ is a particular case of the proof of the corresponding implication in Corollary \ref{parametrization right bounded hn recollements}, where we take into account the additional consideration that the dg categories are in this case ordinary algebras, whence having cohomology concentrated in degree zero.

$2)\Rightarrow 1)$ Taking $\cp=\{P\}$ and $\cq=\{Q\}$, one readily sees that these one-point sets satisfy conditions 2.1, 2.2, 2.3, 2.5 and 2.6 of Corollary \ref{parametrization right bounded hn recollements}. As for condition 2.4 it only remains to prove that $Q$ is compact in $\Tria_{\cd A}(Q)\cap\cd^-A$. For this,
let $(M_{j})_{j\in J}$ be a family of objects in $\Tria(Q)\cap\cd^-A$ having a coproduct, say $M$,  in that subcategory and denote by $q_j:M_{j}\ra M$ the injections. Of course, we have that $sup\{i\in\Z\mid \H i(M_{j})\neq 0\}\leq sup\{i\in\Z\mid \H i(M)\neq 0\}$, for every $j\in J$. Then
the coproduct $\coprod_{j\in J} M_{j}$ of the family in $\cd A$ belongs to $\cd^-A$ and thus to $\Tria(Q)\cap\cd^-A$. This easily implies that 
$M\cong\coprod_{j\in J}M_{j}$ and the injection $q_j:M_{j}\ra M$ gets identified with the canonical injection $M_{j}\ra\coprod_{k\in J}M_{k}$.
For each $j\in J$ we consider the complex $Q_{j,\infty}$ and the filtration
\[0=Q_{j,-1}\ra Q_{j,0}\ra Q_{j,1}\ra ...
\]
given by Lemma \ref{route to Konig2} for $M_{j}$, where we have replaced the letter ``P'' by the letter ``Q'' to avoid confusion with the
object $P$. Notice that $k_{M}\leq k_{M_{j}}$ for every $j\in J$. Therefore, the integer $r:=inf\{k_{M_{j}}\}_{j\in J}$ is well defined. If we fix $n\in\mathbf{N}$ such that $n+r>w(Q)$, then $n>w(Q)-k_{M_{j}}$. Notice that \cite[Lemma 5.3]{Keller1990} implies that a countable composition of inflations of $\cc A$ is again an inflation of $\cc A$. Then, for every $j\in J$ we get a conflation in $\cc A$,
\[Q_{j,n}\ra Q_{j,\infty}\ra Q_{j,\infty}/^jQ_n,
\]
By Lemma \ref{route to Konig2}, the right term of this conflation has a graded support which is disjoint with that of $Q$ and $Q[1]$ (enlarging $n$ if necessary). Then we get a commutative diagram:
\[\xymatrix{\coprod_{J}(\cd A)(Q,Q_{j,n})\ar[d]^{\can}\ar[r]^{\sim} & \coprod_{J}(\cd A)(Q,Q_{j,\infty})\ar[d]^{\can}\ar[r]^{\sim} & \coprod_{J}(\cd A)(Q,M_{j})\ar[d]^{\can} \\
(\cd A)(Q,\coprod_{J}Q_{j,n})\ar[r]^{\sim} & (\cd A)(Q,\coprod_{J}Q_{j,\infty})\ar[r]^{\sim} & (\cd A)(Q,\coprod_{J}M_{j})
}
\]
The fact that the leftmost vertical map is a bijection has been proved in the third step of the proof of Lemma \ref{route to Konig2}, and so we are done.
\end{proof}

\begin{remark}
The reader will have noticed that we changed S.~K\"{o}nig's condition that $Q$ is exceptional for the stronger condition that $(\cd A)(Q,Q[i]^{(\Lambda )})=0$ $(*)$, for all $i\neq 0$ and all sets $\Lambda$. We have not been able to derive this condition from S.~K\"{o}nig's hypotheses. In principle, there is no prohibition for having a module $Q$ over a hereditary algebra $A$ such that $\Ext_A^1(Q,Q)=0$ but $\Ext_A^1(Q,Q^{(\Lambda )})\neq 0$, for some infinite set $\Lambda$ (see Lemma \ref{TrlifajLemma} and Example \ref{TrlifajExample} below). However, S.~K\"{o}nig's theorem implies, together with our Theorem \ref{Konig}, that equality $(*)$ must hold for the particular $Q$ there. 
\end{remark}

The following lemma and example have been suggested to us by J. Trlifaj.

\begin{lemma}\label{TrlifajLemma}
Let $A$ be a countable simple Von Neumann regular ring which is not right Noetherian. Then there exists an injective right $A$-module $Q$ such that:
\begin{enumerate}[1)]
\item The functor $\Hom_A(Q,?):\Mod A\ra\Mod\Z$ preserves small coproducts. 
\item $\Ext_A^1(Q,Q^{(\Lambda )})\neq 0$, for every infinite set $\Lambda$.
\end{enumerate}
\end{lemma}
\begin{proof}
1) Since $A$ is countable its pure global dimension on either side is smaller or equal than $1$ (\cf \cite[Th\`{e}orem 7.10]{GrusonJensen}) and, since $A$ is Von Neumann regular, we conclude that $A$ is hereditary on both sides (\cf \cite[Proposition 10.3]{GrusonJensen}). On the other hand, since every finitely generated right ideal is generated by an idempotent element, one easily constructs an infinite sequence of nonzero orthogonal idempotents $e_n\ko n\in\mathbf{N}$. We then take $I=\bigoplus_{n\geq 0}e_nA$ and take the injective envelope $Q:=E(A/I)$ of $A/I$. Let now $X_i\ko i\in I$, be an infinite family of nonzero right $A$-modules and suppose that the canonical map $\coprod_{i\in I}\Hom_A(Q,X_i)\ra \Hom_A(Q,\coprod_{i\in I}X_i)$ is not bijective and, hence, not surjective. That means that we have a
morphism $f:Q\ra\coprod_{i\in I}X_i$ such that the  composition $f_j:Q\arr{f}\coprod_{i\in I}X_i\arr{\pi_{j}}X_j$ is nonzero, for infinitely many $j\in I$. Selecting an infinite countable subset of the set formed by those $j\in J$ such that $f_j\neq 0$, we see that it is not restrictive to assume that $I=\mathbf{N}$ and the composition 
$f_j:Q\arr{f}\coprod_{i\in\mathbf{N}}X_i\arr{\pi_{j}}X_j$ is nonzero for all $j\in\mathbf{N}$. Put now $M_n:=f^{-1}(\coprod_{0\leq i\leq n}X_i)\ko n\in\mathbf{N}$. Notice that the  sequence $M_0\subseteq M_1\subseteq ...\subseteq M_n\subseteq ...$ is strictly increasing and $Q=\bigcup_{n\geq 0}M_n$. Due to the fact that $A$ is a simple ring, we have that $Ae_nA=A$, for all $n\in\mathbf{N}$. From that we derive the existence, for each $n\in\mathbf{N}$, of an element $x_n\in M_ne_n\setminus M_{n-1}$ (convening that $M_{-1}=0$). Now we get a morphism
$h:I=\bigoplus_{n\geq 0}e_nA\ra Q$ by the rule $h(\sum_{n\geq 0}a_n)=\sum_{n\geq 0}x_na_n$. The fact that $Q$ is injective implies that $h$ is given by multiplication on the left by an element $x\in Q$. In particular $x_n=x_ne_n=h(e_n)=xe_n$. But from the equality $Q=\bigcup_{n\in\mathbf{N}}M_n$ we derive that $x\in M_r$, for some $r\in\mathbf{N}$, which implies that $\im(h)\subseteq xA\subseteq M_r$. But then $x_n\in M_r$, for all $n\in\mathbf{N}$, and in particular $x_{r+1}\in M_r$, which contradicts the choice of the $x_{r+1}$. Therefore the canonical map $\coprod_{i\in I}\Hom_A(Q,X_i)\ra\Hom_A(Q,\coprod_{i\in I}X_i)$ is bijective, for every set $I$.

2) \emph{First step: $|Q|=|Qe_n|$, for every $n\in\mathbf{N}$, where $|?|$ denotes the cardinality.} Indeed, since $A=Ae_nA$ it follows that $Q=Qe_nA=\sum_{a\in A}Qe_na$ (sum as abelian groups). Since $Q$ is infinite and not countable (see for instance \cite{DinhGuilLopez}) and $A$ is countable, one readily sees that $|Q|=|Qe_na|$, for some $a\in A$. But the map $Qe_n\ra Qe_na\ko z\mapsto za$, is surjective, so that we have
$|Q|=|Qe_na|\leq |Qe_n|\leq |Q|$, for some $a\in A$. Then $|Q|=|Qe_n|$ as desired. 

\emph{Second step: $\Ext_A^1(A/I,Q^{(\mathbf{N})})\neq 0$.} By applying the contravariant functor $\Hom_A(?,Q^{(\mathbf{N})})$ to the projective presentation  $0\ra I\hookrightarrow A\ra A/I\ra 0$, we get an exact sequence:
\[0\ra\Hom_A(A/I,Q^{(\mathbf{N})})\ra\Hom_A(A,Q^{(\mathbf{N})})\ra\Hom_A(I,Q^{(\mathbf{N})})\ra
\]
\[\ra\Ext^1_A(A/I,Q^{(\mathbf{N})})\ra 0.
\]
Using the obvious isomorphisms and bearing in mind that 
\[\Hom_A(I,Q^{(\mathbf{N})})=\Hom_A(\bigoplus_{n\geq 0}e_nA,Q^{(\mathbf{N})})\cong\prod_{n\geq 0}\Hom_A(e_nA,Q^{(\mathbf{N})})\cong\prod_{n\geq 0}Q^{(\mathbf{N})}e_n,
\] 
the above exact sequence of abelian groups gets identified with the sequence
\[0\ra\op{ann}_{Q^{(\mathbf{N})}}(I)\hookrightarrow Q^{(\mathbf{N})}\arr{\varphi}\prod_{n\geq 0}Q^{(\mathbf{N})}e_n\ra\Ext^1_A(A/I,Q^{(\mathbf{N})})\ra 0,
\]
where $\varphi (x)=(xe_n)_{n\in\mathbf{N}}$ for every $x\in Q^{(\mathbf{N})}$. Since 
$|Q^{(\mathbf{N})}|=|Q|<|Q^\mathbf{N}|=|(Q^{(\mathbf{N})})^\mathbf{N}|=|\prod_nQ^{(\mathbf{N})}e_n|$ it follows that the map $\varphi$ cannot be surjective, and hence $\Ext^1_A(A/I,Q^{(\mathbf{N})})\neq 0$.

\emph{Third step: $\Ext^1_A(Q,Q^{(\Lambda )})\neq 0$ for every infinite set $\Lambda$.} It suffices to prove that $\Ext^1_A(Q,Q^{(\mathbf{N})})\neq 0$. Bearing in mind that $A$ is right hereditary and $Q^{(\mathbf{N})}$ is not injective, we have an exact sequence of right $A$-modules
\[0\ra Q^{(\mathbf{N})}\hookrightarrow E(Q^{(\mathbf{N})})\ra E'\ra 0,
\]
with $0\neq E'$  injective. We then get the following commutative diagram with exact rows, where the vertical arrows are induced by 
the canonical inclusion $A/I\hookrightarrow Q$:
\[\xymatrix{\Hom_A(Q,E')\ar[r]\ar[d] & \Ext_A^1(Q,Q^{(\mathbf{N})})\ar[r]\ar[d] & 0 \\
\Hom_A(A/I,E')\ar[r]& \Ext_A^1(A/I,Q^{(\mathbf{N})})\ar[r] & 0.
}
\]
The left vertical arrow is an epimorphism, due to the fact that $E'$ is injective, and the lower horizontal arrow is an epimorphism by definition of 
$\Ext_A^1(A/I,?)$. It follows that the right vertical arrow is also an epimorphism and, hence, that $\Ext^1_A(Q,Q^{(\mathbf{N})})\neq 0$ as desired.
\end{proof}

\begin{example}\label{TrlifajExample}
Any countable direct limit of countable simple Artinian rings
satisfies the hypotheses of the above example. A typical case is
given as follows. Consider the direct limit $A=\lid \cm_{2^n\times 2^n}(\mathbb{K})$, where $\mathbb{K}$ is a countable field and the ring morphism
$\cm_{2^n\times 2^n}(\mathbb{K})\ra \cm_{2^{n+1}\times 2^{n+1}}(\mathbb{K})$ maps the matrix $U$ onto the matrix given by the block decomposition 
$\left[\begin{array}{cc} U & 0\\ 0 & U \end{array}\right]$.
\end{example}

\printindex
\addcontentsline{lot}{chapter}{\'Indice terminol\'ogico}

\end{document}